\documentclass[12pt]{amsart}
\usepackage{amsmath,amsthm,eufrak,mathrsfs,amssymb,dsfont,graphicx,bbm,multicol,blkarray,nicematrix,multirow,color,soul,calc,tikz-cd,scalerel,url,setspace,kbordermatrix,scalerel,mathtools,cancel,bm}
\usepackage[all]{xy}
\usepackage[margin=1in]{geometry}
\usepackage[hidelinks]{hyperref}
\usepackage[percent]{overpic}

\author[Manion]{Andrew Manion}
\address{North Carolina State University \\ Department of Mathematics \\ 2108 SAS Hall \\ Box 8205 \\ Raleigh, NC 27695}
\email{ajmanion@ncsu.edu}

\author[Rutter]{Elijah Rutter}
\address{North Carolina State University \\ Department of Mathematics \\ 2108 SAS Hall \\ Box 8205 \\ Raleigh, NC 27695}
\email{esrutter@ncsu.edu}

\newtheorem{theorem}{Theorem}[section]
\newtheorem{lemma}[theorem]{Lemma}
\newtheorem{proposition}[theorem]{Proposition}
\newtheorem{corollary}[theorem]{Corollary}
\theoremstyle{definition} 
\newtheorem{definition}[theorem]{Definition}
\newtheorem{example}[theorem]{Example}
\newtheorem{remark}[theorem]{Remark}

\allowdisplaybreaks
\newcommand{\A}{\mathcal{A}}
\newcommand{\BSDA}{\widehat{\mathrm{BSDA}}}
\newcommand{\BSD}{\widehat{\mathrm{BSD}}}
\newcommand{\BSDD}{\widehat{\mathrm{BSDD}}}
\newcommand{\CFDA}{\widehat{\mathrm{CFDA}}}
\newcommand{\Dc}{\mathcal{D}}

\newcommand{\R}{\mathbb{R}}
\newcommand{\C}{\mathbb{C}}
\newcommand{\Q}{\mathbb{Q}}

\newcommand{\Z}{\mathbb{Z}}
\newcommand{\F}{\mathbb{F}}
\newcommand{\s}{\mathfrak{s}}

\newcommand{\Zc}{\mathcal{Z}}

\newcommand{\Hc}{\mathcal{H}}
\newcommand{\Vc}{\mathcal{V}}
\newcommand{\gr}{\mathrm{-gr}}
\newcommand{\x}{\mathbf{x}}
\newcommand{\Zb}{\mathbf{Z}}
\newcommand{\boldotimes}{\mathbin{\pmb{\otimes}}}

\newcommand{\rank}{\operatorname{rank}}
\newcommand{\im}{\operatorname{im}}

\newcommand{\Spinc}{\mathrm{Spin}^c}

\DeclareMathOperator{\Ab}{Ab}
\DeclareMathOperator{\cell}{cell}
\DeclareMathOperator{\Cob}{Cob}
\DeclareMathOperator{\comb}{comb}
\DeclareMathOperator{\DA}{DA}
\DeclareMathOperator{\DD}{DD}
\DeclareMathOperator{\Eul}{Eul}
\DeclareMathOperator{\FN}{FN}
\DeclareMathOperator{\Lift}{Lift}
\DeclareMathOperator{\id}{id}
\DeclareMathOperator{\std}{std}
\DeclareMathOperator{\inrm}{in}
\DeclareMathOperator{\nb}{nb}

\DeclareMathOperator{\out}{out}

\DeclareMathOperator{\rel}{rel}
\DeclareMathOperator{\SFH}{SFH}
\DeclareMathOperator{\sing}{sing}

\DeclareMathOperator{\Sut}{Sut}
\DeclareMathOperator{\sut}{sut}

\DeclareMathOperator{\tors}{tors}
\DeclareMathOperator{\norm}{norm}

\title{A symmetric monoidal Frohman--Nicas TQFT for sutured manifolds}

\begin{document}
\begin{abstract}
By analyzing the decategorification of bordered sutured Heegaard Floer homology, we reinterpret and generalize the classical Frohman--Nicas TQFT for the Alexander polynomial in the setting of 3d sutured cobordisms between sutured surfaces. In this setting, the Frohman--Nicas TQFT maps for arbitrary cobordisms between surfaces, with no connectivity restrictions, get interpreted as part of an honest symmetric monoidal functor (under disjoint union) with no half-projectivity zeroes. We also relate the decategorified bordered sutured theory with $\Spinc$ structures to a sutured version of Florens--Massuyeau's $G$-analogue of the Frohman--Nicas TQFT.
\end{abstract}
\maketitle

\tableofcontents

\section{Introduction}

Heegaard Floer homology, introduced by Ozsv{\'a}th and Szab{\'o} \cite{Ozsvath:aa,Ozsvath:ab}, is an invariant of closed 3-manifolds and 4-dimensional cobordisms. It satisfies gluing properties reminiscent of a $(3+1)$-dimensional TQFT in the sense of Atiyah \cite{AtiyahTQFT}; the modern perspective on Atiyah's TQFT axioms defines a TQFT to be a symmetric monoidal functor from a category of closed $n$-manifolds and $(n+1)$-dimensional cobordisms to some other symmetric monoidal category (see e.g. \cite[Section 1.1]{LurieCobordism}). However, it is known that Heegaard Floer homology does not literally satisfy these TQFT axioms.

Bordered Heegaard Floer homology \cite{LOT} and sutured Floer homology \cite{Juhasz} offer two distinct approaches to extending Heegaard Floer homology to 3-manifolds with boundary. Bordered sutured Floer homology, developed in \cite{Zarev}, unifies these approaches, recovering each as a special case. One can study the ``decategorification'' of all of these theories, removing some of the technical complexity and hopefully enabling a more complete understanding of their TQFT structure. From general patterns, one could hope that decategorifications of these Heegaard Floer theories involving 2d surfaces and 3-manifolds with boundary form something like a (2+1)-dimensional Atiyah-style TQFT.

The decategorification of bordered Heegaard Floer homology was studied by Petkova \cite{Pet18} and Hom--Lidman--Watson \cite{HLW}; in particular, Hom--Lidman--Watson show that decategorified bordered Heegaard Floer homology recovers a subset of a classical (2+1)-dimensional TQFT framework for the Alexander polynomial, the Frohman--Nicas TQFT \cite{FN} (see also \cite{don99,KerlerHomologyTQFT,FMFunctorial}). However, the Frohman--Nicas TQFT does not literally satisfy the TQFT axioms either; when formulated for general 3d cobordisms whose inputs and outputs can be disconnected, it has a non-standard ``half-projective'' functoriality property that is analyzed by Kerler in \cite{KerlerHomologyTQFT}. 

In a different direction, the decategorification of sutured Floer homology was studied by Friedl--Juhasz--Rasmussen \cite{FJR}, but this theory on its own does not incorporate gluing 3-manifolds along their boundary surfaces as in a (2+1)-dimensional TQFT. Such gluing in the context of sutured Floer homology is the topic of Zarev's bordered sutured theory \cite{Zarev}. The decategorification of bordered sutured invariants for 2-manifolds has been studied by the first author \cite{man22,man23}, but 3-manifolds are not considered in these papers. In this paper we study the decategorification of bordered sutured invariants in dimension 3, generalizing results of \cite{Pet18,HLW} to the bordered sutured setting.

As a result, we show that the Frohman--Nicas TQFT maps for arbitrary cobordisms between closed surfaces with no connectedness restrictions can be reinterpreted in the setting of 3d sutured cobordisms between 2d sutured surfaces, fixing the half-projective functoriality. The sutured Frohman--Nicas TQFT that we define and relate to the 3-dimensional decategorified bordered sutured invariants is a symmetric monoidal functor into a category where morphisms are linear maps modulo sign, i.e. it is ``projective'' but not ``half-projective.'' The main difference from the usual Atiyah setting is that our domain category of sutured surfaces and sutured cobordisms between them is different from the usual category of closed surfaces and cobordisms between them. Neither cobordism category is a subcategory of the other; in particular, closed surfaces are not themselves objects of our category. Objects and morphisms in the usual ``closed'' cobordism category (with no connectedness restrictions) give rise to objects and morphisms in our sutured cobordism category, by removing neighborhoods of points from the surfaces and neighborhoods of acyclic ribbon graphs from the cobordisms, but composition of these new sutured cobordisms disagrees with composition of the original cobordisms precisely when the half-projectivity zero appears in the usual Frohman--Nicas TQFT. See Remark~\ref{rem:FixingHalfProj} below.

\begin{remark}
    We note that other symmetric monoidal versions of the Frohman--Nicas TQFT have appeared previously in the literature; in particular, the versions defined by Frohman--Nicas in \cite[Section 4]{FN} on categories $\mathcal{SC}^+$ and $\mathcal{SC}^-$ of monotone increasing and monotone decreasing cobordisms are symmetric monoidal with disjoint union as the symmetric monoidal category on the domain. These restricted cobordism categories $\mathcal{SC}^{\pm}$ are chosen so that they do not have compositions in which the half-projectivity zero would appear. However, they have a corresponding limitation: in the usual (oriented) cobordism category, the dual of an object $X$ is the orientation reversal $-X$ of $X$, with evaluation and coevaluation morphisms for the duality given by $X \times [0,1]$ viewed as a cobordism from $X \sqcup -X$ to $\emptyset$ and from $\emptyset$ to $-X \sqcup X$. This duality is often the starting point for reasoning about the basic properties of Atiyah-style TQFTs, e.g. as in \cite[Section 1.1]{LurieCobordism}. For any object $X$ of $\mathcal{SC}^{\pm}$, the evaluation cobordism is monotone decreasing and thus in $\mathcal{SC}^-$, while the coevaluation cobordism is monotone increasing; thus, neither $\mathcal{SC}^+$ nor $\mathcal{SC}^-$ contain both of the usual cobordisms witnessing the duality between $X$ and $-X$. In our version, by contrast, for a sutured surface $(F,\Lambda)$, the sutured cobordism $(F,\Lambda) \times [0,1]$ does give valid evaluation and coevaluation cobordisms witnessing a duality between $(F,\Lambda)$ and $-(F,\Lambda)$; see Remark~\ref{rem:Duals}.

    We also note that when the Frohman--Nicas TQFT is formulated on the Crane--Yetter category $\mathrm{Cob}$ where the objects are connected surfaces with one boundary component and the morphisms are relative cobordisms, as in Kerler \cite{KerlerHomologyTQFT} and the $G=1$ case of Florens--Massuyeau \cite{FMFunctorial}, then it gives a braided monoidal functor where the domain category is braided (but not symmetric) monoidal under boundary connected sum and admits duals of objects given by a relative analogue of the usual $X \times [0,1]$ construction; see e.g. \cite[Figure 5.2]{CHM} for an example of a coevaluation cobordism in this setting. 
\end{remark} 

We also bring $\Spinc$ structures into the analysis of the decategorified bordered theory for the first time; we relate the decategorified bordered sutured invariant with $\Spinc$ structures to a sutured version of Florens--Massuyeau's Alexander functor \cite{FMFunctorial}. This allows, for example, the Alexander polynomial of a knot in $S^3$ to be directly visible from the decategorified bordered sutured invariant of the knot complement.

\vspace{0.15in}

\noindent \textbf{Frohman--Nicas TQFT.} We begin with a review of the Frohman--Nicas TQFT, which we will call $\Vc^{\FN}$.

\begin{definition}\label{def:FNTQFT}
Let $\Sigma$ be a closed oriented surface and let $g$ be the sum of the genera of all components of $\Sigma$. To $\Sigma$, the \emph{Frohman--Nicas TQFT (with $\Z$ coefficients)} assigns $\Vc^{\FN}(\Sigma) := \wedge^* H_1(\Sigma)$, a free abelian group of rank $2^{2g}$. 

Let $W$ be an oriented cobordism from $\Sigma_0$ to $\Sigma_1$; by convention $\partial W$ is identified with $\Sigma_1 \sqcup -\Sigma_0$. To $W$, the Frohman--Nicas TQFT assigns a $\Z$-linear map from $\wedge^* H_1(\Sigma_0)$ to $\wedge^* H_1(\Sigma_1)$, defined up to sign, as follows.
\begin{itemize}
    \item If $i_* \colon H_1(\partial W;\Q) \to H_1(W;\Q)$ is not surjective (i.e. $W$ is not ``rationally homologically trivial''), $\Vc^{\FN}(W)$ is defined to be zero.

    \item If $W$ is rationally homologically trivial, consider the up-to-sign element
    \begin{align*}
    |K_{W}| &:= \pm |\tors(H_1(W))| \wedge^{\mathrm{top}} \mathrm{ker}(i_* \colon H_1(\partial W) \to H_1(W)) \\
    &\in \wedge^* H_1(\Sigma_0) \otimes \wedge^* H_1(\Sigma_1);
    \end{align*}
    note that $\ker i_*$ has rank $g_0 + g_1$, so $\wedge^{\mathrm{top}} = \wedge^{g_0 + g_1}$ in the above expression. The up-to-sign map $\Vc^{\FN}(W)$ is defined to be the composition
    \[
    \wedge^* H_1(\Sigma_0) \xrightarrow{\id \otimes |K_{W}|} \wedge^* H_1(\Sigma_0) \otimes \wedge^* H_1(\Sigma_0) \otimes \wedge^* H_1(\Sigma_1) \xrightarrow{\varepsilon \otimes \id} \wedge^* H_1(\Sigma_1),
    \]
    where $\varepsilon \colon \wedge^* H_1(\Sigma_0) \otimes \wedge^* H_1(\Sigma_0) \to \Z$ sends $\alpha \otimes \beta$ to zero if $\alpha \wedge \beta$ is not a top-degree element of $\wedge^* H_1(\Sigma_0)$, and otherwise sends $\alpha \otimes \beta$ to the coefficient of $\alpha \wedge \beta$ on an arbitrarily chosen basis element $\Omega$ of $\wedge^{\mathrm{top}} H_1(\Sigma_0) \cong \Z$.
\end{itemize}
\end{definition}

By \cite[Lemma 2]{KerlerHomologyTQFT}, the Frohman--Nicas TQFT satisfies the following ``half-projective'' functoriality property when composing cobordisms $W$ and $W'$:
\[
\Vc^{\FN}(W' \circ W) = \pm 0^{\mu(W',W)} \Vc^{\FN}(W') \circ \Vc^{\FN}(W),
\]
where $0^0 := 1$ and $\mu(W',W) \in \Z_{\geq 0}$ is defined by letting $b(W)$ for a cobordism $W$ be the number of components of $W$ minus half the number of components of $\partial W$ and defining $\mu(W', W) := b(W' \circ W) - b(W') - b(W)$. See Figure~\ref{fig:HalfProjectivity} for a prototypical example of when $\mu(W',W)$ is nonzero; in this case we have $b(W) = b(W') = -1/2$ while $b(W' \circ W) = 0$, so $\mu(W',W) = 1$.

\begin{figure}
    \centering
    \begin{overpic}[width=0.7\textwidth]{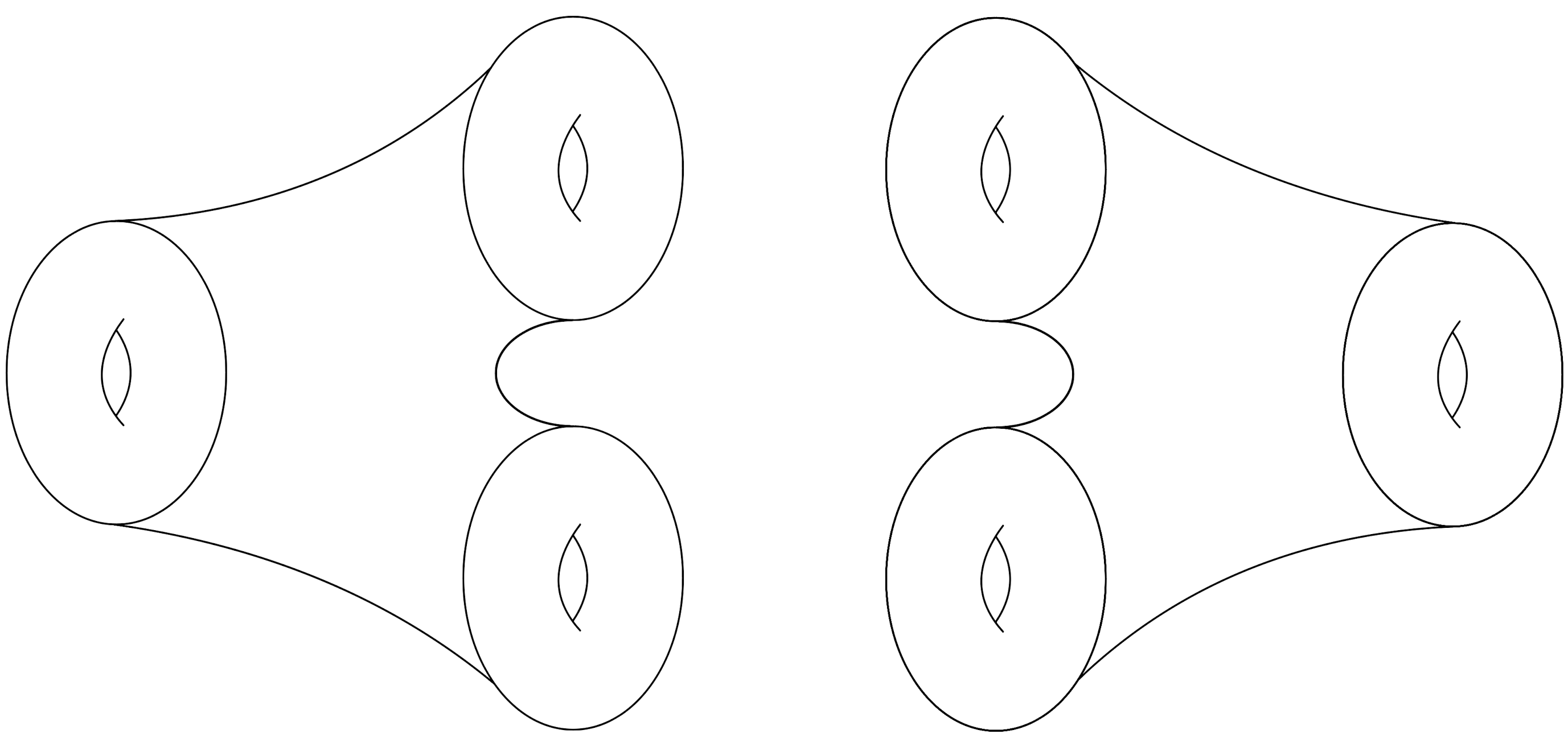}
        \put (18,-4) {$W'$}
        \put (78,-4) {$W$}
    \end{overpic}
    \vspace{0.22in}
    \caption{A prototypical example where the half-projectivity zero appears in the Frohman--Nicas TQFT.}
    \label{fig:HalfProjectivity}
\end{figure}

The relationship in the literature between $\Vc^{\FN}$ and decategorified bordered Heegaard Floer homology holds for cobordisms $W \colon \Sigma_0 \to \Sigma_1$ with $W$ and $\Sigma_i$ connected. In this case there is a bordered Heegaard Floer invariant $\CFDA(W)$ which is a certain type of bimodule. It follows from results of Hom--Lidman--Watson \cite{HLW} that the ``decategorification'' $[\CFDA(W)]$ of $\CFDA(W)$, an up-to-sign map from $\wedge^* H_1(\Sigma_0)$ to $\wedge^* H_1(\Sigma_1)$, agrees with $\Vc^{\FN}(W)$.

Following Petkova and Hom--Lidman--Watson, there are two main steps to establishing the relationship between decategorified bordered Heegaard Floer homology and the Frohman--Nicas TQFT. The first step is to give a combinatorial interpretation, based on local intersection signs of curves in Heegaard diagrams, of certain $\Z_2$ gradings on the bordered Heegaard Floer invariant $\CFDA(W)$. The proof \cite[Appendix A]{HLW} that the abstract and combinatorial $\Z_2$ gradings agree is technical and does not readily generalize to the bordered sutured setting. The second step is to relate $[\CFDA(W)]$, defined with the combinatorial gradings, to the Frohman--Nicas TQFT.

\vspace{0.15in}

\noindent \textbf{Main results over $\Z$, no $\Spinc$ structures.} One aim of this paper is to generalize this second step to the bordered sutured setting, leaving aside the first step. Slightly generalizing \cite{Zarev}, we use the following definitions.

\begin{definition}\label{def:SuturedSurf}
    A \emph{sutured surface} $(F,\Lambda,S^+,S^-)$ is a compact oriented surface $F$ equipped with a finite set $\Lambda$ of points in $\partial F$ dividing $\partial F$ into subsets $S^+$ and $S^-$ that alternate across points of $\Lambda$. In this paper we require that each component of $F$ intersects both $S^+$ and $S^-$ nontrivially; in particular, $F$ has no closed components. We will often write $(F,\Lambda)$ as shorthand for $(F,\Lambda,S^+,S^-)$.
\end{definition}

\begin{definition}\label{def:SuturedCob}
    A \emph{sutured cobordism} $(Y,\Gamma,R^+,R^-)$ (or just $(Y,\Gamma)$) from a sutured surface $(F_0,\Lambda_0)$ to a sutured surface $(F_1,\Lambda_1)$ is a compact oriented 3-manifold $Y$, equipped with an embedding of $F_1\sqcup -F_0$ into $\partial Y$ and a finite set $\Gamma$ of properly embedded arcs and circles in $\partial Y \setminus \mathrm{int} (F_1 \sqcup -F_0)$, with $\partial \Gamma = \Lambda_1 \sqcup \Lambda_0$, dividing $\partial Y \setminus \mathrm{int} (F_1 \sqcup -F_0)$ into subsets $R^+$ and $R^-$ that alternate across curves of $\Gamma$ and satisfy $R^{\pm} \cap F_i = S_i^{\pm}$. We require that each component of $Y$ intersects both $R^+$ and $R^-$ nontrivially; in particular, $Y$ has no closed components.
\end{definition}

\begin{example}\label{ex:SuturedFromOrdinary}
     For a closed oriented surface $\Sigma$ and a set $p$ of framed points in $\Sigma$ (one in each connected component of $\Sigma$), let $\Sut(\Sigma,p)$ be the sutured surface $(F,\Lambda)$ where $F$ is the complement of an open neighborhood of the points $p$ and $\Lambda$ consists of two points (specified by the framing of $p$) on each boundary component of $F$.
    
    For a connected oriented cobordism $W \colon \Sigma_0 \to \Sigma_1$, suppose that $W$ is equipped with an acyclic ribbon graph $\gamma$ embedded in $W$ whose external vertices are the points $p_i$ on $\Sigma_0$ and $\Sigma_1$. In the special case where one $\Sigma_i$ is connected and nonempty and the other is empty, we suppose that $\gamma$ is a single point $p_i$ in the nonempty $\Sigma_i$. When both $\Sigma_i$ are empty, we suppose that $\gamma$ is a single point in $W$. If $W$ is disconnected, we suppose each component of $W$ is equipped with an acyclic ribbon graph as above, and we let $\gamma$ be the union of these graphs.
    
    Let $\Sut(W,\gamma)$ be the sutured cobordism $(Y,\Gamma)$ where $Y$ is the complement in $W$ of a tubular neighborhood of $\gamma$ and the sutures $\Gamma$ on $Y$ are the boundary of the ribbon graph $\gamma$ (after thickening $\gamma$ to form a surface). In a component of $W$ where one $\Sigma_i$ is connected and nonempty and the other is empty, take $\Gamma$ to be a single arc connecting the two points of $\Lambda$. In a closed component of $W$, take $\Gamma$ to be a single circle in the $S^2$ boundary of  the corresponding component of $Y$.
\end{example}

As in \cite{man22,man23}, for a sutured surface $(F,\Lambda)$, we will interpret $\wedge^* H_1(F,S^+)$ as the decategorified bordered sutured Heegaard Floer invariant of $(F,\Lambda)$. Note that for a closed oriented surface $\Sigma$ with $(F,\Lambda) := \Sut(\Sigma,p)$ for any choice of $p$, there is one $S^+$ interval on each boundary component of $F$. We have $H_1(\Sigma) \cong H_1(F) \cong H_1(F,S^+)$, so that $\Vc^{\FN}(\Sigma)$ is canonically identified with $\wedge^* H_1(F,S^+)$.

Let $(Y,\Gamma)$ be a sutured cobordism from a sutured surface $(F_0,\Lambda_0)$ to a sutured surface $(F_1,\Lambda_1)$. In Definition~\ref{def:BSDAZComb}, after making additional choices such as arc diagrams $\Zc_i$ representing $(F_i, \Lambda_i)$ and a Heegaard diagram representing $(Y,\Gamma)$, we will use intersection-based gradings to define a combinatorial stand-in $[\BSDA(Y,\Gamma)]^{\Z}_{\mathrm{comb}}$, a $\Z$-linear map defined up to sign from $\wedge^* H_1(F_0,S^+_0)$ to $\wedge^* H_1(F_1,S^+_1)$, for the ``actual'' decategorified bordered sutured Heegaard Floer invariant $[\BSDA(Y,\Gamma)]$ of $(Y,\Gamma)$ (this latter invariant is only defined given stronger assumptions on $(Y,\Gamma)$ than we have imposed here). 

\begin{remark}
    For a connected cobordism $W$ between closed connected surfaces $\Sigma_i$, the definition of $\CFDA(W)$ depends \emph{a priori} on a choice of $\gamma$ as in Example~\ref{ex:SuturedFromOrdinary}; see e.g. \cite[Section 1.2]{LOTmor}. For example, if both $\Sigma_i$ are nonempty, then $\CFDA(W)$ depends on a choice of framed arc $\gamma$ from $\Sigma_0$ to $\Sigma_1$. If both $\Sigma_i$ are empty, we have the usual dependence of $\CFDA(W) = \widehat{\mathrm{HF}}(W)$ on a choice of basepoint. In general, if we let $(Y,\Gamma) = \Sut(W,\gamma)$, we have
    \[
    [\CFDA(W)] = [\BSDA(Y,\Gamma)]^{\Z}_{\comb} 
    \]
    by construction (using combinatorial gradings for $[\CFDA(W)]$), where the definition of $\CFDA(W)$ implicitly uses the given choice of $\gamma$.
\end{remark}

To state our main theorem relating $[\BSDA(Y,\Gamma)]^{\Z}_{\comb}$ without $\Spinc$ structures to the sutured Frohman--Nicas TQFT, let
\[
K = \rank H_1(F_0,S^+_0) + \rank H_1(F_1,S^+_1) + \chi(Y,R^+).
\]
We will show in Proposition~\ref{prop:RankKerLeqK} that, assuming $H_2(Y,R^-) = 0$, we have $\rank \ker(i_*) \leq K$ with equality if and only if $H_2(Y,R^+) = 0$. Note that by Poincar{\'e}--Lefschetz duality, the condition $H_2(Y,R^-) = 0$ is equivalent to requiring that $H^1(Y,F_1 \cup R^+ \cup F_0) = 0$ or (by the universal coefficient theorem) that $H_1(Y, F_1 \cup R^+ \cup F_0)$ is finite. In this case we have
\[
|H_1(Y,F_1 \cup R^+ \cup F_0)| = |H^2(Y,R^-)| = |\tors(H_1(Y,R^-))|.
\]

\begin{theorem}\label{thm:IntroFirstThm}
         If $H_2(Y,R^-) = 0$, consider the up-to-sign element
    \begin{align*}
    &|K_{Y,\Gamma}| := \pm |H_1(Y,F_1 \cup R^+ \cup F_0)| \wedge^{K} \mathrm{ker}( i_*\colon H_1(F_0 \sqcup F_1, S^+_0 \sqcup S^+_1) \to H_1(Y,R^+)) \\
    &\subset \wedge^* H_1(F_0,S^+_0) \otimes \wedge^* H_1(F_1,S^+_1);
    \end{align*}
    note that if $\rank \ker(i_*) < K$ (i.e. $H_2(Y,R^+) \neq 0$) then $|K_{Y,\Gamma}| = 0$, while if $\rank \ker(i_*) = K$ (i.e. $H_2(Y,R^+) = 0$) then $\wedge^K \ker(i_*)$ can be viewed as an up-to-sign element rather than a subspace. If $H_2(Y,R^-) \neq 0$, let $|K_{Y,\Gamma}| := 0$. The up-to-sign map $[\BSDA(Y,\Gamma)]^{\Z}_{\comb}$ agrees with the composition
    \begin{align*}
    \wedge^* H_1(F_0,S^+_0) & \xrightarrow{\id \otimes |K_{Y,\Gamma}|} \wedge^* H_1(F_0,S^+_0) \otimes \wedge^* H_1(F_0,S^+_0) \otimes \wedge^* H_1(F_1,S^+_1) \\
    & \xrightarrow{\varepsilon \otimes \id} \wedge^* H_1(F_1,S^+_1)
    \end{align*}
    where $\varepsilon$ is defined as in Definition~\ref{def:FNTQFT} and the tensor product of maps $f \otimes g$ is computed using the sign rule $(f \otimes g)(x \otimes y) := (-1)^{|f||y|} f(x) \otimes g(y)$. Here the degree of an element of $\wedge^j H_1(F_i,S^+_i)$ is defined to be $j$; note that the degree of the map $\varepsilon$ is $n_0 := \rank  H_1(F_0,S^+_0)$.
\end{theorem}

We view the composition $(\varepsilon \otimes \id) \circ (\id \otimes |K_{Y,\Gamma}|)$ in Theorem~\ref{thm:IntroFirstThm} as defining the value of the ``sutured Frohman--Nicas TQFT'' $\Vc^{\FN}_{\sut}$ on the sutured cobordism $(Y,\Gamma)$. The functoriality properties of this sutured Frohman--Nicas TQFT are discussed later in the introduction and in Section~\ref{sec:SymmetricMonoidal}; for establishing these properties, the relationship to $[\BSDA]$ from Theorem~\ref{thm:IntroFirstThm} is especially useful.

\begin{example}\label{ex:VFNSutForOrdinary}
    For an ordinary (non-bordered) sutured 3-manifold $(Y,\Gamma)$ satisfying the assumptions of Definition~\ref{def:SuturedCob} (i.e. each component of $Y$ intersects both $R^+$ and $R^-$ nontrivially), view $(Y,\Gamma)$ as a sutured cobordism from $\emptyset$ to itself, so that $\Vc^{\FN}_{\sut}(Y,\Gamma)$ can be viewed as an element of $\Z/_{\pm 1}$. We have
    \[
    \Vc^{\FN}(Y,\Gamma) = \begin{cases}
        \pm |H_1(Y,R^+)| & \textrm{if } H_2(Y,R^+) = H_2(Y,R^-) = 0; \\
        0 & \textrm{otherwise.}
    \end{cases}
    \]
\end{example}

To prove Theorem~\ref{thm:IntroFirstThm}, we establish the following lemma, which is analogous to \cite[Theorem 4]{Pet18}.

\begin{lemma}\label{lem:IntroMainLemma}
    For a ``one-sided'' sutured cobordism $(Y,\Gamma)$ from $(\emptyset, \emptyset)$ to $(F,\Lambda)$, following the usual terminology, write $[\BSD(Y,\Gamma)]^{\Z}_{\comb}$ for the up-to-sign element $[\BSDA(Y,\Gamma)]^{\Z}_{\comb} \in \wedge^* H_1(F,S^+)$. We have
    \[
    [\BSD(Y,\Gamma)]^{\Z}_{\comb} = |K_{Y,\Gamma}|
    \]
    up to sign, where $|K_{Y,\Gamma}| \in \wedge^* H_1(F,S^+)$ is defined as in Theorem~\ref{thm:IntroFirstThm}.
\end{lemma}

To prove Lemma~\ref{lem:IntroMainLemma}, we relate $[\BSDA(Y,\Gamma)]^{\Z}_{\comb}$ for general sutured cobordisms $(Y,\Gamma)$ to a sutured version of Florens--Massuyeau's Alexander functor \cite{FMFunctorial} over $\Z$ in Theorem~\ref{thm:BSDAAlexanderZ}. This relationship also allows us to show that $[\BSDA(Y,\Gamma)]^{\Z}_{\comb}$ is independent up to overall sign of the choices made in its definition, like the choices of arc diagrams for $(F_i,\Lambda_i)$ and Heegaard diagram for $(Y,\Gamma)$. In Section~\ref{sec:SuturedFNTQFT} we deduce Lemma~\ref{lem:IntroMainLemma} from the one-sided case of Theorem~\ref{thm:BSDAAlexanderZ}.

Given Lemma~\ref{lem:IntroMainLemma}, Theorem~\ref{thm:IntroFirstThm} is proved by, for a general sutured cobordism $(Y,\Gamma)$ from $(F_0,\Lambda_0)$ to $(F_1,\Lambda_1)$, concretely analyzing the difference between the definitions of $[\BSDA(Y,\Gamma)]^{\Z}_{\comb}$ and $[\widehat{\mathrm{BSDD}}(Y,\Gamma)]^{\Z}_{\comb} := [\widehat{\mathrm{BSD}}(Y,\Gamma)]^{\Z}_{\comb}$, where the latter is defined by viewing $(Y,\Gamma)$ as a one-sided cobordism from $(\emptyset, \emptyset)$ to $(-F_0 \sqcup F_1, \Lambda_0 \sqcup \Lambda_1)$. This argument largely parallels the approach of \cite{HLW}.

We also show how the sutured Frohman--Nicas TQFT recovers the ordinary Frohman--Nicas TQFT. Let $W \colon \Sigma_0 \to \Sigma_1$ be a cobordism between closed surfaces (no connectedness assumptions), and let $(F_i,\Lambda_i) = \Sut(\Sigma_i)$ and $(Y,\Gamma) = \Sut(W,\gamma)$ for any choice of acyclic ribbon graph $\gamma$ as in Example~\ref{ex:SuturedFromOrdinary}. We will establish the following proposition in Section~\ref{sec:SuturedFNTQFT} below.
\begin{proposition}\label{prop:OrdinaryFNFromSuturedFN}
    Under the identifications $\wedge^* H_1(\Sigma_i) \cong \wedge^* H_1(F_i,S^+_i)$, the ``sutured Frohman--Nicas TQFT'' composition associated to $(Y,\Gamma)$ in Theorem~\ref{thm:IntroFirstThm} agrees with the Frohman--Nicas TQFT invariant $\Vc^{\FN}(W)$.
\end{proposition}

\begin{remark}\label{rem:FixingHalfProj}
Proposition~\ref{prop:OrdinaryFNFromSuturedFN} implies that the maps $[\BSDA(Y,\Gamma)]^{\Z}_{\comb}$ recover the whole content of the Frohman--Nicas TQFT, not just the restriction to connected surfaces and cobordisms. In this general setting the half-projective functoriality of the Frohman--Nicas TQFT becomes relevant; for cobordisms $W$ and $W'$, sometimes we have $\Vc^{\FN}(W') \circ \Vc^{\FN}(W) \neq 0$ but $\Vc^{\FN}(W' \circ W) = 0$. On the other hand, functoriality for $[\BSDA(Y,\Gamma)]^{\Z}_{\comb}$ is more transparent and, as discussed below, we always have
\[
[\BSDA((Y',\Gamma') \circ (Y,\Gamma))]^{\Z}_{\comb} = [\BSDA(Y',\Gamma')]^{\Z}_{\comb} \circ [\BSDA(Y,\Gamma)]^{\Z}_{\comb}
\]
up to sign. This discrepancy can be explained as follows: given $W$ and $W'$, let $(Y,\Gamma) = \Sut(W,\gamma)$ and $(Y',\Gamma') = \Sut(W', \gamma')$ for any choice of acyclic ribbon graphs $\gamma$ and $\gamma'$. The condition $\mu(W',W) = 0$ is equivalent to requiring that when gluing $\gamma$ and $\gamma'$, the result $\gamma' \circ \gamma$ is still acyclic; for example, in Figure~\ref{fig:HalfProjectivity}, $\gamma' \circ \gamma$ would not be acyclic. Thus, the half-projectivity zero appears in the functoriality formula for $\Vc^{\FN}$ precisely when $(Y', \Gamma') \circ (Y, \Gamma)$ is no longer the same as $\Sut(W' \circ W, \tilde{\gamma})$ for any acyclic choice of ribbon graph $\tilde{\gamma}$ in $W' \circ W$; the composition $(Y', \Gamma') \circ (Y, \Gamma)$ is $\Sut(W' \circ W, \gamma' \circ \gamma)$, but $\gamma' \circ \gamma$ is not acyclic. In this sense, reinterpreting $\Vc^{\FN}(W)$ as $[\BSDA(\Sut(W,\gamma))]^{\Z}_{\comb} = \Vc^{\FN}_{\sut}(\Sut(W,\gamma))$ fixes the half-projectivity. In \cite{KerlerHomologyTQFT} half-projectivity is considered to be a characteristic feature of ``non-semisimple'' TQFTs; this may be evidence that non-semisimple TQFTs other than $\Vc^{\FN}$ might also be reinterpreted fruitfully in the sutured setting.
\end{remark}

\vspace{0.15in}

\noindent \textbf{Sutured Frohman--Nicas TQFT as a symmetric monoidal functor.} The perspective of $[\BSDA(Y,\Gamma)]^{\Z}_{\comb}$ is useful for investigating functoriality of the sutured analogue of the Frohman--Nicas TQFT in Theorem~\ref{thm:IntroFirstThm}. By concrete computations, we will show that the maps $[\BSDA(Y,\Gamma)]^{\Z}_{\comb}$ associated to sutured cobordisms $(Y,\Gamma)$ are functorial up to sign with no half-projectivity zeroes and no connectedness restrictions. Given this result, it is reasonable to ask whether the maps $[\BSDA(Y,\Gamma)]^{\Z}_{\comb}$ assemble into a symmetric monoidal functor, as in the usual TQFT axioms, from a category of cobordisms with disjoint union as the symmetric monoidal structure to a category of abelian groups with tensor product as the symmetric monoidal structure. We show that this is the case when the domain and codomain symmetric monoidal categories are appropriately chosen.

\begin{definition}\label{def:SutCobCategory}
Let $\Cob^{\sut}_{2+1}$ be the symmetric monoidal category whose objects are $(F,\Lambda)[d]$ where $(F,\Lambda)$ is a sutured surface and $d \in \Z$. We will view $d$ as a degree shift downward by $d$, so that when we define $\Vc^{\FN}_{\sut}((F,\Lambda)[d])$, the summand $\wedge^k H_1(F,S^+)$ will be placed in degree $k-d$.

A morphism from $(F_0,\Lambda_0)[d_0]$ to $(F_1,\Lambda_1)[d_1]$ in $\Cob^{\sut}_{2+1}$ is a sutured cobordism $(Y,\Gamma)$ from $(F_0,\Lambda_0)$ to $(F_1,\Lambda_1)$ as above (disregarding the degree shifts), up to isomorphism of cobordisms. Composition is the usual gluing of cobordisms. The identity cobordism on $(F,\Lambda)[d]$ is $(F \times [0,1], \Lambda \times [0,1]) =: (\id_{(F,\Lambda)},\Gamma_{\id})$, oriented so that $F \times \{0\} = -F$ and $F \times \{1\} = F$ (we think of $[0,1]$ as having $0$ on the right and $1$ on the left). The monoidal structure on objects is disjoint union with addition of degree shifts; the monoidal structure on morphisms is disjoint union. The symmetrizer for $(F,\Lambda)[d]$ and $(F',\Lambda')[d']$ is $((F \sqcup F') \times [0,1], (\Lambda \sqcup \Lambda') \times [0,1])$ with incoming boundary identified with $-(F \sqcup F')$ and outgoing boundary identified with $F' \sqcup F$ using the evident identifications.
\end{definition}

\begin{definition}\label{def:AbZgrCategory}
Let $\Ab^{\Z\gr}_{\pm 1}$ be the symmetric monoidal category whose objects are $\Z$-graded abelian groups $V_* = \oplus_{k \in \Z} V_k$. Morphisms from $V_*$ to $W_*$ are homogeneous maps of arbitrary degree modulo multiplication by $\pm 1$, with the usual composition and identities. The monoidal structure on objects is the usual tensor product of $\Z$-graded abelian groups. The monoidal structure on morphisms is defined by
\[
(f \otimes g)(x \otimes y) := (-1)^{|f||y|} f(x) \otimes g(y);
\]
since we mod out by overall sign on morphisms, the interchange law holds and we have a monoidal category. The symmetrizer for $V_*$ and $W_*$ is the ``super'' symmetrizer $v \otimes w \mapsto (-1)^{|v||w|} w \otimes v$.
\end{definition}

Our main theorem about the symmetric monoidal functor structure of $\Vc^{\FN}_{\sut}$ is as follows.

\begin{theorem}\label{Thm:VFNSymmetricMonoidal}
    The following data define a symmetric monoidal functor $\Vc^{\FN}_{\sut}$ from $\Cob^{\sut}_{2+1}$ to $\Ab^{\Z\gr}_{\pm 1}$.
    \begin{itemize}
        \item On objects, $\Vc^{\FN}_{\sut}((F,\Lambda)[d])$ is the $\Z$-graded abelian group whose summand in degree $k$ is $\wedge^{k+d} H_1(F,S^+)$.

    \item On a morphism $(Y,\Gamma)$ from $(F_0,\Lambda_0)[d_0]$ to $(F_1,\Lambda_1)][d_1]$, $\Vc^{\FN}_{\sut}(Y,\Gamma)$ is the up-to-sign map from $\wedge^* H_1(F_0,S^+_0)$ to $\wedge^* H_1(F_1,S^+_1)$ determined by the composition in Theorem~\ref{thm:IntroFirstThm} involving $|K_{Y,\Gamma}|$; equivalently, $\Vc^{\FN}_{\sut}(Y,\Gamma) =[\BSDA(Y,\Gamma)]^{\Z}_{\comb}$. 
    As a map between ungraded abelian groups, $\Vc^{\FN}_{\sut}(Y,\Gamma)$ is independent of the degree shifts $d_i$. Only the degree of the map $\Vc^{\FN}_{\sut}(Y,\Gamma)$ depends on $d_i$.

    \item The monoidal functor structure map 
    \[
    \Phi_{(F,\Lambda)[d],(F',\Lambda')[d']}: \Vc^{\FN}_{\sut}((F,\Lambda)[d]) \otimes \Vc^{\FN}_{\sut}((F',\Lambda')[d']) \to \Vc^{\FN}_{\sut}((F \sqcup F', \Lambda \sqcup \Lambda')[d + d'])
    \]
    sends $\alpha \otimes \beta$ to $(-1)^{d|\beta|}\alpha \wedge \beta$. Informally, we think of $\alpha$ as ``$\alpha$ times $d$ odd variables'' and $\beta$ as ``$\beta$ times $d'$ odd variables;'' then when forming $\alpha \wedge \beta$, the $d$ odd variables for $\alpha$ must move past $\beta$ to the right, picking up the sign $(-1)^{d|\beta|}$.\footnote{Compare with \cite[Remark 1.1]{man23}, where the conventions differ slightly in that the odd variables are placed on the left instead of the right.}
    \end{itemize}
\end{theorem}

\begin{remark}\label{rem:DegreeShifts}
    For an ordinary cobordism $W \colon \Sigma_0 \to \Sigma_1$ and any acyclic choice of $\gamma$, let $(F_i,\Lambda_i) = \Sut(\Sigma_i, (\partial \gamma) \cap \Sigma_i)$. As shown in Section~\ref{sec:SuturedFNTQFT}, we can interpret $\Vc^{\FN}(W)$ as
    \[
    \Vc^{\FN}_{\sut}(\Sut(W,\gamma)) \colon \Vc^{\FN}_{\sut}((F_0,\Lambda_0)[0]) \to \Vc^{\FN}_{\sut}((F_1,\Lambda_1)[0])
    \]
    without building in any degree shifts. If we do this, then $\Vc^{\FN}(W)$ is sometimes an odd morphism. On the other hand, in Geer--Young \cite{GeerYoung}, the target symmetric monoidal category for the TQFT from unrolled $U_q(\mathfrak{gl}(1|1))$ has only even morphisms between super vector spaces. This may be a sign that $\Vc^{\FN}(W)$ is most naturally interpreted as
    \[
    \Vc^{\FN}_{\sut}(\Sut(W,\gamma)) \colon \Vc^{\FN}_{\sut}((F_0,\Lambda_0)[g_0]) \to \Vc^{\FN}_{\sut}((F_1,\Lambda_1)[g_1]),
    \]
    which is always an even morphism.
\end{remark}

\begin{remark}
    In \cite{man23,man24}, degree shifts ($\Q$-valued in general) are defined for a sutured surface $(F,\Lambda)$ based on the topological structure of $(F,\Lambda)$ plus a choice of values for certain parameters. The setting of \cite{man24} is closest to the current paper; in \cite{man23} there are separate $\Q$-degree and parity shifts, with the $\Q$-degree taken to be irrelevant for the ``super'' signs, and the parity shift is always a genus-independent quantity plus the rank of $H_1(F,S^+)$. In \cite{man24}, by contrast, there is a $\Q$-degree shift depending on topological data plus a single parameter $A \in \Q$; the shift is often integral, and in this case the parity shift for ``super'' signs is taken to be the $\Z$-degree shift mod 2. When $A = 1/2$ the $\Q$-degree shift is a genus-independent quantity minus half the rank of $H_1(F,S^+)$, in line with the shifts of $g_i = (1/2)(2g_i)$ in Remark~\ref{rem:DegreeShifts}.
    
    The $\Q$-degree shift in \cite{man24} is applied to the abelian group $\wedge^* H_1(F,P)$ where $P$ consists of one point in each component of $S^+$ (circle or interval); this group agrees with $\wedge^* H_1(F',(S^+)')$ where $(F',\Lambda')$ is obtained from $(F,\Lambda)$ by replacing all $S^+$ circles with pairs of an $S^+$ interval and $S^-$ interval. Note that if all components of $F$ intersect $S^+$ nontrivially then all components of $F'$ intersect both $(S^+)'$ and $(S^-)'$ nontrivially, so $(F',\Lambda')$ satisfies the topological hypotheses of this paper. If $F$ is connected of genus $g$ with one boundary component, placed entirely in $S^+$, the $A = 1/2$ degree shift of \cite{man24} is $-g$, in agreement with Remark~\ref{rem:DegreeShifts}. This shift applies to $\wedge^* H_1(F',(S^+)')$, and we have $(F',\Lambda') = \Sut(\Sigma_g,P)$ where $\Sigma_g$ is closed and $P$ is viewed as a point in $\Sigma_g$.
\end{remark}

\begin{remark}\label{rem:Duals}
    We briefly discuss duals here. For a sutured surface $(F,\Lambda)$, the various degree-shifted objects $(F,\Lambda)[d]$ of $\Cob^{\sut}_{2+1}$ are isomorphic for all $d \in \Z$, so we will just consider $(F,\Lambda) = (F,\Lambda)[0]$. The objects $(F,\Lambda)$ and $(-F,\Lambda)$ are dual in $\Cob^{\sut}_{2+1}$; specifically, the identity cobordism $(\id_{(F,\Lambda)}, \Gamma_{\id})$ can be viewed as a cobordism from $(F,\Lambda) \sqcup (-F,\Lambda)$ to $\emptyset$ and as a cobordism from $\emptyset$ to $(-F,\Lambda) \sqcup (F,\Lambda)$. These cobordisms satisfy the triangle identities required for them to be the evaluation and coevaluation morphisms for a duality between $(F,\Lambda)$ and $(-F,\Lambda)$ in $\Cob^{\sut}_{2+1}$.

Since $\Vc^{\FN}_{\sut}$ is a symmetric monoidal functor, applying $\Vc^{\FN}_{\sut}$ to these evaluation and coevaluation cobordisms gives the evaluation and coevaluation maps of a duality between $\Vc^{\FN}_{\sut}(F,\Lambda)$ and $\Vc^{\FN}_{\sut}(-F,\Lambda)$ in $\Ab^{\Z\gr}_{\pm 1}$. In bordered-Floer terms, $\Vc^{\FN}_{\sut}$ of these cobordisms amounts to $[\widehat{\mathrm{BSAA}}]$ and $[\widehat{\mathrm{BSDD}}]$ of the identity cobordism; the fact that decategorified $\mathrm{DD}$ and $\mathrm{AA}$ bimodules of identity cobordisms give such a duality is also used by Hom--Lidman--Watson in \cite[Section 4.3]{HLW}. One nice feature of our setup is that, for us, the duality maps ``$[\widehat{\mathrm{BSAA}}]$ and $[\widehat{\mathrm{BSDD}}]$ of the identity'' are themselves part of the functorial TQFT structure; they are $\Vc^{\FN}_{\sut}$ applied to the evaluation and coevaluation cobordisms in $\Cob^{\sut}_{2+1}$.

Finally, we note that for a sutured surface $(F,\Lambda)$, we can view $(F \times S^1, \Lambda \times S^1)$ as a sutured cobordism from $\emptyset$ to $\emptyset$. We have $\Vc^{\FN}_{\sut}((F \times S^1,\Lambda \times S^1)) \in \Z_{\pm 1}$, which we can compute as a composition of evaluation and coevaluation maps. One can show that this composition gives an alternating sum of the ranks of the homogeneous pieces of the graded abelian group $\Vc^{\FN}_{\sut}(F,\Lambda)$; this sum is $\pm 1$ in the degenerate case $H_1(F,S^+) = 0$ and is zero otherwise. Correspondingly, if $H_1(F,S^+)$ is nonzero then $H_2(F \times S^1, S^+ \times S_1)$ is nonzero, so $\Vc^{\FN}_{\sut}((F \times S^1,\Lambda \times S^1)) = 0$ as in Example~\ref{ex:VFNSutForOrdinary}.
\end{remark}

\vspace{0.15in}

\noindent \textbf{$\Spinc$ structures.} Suppose that $W$ is the complement of a knot $K$ in $S^3$, viewed as a cobordism from $\emptyset$ to $\Sigma = T^2$. The up-to-sign element $|K_W| = \wedge^{\mathrm{top}} \ker (i_* \colon H_1(\Sigma) \to H_1(W))$ lives in the middle summand of $\wedge^* H_1(\Sigma) \cong \Z \oplus \Z^2 \oplus \Z$; it is the class of the Seifert longitude of $K$ inside $H_1(\Sigma) = \wedge^1 H_1(\Sigma)$. This class does not encode enough data to recover the Alexander polynomial of $K$; the usual Frohman--Nicas TQFT only sees the Alexander polynomial of $K$ via a more complicated graded-trace construction which is not actually part of the functorial TQFT structure.

Florens--Massuyeau \cite{FMFunctorial} define a twisted version of the Frohman--Nicas TQFT, called the Alexander functor, whose invariant for a knot complement $W$ as above does see the Alexander polynomial of $K$. Their Alexander functor $\mathsf{A}_G$ assigns invariants to cobordisms equipped with homomorphisms from their fundamental groups to a fixed finitely generated free abelian group $G$. In particular, one could focus on a specific cobordism $W$ and take $G = H_1(W) / \tors(H_1(W))$, disregarding the values of the Alexander functor on other cobordisms. 

Let $(Y,\Gamma)$ be a sutured cobordism from $(F_0,\Lambda_0)$ to $(F_1,\Lambda_1)$. For results involving $\Spinc$ structures, we will assume $Y$ is connected; with a bit of work on sign bookkeeping, one could extend to the disconnected case using the sign patterns for disjoint unions that appear in the symmetric monoidal functor structure discussed above.

Let $H := H_1(Y)$ and $G = H / \tors(H)$. Let $\overline{Y} \xrightarrow{p} Y$ be the maximal free abelian cover of $Y$, so that the deck group of $\overline{Y} \xrightarrow{p} Y$ is $G$. Let $\overline{F}_i = p^{-1}(F_i)$ and $\overline{S}^+_i = p^{-1}(S^+_i)$; note that the $\Z[G]$-module $H_1(\overline{F}_i,\overline{S}^+_i)$ is non-canonically isomorphic to $H_1(F_i,S^+_i) \otimes_{\Z} \Z[G]$.

Using $\Spinc$ structures on sutured manifolds associated to $(Y,\Gamma)$, we will define a decategorified bordered sutured invariant 
\[
[\BSDA(Y,\Gamma)]^{\Z[G]}_{\comb} \colon \wedge^*_{\Z[G]} H_1(\overline{F}_0,\overline{S}^+_0) \to \wedge^*_{\Z[G]} H_1(\overline{F}_1,\overline{S}^+_1),
\]
depending a priori on choices such as arc diagrams and Heegaard diagrams, and relate it to a version of Florens--Massuyeau's Alexander functor, for this choice of $G$, applied to $(Y,\Gamma)$. From the relationship with the Alexander functor, we will see that $[\BSDA(Y,\Gamma)]^{\Z[G]}_{\comb}$ is independent of these choices up to multiplication by $\pm G \subset \Z[G]$.

Let $\overline{R}^+ = p^{-1}(R^+)$. The version of the Alexander functor we will consider is based on the finitely presented $\Z[G]$-module $H_1(\overline{Y}, \overline{R}^+)$. Let 
\begin{equation}\label{eq:deficiency}
d = -\chi(Y,R^+) = \rank H_1(Y,R^+) - \rank H_2(Y,R^+).
\end{equation}

Note that we also have
\[
d = \rank_{\Z[G]} H_1(\overline{Y}, \overline{R}^+) - \rank_{\Z[G]} H_2(\overline{Y}, \overline{R}^+) 
\]
because (as we will see) $H_*(\overline{Y}, \overline{R}^+)$ can be computed from a chain complex of free $\Z[G]$-modules with the same number of basis elements in each degree as a chain complex of free $\Z$-modules computing $H_*(Y,R^+)$ (these complexes are only nonzero in degrees 1 and 2).

When $H_2(\overline{Y}, \overline{R}^+) \neq 0$, the Alexander functor will be defined to be zero, and the map  $[\BSDA(Y,\Gamma)]^{\Z[G]}_{\comb}$ will also be zero, so assume that $H_2(\overline{Y}, \overline{R}^+) = 0$ and thus 
\[
d = \rank_{\Z[G]} H_1(\overline{Y}, \overline{R}^+).
\]
In this case, $H_1(\overline{Y}, \overline{R^+})$ admits an injective presentation matrix $M$ of deficiency $d$ over $\Z[G]$ (the deficiency is defined to be the number of generators minus the number of relations, i.e. the number of rows of $M$ minus the number of columns, and ``injective'' means that the kernel of $M$ is zero).
\begin{definition}\label{def:AlexFunction}
As in Lescop \cite{LescopSumFormula}, define the \emph{Alexander function}
\[
\mathcal{A}^{\Z[G]}_{Y,\Gamma} \colon \wedge^d_{\Z[G]} H_1(\overline{Y}, \overline{R}^+) \to \Z[G],
\]
uniquely specified up to multiplication of all values by the same element of $\pm G \subset \Z[G]$, to be zero if $H_2(\overline{Y}, \overline{R}^+) \neq 0$. If $H_2(\overline{Y}, \overline{R}^+) = 0$, choose any injective deficiency-$d$ presentation matrix $M$ for $H_1(\overline{Y}, \overline{R^+})$ as a $\Z[G]$-module. Then $\mathcal{A}^{\Z[G]}_{Y,\Gamma}(u_1 \wedge \cdots \wedge u_d) \in \Z[G]$, where $u_i \in H_1(\overline{Y}, \overline{R}^+)$, is defined as follows. Choose any expression $\overline{u}_i$ of $u_i$ as a $\Z[G]$-linear combination of the generators, form a square matrix from $M$ by adding new columns $\overline{u}_1, \ldots, \overline{u}_d$ to the right, and take the determinant, an element of $\Z[G]$. As we will see in Proposition~\ref{prop:AlexanderFunctionZ[G]IndependentOfM}, changing the presentation matrix $M$ multiplies all values of $\mathcal{A}^{\Z[G]}_{Y,\Gamma}$ by the same element of $\pm G \subset \Z[G]$. 
\end{definition}

Now let $n_1 = \mathrm{rank}\, H_1(F_1,S^+_1)$ and $c = n_1 + \chi(Y,R^+)$. Note that $d = -\chi(Y,R^+) = n_1 - c$. Define a map
\[
\mathsf{A}_{\Z[G]}(Y,\Gamma) \colon \wedge^*_{\Z[G]} H_1(\overline{F}_0, \overline{S}^+_0) \to \wedge^{*}_{\Z[G]} H_1(\overline{F}_1, \overline{S}^+_1),
\]
well-defined up to multiplication by $\pm G$ and homogenous of degree $c$, by the equation
\[
\omega(\wedge(\mathsf{A}_{\Z[G]}(Y,\Gamma) \otimes \id) (x \otimes y)) = \mathcal{A}^{\Z[G]}_{Y,\Gamma}\left(\overline{i}_*(x) \wedge \overline{i}_*(y)\right)
\]
where $x \in \wedge^*_{\Z[G]} H_1(\overline{F}_0,\overline{S}^+_0)$, $y \in \wedge^{*}_{\Z[G]} H_1(\overline{F}_1, \overline{S}^+_1)$, the map
\[
\wedge \colon \wedge^{p}_{\Z[G]} H_1(\overline{F}_1,\overline{S}^+_1) \otimes \wedge^{q}_{\Z[G]} H_1(\overline{F}_1,\overline{S}^+_1) \to \wedge^{n_1}_{\Z[G]} H_1(\overline{F}_1,\overline{S}^+_1)
\]
sends $z \otimes w$ to $z \wedge w$ when $p+q = n_1$ and is zero otherwise, $\omega$ is any choice of volume form on the free $\Z[G]$-module $H_1(\overline{F}_1,\overline{S}^+_1)$, $\overline{i}_*$ denotes the inclusion-induced map from $H_1(\overline{F}_i,\overline{S}^+_i)$ to $H_1(\overline{Y},\overline{R}^+)$, and the Alexander function $\A^{\Z[G]}_{Y,\Gamma}$ is defined to be zero on inputs in $\wedge^{d'} H_1(Y,R^+)$ for $d' \neq d$. Note that by our sign conventions, $(\mathsf{A}_{\Z}(Y,\Gamma) \otimes \id)(x \otimes y)$ in the above equation should be interpreted as $(-1)^{|\mathsf{A}_{\Z}(Y,\Gamma)||y|} \mathsf{A}_{\Z}(Y,\Gamma)(x) \otimes y$, where the degree $|\mathsf{A}_{\Z}(Y,\Gamma)|$ of $\mathsf{A}_{\Z}(Y,\Gamma)$ is $c$.

Our main $\Spinc$ result in the setting of $G = H/\tors(H)$ is as follows.

\begin{theorem}\label{thm:IntroSpinc}
    We have
    \[
    [\BSDA(Y,\Gamma)]^{\Z[G]}_{\comb} = \mathsf{A}_{\Z[G]}(Y,\Gamma)
    \]
    up to multiplication by $\pm G \subset \Z[G]$. In particular, if $H_2(\overline{Y}, \overline{R}^+) \neq 0$, then both sides are zero.
\end{theorem}

\begin{remark}
Theorem~\ref{thm:BSDAAlexanderZ}, mentioned above as a key step in proving Lemma~\ref{lem:IntroMainLemma}, is the $\Z$-analogue of Theorem~\ref{thm:IntroSpinc}; on the Alexander functor side we have a map $\mathsf{A}_{\Z}(Y,\Gamma)$ based on a $\Z$-coefficient sutured Alexander function $\A^{\Z}_{Y,\Gamma}$. In line with Zarev's interpretation of bordered Floer invariants as sums of sutured Floer invariants \cite{joiningandgluing}, the quantity $\A^{\Z}_{Y,\Gamma}(i_* x \wedge i_* y)$ appearing in the definition of $\mathsf{A}_{\Z}(Y,\Gamma)$ has a natural interpretation as an explicit sign times the Euler characteristic of a sutured Floer homology group; the relevant sutured manifold is obtained from $(Y,\Gamma)$ by pairing with ``Zarev caps'' on the right and left corresponding to the choice of $x$ and $y$. See Corollary~\ref{cor:SFHInterpOfAlexFunctor} below. In the $\Z[G]$ case, this interpretation will form an important part of our proof of Theorem~\ref{thm:IntroSpinc}; we will make use of Friedl--Juhasz--Rasmussen's analysis \cite{FJR} of the Euler characteristic of sutured Floer homology with $\Spinc$ structures taken into account.
\end{remark}

\vspace{0.15in}

\noindent \textbf{Torsion in $H_1(Y).$} As above, let $(Y,\Gamma)$ be a connected sutured cobordism from $(F_0,\Lambda_0)$ to $(F_1,\Lambda_1)$; let $H = H_1(Y)$. Rather than passing to $G = H / T$ where $T = \tors(H)$, we could let $\widehat{Y} \xrightarrow{\pi} Y$ be the maximal abelian cover of $Y$, with $\widehat{R}^+ = \pi^{-1}(R^+), \widehat{F}_i = \pi^{-1}(F_i)$, and $\widehat{S}^+_i = \pi^{-1}(S^+_i)$. We will see that, in the cellular model we will use for $C_*(\widehat{F_i},\widehat{S}^+_i)$, there are only 1-cells and the differential vanishes, so we have a non-canonical isomorphism $H_1(\widehat{F}_i,\widehat{S}^+_i) \cong H_1(F_i,S^+_i) \otimes_{\Z} \Z[H]$ of $\Z[H]$-modules. We will define
\[
[\BSDA(Y,\Gamma)]^{\Z[H]}_{\comb} \colon \wedge^*_{\Z[H]} H_1(\widehat{F}_0, \widehat{S}^+_0) \to \wedge^*_{\Z[H]} H_1(\widehat{F}_1, \widehat{S}^+_1)
\]
depending a priori on choices such as arc diagrams and Heegaard diagrams.

\begin{remark}\label{rem:Z[H]IndependenceAndHeegaardMoves}
    It will follow from our results that $[\BSDA(Y,\Gamma)]^{\Z[H]}_{\comb}$ is independent of the choices in its definition up to multiplication by units in $\Q[H]$. This is not the strongest independence statement one could hope for; we expect that $[\BSDA(Y,\Gamma)]^{\Z[H]}_{\comb}$ is independent of choices up to multiplication by $\pm H \subset \Z[H]$, and that this could be proved by analyzing how $[\BSDA(Y,\Gamma)]^{\Z[H]}_{\comb}$ changes under Heegaard moves. We leave this as future work, however, since it is largely tangential to the goals of the present paper. 
\end{remark}

As in Turaev's theory of maximal abelian torsion, we can understand the group ring $\Q[H]$ as follows. For any character $\chi \colon T \to \C^*$, there is a corresponding algebra homomorphism $\Q[T] \to \C$ whose image is a cyclotomic field $\F_{\chi}$. If we choose a splitting $H \cong T \oplus G$, we get a homomorphism from $\Q[H] = (\Q[T])[G]$ to $\F_{\chi}[G]$, which we will still call $\chi$. In particular, $\F_{\chi}[G]$ is a module over $\Z[H] \subset \Q[H]$.

Let $(\chi_1, \ldots, \chi_m)$ be a complete set of representatives modulo equivalence for characters $\chi \colon T \to \C^*$, where characters $\chi$ and $\chi'$ are equivalent if $\F_{\chi} = \F_{\chi'}$ and $\chi'$ is the composition of $\chi$ with a Galois automorphism of $\F_{\chi}$ over $\Q$. In particular, we take $\chi_1$ to be the trivial character, so $\F_{\chi_1} = \Q$. Then
\[
    (\chi_1,\ldots,\chi_m) \colon \Q[H] \to \bigoplus_{i=1}^m \F_{\chi_i}[G]
\]
is an isomorphism of $\Q$-algebras (and of $\Z[H]$-modules).

Write $H_*(Y,R^+;\Q[H])$ (resp. $H_*(Y,R^+;\F_{\chi_i}[G])$) for the homology groups of the complex $C_*(\widehat{Y},\widehat{R}^+) \otimes_{\Z[H]} \Q[H]$ (resp. $C_*(\widehat{Y},\widehat{R}^+) \otimes_{\Z[H]} \F_{\chi_i}[G]$).

\begin{definition}
    Let $d$ be defined as in equation~\eqref{eq:deficiency}. For $1 \leq i \leq n$, define the Alexander function
    \[
    \A^{\F_{\chi_i}[G]}_{Y,\Gamma} \colon \wedge^d_{\F_{\chi_i}[G]} H_1(Y, R^+; \F_{\chi_i}[G]) \to \F_{\chi_i}[G],
    \]
    well-defined up to (global) multiplication by $\F_{\chi_i}^* \cdot G \subset \F_{\chi_i}[G]$, to be zero if $H_2(Y,R^+;\F_{\chi_i}[G]) \neq 0$. If $H_2(Y,R^+;\F_{\chi_i}[G]) = 0$, define $\A^{\F_{\chi_i}[G]}_{Y,\Gamma}$ as in Definition~\ref{def:AlexFunction}, replacing the $\Z[G]$-module $H_1(\overline{Y},\overline{R}^+)$ with the $\F_{\chi_i}[G]$-module $H_1(Y,R^+;\F_{\chi_i}[G])$. 
\end{definition}

Note that the module $H_1(Y,R^+;\F_{\chi_i}[G])$, and thus the Alexander function $\A^{\F_{\chi_i}[G]}_{Y,\Gamma}$, depends on the splitting of $H$ as $T \oplus G$. However, the below definition gives an Alexander function that we will show is independent of the choice of splitting.

\begin{definition}
    Define the Alexander function
    \[
    \A^{\Q[H]}_{Y,\Gamma} \colon \wedge^d_{\Q[H]} H_1(Y, R^+; \Q[H]) \to \Q[H],
    \]
    well-defined up to (global) multiplication by units in $\Q[H]$, to be the composition
    \begin{align*}
    \wedge^d_{\Q[H]} H_1(Y, R^+; \Q[H]) &\xrightarrow{\cong} \bigoplus_i \wedge^d_{\F_{\chi_i}[G]} H_1(Y, R^+; \F_{\chi_i}[G]) \\
    &\xrightarrow{\bigoplus_i \A^{\F_{\chi_i}[G]}_{Y,\Gamma}} \bigoplus_i \F_{\chi_i}[G]\\
    &\xrightarrow{\cong} \Q[H]
    \end{align*}
    where all arrows are defined using the same choice of splitting $H \cong T \oplus G$.
\end{definition}

Write $H_*(F_i,S^+_i;\Q[H])$ for the homology groups of the complex $C_*(\widehat{F_i},\widehat{S}^+_i) \otimes_{\Z[H]} \Q[H]$. We will see that $H_*(F_i,S^+_i;\Q[H]) \cong H_*(\widehat{F_i},\widehat{S}^+_i) \otimes_{\Z[H]} \Q[H]$ canonically. Define a map
\[
    \mathsf{A}_{\Q[H]}(Y,\Gamma) \colon \wedge^*_{\Q[H]} H_1(F_0,S^+_0;\Q[H]) \to \wedge^*_{\Q[H]} H_1(F_1,S^+_1; \Q[H])
    \]
    well-defined up to multiplication by units in $\Q[H]$ and homogeneous of degree $c$, by 
    \[
    \omega(\wedge(\mathsf{A}_{\Q[H]}(Y,\Gamma) \otimes \id)(x \otimes y)) = \mathcal{A}^{\Q[H]}_{Y,\Gamma}\left(\widehat{i}_*(x) \wedge \widehat{i}_*(y)\right)
    \]
    as in the $\Z[G]$ case; here $\widehat{i}_*$ denotes the inclusion-induced map from $H_1(F_i,S^+_i;\Q[H])$ to $H_1(Y,R^+;\Q[H])$ and $\omega$ is any choice of volume form for the free $\Q[H]$-module $H_1(F_1,S^+_1;\Q[H])$.

On the $[\BSDA]$ side, from $[\BSDA(Y,\Gamma)]^{\Z[H]}_{\comb}$ we will define
\[
[\BSDA(Y,\Gamma)]^{\Q[H]}_{\comb} \colon \wedge^*_{\Q[H]} H_1(F_0,S^+_0;\Q[H]) \to \wedge^*_{\Q[H]} H_1(F_1,S^+_1;\Q[H])
\]
by passing from $\Z[H]$ to $\Q[H]$. Our main $\Spinc$ result taking torsion in $H$ into account is as follows.

\begin{theorem}\label{thm:IntroSpincTorsion}    
    We have
    \[
    [\BSDA(Y,\Gamma)]^{\Q[H]}_{\comb} = \mathsf{A}_{\Q[H]}(Y,\Gamma)
    \]
    up to multiplication by units in $\Q[H]$.
\end{theorem}

\begin{remark}\label{rem:IntroWhyPassToQ[H]}
    The passage to $\Q[H]$ is essential to establish the above relationship for general connected cobordisms $(Y,\Gamma)$ when $H$ has torsion. If we instead worked over $\Z[H]$, whenever $H_2(\widehat{Y},\widehat{R}^+) \neq 0$, an analogous result relating $[\BSDA(Y,\Gamma)]^{\Z[H]}_{\comb}$ to $\mathsf{A}_{\Z[H]}(Y,\Gamma)$ would fail to hold; only when $H_2(\widehat{Y},\widehat{R}^+) = 0$ does the theory go through. See also Remark~\ref{rem:ZHAlexanderWhenH2Zero}. 
\end{remark}

\vspace{0.15in}
\noindent \textbf{Additional remarks.} Here we make some additional remarks.

\begin{remark}\label{rem:MotivationForUpToSign}
    Most of the TQFT maps in this paper are only well-defined up to an overall sign. One reason we do not attempt to resolve this sign ambiguity is that to have a symmetric monoidal functor, we need a symmetric monoidal target category to map into. As morphisms in this target category, we want to allow both even and odd maps between $\Z$-graded abelian groups (analogous to super vector spaces); this may be avoidable in the ordinary non-sutured Frohman--Nicas TQFT as in Remark~\ref{rem:DegreeShifts}, but not in the sutured version. If one does not work modulo an overall sign, even and odd maps together do not form a monoidal category because the interchange law does not hold; rather, a ``super'' interchange law holds as in Brundan--Ellis \cite{BrundanEllis}.

    We note that in Kerler \cite{KerlerHomologyTQFT}, the sign ambiguity is resolved in the Hennings TQFT perspective when working with 2-framed cobordisms. It may be possible to add 2-framings in our setting to resolve the sign ambiguity, but because of the above issue with monoidal categories, we do not investigate this possibility further here.

    In another direction, Hom--Lidman--Watson use the choices made in bordered Heegaard Floer homology to specify a sign for the Frohman--Nicas TQFT map of a cobordism; see \cite[Section 2.1]{HLW}. However, with these signs, the cobordism maps are not functorial on-the-nose; composition is respected only up to an overall sign.
\end{remark}

\begin{remark}
    In the introduction of \cite{Pet18}, Petkova mentions that a bordered manifold (say with boundary $\Sigma$) equipped with an additional knot in its interior can be assigned bordered Heegaard Floer invariants (minus and hat versions) with an additional ``internal'' $\Z$-grading, whose decategorification lives in $\wedge^* H_1(\Sigma) \otimes_{\Z} \Z[t,t^{-1}]$. From the bordered sutured perspective, one should be able to interpret the hat version as the bordered sutured invariant of a related sutured cobordism obtained by cutting out the knot; the additional $\Z$-grading should come from $\Spinc$ structures on this sutured cobordism. We will not work out the details here, though.
\end{remark}

\begin{remark}\label{rem:TopologicalRestrictions}
In \cite{man22,man23,man24}, the sutured surfaces are more general than the ones considered here; they are allowed to have all-$S^+$ components, all-$S^-$ components, and closed components. Such sutured surfaces cannot be represented by arc diagrams, so we disallow them in general here, but see Remark~\ref{rem:AdditionalModuleStructure} below.
\end{remark}

\begin{remark}\label{rem:FitsWithLiterature}
    The sutured perspective on the Frohman--Nicas TQFT fits well with various phenomena in the literature:
    \begin{itemize}
        \item When relating the Frohman--Nicas TQFT to a Hennings TQFT, Kerler \cite{KerlerHomologyTQFT} works with a category where the objects are (connected) surfaces with one boundary component, rather than closed surfaces. More generally, surfaces with nonempty boundary also play a central role in Kerler--Lyubashenko's theory of non-semisimple TQFTs for 3-manifolds with corners \cite{KerlerLyubashenko}.
        
        \item Florens--Massuyeau also use connected surfaces $F_g$ with one boundary component, and they choose a basepoint $\star$ in $\partial F_g$ (we would thicken $\star$ to get an $S^+$ interval). To $(F_g, \star)$ equipped with a homomorphism $\varphi$ from $H_1(F_g)$ to a finitely generated free abelian group $G$, they associate the $\Z[G]$-module $\wedge^*_{\Z[G]} H_1^{\varphi}(F_g,\star)$, rather than (e.g.) $\wedge^*_{\Z[G]} H_1^{\varphi}(F_g)$, matching our use of the relative homology $H_1(F,S^+)$.

        \item For a two-sided cobordism $W$ between connected closed surfaces, with an arc $\mathbf{z}$ in $W$ going between the boundary components, \cite[Theorem 4.5]{HLW} has the prefactor $|H_1(W,\partial W \cup \mathbf{z})|$; one cannot replace this with $|H_1(W,\partial W)|$ because $H_1(W,\partial W)$ is infinite. If we let $\gamma$ denote any framing of $\mathbf{z}$ and $(Y,\Gamma) = \Sut(W,\gamma)$, with $(F_i,\Lambda_i)$ defined accordingly, then $|H_1(W,\partial W \cup \mathbf{z})|$ agrees with our sutured prefactor $|H_1(Y, F_1 \cup R^+ \cup F_0)|$.

        \item The ribbon graphs in Example~\ref{ex:SuturedFromOrdinary} are reminiscent of the ribbon graphs in the type of non-semisimple TQFT in de Renzi \cite{deRenziNSS}, of which one example is Geer--Young's $\Dc^{q,int}$ TQFT from \cite{GeerYoung}. In the framework of \cite{deRenziNSS} and \cite{GeerYoung}, morphisms in the domain category of the TQFT are 3d cobordisms equipped with colored ribbon graphs and cohomology classes satisfying admissibility criteria. Our extension of the Frohman--Nicas formalism to sutured cobordisms gives, in particular, Frohman--Nicas-style maps for 3d cobordisms equipped with ribbon graphs via the evident generalization of Example~\ref{ex:SuturedFromOrdinary}. 
    \end{itemize}    
\end{remark}

\begin{remark}\label{rem:CohomologyClasses}
     There is a basic issue with functoriality for the maps in Theorem~\ref{thm:IntroSpinc} and Theorem~\ref{thm:IntroSpincTorsion}; the state space associated to a sutured surface $(F,\Lambda)$ depends on the cobordism $(Y,R^+)$, not just on $(F,\Lambda)$. Florens--Massuyeau's approach is to assign invariants to cobordisms decorated with a homomorphism $\varphi$ from their first homology group to a fixed finitely generated free abelian group (call it $G_{\mathrm{FM}}$). Since for any fixed $G_{\mathrm{FM}}$, there are cobordisms whose $H_1$ has rank larger than the rank of $G_{\mathrm{FM}}$, this approach will always lose data for some cobordisms.

    A more universal approach, as in Pontryagin duality, would be to try taking $G_{\mathrm{FM}} = U(1)$ or $\C^*$, even though these abelian groups are not finitely generated. Then the extra data carried by surfaces and cobordisms would be cohomology classes as in \cite{GeerYoung}. We note that the ``decorated TQFTs'' discussed in general terms by Costantino--Gukov--Putrov in \cite[Section 5.2]{CGP-NSS} are also defined on surfaces and cobordisms decorated by cohomology classes.
\end{remark}

\begin{remark}\label{rem:ModuleStructure}
    This paper is about TQFT structure in dimensions 2+1; by contrast, the papers \cite{man22,man23,man24} of the first author show that the state spaces $\wedge^* H_1(F,S^+)$ for sutured surfaces satisfy gluing properties as in an open-closed TQFT of dimension 1+1, valued in algebras and bimodules. In this lower-dimensional TQFT structure, an $S^+$ interval gets assigned the algebra $\Z[E]/(E^2)$; a sutured surface with $m$ $S^+$ intervals gets assigned $\wedge^* H_1(F,S^+)$ viewed as a module over $(\Z[E]/(E^2))^{\otimes m}$, with an explicitly described algebra action.

    It would be desirable to unify the constructions of this paper with \cite{man22,man23,man24} into some type of 2-functorial ``extended'' TQFT structure in dimension 1+1+1. A first step would be to investigate when, for a sutured cobordism $(Y,\Gamma)$ from $(F_0,\Lambda_0)$ to $(F_1,\Lambda_1)$, the up-to-sign map $\Vc^{\FN}_{\sut}(Y,\Gamma)$ intertwines any of the $\Z[E]/(E^2)$ actions on $\wedge^* H_1(F_0,S^+_0)$ with those on $\wedge^* H_1(F_1,S^+_1)$. For example, if the $R^+$ region of $(Y,\Gamma)$ contains a strip connecting an $S^+$ interval $I_0$ of $F_0$ to an $S^+$ interval $I_1$ of $F_1$, one could ask if $\Vc^{\FN}_{\sut}(Y,\Gamma)$ intertwines the $\Z[E]/(E^2)$ action for $I_0$ with the $\Z[E]/(E^2)$ action for $I_1$. We leave this question for future work.
\end{remark}

\begin{remark}\label{rem:AdditionalModuleStructure}
        When $F$ has $S^+$ circles, there is a more elaborate module structure in \cite{man24} involving actions for circles as well as intervals. Choose a set of points $P$, one on each $S^+$ circle; then $\wedge^* H_1(F,P)$ has an action of $\Z[E]/(E^2) = U^+_{\Z}(\mathfrak{psl}(1|1))$ for each $S^+$ interval and an action of $\Z\langle E,F\rangle /(E^2,F^2, EF + FE) = U_{\Z}(\mathfrak{psl}(1|1))$ for each $S^+$ circle. 
        
        Equivalently, we can form a sutured surface $(F',\Lambda')$ by replacing each $S^+$ circle of $F$ by a pair of an $S^+$ interval and an $S^-$ interval; we have $\wedge^* H_1(F,P) = \wedge^* H_1(F',(S^+)')$, so this latter group has actions of $U^+_{\Z}(\mathfrak{psl}(1|1))$ for $S^+$ intervals of $F$ and $U_{\Z}(\mathfrak{psl}(1|1))$ for $S^+$ circles of $F$.

        We could then choose colorings of the $S^+$ components of $F$, namely a module over $U_{\Z}^+(\mathfrak{psl}(1|1))$ for each $S^+$ interval of $F$ and a module over $U_{\Z}(\mathfrak{psl}(1|1))$ for each $S^+$ circle of $F$. If we tensor $\wedge^* H_1(F',(S^+)')$ with these modules over the appropriate algebras, we get a TQFT state space for a ``colored sutured surface.'' 
        
        These state spaces are getting a bit closer to being comparable to those in \cite{GeerYoung}, where punctures in surfaces are colored by certain representations of Geer--Young's ``unrolled'' version $U_q^E(\mathfrak{gl}(1|1))$ of $U_q(\mathfrak{gl}(1|1))$. If we incorporate cohomology classes via homomorphisms $\varphi \colon H_1(F) \to \C^*$ as in Remark~\ref{rem:CohomologyClasses}, it is reasonable to suppose that we get actions of something resembling $U_q^E(\mathfrak{gl}(1|1))$ on the state spaces $\wedge^* H_1^{\varphi}(F',(S^+)')$, and thus (by tensor product over $U_q^E(\mathfrak{gl}(1|1))$) state spaces for colored sutured surfaces with colors as in \cite{GeerYoung}. We would hope that these state spaces recover at least some subset of Geer--Young's state spaces; see \cite{man23,man24} for more specific conjectures when $\varphi$ is trivial. We also note that in \cite[Section 5.2.2]{Mikhaylov}, Mikhaylov describes the state spaces of the ``$\mathfrak{psl}(1|1)$ Chern--Simons TQFT'' (a physics construction of which Geer--Young's TQFT is meant to be a mathematical instantiation) on genus-zero punctured surfaces as exterior algebras of twisted relative homology groups. 
        
        Unlike in \cite{man22,man23,man24}, this paper deals with invariants for 3-dimensional cobordisms, not just invariants in dimensions $\leq 2$, so we hope that the results of this paper can go beyond state spaces and lead to an elementary homological construction in the style of Frohman--Nicas of some of the 3-dimensional content of Geer--Young's TQFT, which is currently defined by more indirect methods involving a universal construction. We leave this investigation for future work, however.
\end{remark}

\vspace{0.15in}
\noindent \textbf{Organization.} In Section~\ref{sec:Background} we review the relevant material on arc diagrams and Heegaard diagrams representing sutured surfaces and sutured cobordisms. In Section~\ref{sec:CWDecomp} we discuss CW decompositions coming from Heegaard diagrams. In Section~\ref{sec:BSDAwithoutSpinc} we define $[\BSDA(Y,\Gamma)]^{\Z}_{\comb}$ using additional sets $\Xi$ of choices. In Section~\ref{sec:SuturedAlexanderOverZ} we relate $[\BSDA(Y,\Gamma)]^{\Z}_{\comb}$ to a sutured version of Florens--Massuyeau's Alexander functor over $\Z$ and prove that $[\BSDA(Y,\Gamma)]^{\Z}_{\comb}$ is independent of the choices $\Xi$ up to overall sign. In Section~\ref{sec:SuturedFNTQFT} we prove Lemma~\ref{lem:IntroMainLemma} and Theorem~\ref{thm:IntroFirstThm} relating $[\BSDA(Y,\Gamma)]^{\Z}_{\comb}$ to the sutured Frohman--Nicas TQFT. In Section~\ref{sec:SymmetricMonoidal} we establish the symmetric monoidal structure of the sutured Frohman--Nicas TQFT. In Section~\ref{sec:RelationshipWithSFH} we discuss Zarev caps and the relationship with sutured Floer homology. In Section~\ref{sec:SpincStructures} we define $\Spinc$-versions $[\BSDA(Y,\Gamma)]^{\Z[G]}_{\comb}$ and $[\BSDA(Y,\Gamma)]^{\Q[H]}_{\comb}$ of $[\BSDA(Y,\Gamma)]^{\Z}_{\comb}$, and in Section~\ref{sec:AlexFunctorsInOtherSettings} we relate these to sutured versions of Florens--Massuyeau's Alexander functor over $\Z[G]$ and $\Q[H]$.

\vspace{0.15in}
\noindent \textbf{Acknowledgments.} The authors would like to thank Tye Lidman and Gw{\'e}na{\"e}l Massuyeau for useful conversations. The first author was supported by NSF grants DMS-2151786 and DMS-2502205.

\section{Background}\label{sec:Background}

\subsection{Arc diagrams and sutured surfaces}

Sutured surfaces satisfying the assumptions of Definition~\ref{def:SuturedSurf} can be represented by combinatorial data called arc diagrams. We start by reviewing how this works.

\begin{definition}\label{def:ArcDiagram}
An \emph{arc diagram} $\mathcal{Z} = (\mathbf{Z}, \mathbf{a}, \sim, t)$ is the data of:
\begin{itemize}
    \item a compact oriented 1-manifold $\mathbf{Z}=\bigsqcup_i Z_i$ consisting of finitely many oriented closed intervals and oriented circles $Z_i$;
    \item a set $\mathbf{a}$ of finitely many distinct points in $\mathbf{Z}$ (none on the boundary of interval components of $\mathbf{Z}$) with $|\mathbf{a}| = 2k$ for some $k \geq 0$;
    \item an equivalence relation (``2-to-1 matching'') $\sim$ on $\mathbf{a}$ whose equivalence classes have size 2;
    \item a choice $t$ of whether $\Zc$ is an ``$\alpha$-arc diagram'' (we say $t = \alpha$) or a ``$\beta$-arc diagram'' (we say $t = \beta$).
\end{itemize}
\end{definition}

We draw the intervals and circles of $\mathbf{Z}$ vertically and (for intervals) pointed upwards. For $\alpha$-arc diagrams we draw the circles oriented clockwise; for $\beta$-arc diagrams we draw the circles oriented counterclockwise. We indicate the equivalence relation $\sim$ by joining the two points in each equivalence class by an arc, called a \emph{matching arc}. For $\alpha$-arc diagrams, we draw the matching arcs in red and to the left of $\mathbf{Z}$; for $\beta$-arc diagrams, we draw the matching arcs in blue and to the right of $\mathbf{Z}$. 

See the left side of Figure~\ref{fig:dual arc diagram} for an example of an $\alpha$-arc diagram and the right side of Figure~\ref{fig:dual arc diagram} for an example of a $\beta$-arc diagram. To highlight that we allow circles as well as intervals in $\mathbf{Z}$, the $\alpha$-arc diagram in Figure~\ref{fig:dual arc diagram} has both a circle and an interval in $\mathbf{Z}$. 

\begin{remark}
    Figure 7 in Zarev \cite{joiningandgluing} shows some more examples of $\alpha$-arc diagrams and $\beta$-arc diagrams. Zarev uses slightly different visual conventions; he draws $\alpha$-arcs to the right of $\mathbf{Z}$ and $\beta$-arcs to the left of $\mathbf{Z}$, so to match our conventions, the diagrams in \cite[Figure 7]{joiningandgluing} should be reflected left-to-right in the plane. See Remark~\ref{rem:ZarevOppositeOrientation} and Remark~\ref{rem:ChangeOfConventions} for an explanation of why we change the conventions.
\end{remark}

\begin{remark}
    We do not include data like orientations of the matching arcs and an ordering of the set of matching arcs in the data of an arc diagram; we will choose such data later when necessary.
\end{remark}

\begin{definition}[cf. Section 2.1 of \cite{Zarev}]\label{def:SuturedSurfaceOfZ}
If $\mathcal{Z} = (\mathbf{Z}, \mathbf{a}, \sim, t = \alpha)$ is an $\alpha$-arc diagram, the \emph{sutured surface $F(\mathcal{Z}) = (F,\Lambda)$ associated to $\mathcal{Z}$} is defined in our conventions as follows.
\begin{itemize}
    \item Start with $\mathbf{Z} \times [0,1]$, a disjoint union of rectangles and annuli. Orient $\mathbf{Z} \times [0,1]$ so that the induced orientation on $\mathbf{Z} \times \{0\}$ agrees with the orientation of $\mathbf{Z}$ and the induced orientation on $\mathbf{Z} \times \{1\}$ is the opposite of the orientation on $Z_i$. 
    \item Let $\Lambda$ be $\partial \mathbf{Z} \times \{1/2\}$. Let $S^+$ be
    \[
    (\partial \mathbf{Z} \times [0,1/2]) \cup (\mathbf{Z} \times \{0\});
    \]
    note that $S^+$ is homeomorphic to $\mathbf{Z}$, and the orientation on $\mathbf{Z}$ agrees with the boundary orientation on $S^+$.
    \item For each equivalence class of $\sim$, attach a 2-dimensional 1-handle to $\mathbf{Z} \times \{1\}$ at the two points of the equivalence class, in an orientable way. Note that if such a handle addition is contained in a single component of the surface being built, then there is only one orientable way to add the handle, while if the handle addition connects two previously separate components of the surface, then the two ways to add the handle give isomorphic surfaces. Thus, we do not need to specify a framing on the points of the equivalence class to do the handle addition.
\end{itemize}
Note that $S^-$ is the 1-manifold obtained by surgery (the boundary effect of the handle additions) on
\[
(\partial \mathbf{Z} \times [1/2,1]) \cup (\mathbf{Z} \times \{1\}).
\]
If $\Zc$ is a $\beta$-arc diagram rather than an $\alpha$-arc diagram ($t = \beta$), we define $F(\Zc)$ by the following modification to the definition in the $\alpha$ case: we keep the orientation of $\mathbf{Z} \times [0,1]$ the same but we do the handle additions to $\mathbf{Z} \times \{0\}$ rather than to $\mathbf{Z} \times \{1\}$. In this case $S^-$ is homeomorphic to $\mathbf{Z}$ while $S^+$ is obtained by surgery; the orientation on $\mathbf{Z}$ is opposite to the boundary orientation on $S^-$.
\end{definition}

See Figure~\ref{fig:dual arc diagram} for an example of the sutured surface associated to an $\alpha$-arc diagram; the same sutured surface is associated to the $\beta$-arc diagram shown in the figure. We often depict $F(\Zc)$ with codimension-2 corners in order to emphasize the manner by which $F(\Zc)$ is constructed. These corners can be smoothed out uniquely (up to diffeomorphism) to produce a smooth surface with $S^1$ boundary components.

\begin{remark}\label{rem:ZarevOppositeOrientation}
    Zarev's orientation on $F(\Zc)$ is the opposite of the one specified in Definition~\ref{def:SuturedSurfaceOfZ}; see Remark~\ref{rem:ChangeOfConventions} for an explanation. If we follow our visual conventions for drawing arc diagrams $\Zc$ in the plane, then if we thicken the arc diagram in the plane and give it the standard orientation as a region of $\R^2$, we get the correct orientation on $F(\Zc)$. The same is true with Zarev's visual conventions and Zarev's definition of $F(\Zc)$. 
    
    Figure 8 of \cite{joiningandgluing} shows the sutured surfaces associated to the arc diagrams of \cite[Figure 7]{joiningandgluing}; to get our conventions, one should reflect the surfaces of \cite[Figure 8]{joiningandgluing} left-to-right in the plane and reverse their orientation, so that after reflection they are given the standard orientation as regions in $\R^2$.
\end{remark}

All sutured surfaces $F(\Zc)$ constructed by Definition~\ref{def:SuturedSurfaceOfZ} satisfy the topological assumptions of Definition~\ref{def:SuturedSurf}; indeed, in $\mathbf{Z} \times [0,1]$ before handle additions, each component $Z_i \times [0,1]$ has both $S^+$ boundary (e.g. $Z_i \times \{0\}$) and $S^-$ boundary (e.g. $Z_i \times \{1\}$), and the handle additions never produce a surface component whose boundary avoids $S^+$ or $S^-$. In the other direction, we have the following proposition, which we prove below.

\begin{proposition}\label{prop:SuturedSurfaceRepresentable}
    Every sutured surface $(F,\Lambda)$ satisfying the assumptions of Definition~\ref{def:SuturedSurf} can be written as $F(\Zc)$ for some arc diagram $\Zc$, which can be taken to be either $\alpha$ or $\beta$ type.
\end{proposition}

There are moves called \emph{arc-slides} that can relate any two arc diagrams for the same sutured surface $(F,\Lambda)$, but we will not need them here.

\begin{definition}\label{def:ArcDiagramOps}
    Let $\Zc = (\mathbf{Z},\mathbf{a},\sim,t)$ be an arc diagram.
    \begin{itemize}
        \item Let $-\Zc$ be the diagram $(-\mathbf{Z},\mathbf{a},\sim,t)$ where the orientation on $\mathbf{Z}$ has been reversed.
        \item Let $\overline{\Zc}$ be the same diagram as $\Zc$ with the type $t$ changed from $\alpha$ to $\beta$ or vice-versa.
        \item Let $\Zc^* = (\mathbf{Z}^*,\mathbf{a}^*,\sim^*,t^*)$ be the diagram defined as follows; first, form the sutured surface $F(\Zc)$. If $\Zc$ is an $\alpha$-arc diagram, let $\mathbf{Z}^*$ denote the $S^-$ boundary of $F(\Zc)$ with the opposite of the boundary orientation. If $\Zc$ is a $\beta$-arc diagram, let $\mathbf{Z}^*$ denote the $S^+$ boundary of $F(\Zc)$ with the boundary orientation. Let $\mathbf{a}^*$ denote the union of the belt spheres (two points each) of the 1-handles added in the definition of $F(\Zc)$. Let $\sim^*$ be the equivalence relation whose equivalence classes are these belt spheres. Let $t^*$ be the opposite of $t$, so if $t = \alpha$ then $t^* = \beta$ and vice-versa. We say that $\Zc^*$ is \emph{the arc diagram dual to $\Zc$.}
    \end{itemize}
\end{definition}

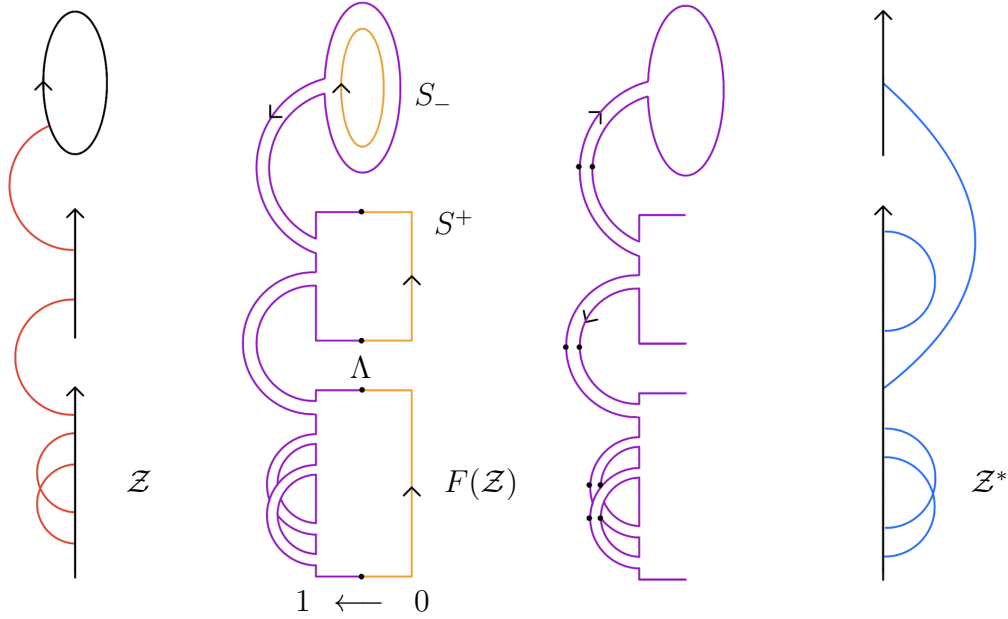
\begin{figure}
    \centering
    \begin{overpic}[width=0.8\textwidth]{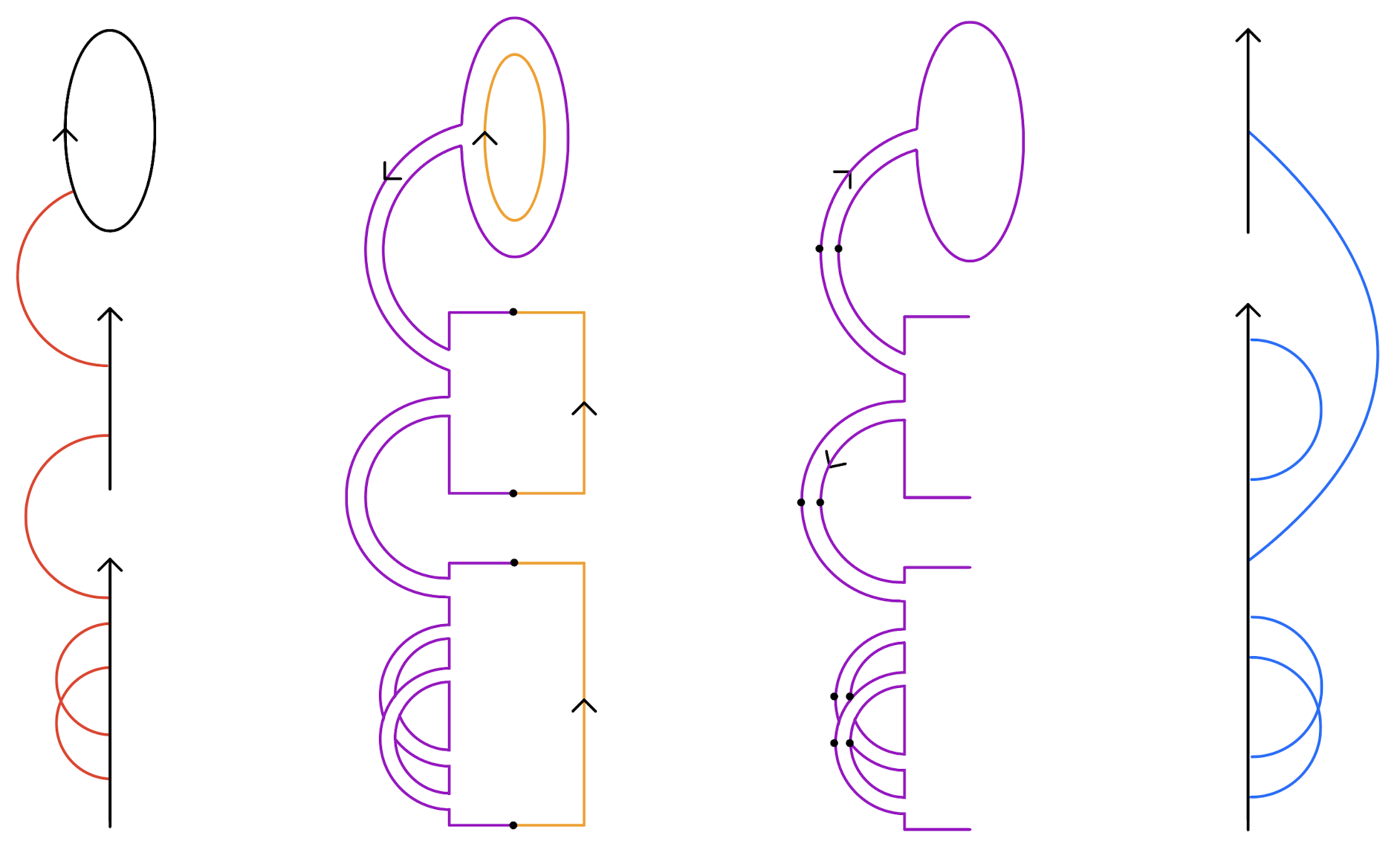}
    \put (13,10) {$\mathcal{Z}$} 
    \put (45,10) {$F(\mathcal{Z})$} 
    \put (30,-2) {$1 \:\: \longleftarrow$} 
    \put (42,-2) {$0$} 
    \put (44,36) {$S^+$} 
    \put (35.5,21.5) {$\Lambda$} 
    \put (42,49) {$S_-$} 
    \put (98,10) {$\mathcal{Z}^*$} 
    \end{overpic}
    \vspace{0.1in}
    \caption{Left to right: An $\alpha$-arc diagram $\mathcal{Z}$; the sutured surface $F(\mathcal{Z})$ associated to $\mathcal{Z}$ (we draw $S^+$ in orange and $S^-$ in purple); $S_-$ with belt spheres indicated and boundary orientation reversed; the $\beta$-arc diagram $\mathcal{Z}^*$ dual to $\mathcal{Z}$.}
    \label{fig:dual arc diagram}
\end{figure}

Figure 7 in \cite{joiningandgluing} shows, for an $\alpha$-arc diagram $\Zc$, the arc diagrams $-\Zc$, $\overline{\Zc}$, and $-\overline{\Zc}$. The construction of $\Zc^*$ from $\Zc$ is shown in Figure~\ref{fig:dual arc diagram}.

\begin{proposition}
    The operations on an arc diagram $\Zc$ in Definition~\ref{def:ArcDiagramOps} have the following effect on the sutured surface $F(\Zc)$.
    \begin{itemize}
        \item We have $F(-\Zc) = -F(\Zc)$, where for a sutured surface $(F,\Lambda)$, we let $-(F,\Lambda)$ denote the same sutured surface with the orientation on $F$ reversed but the roles of $S^+$ and $S^-$ left unchanged.
        \item We have $F(\overline{\Zc}) = \overline{F(\Zc)}$, where for a sutured surface $(F,\Lambda)$, we let $\overline{(F,\Lambda)}$ denote $(F,\Lambda)$ with the orientation on $F$ reversed and the roles of $S^+$ and $S^-$ reversed. 
        \item We have $F(\Zc^*) = F(\Zc)$.
    \end{itemize}
\end{proposition}

Figure 8 in \cite{joiningandgluing} shows the sutured surfaces associated to the arc diagrams $\Zc$, $-\Zc$, $\overline{\Zc}$, and $-\overline{\Zc}$ of \cite[Figure 7]{joiningandgluing}; to get our conventions, Zarev's visual conventions should be modified as in Remark~\ref{rem:ZarevOppositeOrientation}.

\begin{proof}[Proof of Proposition~\ref{prop:SuturedSurfaceRepresentable}]
    It suffices to prove the result for $(F,\Lambda)$ connected. First suppose that $F$ has no $S^+$ intervals. In this case, by the classification of surfaces, $(F,\Lambda)$ is characterized by the following data:
    \begin{itemize}
        \item a genus $g \geq 0$;
        \item some number $p \geq 1$ of $S^+$ boundary circles;
        \item some number $q \geq 1$ of $S^-$ boundary circles.
    \end{itemize}
    The $\alpha$-arc diagram $\Zc$ shown at the top of Figure~\ref{fig:classification of surfaces} represents $(F,\Lambda)$; the corresponding diagram $\Zc^*$ is a $\beta$-arc diagram representing $(F,\Lambda)$. 

    \begin{figure}
    \centering
    \vspace{0.3in}
    \begin{overpic}[width=\textwidth]
    {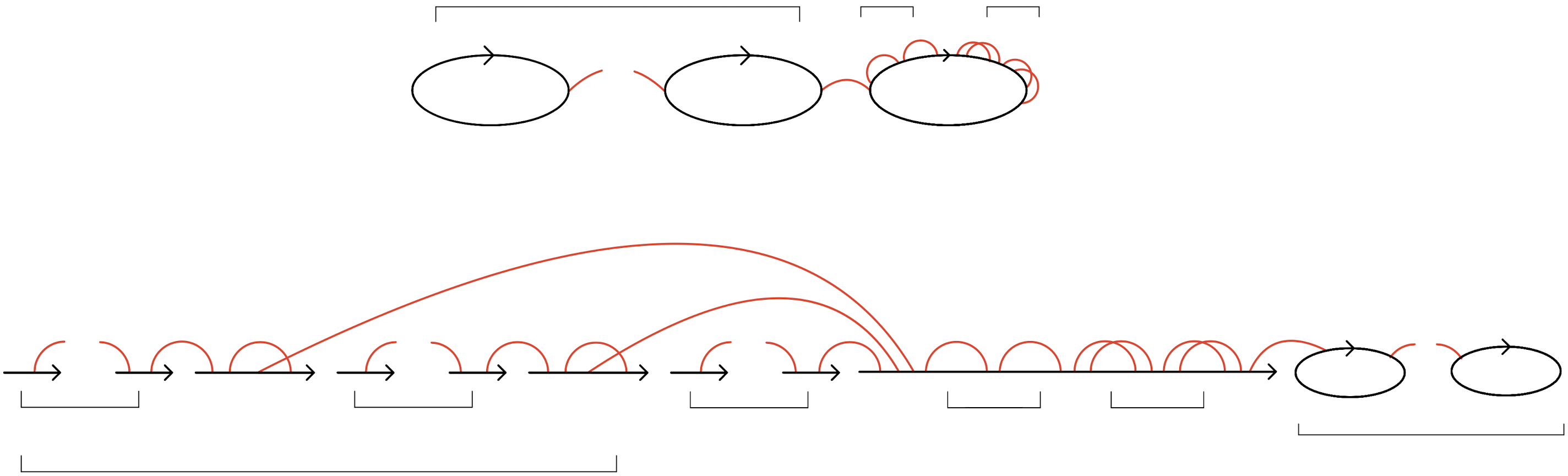}
    \put (37,31.5) {$p-1$} 
    \put (54.3,31.5) {\scalebox{0.8}{$q-1$}} 
    \put (64,31.5) {\scalebox{0.8}{$g$}}
    \put (38.5,26.5) {\scalebox{0.7}{$\cdots$}}
    \put (55.5,27.1) {\rotatebox{23}{\scalebox{0.75}{$\cdots$}}}
    \put (63.5,27.9) {\rotatebox{-28}{\scalebox{0.7}{$\cdots$}}}

    \put (2.5,2) {\scalebox{0.8}{$n_1-1$}}
    \put (23,2) {\scalebox{0.8}{$n_{r-1}-1$}}
    \put (45,2) {\scalebox{0.8}{$n_r-1$}}
    \put (18,-2.5) {$r-1$}
    \put (42.6,12.1) {\scalebox{0.75}{$r-1$}}
    \put (63,2) {$q$} 
    \put (73,2) {$g$} 
    \put (91,0) {$p$}
    
    \put (4.2,9) {\scalebox{0.7}{$\cdots$}}
    \put (20.6,7.3) {\scalebox{0.6}{$\cdots$}}
    \put (25.6,9) {\scalebox{0.7}{$\cdots$}}
    \put (46.7,9) {\scalebox{0.7}{$\cdots$}}
    
    \put (47.8,12.1) {\scalebox{0.7}{$\vdots$}}

    \put (62.3,9) {\scalebox{0.7}{$\cdots$}}
    \put (72.8,9) {\scalebox{0.7}{$\cdots$}}
    \put (90,9) {\scalebox{0.7}{$\cdots$}}
    \end{overpic}
    \vspace{0.1in}
    \caption{Top: an arc diagram for a genus $g$ sutured surface $(F,\Lambda)$ with no $S^+$ intervals. Bottom: an arc diagram for a genus $g$ sutured surface $(F,\Lambda)$ with $\left(\sum_{i=1}^r n_i\right)$-many $S^+$ intervals on $r$ mixed boundary circles. To conserve space, the above arc diagrams have been rotated $90^{\circ}$ clockwise from the conventions under Definition~\ref{def:ArcDiagram}.}
    \label{fig:classification of surfaces}
    \end{figure}
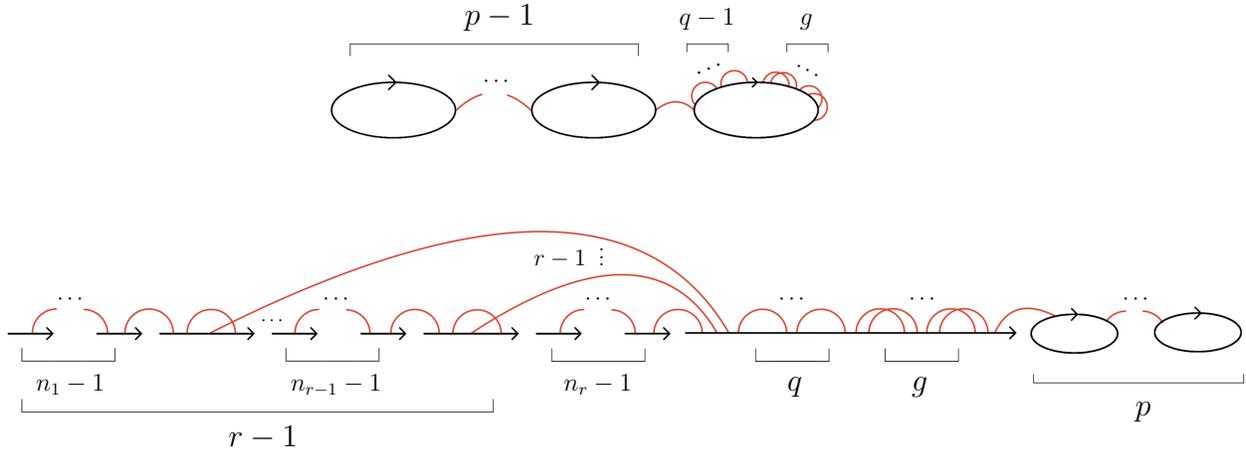
    
    Now suppose that $F$ has at least one $S^+$ interval. In this case, by the classification of surfaces, $(F,\Lambda)$ is characterized by the following data:
    \begin{itemize}
        \item a genus $g \in \Z_{\geq 0}$;
        \item some number $p \geq 0$ of $S^+$ boundary circles;
        \item some number $q \geq 0$ of $S^-$ boundary circles;
        \item some number $r \geq 1$ of ``mixed'' boundary circles with both $S^+$ and $S^-$ intervals;
        \item for $1 \leq i \leq r$, some number $n_i \geq 1$ of $S^+$ intervals for the $i^{th}$ mixed boundary circle.
    \end{itemize}
    The $\alpha$-arc diagram $\Zc$ shown at the bottom of Figure~\ref{fig:classification of surfaces} represents $(F,\Lambda)$;  the corresponding diagram $\Zc^*$ is a $\beta$-arc diagram representing $(F,\Lambda)$.
    \end{proof}

\subsection{Heegaard diagrams and sutured cobordisms}

\begin{definition}[cf. Definition 4.1 of \cite{Zarev}]\label{def:AlphaBorderedDiagram}
An \emph{$\alpha$-$\alpha$ bordered sutured Heegaard diagram} $\Hc = (\Sigma, \boldsymbol{\alpha}, \boldsymbol{\beta}, \mathcal{Z}_0, \mathcal{Z}_1, \psi)$ from the arc diagram $\Zc_0$ (drawn on the right) to the arc diagram $\Zc_1$ (drawn on the left) is the data of:
\begin{itemize}
    \item a compact oriented surface $\Sigma$ called the Heegaard surface;
    \item two $\alpha$-arc diagrams $\mathcal{Z}_0 = (\mathbf{Z}_0, \mathbf{a}_0, \sim_0, t_0 = \alpha)$ and $\mathcal{Z}_1 = (\mathbf{Z}_1, \mathbf{a}_1, \sim_1, t_1 = \alpha)$;
    \item an orientation-preserving identification $\psi$ of $\mathbf{Z}_1$ and $-\mathbf{Z_0}$ with two disjoint subsets of $\partial \Sigma$; let 
    \[
    \partial_{\mathrm{non-gluing}}(\Sigma) := \partial \Sigma \setminus \psi(\mathbf{Z}_1 \sqcup -\mathbf{Z}_0).
    \]
    \item A finite set $\boldsymbol{\alpha}^{a,\mathrm{in}}$ of pairwise disjoint embedded arcs (\emph{incoming alpha arcs}) in $\Sigma$, intersecting $\partial \Sigma$ only in their endpoints, which are $\psi(\mathbf{a}_0) \subset \psi(\mathbf{Z}_0)$, such that two points of $\mathbf{a}_0$ are matched by $\sim_0$ if and only if under $\psi$ they are two endpoints of the same incoming alpha arc;
    \item a finite set $\boldsymbol{\alpha}^{a,\mathrm{out}}$ of pairwise disjoint embedded arcs (\emph{outgoing alpha arcs}) in $\Sigma$, intersecting $\partial \Sigma$ only in their endpoints, which are $\psi(\mathbf{a}_1) \subset \psi(\mathbf{Z}_1)$, such that two points of $\mathbf{a}_1$ are matched by $\sim_1$ if and only if under $\psi$ they are two endpoints of the same outgoing alpha arc, and which are pairwise disjoint from the incoming alpha arcs;
    \item a finite set $\boldsymbol{\alpha}^c$ of pairwise disjoint simple closed curves (\emph{alpha circles}) in the interior of $\Sigma$ that are also pairwise disjoint with the set $\boldsymbol{\alpha}^a := \boldsymbol{\alpha}^{a,\mathrm{out}} \cup \boldsymbol{\alpha}^{a,\mathrm{in}}$ of alpha arcs;
    \item a finite set $\boldsymbol{\beta} = \boldsymbol{\beta}^c$ of pairwise disjoint simple closed curves (\emph{beta circles}) in the interior of $\Sigma$ that intersect $\boldsymbol{\alpha} := \boldsymbol{\alpha}^a \cup \boldsymbol{\alpha}^c$ transversely.
\end{itemize}
\end{definition}

Figure~\ref{fig:Heegaard diagram defn} shows an example of an $\alpha$-$\alpha$ bordered sutured Heegaard diagram.

\begin{remark}
    We do not include data such as ordering or orientations of the $\alpha$- and $\beta$-curves in the data of a Heegaard diagram; we will choose such data later when necessary.
\end{remark}

\begin{definition}\label{def:OtherTypesOfHeegaardDiagrams}
    The following are other types of Heegaard diagrams that arise as simple adaptations or special cases of Definition~\ref{def:AlphaBorderedDiagram}:
    \begin{enumerate}
        \item[(a)] The definitions of $\alpha$-$\beta$ \emph{bordered}, $\beta$-$\alpha$ \emph{bordered}, and $\beta$-$\beta$ \emph{bordered sutured Heegaard diagram} are analogous to Definition~\ref{def:AlphaBorderedDiagram}, except that for the diagram to be $\beta$-bordered on the right (resp. left), $\Zc_0$ (resp. $\Zc_1$) must have $t_0=\beta$ (resp. $t_1=\beta$), and then instead of incoming (resp. outgoing) $\alpha$-arcs, there are incoming (resp. outgoing) $\beta$-arcs; these must intersect $\alpha$ curves transversely and are required to be disjoint from all the other $\beta$ curves in $\Hc$. 
        
        \item[(b)] If $\Zc_i$ has $t_i=\alpha$ (resp. $t_i=\beta$) and $\Zc_j=\emptyset$ where $j=i+1 \text{ mod } 2$, then we say $\Hc$ is an \emph{$\alpha$-bordered sutured Heegaard diagram} (resp. \emph{$\beta$-bordered sutured Heegaard diagram}). These coincide with \cite[Definition 4.1]{Zarev}, (compare \cite[p.~19]{joiningandgluing} for $\beta$-bordered diagrams) aside from the differences explained in Remark~\ref{rem:HDDifferencesFromZarev}.
        
        \item[(c)] If $\Zc_0=\Zc_1=\emptyset$, we say that $\Hc$ is a \emph{sutured Heegaard diagram}. These coincide with the sutured Heegaard diagrams of \cite[Definition 2.7]{Juhasz} and \cite[Definition 2.4]{FJR}.
    \end{enumerate}
\end{definition}

\begin{remark}\label{rem:HDDifferencesFromZarev}
    Definition~\ref{def:OtherTypesOfHeegaardDiagrams}(b) is more general than Zarev's definition of bordered sutured Heegaard diagram \cite[Definition 4.1]{Zarev} both in that we allow circles as well as intervals in $\mathbf{Z}_i$ and in that we do not impose any homological linear independence restrictions (in particular, our Heegaard surface $\Sigma$ may have closed components).
\end{remark}

\begin{figure}
    \centering
    \vspace{0.2in}
    \begin{overpic}[scale=0.25]
    {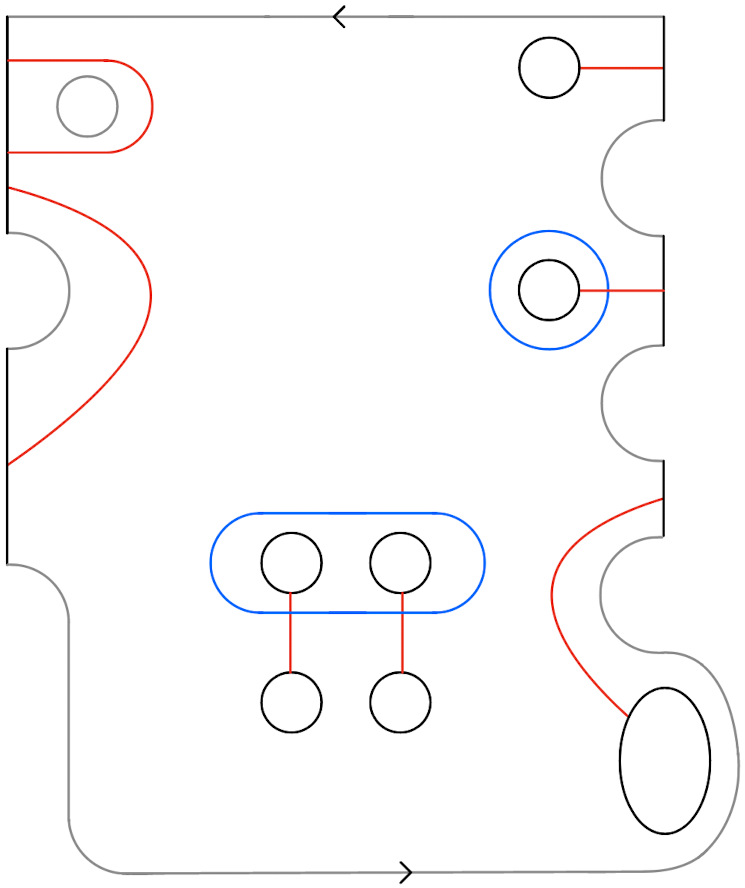}
    \put (31,22) {\scalebox{0.8}[-0.8]{$C$}} 
    \put (31,35) {\scalebox{0.8}{$C$}}
    \put (43.5,22) {\scalebox{0.8}[-0.8]{$B$}} 
    \put (43.5,35) {\scalebox{0.8}{$B$}} 
    \put (60.5,68.5) {\scalebox{0.8}[-0.8]{$A$}} 
    \put (60.5,91) {\scalebox{0.8}{$A$}} 
    \put (-1,103) {$\Zc_1$} 
    \put (72,103) {$-\Zc_0$} 
    \end{overpic}
    \caption{
    An $\alpha$-$\alpha$ bordered sutured Heegaard diagram $\Hc$ from $\Zc_0$ to $\Zc_1$. The subset $\psi(\mathbf{Z}_1 \sqcup -\mathbf{Z}_0) \subset \partial \Sigma$ is drawn in black, and $\partial_{\mathrm{non-gluing}}(\Sigma) \subset \partial \Sigma$ is drawn in grey. The pairs of circles (also drawn in black) labeled $A$, $B$, and $C$ should be glued together to form cylinders; see Figure~\ref{fig:CobordismOrientations} for a more 3-dimensional picture of such a cylinder. We draw $\Sigma$ with the standard orientation as a region of $\R^2$, consistent with Figure~\ref{fig:CobordismOrientations} and Remark~\ref{rem:ChangeOfConventions}. With this choice, the arc diagrams are drawn pointing downward rather than upward as in the conventions below Definition~\ref{def:ArcDiagram}; this is purely a matter of presentation, since arc diagrams have no preferred planar direction.}
    \label{fig:Heegaard diagram defn}
\end{figure}

\begin{definition}[cf. \cite{Zarev}, proof of Proposition 4.4]\label{def:AssocBorderedSuturedCob}
 Let $\Hc$ be a bordered sutured Heegaard diagram from $\mathcal{Z}_0$ to $\mathcal{Z}_1$ with Heegaard surface $\Sigma$. \emph{The sutured cobordism $Y(\Hc)$ from $F(\mathcal{Z}_0)$ to $F(\mathcal{Z}_1)$ associated to $\Hc$} is formed by the following steps.
\begin{itemize}
    \item Thicken $\Sigma$ to get $Y^0 := \Sigma \times [0,1]$, oriented so that with the induced boundary orientation, $\Sigma \times \{1\}$ has the same orientation as $\Sigma$ and $\Sigma \times \{0\}$ has the opposite orientation.
    \item Define $\Gamma$ to be $\partial_{\mathrm{non-gluing}}(\Sigma) \times \{1/2\}$.
    \item Identify $F(\Zc_0)$ with the subset of $\partial Y^0$ given by the union of $\psi(\mathbf{Z}_0) \times [0,1]$ with a small strip around each incoming $\alpha$-arc in $\Sigma \times \{1\}$, or around each incoming $\beta$-arc in $\Sigma \times \{0\}$, depending on whether $\Hc$ is $\alpha$-bordered or $\beta$-bordered on the incoming side. Specifically, the subset $\mathbf{Z}_0 \times [0,1] \subset F(\Zc_0)$ should be identified with $\psi(\mathbf{Z}_0) \times [0,1] \subset \partial Y^0$ by the map $\psi \times \id_{[0,1]}$, and the 1-handles in $F(\Zc_0)$ should be identified with the strips around the incoming $\alpha$-arcs in $\Sigma \times \{1\}$ or $\beta$-arcs in $\Sigma \times \{0\}$; then the identification is unique up to inessential choices and is orientation-reversing. 
    \item Similarly, identify $F(\Zc_1)$ with the subset of $\partial Y^0$ given by the union of $\psi(\mathbf{Z}_1) \times [0,1]$ with a small strip around each outgoing $\alpha$-arc in $\Sigma \times \{1\}$ or each outgoing $\beta$-arc in $\Sigma \times \{0\}$. The identification is orientation-preserving. 
    \item Let
    \[
    R^+_0 := (\partial_{\mathrm{non-gluing}}(\Sigma) \times [0,1/2])\cup (\Sigma \times \{0\} \setminus \mathrm{int}(F(\Zc_1) \cup F(\Zc_0)))
    \]
    and
    \[
    R^-_0 := (\partial_{\mathrm{non-gluing}}(\Sigma) \times [1/2,1]) \cup (\Sigma \times \{1\} \setminus \mathrm{int}(F(\Zc_1) \cup F(\Zc_0))).
    \]
    Note that if $\Hc$ is $\alpha$-$\alpha$ bordered, then $\Sigma \times \{0\}$ is disjoint from $\mathrm{int}(F(\Zc_1) \cup F(\Zc_0))$, so that in this case we have $R^+_0 = (\partial_{\mathrm{non-gluing}}(\Sigma) \times [0,1/2])\cup (\Sigma \times \{0\})$. Similarly, if $\Hc$ is $\beta$-$\beta$ bordered, then $R^-_0 = (\partial_{\mathrm{non-gluing}}(\Sigma) \times [1/2,1]) \cup (\Sigma \times \{1\})$.    
    \item Add a 3d 2-handle to $Y^0$ for each $\alpha$-circle and $\beta$-circle, with the $\alpha$ 2-handles attached to $\alpha$-circles viewed inside $\Sigma \times \{1\}$ and the $\beta$ 2-handles attached to $\beta$-circles viewed inside $\Sigma \times \{0\}$. Let $Y$ denote the resulting 3-manifold. These handle additions leave $F(\Zc_i)$, $\Lambda_i$, and $\Gamma$ unchanged; the effect of the $\alpha$ handle additions on $\partial Y^0$ is a surgery on $R^-_0$. Let $R^- \subset \partial Y$ denote the result of this surgery. Similarly, the effect of the $\beta$ handle additions is a surgery on $R^+_0$; let $R^+ \subset \partial Y$ denote the result of this surgery.
\end{itemize}
\end{definition}

Figure~\ref{fig:CobordismOrientations} illustrates the construction of a sutured cobordism $Y(\Hc)$ from a bordered sutured Heegaard diagram $\Hc$, focusing on the outgoing or $\mathbf{Z}_1$ side of $\Hc$ and $Y(\Hc)$.

\begin{figure}
    \centering
    \vspace{0.2in}
    \begin{overpic}[width=0.95\textwidth]
    {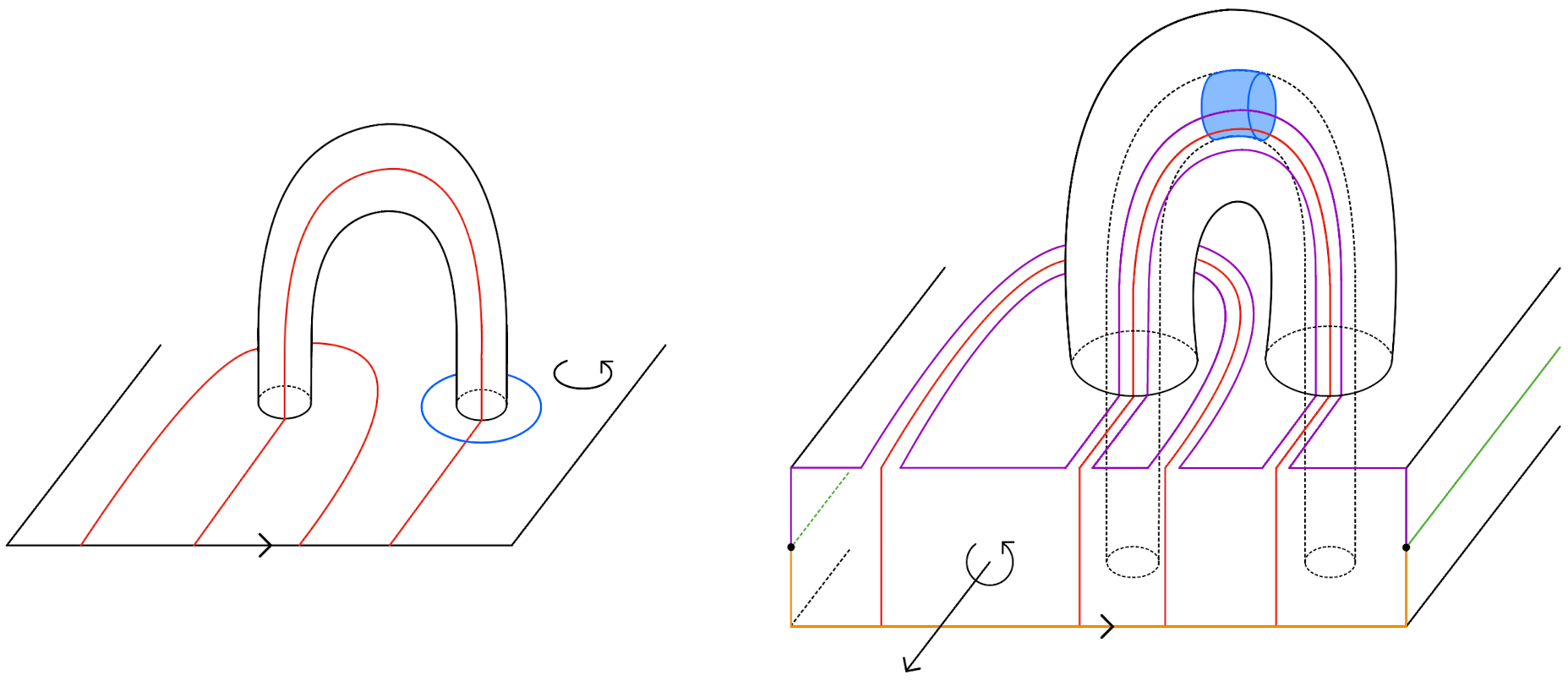}
    \put (16,3) {$\mathbf{Z}_1$} 
    \put (36.5,23) {$\Sigma$} 
    
    \put (63,-2) {$\mathbf{Z}_1\times [0,1]$} 
    \end{overpic}
    \vspace{0.1in}
    \caption{Left: portion of a bordered sutured Heegaard diagram $\Hc$ with Heegaard surface $\Sigma$. Right: outgoing or $\mathbf{Z}_1$ side of the sutured cobordism $Y(\Hc)$; the front rectangular face is $\mathbf{Z}_1 \times [0,1]$. By convention we will usually shade $R^+$ in grey, but this convention may be relaxed in more intricate figures such as this one.}
    \label{fig:CobordismOrientations}
    \end{figure}

\begin{remark}\label{rem:ChangeOfConventions}
    Our choices of conventions stem from our decision to view the decategorified bordered sutured invariant of a sutured surface $(F,\Lambda)$ as $\wedge^* H_1(F,S^+)$ rather than $\wedge^* H_1(F,S^-)$. Equivalently, we choose to view $\alpha$-arc diagrams as the primary way to represent sutured surfaces, rather than $\beta$-arc diagrams; the matching arcs of an $\alpha$-arc diagram give elements of $H_1(F,S^+)$ while the matching arcs of a $\beta$-arc diagram give elements of $H_1(F,S^-)$. Correspondingly, for a sutured cobordism $(Y,\Gamma)$, we will be more concerned with $H_1(Y,R^+)$ than with $H_1(Y,R^-)$, and from an $\alpha$-$\alpha$ Heegaard diagram for $(Y,\Gamma)$, we will extract a relative CW decomposition of a space homotopy equivalent to $Y$ that is relative to $R^+$ (rather than $R^-$ as in e.g. \cite{FJR}).
    
    Definition~\ref{def:AssocBorderedSuturedCob} also uses the usual convention that when constructing a sutured cobordism from a Heegaard diagram, the $\alpha$-side of the diagram gives $R^-$ while the $\beta$-side gives $R^+$. This choice corresponds to our association of relative $\Spinc$ structures to Heegaard diagram generators in Definition~\ref{def:SpincStrsOfGens} below, which also agrees with the usual conventions. As a result, in CW decompositions relative to $R^+$, we want $\beta$-circles of the diagram to give 1-cells and $\alpha$-circles to give 2-cells. In terms of handle decompositions, we want $\beta$-circles to give 1-handles and $\alpha$-circles to give 2-handles, the opposite of the usual convention. 

    These choices, plus the convention that relative handle decompositions or CW decompositions are built from the ``bottom up,'' determine our orientation conventions as follows. Given a bordered sutured Heegaard diagram $\Hc$, when we thicken $\Sigma$ to begin forming $Y(\Hc)$, we imagine that (a local piece of) $\Sigma$ is embedded in the $z=0$ plane in $\R^3$ (with its usual ``counterclockwise'' orientation as a subset of $\R^2$), and that $\Sigma \times [0,1]$ is embedded in the region from $z = 0$ to $z = 1$ in $\R^3$. We give $\Sigma \times [0,1]$ the usual ``right-hand rule'' orientation for regions in $\R^3$; this choice determines the orientation on $Y(\Hc)$. The boundary orientation on $\partial (Y(\Hc))$ is then visualized as the usual ``outward-pointing normal last'' orientation on the boundary of a region in $\R^3$.

    For the relationship between arc diagram orientations and Heegaard diagram orientations, we make the additional choice that the outgoing or left subset of $\partial \Sigma$ is $\mathbf{Z}_1$ while the incoming or right subset of $\partial \Sigma$ is $-\mathbf{Z}_0$. Thus, on the left side of Figure~\ref{fig:CobordismOrientations}, $\mathbf{Z}_1$ should be oriented from left to right.
    
    We can compare with the boundary orientation on the subset $F_1$ of $\partial(Y(\Hc))$ corresponding to $F(\Zc_1)$; by the outward-normal rule, $F_1$ is given the counterclockwise orientation from the perspective of Figure~\ref{fig:CobordismOrientations}. Thus, if $\Hc$ is $\alpha$-bordered on the left, the bottom portion of $F_1$ (in the $z=0$ plane) should be oriented from left to right, in agreement with the orientation on $\mathbf{Z}_1$. In other words, for an $\alpha$-arc diagram $\Zc$, the $S^+$ region of $F(\Zc)$ should be oriented the same as $\mathbf{Z}$. 

    By contrast, the usual conventions are the same as above except that the $\alpha$ handles are added on the bottom of $\Sigma \times [0,1]$ and the $\beta$ handles are added on the top; see \cite[Figure 4.1]{LOT} and \cite[Figure 3]{LOTmor} as well as \cite[Figure 6]{Zarev}. The result is that, for an $\alpha$-arc diagram $\Zc$, the $S^+$ region of $F(\Zc)$ should be oriented oppositely to $\mathbf{Z}$, explaining our change of convention in Definition~\ref{def:SuturedSurfaceOfZ}.
\end{remark}

\begin{remark}\label{rem:FJROppositeOrientation}    
    In the special case that $\Hc$ is an ordinary sutured Heegaard diagram (when $\Zc_0=\Zc_1=\emptyset$, as in Definition~\ref{def:OtherTypesOfHeegaardDiagrams}(c)), the sutured cobordism $Y(\Hc)$ associated to $\Hc$ is an ordinary sutured manifold with no bordered boundary, i.e. satisfying $F_0=F_1=\emptyset$. Relative to the conventions of \cite{FJR} and \cite{Juhasz}, our construction of $Y(\Hc)$ yields the sutured manifold with the opposite overall orientation, but is such that the regions $R^+$ and $R^-$ coincide. 
\end{remark}

All sutured cobordisms $Y(\Hc)$ constructed by Definition~\ref{def:AssocBorderedSuturedCob} satisfy the topological assumptions of Definition~\ref{def:SuturedCob}; indeed, each component of the 3-manifold $Y^0$ in Definition~\ref{def:AssocBorderedSuturedCob} intersects both $R^+_0$ and $R^-_0$ nontrivially, because both $\Sigma \times \{0\} \setminus \mathrm{int}(F(\Zc_1) \cup F(\Zc_0))$ and $\Sigma \times \{1\} \setminus \mathrm{int}(F(\Zc_1) \cup F(\Zc_0))$ are nonempty. The handle additions never produce a component of $Y$ whose boundary avoids $R^+$ or $R^-$. In the other direction, we have the following proposition.

\begin{proposition}\label{prop:RepresentabilityByHeegaardDiagrams}
    Every sutured cobordism $(Y,\Gamma)$ satisfying the assumptions of Definition~\ref{def:SuturedCob} can be written as $Y(\Hc)$ for some bordered sutured Heegaard diagram $\Hc$, which can be taken to be of type $\alpha$-$\alpha$, $\alpha$-$\beta$, $\beta$-$\alpha$, or $\beta$-$\beta$.
\end{proposition}

\begin{proof}
    We first construct an $\alpha$-$\alpha$ bordered sutured Heegaard diagram representing $(Y,\Gamma)$. Let $(\widetilde{Y}, \widetilde{\Gamma}, \widetilde{R}^+,\widetilde{R}^-)$ be the sutured 3-manifold such that $\widetilde{Y} = Y$, $\widetilde{R}^+ = R^+$, and $\widetilde{R}^- = F_1 \cup R^- \cup F_0$. By assumption, every component of $\widetilde{Y}$ intersects both $\widetilde{R}^+$ and $\widetilde{R}^-$ nontrivially. By Juhasz \cite[Proposition 2.13]{Juhasz}, there exists a sutured Heegaard diagram $\widetilde{\Hc} = (\widetilde{\Sigma}, \widetilde{\alpha}, \widetilde{\beta})$ defining $(\widetilde{Y}, \widetilde{\Gamma}, \widetilde{R}^+,\widetilde{R}^-)$. Note that the subset $\widetilde{R}^-$ of $\partial \widetilde{Y}$ is obtained by surgery (boundary effect of $\alpha$ handle addition) on the Heegaard surface $\widetilde{\Sigma}$. 

    Pick $\alpha$-arc diagrams $\Zc_i$ and identifications of $(F_i, \Lambda_i)$ with $F(\Zc_i)$. Under these identifications, the matching arcs of $\Zc_i$ become arcs in $F_i$ and thus in $\widetilde{R}^-$. After an isotopy, we can arrange that these arcs in $\widetilde{R}^-$ avoid the ``new disks'' of $\widetilde{R}^-$ that are not present in $\widetilde{\Sigma}$ and instead are added during the surgery. Viewing these arcs in $\widetilde{\Sigma}$ as incoming and outgoing $\alpha$-arcs, we get a bordered sutured Heegaard diagram representing $(Y,\Gamma)$. See Figure~\ref{fig:HDiagRepresentabilityIsotopy}.
    
    For $\beta$-$\beta$ diagrams, the argument is parallel, but we define $(\widetilde{Y}, \widetilde{\Gamma}, \widetilde{R}^+,\widetilde{R}^-)$ as the sutured 3-manifold such that $\widetilde{R}^- = R^-$ and $\widetilde{R}^+ = F_1 \cup R^+ \cup F_0$, pick $\beta$-arc diagrams $\Zc_i$, and isotope away $\beta$-arcs in $\widetilde{R}^+$ from the boundary effect of the $\beta$ handle additions. For $\alpha$-$\beta$ and $\beta$-$\alpha$ diagrams, we would define $\widetilde{R}^- =F_1 \cup R^- $ (resp. $\widetilde{R}^- = R^- \cup F_0$) and $\widetilde{R}^+ = R^+ \cup F_0$ (resp. $\widetilde{R}^+ = F_1 \cup R^+$), pick a $\beta$ (resp. $\alpha$) arc diagram $\Zc_0$ and an $\alpha$ (resp. $\beta$) arc diagram $\Zc_1$, then isotope away $\alpha$-arcs in $\widetilde{R}^-$ from the boundary effect of the $\alpha$ handle additions, and isotope away $\beta$-arcs in $\widetilde{R}^+$ from the boundary effect of the $\beta$ handle additions.
\end{proof}

\begin{figure}
    \centering
    \vspace{0.2in}
    \begin{overpic}[width=0.87\textwidth]
    {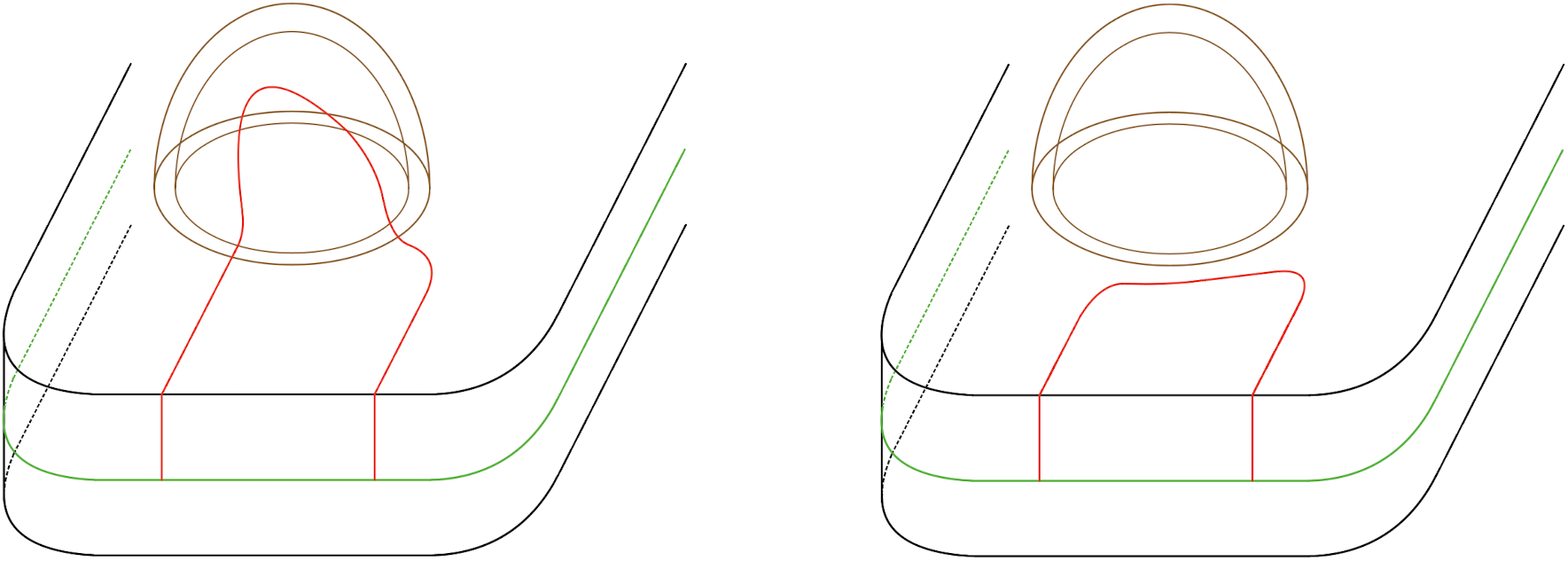}
    \put (45,10) {$\longrightarrow$} 
    
    \put (16,2) {$R^+$} 
    \put (16,7) {$R^-$} 
    \put (73,2) {$R^+$} 
    \put (73,7) {$R^-$} 
    \end{overpic}
    \vspace{0.1in}
    \caption{Performing an isotopy on an arc of $F_1$ away from the disks added to $\widetilde{R}^- \setminus \widetilde{\Sigma}$ during surgery, as in the proof of Proposition~\ref{prop:RepresentabilityByHeegaardDiagrams}}
    \label{fig:HDiagRepresentabilityIsotopy}
\end{figure}
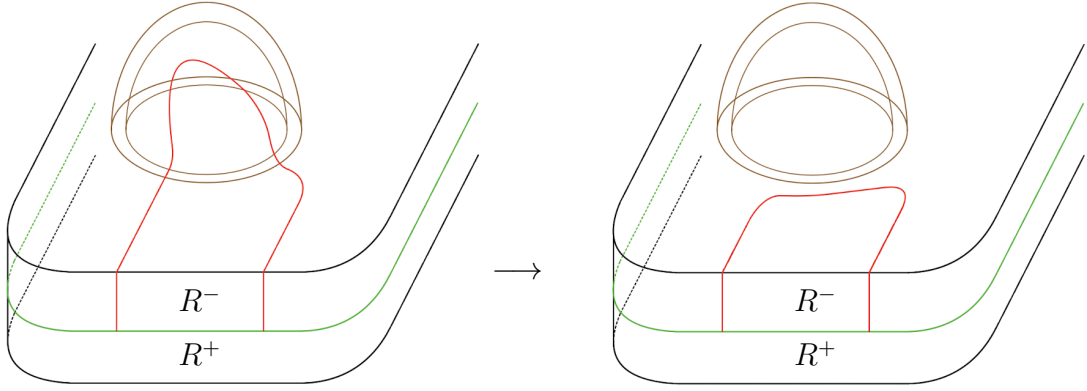
    
Note that once the second paragraph in the above proof is established, type $\alpha$-$\beta$, $\beta$-$\alpha$, or $\beta$-$\beta$ representability follows as well by appealing to the half-identity Heegaard diagrams introduced below in Example~\ref{ex:HalfIdentityDiags}.

There are Heegaard moves that can relate any two choices of Heegaard diagram for a given sutured cobordism, but we will not need them here. 

\begin{definition}[cf. Section 4.6 of \cite{Zarev}]\label{def:BorderedSuturedHDGluing}
    Let $\Hc$ and $\Hc'$ be bordered sutured Heegaard diagrams from $\mathcal{Z}_0$ to $\mathcal{Z}_1$ and from $\mathcal{Z}_1$ to $\mathcal{Z}_2$ respectively, with Heegaard surfaces $\Sigma$ and $\Sigma'$. Let $\mathcal{Z}_1 = (\mathbf{Z}_1, \mathbf{a}_1, \sim_1, t_1)$. Assume that on their $\mathcal{Z}_1$-sides, $\Hc$ and $\Hc'$ are either both $\alpha$-bordered or both $\beta$-bordered. We can form a bordered sutured Heegaard diagram $\Hc' \cup_{\Zc_1} \Hc$ from $\mathcal{Z}_0$ to $\mathcal{Z}_2$ by gluing the outgoing (left) boundary $\mathbf{Z}_1 \subset \Sigma$ of $\Hc$ to the incoming (right) boundary $-\mathbf{Z}_1 \subset \Sigma'$ of $\Hc'$ using the given identifications between $\mathbf{Z}_1$ and subsets of the boundary of $\Sigma$ and $\Sigma'$. The $\alpha$-arcs or $\beta$-arcs at the interface between $\Sigma$ and $\Sigma'$ glue to give closed circles. 
\end{definition}

\begin{definition}\label{def:SuturedCobGluing}
    If $(Y,\Gamma)$ is a sutured cobordism from $(F_0,\Lambda_0)$ to $(F_1,\Lambda_1)$ and $(Y',\Gamma')$ is a sutured cobordism from $(F_1,\Lambda_1)$ to $(F_2,\Lambda_2)$, define a sutured cobordism $Y' \cup_{F_1} Y$ from $F_0$ to $F_2$ by using the embeddings of $F_1$ into $\partial Y$ and of $-F_1$ into $\partial Y'$ to glue $Y'$ and $Y$ along the given orientation-reversing identification between $-F_1 \subset \partial Y'$ and $F_1 \subset \partial Y$. 
\end{definition}

\begin{proposition}
    If we have two bordered sutured Heegaard diagrams $\Hc$ and $\Hc'$ as in Definition~\ref{def:BorderedSuturedHDGluing}, then
    \[
    Y(\Hc' \cup_{\Zc_1} \Hc) \cong Y(\Hc') \cup_{F(\Zc_1)} Y(\Hc).
    \]
\end{proposition}

\begin{proof}
    The main apparent difference between the two sides is that if one glues the diagrams first to form $\Hc' \cup_{\Zc_1} \Hc$ and then takes $Y(\Hc' \cup_{\Zc_1} \Hc)$, the result at first glance looks glued from $Y(\Hc')$ and $Y(\Hc)$ along $\mathbf{Z}_1 \times [0,1]$, rather than along the larger surface $F_1 = F(\Zc_1)$ that also contains strips around the $\alpha$ or $\beta$-arcs being glued on either side. However, there is also a new $\alpha$ or $\beta$-circle in $\Hc' \cup_{\Zc_1} \Hc$ whenever an arc in $\Hc$ gets joined to an arc in $\Hc'$ by the diagram gluing. When adding 3d 2-handles to these new circles, the effect is to glue $(F_1 \setminus (\mathbf{Z}_1 \times [0,1])) \subset Y(\Hc)$ to $(F_1 \setminus (\mathbf{Z}_1 \times [0,1])) \subset Y(\Hc')$, producing $Y(\Hc') \cup_{F(\Zc_1)} Y(\Hc)$ up to isomorphism of sutured cobordisms.
\end{proof}

\begin{example}[cf. Construction 8.18 of \cite{LOTmor}]\label{ex:HalfIdentityDiags}
Let $\mathcal{Z} = (\mathbf{Z},\mathbf{a},\sim, t = \alpha)$ be an $\alpha$-arc diagram, and let $(F,\Lambda) = F(\Zc)$. The $\alpha$-$\beta$ \emph{half-identity Heegaard diagram} for $\mathcal{Z}$ is the $\alpha$-$\beta$ bordered sutured Heegaard diagram $\Hc^{\alpha \beta}_{1/2} =(\Sigma,\boldsymbol{\alpha},\boldsymbol{\beta}, \mathcal{Z}^*,\mathcal{Z}, \psi)$ where:
\begin{itemize}
    \item $\Sigma$ is the oriented surface $F$; note that we can write $\partial \Sigma$ as $\Zb \cup_{\Lambda} -\Zb^*$. Thicken each point $p$ in $\Lambda$ into a small closed interval $p \times [0,1]$ in $\partial \Sigma$, shrinking $\Zb$ and $-\Zb^*$ accordingly, so that
    \[
    \partial \Sigma = \Zb \cup_{\Lambda \times \{0\}} (\Lambda \times [0,1]) \cup_{\Lambda \times \{1\}} -\Zb^*.
    \]
    \item On the output side (left), $\psi$ identifies $\mathbf{Z}$ with $\mathbf{Z} \subset \partial \Sigma$ in the evident way (this identification is orientation-preserving);
    \item On the input side (right), $\psi$ identifies $-\mathbf{Z}^*$ with $-\mathbf{Z}^* \subset \partial \Sigma$ in the evident way (this identification is orientation-preserving);
    \item There are no $\alpha$-circles, and the outgoing arcs are $\alpha$-arcs given by the cores of the 2d 1-handles used to build $\Sigma = F$ from $\Zc$;
    \item There are no $\beta$-circles, and the incoming arcs are $\beta$-arcs given by the co-cores of the 2d 1-handles used to build $\Sigma = F$ from $\Zc$.
\end{itemize}
The $\beta$-$\alpha$ \emph{half-identity Heegaard diagram} for $\mathcal{Z}$ is the $\beta$-$\alpha$ bordered sutured Heegaard diagram $\Hc^{\beta \alpha}_{1/2} =(\Sigma,\boldsymbol{\alpha},\boldsymbol{\beta}, \mathcal{Z},\mathcal{Z}^*, \psi)$ where:
\begin{itemize}
    \item $\Sigma$ is $F$ with its orientation reversed; note that we can write $\partial \Sigma$ as $\Zb^* \cup_{\Lambda} -\Zb$. Thicken each point $p$ in $\Lambda$ into a small closed interval $p \times [0,1]$ in $\partial \Sigma$, shrinking $\Zb^*$ and $-\Zb$ accordingly, so that
    \[
    \partial \Sigma = \Zb \cup_{\Lambda \times \{0\}} (\Lambda \times [0,1]) \cup_{\Lambda \times \{1\}} -\Zb^*.
    \]
    \item On the output side (left), $\psi$ identifies $\mathbf{Z}^*$ with $\mathbf{Z}^* \subset \partial \Sigma$ in the evident way (this identification is orientation-preserving);
    \item On the input side (right), $\psi$ identifies $-\mathbf{Z}$ with $-\mathbf{Z} \subset \partial \Sigma$ in the evident way (this identification is orientation-preserving);
    \item There are no $\beta$-circles, and the outgoing arcs are $\beta$-arcs given by the co-cores of the 2d 1-handles used to build $\Sigma = F$ from $\Zc$;
    \item There are no $\alpha$-circles, and the incoming arcs are $\alpha$-arcs given by the cores of the 2d 1-handles used to build $\Sigma = F$ from $\Zc$.
\end{itemize}
\end{example}

Both types of half-identity Heegaard diagrams for $\mathcal{Z}$ represent the identity sutured cobordism on $F(\mathcal{Z})$. Figure~\ref{fig:Half identity diagram} shows the $\alpha$-$\beta$ and $\beta$-$\alpha$ half-identity Heegaard diagrams for the arc diagram $\mathcal{Z}$ from Figure~\ref{fig:dual arc diagram}. 

\begin{example}[cf. Definition 5.35 of \cite{LOTbimod}]\label{ex:identityHD} 
Let $\mathcal{Z}$ be an arc diagram. The $\alpha$-$\alpha$ \emph{identity Heegaard diagram} $\Hc_{\id}$ for $\mathcal{Z}$ is the bordered sutured Heegaard diagram $\Hc^{\alpha \beta}_{1/2} \cup_{\Zc^*} \Hc^{\beta \alpha}_{1/2}$.
\end{example}

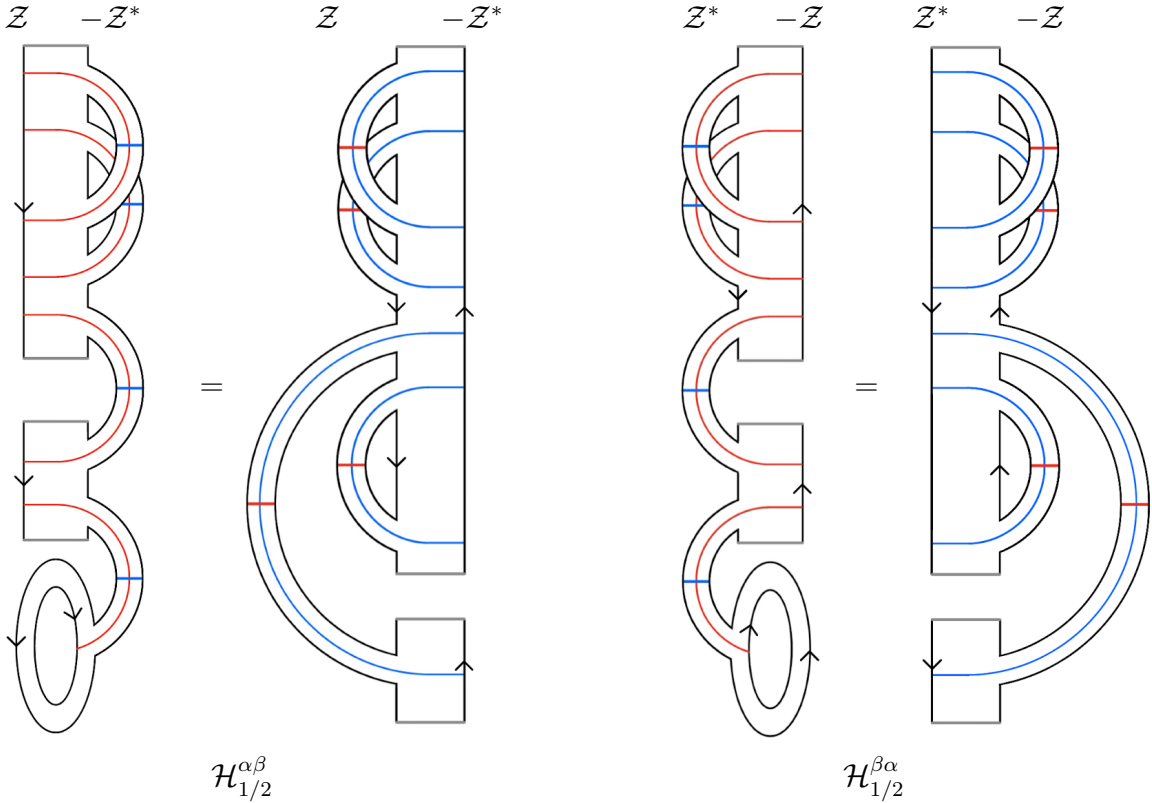
\begin{figure}
    \centering
    \vspace{0.33in}
    \begin{overpic}[width=0.925\linewidth]
    {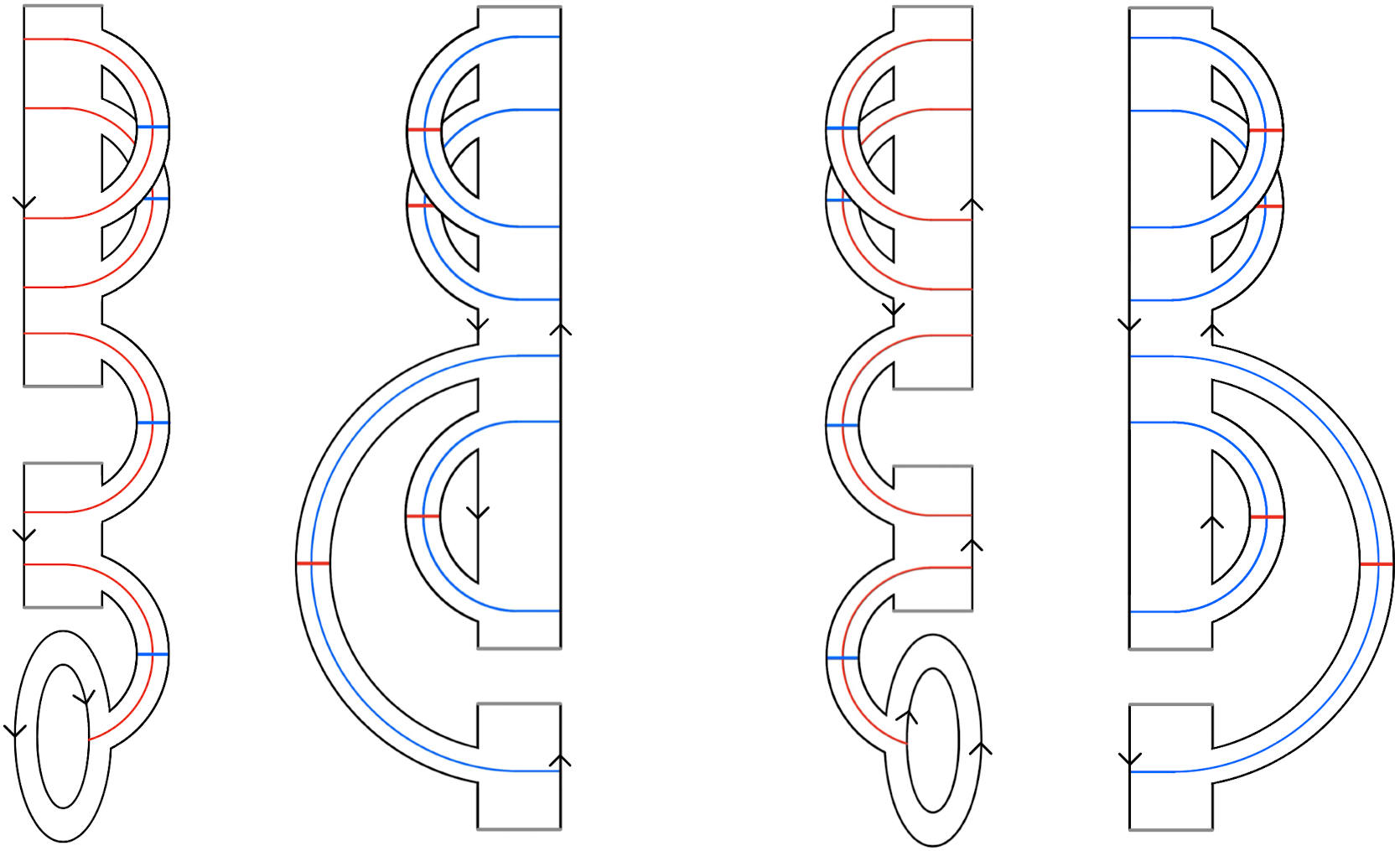}
    \put (18,-4) {$\Hc^{\alpha \beta}_{1/2}$} 
    \put (73,-4) {$\Hc^{\beta \alpha}_{1/2}$} 

    \put (74,30) {$=$} 
    \put (17,30) {$=$} 
    
    \put (0,62) {$\Zc$} 
    \put (6.5,62) {$-\Zc^*$} 
    \put (27,62) {$\Zc$} 
    \put (38,62) {$-\Zc^*$} 

    \put (59,62) {$\Zc^*$} 
    \put (67,62) {$-\Zc$} 
    \put (79,62) {$\Zc^*$} 
    \put (88,62) {$-\Zc$} 
    \end{overpic}
    \vspace{0.3in}
    \caption{Leftmost pair: Two representations of the $\alpha$-$\beta$ half-identity diagrams for the arc diagram $\mathcal{Z}$ from Figure~\ref{fig:dual arc diagram}. Rightmost pair: Two representations of the $\beta$-$\alpha$ half-identity diagrams for the same arc diagram $\mathcal{Z}$.
    The construction outlined in Example~\ref{ex:HalfIdentityDiags} \emph{a priori} produces the first and third figures. We choose to display the equivalent second and fourth diagrams, as these are more convenient for the construction of $\Hc_{\norm} =\Hc^{\alpha \beta}_{1/2} \cup_{\Zc_1^*} \Hc' \cup_{\Zc_0^*} \Hc^{\beta \alpha}_{1/2}$ appearing in Corollary~\ref{cor:NormalizedHeegaardDiagrams}.}
    \label{fig:Half identity diagram}
\end{figure}

Since the half-identity diagrams for $\Zc$ represent the identity sutured cobordism on $F(\Zc)$, so does the identity Heegaard diagram. We could define $\beta$-$\beta$ identity Heegaard diagrams similarly, although we will not use them here.

Below, we show that every sutured cobordism is representable by a Heegaard diagram with a distinguished decomposition. We view this decomposition as placing an $\alpha$-$\alpha$ bordered sutured Heegaard diagram into a \emph{normal form}. This form will be used throughout the paper to streamline many arguments, including the proof of our main result over $\Z$, Theorem~\ref{thm:IntroFirstThm}. This form also primes our constructions for bridging over to the purely sutured setting (see Proposition~\ref{prop:BSDAandSFH}), and for defining the $\Spinc$-aware version of $[\BSDA]_{\comb}$ (see Section~\ref{sec:SpincStructures}). We show in Corollary~\ref{cor:CanChooseXiNormalized} that assuming this form in the setting over $\Z$ is irrelevant up to overall sign. 

\begin{corollary}\label{cor:NormalizedHeegaardDiagrams}
    Every sutured cobordism $(Y,\Gamma)$ satisfying the assumptions of Definition~\ref{def:SuturedCob} can be written as $Y(\Hc)$ for some $\alpha$-$\alpha$ bordered sutured Heegaard diagram $\Hc$ of the form $\Hc_{\norm} \coloneq \Hc^{\alpha \beta}_{1/2} \cup_{\Zc_1^*} \Hc' \cup_{\Zc_0^*} \Hc^{\beta \alpha}_{1/2}$, where $\Hc'$ is a $\beta$-$\beta$ bordered Heegaard diagram for $Y$ with left arc diagram $\Zc_1^*$ and right arc diagram $\Zc_0^*$ for two $\alpha$-arc diagrams $\Zc_0,\Zc_1$.
\end{corollary}

\begin{proof}
    By Proposition~\ref{prop:RepresentabilityByHeegaardDiagrams}, we can write $Y$ as $Y(\widetilde{\Hc})$ for some $\alpha$-$\alpha$ bordered sutured Heegaard diagram $\widetilde{\Hc}$, say with left arc diagram $\Zc_1$ and right arc diagram $\Zc_0$. Let
    \begin{align*}
    \Hc_{\norm} &= \Hc_{\id} \cup_{\Zc_1} \widetilde{\Hc} \cup_{\Zc_0} \Hc_{\id} \\
    &= \Hc^{\alpha \beta}_{1/2} \cup_{\Zc_1^*} \Hc^{\beta \alpha}_{1/2} \cup_{\Zc_1} \widetilde{\Hc} \cup_{\Zc_0} \Hc^{\alpha \beta}_{1/2} \cup_{\Zc_0^*} \Hc^{\beta \alpha}_{1/2} \\
    &= \Hc^{\alpha \beta}_{1/2} \cup_{\Zc_1^*} \Hc' \cup_{\Zc_0^*} \Hc^{\beta \alpha}_{1/2}
    \end{align*}
    where $\Hc' := \Hc^{\beta \alpha}_{1/2} \cup_{\Zc_1} \widetilde{\Hc} \cup_{\Zc_0} \Hc^{\alpha \beta}_{1/2}$.
\end{proof}

\begin{remark}
    The $\beta$-$\beta$ bordered diagram $\Hc'$ in Corollary~\ref{cor:CanChooseXiNormalized} also represents $(Y,\Gamma)$, since the half-identity diagrams represent identity sutured cobordisms. 
\end{remark}

\section{CW decompositions}\label{sec:CWDecomp}

\subsection{Decompositions from Heegaard diagrams}

Let $(Y,\Gamma)$ be a sutured cobordism from $(F_0,\Lambda_0)$ to $(F_1,\Lambda_1)$ satisfying the assumptions of Definition~\ref{def:SuturedCob}. Choose an $\alpha$-$\alpha$ Heegaard diagram $\Hc$ to represent $(Y,\Gamma)$. Let $\Sigma$ be the Heegaard surface of $\Hc$. Recall that a sutured cobordism $Y$ is built from the Heegaard diagram $\Hc$ as follows.
\begin{itemize}
    \item Start with $\Sigma  \times [0,1]$ and add 2-handles on the bottom side for all $\beta$-circles of $\Sigma$. The thickened bottom boundary is now $R^+ \times [0,1]$.
    \item Add 2-handles on top for all $\alpha$-circles of $\Sigma$.
\end{itemize}
Equivalently:
\begin{itemize}
    \item Start with $R^+ \times [0,1]$ and add 1-handles on the top side for all $\beta$-circles of $\Sigma$. The thickened top boundary is now $\Sigma  \times [0,1]$.
    \item Add 2-handles on top for all $\alpha$-circles of $\Sigma$.
\end{itemize}
This second construction presents $Y$ as a \emph{sutured handle complex} 
\[
\mathcal{A} = (Y,R^+ \times I,\mathcal{E})
\] 
as defined in \cite[Definition 3.8]{FJR}. In particular, for any choice of orderings of the sets of 1-handles and 2-handles, $\mathcal{E}$ is the ordered list of 3-dimensional handles
\[
\mathcal{E} = (e_1,\ldots,e_b,e_{b+1},\ldots,e_{b+a})
\]
where:
\begin{itemize}
    \item $e_1,\ldots,e_b$ are the 3d 1-handles corresponding to the $\beta$-circles of $\Hc$;
    \item $e_{b+1},\ldots,e_{b+a}$ are the 3d 2-handles corresponding to the $\alpha$-circles of $\Hc$.
\end{itemize}
The 1-handles appear before the 2-handles in the ordering, the 1-handles are pairwise disjoint, and the 2-handles are pairwise disjoint, so $\mathcal{A}$ is a nice sutured handle complex in the language of \cite[Section 3.5]{FJR}. Since $\mathcal{A}$ is nice, the ordering of handles is actually irrelevant at this point, although later we will make various choices of orderings. As discussed in \cite[Remark 3.11]{FJR}, by shrinking handles onto their cores we get a CW decomposition of a space $\mathfrak{Y}$ homotopy equivalent to $Y$ relative to $R^+ = R^+ \times \{0\} \subset R^+ \times [0,1]$. 

The 1-cells of this relative CW decomposition consist of $b$ arcs, each of whose endpoints is in the interior of $R^+$. Similarly to \cite[Remark 3.11]{FJR}, pick a CW decomposition $\mathfrak{R}^+$ of $R^+$ (with cells of dimension 0, 1, and 2); we will furthermore assume this CW decomposition is obtained as follows (see Figure~\ref{fig:CWDecompFromH}).
\begin{itemize}
    \item Pick a 1-dimensional CW decomposition of $\partial \Sigma$ such that all points of $\Lambda \subset \partial \Sigma$ and all endpoints of $\alpha$ arcs in $\partial \Sigma$ are 0-cells of the decomposition.
    \item Let $\Sigma^{\circ}$ denote $\Sigma$ with a tubular neighborhood of each $\beta$ circle (open annulus) removed.
    \item Note that $\partial \Sigma^{\circ}$ is the disjoint union of $\partial \Sigma$ with two circles $C,C'$ for each $\beta$ circle of $\Sigma$. Consider a $\beta$ circle of $\Sigma$; let $x_1,\ldots,x_m$ be the points of intersection of this $\beta$ circle with $\alpha$ arcs and circles. Since $C$ and $C'$ are parallel pushoffs of the $\beta$ circle, each point $x_i$ determines a point $x_i^C$ in $C$ and a point $x_i^{C'}$ in $C'$. Choose CW decompositions of $C$ and $C'$ such that the points $x_i^C$ are $0$-cells of $C$ and the points $x_i^{C'}$ are $0$-cells of $C'$. We get a CW decomposition of $\partial \Sigma^{\circ}$.
    
    \item Extend this CW decomposition of $\partial \Sigma^{\circ}$ to a CW decomposition of all of $\Sigma^{\circ}$ in such a way that each $\alpha$ arc and circle of $\Sigma$ that intersects a $\beta$ circle at least once, (so the $\alpha$ arc or circle is now cut into $\geq 1$ arcs of $\Sigma^{\circ}$, each ending on 0-cells of $\partial \Sigma^{\circ}$), has each of these $\geq 1$ arcs of $\Sigma^{\circ}$ as a 1-cell of the CW decomposition. For an $\alpha$ arc of $\Sigma$ that does not intersect any $\beta$ circles, make the arc into a 1-cell of $\Sigma^{\circ}$. For an $\alpha$ circle of $\Sigma$ that does not intersect any $\beta$ circles, choose a point on the $\alpha$ circle to make into a 0-cell of $\Sigma^{\circ}$, and make the rest of the $\alpha$ circle into a 1-cell of $\Sigma^{\circ}$. 
    
    \item We obtain $R^+$ from $\Sigma^{\circ}$ by gluing in one disk along each boundary component of $\Sigma^{\circ}$ that is not on the boundary of $\Sigma$. For such a boundary component $C$, we have already chosen a CW decomposition of $C$, say with $m$ 0-cells. Extend to a CW decomposition of $D^2$ with $(m+1)$-many 0-cells ($m$ on the boundary and one in the center), $2m$ 1-cells ($m$ on the boundary and $m$ connecting the boundary 0-cells to the center), and $m$ 2-cells. Together with our CW decomposition of $\Sigma^{\circ}$, these CW decompositions of $D^2$ glue together into a CW decomposition $\mathfrak{R}^+$ of $R^+$.
\end{itemize}
We thus have an absolute CW decomposition of $\mathfrak{Y}$, and $(\mathfrak{Y}, \mathfrak{R}^+)$ is a CW pair. 

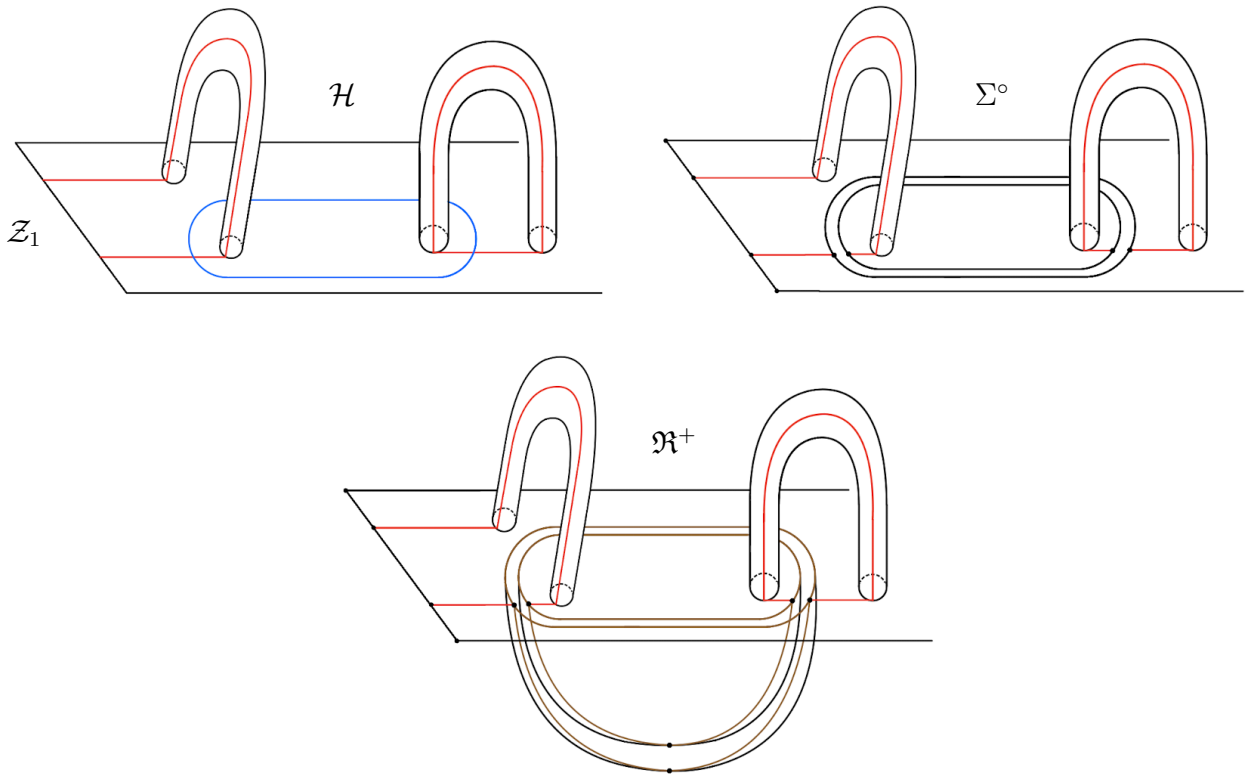
\begin{figure}
    \centering
    \vspace{1em}
    \begin{overpic}[width=\textwidth]
    {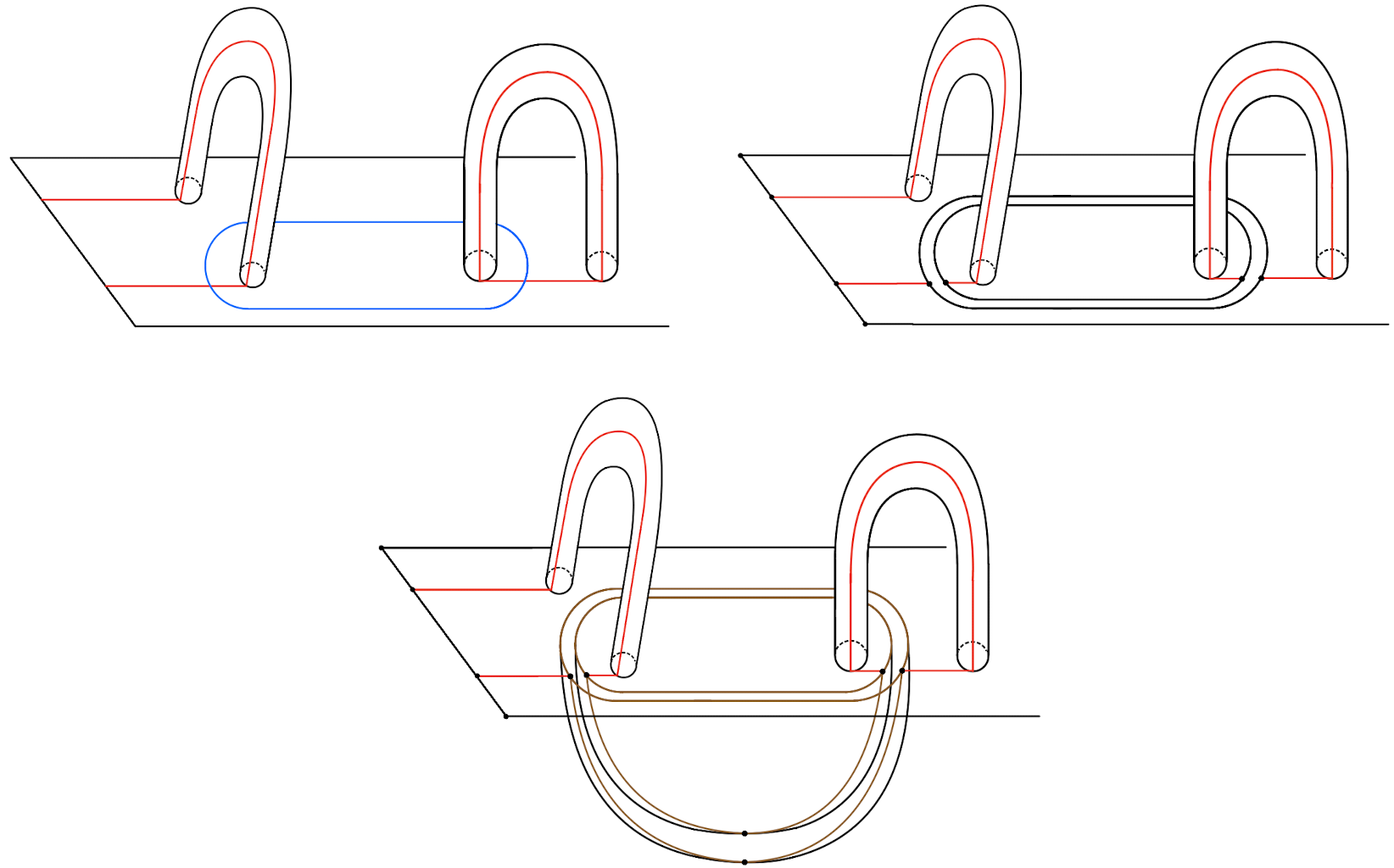} 
    \put (26,54) {$\Hc$} 
    \put (0,43) {$\Zc_1$} 
    \put (78,54) {$\Sigma^{\circ}$} 
    \put (52,26) {$\mathfrak{R}^+$} 
    \end{overpic}
    \caption{Top left: a local portion of $\Hc$ showing the outgoing boundary. Top right: the CW decomposition of $\Sigma^{\circ}$. Bottom: the CW decomposition $\mathfrak{R}^+$. The brown arcs are the 1-cells in the CW decompositions of the $D^2$'s; the neighboring black curves depict the disks themselves drawn in perspective, which are not part of the CW structure. In some subsequent figures we relax this distinction and draw the brown $1$-cells overlapping these black curves for simplicity.}
    \label{fig:CWDecompFromH}
\end{figure}

\subsection{Orientations of cells}\label{sec:CellOrientations}

In the above CW complexes, we have not yet chosen orientations of cells. We need to do this in order to talk about cellular chain complexes below.

In the Heegaard diagram $\Hc$, choose an orientation for each closed $\alpha$ circle and $\beta$ circle. We use these orientations to orient the cells of $\mathfrak{Y} \setminus \mathfrak{R}^+$ as follows.
\begin{itemize}
    \item The 2-cells of $\mathfrak{Y} \setminus \mathfrak{R}^+$ correspond to the $\alpha$ circles of $\Hc$; the boundary of the 2-cell is the $\alpha$ circle. Orient the 2-cell so its oriented boundary agrees with the given orientation on the $\alpha$ circle.
    \item The 1-cells of $\mathfrak{Y} \setminus \mathfrak{R}^+$ correspond to the $\beta$ circles of $\Hc$. The 3d $\beta$ 1-handle corresponding to a $\beta$ circle has co-core $D^2$ with boundary the $\beta$ circle; it has core $D^1$ and this $D^1$ is the 1-cell of $\mathfrak{Y} \setminus \mathfrak{R}^+$. Orient the co-core so its oriented boundary agrees with the given orientation on the $\beta$ circle. Then orient the core so that if $0$ denotes the center point of the 3d $\beta$ handle and $T_0$ denotes the tangent space at $0$, then (oriented basis for $T_0(\textrm{core})$) $\cup$ (oriented basis for $T_0(\textrm{co-core})$) is an oriented basis for the tangent space of the 3d $\beta$ handle at $0$. See Figure~\ref{fig:CWOrientations}. 
\end{itemize}
Orient the cells of $\mathfrak{R}^+$ arbitrarily; nothing we do will depend on these orientations. 

Given our choices of orientations, we have a cellular chain complex $C_*^{\cell}(\mathfrak{Y},\mathfrak{R}^+)$ that is only nonzero in degrees 1 and 2, and we have a basis of cells for the chain groups $C_1$ and $C_2$. We have canonical identifications
\[
H_1^{\cell}(\mathfrak{Y},\mathfrak{R}^+) \cong H_1^{\sing}(\mathfrak{Y},\mathfrak{R}^+) \cong H_1^{\sing}(Y,R^+).
\]

To have an explicit matrix for the cellular differential $\partial_2$, we choose orderings of the sets of $\alpha$ and $\beta$ circles of $\Hc$. Given these choices of ordered bases, let $M_{\Hc}$ be the matrix for the cellular differential $\partial_2$ in this cellular chain complex. Note that $M_{\Hc}$ is a presentation matrix for $H_1(Y,R^+)$. 

\begin{figure}
    \centering
    \begin{overpic}[width=0.8\textwidth]
    {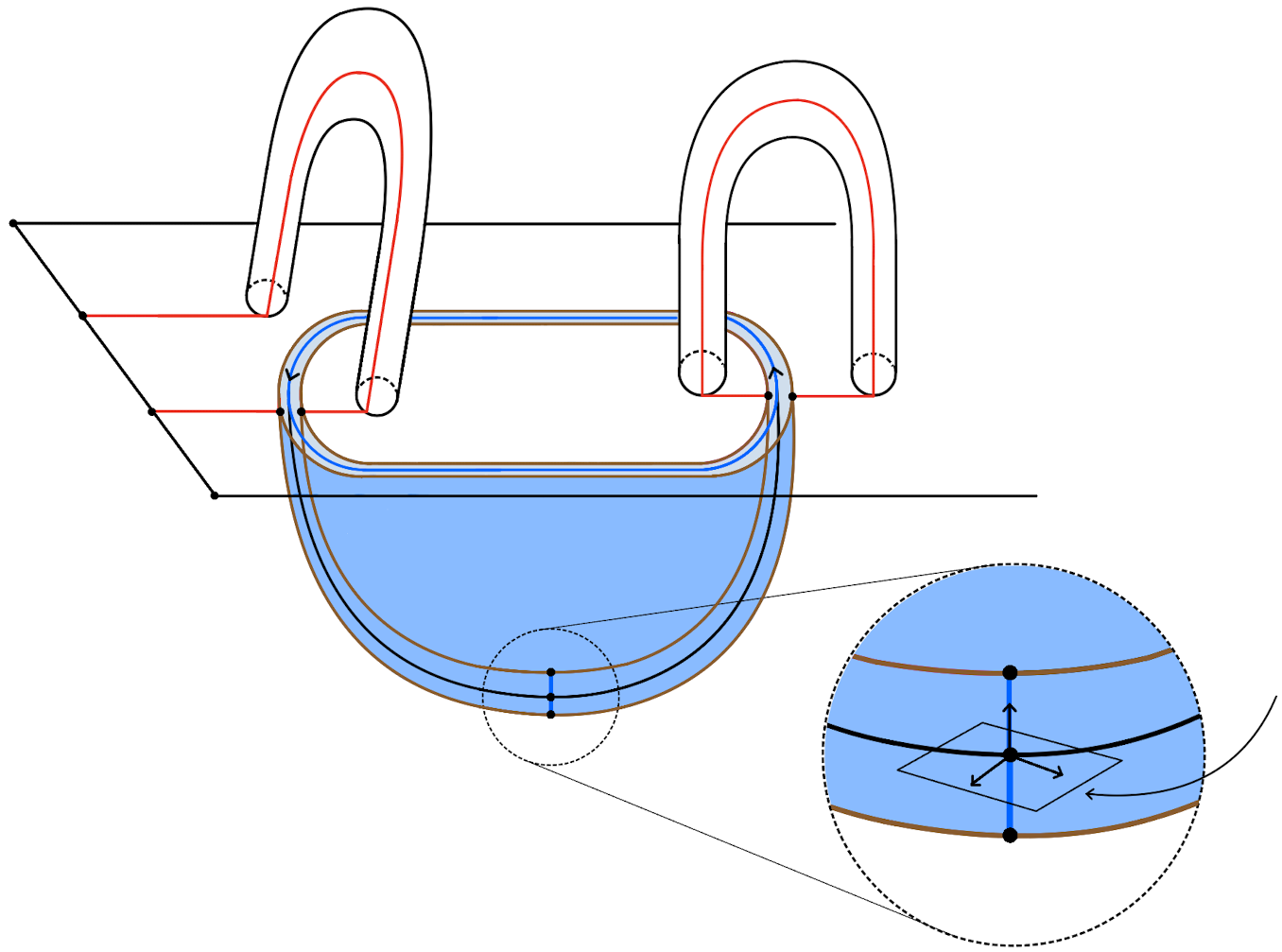}
    \put (3.5,60) {$R^+ \times \{1\}$} 
    \put (42,40.5) {$\beta$}
    \put (80.3,18.3) {\scalebox{0.75}{$T_0(\textrm{core})$}} 
    \put (95,22) {\scalebox{0.87}{$T_0(\textrm{co-core})$}}
    \end{overpic}
    \caption{The 1-handle corresponding to the $\beta$ circle of $\Hc$ shown in Figure~\ref{fig:CWDecompFromH} is shaded in blue; the attaching boundary consists of the two brown circles and its co-core is $\beta$. If, as pictured, the chosen orientation on the $\beta$ circle is counterclockwise when viewed from above, the 1-handle's co-core must be oriented using the inward–pointing normal, so that the induced boundary orientation agrees with this given orientation on the $\beta$ circle.}
    \label{fig:CWOrientations}
\end{figure}
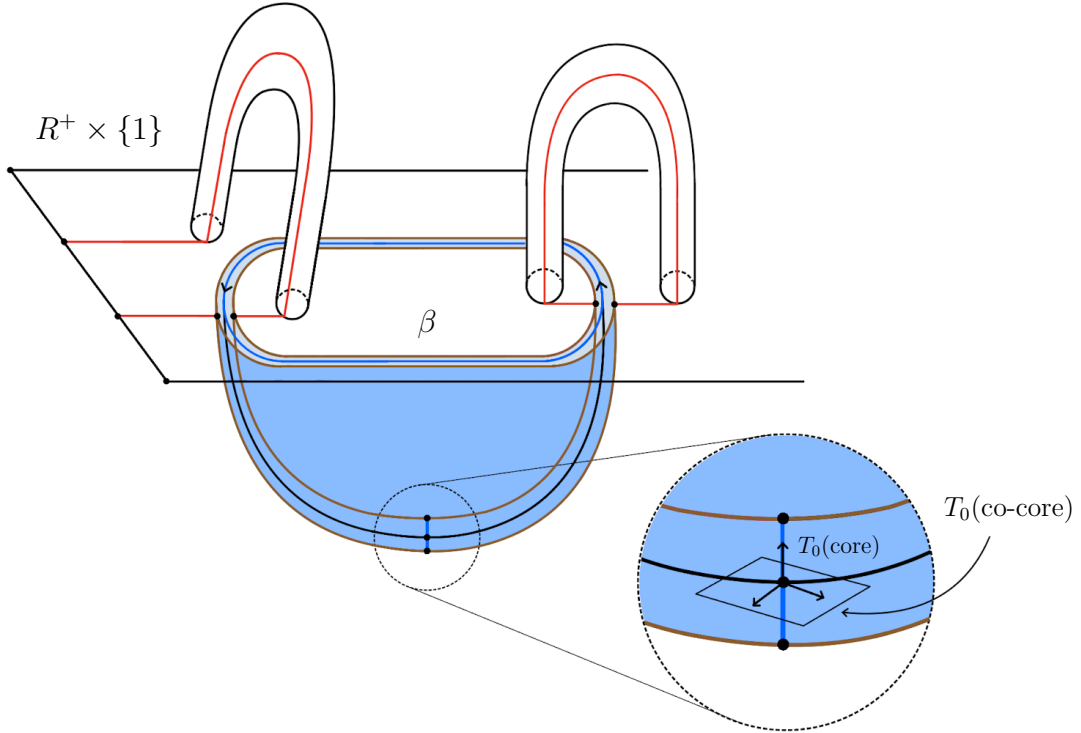

\begin{proposition}\label{prop:SignsInCellularDiff}
    In the matrix $M_{\Hc}$ for the cellular differential $\partial_2$ for $C_*^{\cell}(\mathfrak{Y},\mathfrak{R}^+)$, the entry corresponding to an input $\alpha$ circle $C$ and an output $\beta$ circle $C'$ is the sum over intersection points $x \in C \cap C'$ of the local intersection sign $C'_x \cdot C_x$ (i.e. ``$+1$ if $\beta \cdot \alpha = +1$ and $-1$ if $\beta \cdot \alpha = -1$'').
\end{proposition}

\begin{proof}
    On general principles, the matrix entry is the sum over $x \in C \cap C'$ of some local sign associated to $x$; we check that the sign is as described. The argument is shown in Figure~\ref{fig:CWIntersectionSigns}; the left side shows a local intersection of the type $\beta \cdot \alpha = +1$ and checks that the boundary of the $\alpha$ 2-cell traverses the $\beta$ 1-cell in the positive direction. The right side shows a local intersection of the type $\beta \cdot \alpha = -1$ and checks that the boundary of the $\alpha$ 2-cell traverses the $\beta$ 1-cell in the negative direction.
\end{proof}

\begin{figure}
    \centering
    \begin{overpic}[width=0.78\textwidth]
    {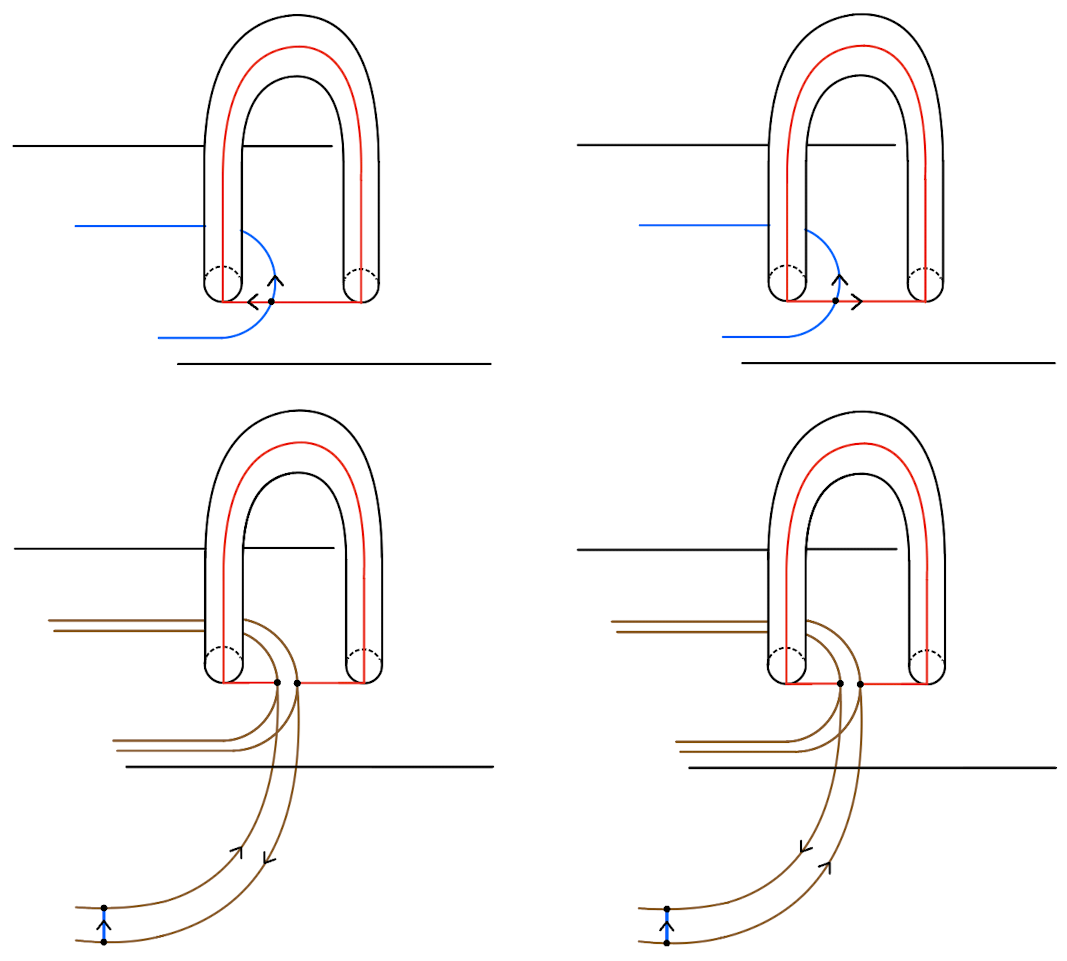}
    \put (26,57.5) {\scalebox{0.9}{$\beta \cdot \alpha = +1$}}
    \put (80,57.5) {\scalebox{0.9}{$\beta \cdot \alpha = -1$}}
    \end{overpic}
    \vspace{0.5em}
    \caption{The visual argument for Proposition~\ref{prop:SignsInCellularDiff} on a local portion of the CW complex $\mathfrak{Y}$ associated to $\Hc$ from Figures~\ref{fig:CWDecompFromH} and \ref{fig:CWOrientations}.}
    \label{fig:CWIntersectionSigns}
\end{figure}

Now assume $\Hc = \Hc_{\norm}$ as in Corollary~\ref{cor:NormalizedHeegaardDiagrams}. When choosing the ordering of $\beta$ circles of $\Hc_{\norm}$, we will require that the $\beta$ circles of $\Hc_{\norm}$ glued from outgoing $\beta$ arcs of $\Hc'$ come first, followed by $\beta$ circles of $\Hc'$, followed by $\beta$ circles of $\Hc_{\norm}$ glued from incoming $\beta$ arcs of $\Hc'$. We can write $M_{\Hc_{\norm}}$ schematically as 
\[
M_{\Hc_{\norm}} = \kbordermatrix{ 
& \alpha^c  \\
\beta^{\out} & * \\
\beta^c & * \\
\beta^{\inrm} & * }.
\]
In Section~\ref{sec:SuturedFNTQFT} we will encounter the submatrix $\kbordermatrix{ 
& \alpha^c  \\
\beta^c & * \\}$ of $M_{\Hc_{\norm}}$.

\begin{proposition}\label{prop:SubmatrixPresentation}
    The submatrix $\kbordermatrix{ 
     & \alpha^c  \\
     \beta^c & * \\}$ of $M_{\Hc_{\norm}}$ is a presentation matrix for $H_1(Y,F_1 \cup R^+ \cup F_0)$.
\end{proposition}

\begin{proof}
    View $Y$ as the sutured cobordism associated to the $\beta$-$\beta$ bordered sutured Heegaard diagram $\Hc'$ and apply the above analysis, noting that the construction of $Y$ from $\Hc'$ presents $Y$ as a sutured handle complex relative to $(F_1 \cup R^+ \cup F_0) \times I$.
\end{proof}

\subsection{Normalized choices and subcomplexes for \texorpdfstring{$F_i$}{Fi}}\label{sec:CellularChainCx}

Continue to assume $\Hc = \Hc_{\norm}$ as in Corollary~\ref{cor:NormalizedHeegaardDiagrams}. In this case we will define subcomplexes $\mathfrak{F}_i \subset \mathfrak{Y}$ and $\mathfrak{S}^+_i \subset \mathfrak{F}_i$ such that $\mathfrak{F}_i \cap \mathfrak{R}^+ = \mathfrak{S}^+_i$ and $(\mathfrak{F}_i,\mathfrak{S}^+_i)$ is homotopy equivalent to $(F_i,S^+_i)$. 

First note that since $\partial \Sigma^{\circ}$ and thus $\partial \Sigma$ are subcomplexes of $\mathfrak{R}^+$, and all points of $\Lambda$ are 0-cells of $\partial \Sigma$, we already have CW decompositions of $S^+_0$ and $S^+_1$ as subcomplexes of $\mathfrak{R}^+$. We will define $\mathfrak{S}^+_i$ below by enlarging these subcomplexes slightly.

Now, since we took $\Hc = \Hc_{\norm}$, each incoming $\alpha$ arc of $\Hc_{\norm}$ intersects exactly one $\beta$ circle in the Heegaard surface $\Sigma$. Thus, the intersection between the $\alpha$ arc and $\Sigma^{\circ}$ consists of two intervals, each of which is a 1-cell of $\mathfrak{R}^+$ with one endpoint in $S^+$ and the other endpoint on one of the two circles $C,C' \subset \partial \Sigma^{\circ}$ obtained from parallel pushoffs of the $\beta$ circle intersecting the $\alpha$ arc. There are two more 1-cells of $\mathfrak{R}^+$ that live in the $D^2$s that are glued to $\Sigma^{\circ}$ along $C$ and $C'$ to form $R^+$, and run from the centers of the $D^2$s to the endpoints of the cut $\alpha$ arc on $C$ and $C'$. Finally, there is a 1-cell of $\mathfrak{Y} \setminus \mathfrak{R}^+$ connecting the centers of the two $D^2$s.

Let $\mathfrak{F}_0$ denote the union of all these cells and their boundary 0-cells, over all incoming $\alpha$ arcs of $\Hc$, together with the given CW decomposition of $S^+_0$ inside $\mathfrak{R}^+$. Let $\mathfrak{S}^+_0$ denote the union of all cells of $\mathfrak{F}_0$ except the ones in $\mathfrak{Y} \setminus \mathfrak{R}^+$ (i.e. let $\mathfrak{S}^+_0 = \mathfrak{F}_0 \cap \mathfrak{R}^+$). Note that $\mathfrak{S}^+_0$ contains our original CW decomposition of $S^+_0$ along with four extra 1-cells for each incoming $\alpha$ arc of $\Hc_{\norm}$. Define $\mathfrak{F}_1$ similarly, using outgoing $\alpha$ arcs of $\Hc_{\norm}$, and let $\mathfrak{S}^+_1 = \mathfrak{F}_1 \cap \mathfrak{R}^+$. See Figure~\ref{fig:CWDecompNormalized}. 

\begin{figure}
    \centering
    \vspace{1em}
    \begin{overpic}[width=0.6\textwidth]
    {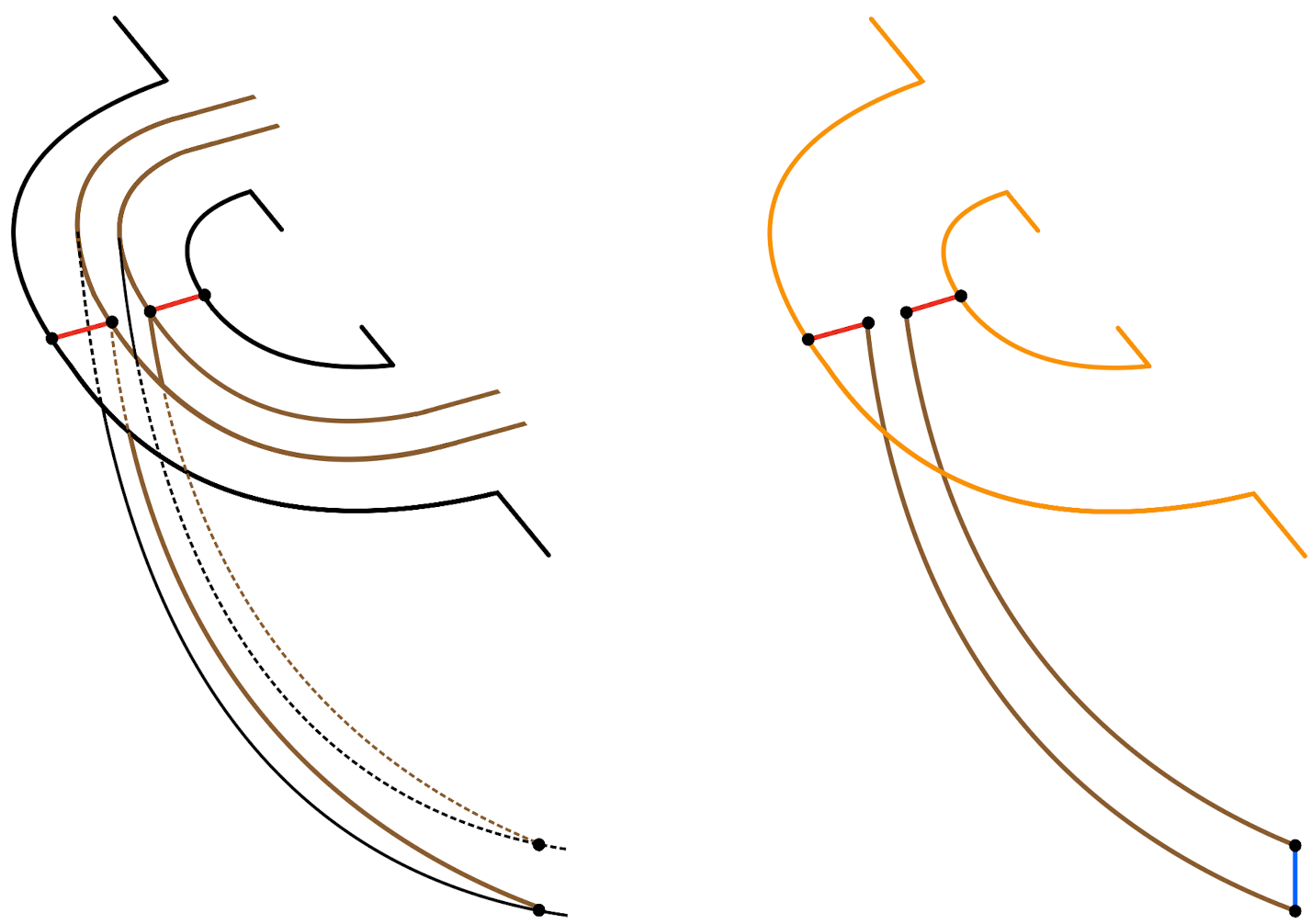} 
    \put (-5,35) {$\Sigma_{\norm}^{\circ}$} 
    \put (56.5,35) {$\mathfrak{F}^+_1$} 
    \put (23,51) {\rotatebox{-45}{$\ldots$}}
    \put (80.5,50.5) {\rotatebox{-45}{$\ldots$}}
    \end{overpic}
    \vspace{0.1in}
    \caption{Left: a local view of the CW decomposition of $\Sigma_{\norm}^{\circ}$ near an outgoing $\alpha$ arc. Right: the CW decompositions $\mathfrak{F}^+_1$ and $\mathfrak{S}^+_1$. All 0- and 1-cells shown lie in $\mathfrak{F}^+_1$; the subcomplex $\mathfrak{S}^+_1$ is obtained by deleting the blue 1-cell.}
    \label{fig:CWDecompNormalized}
\end{figure}

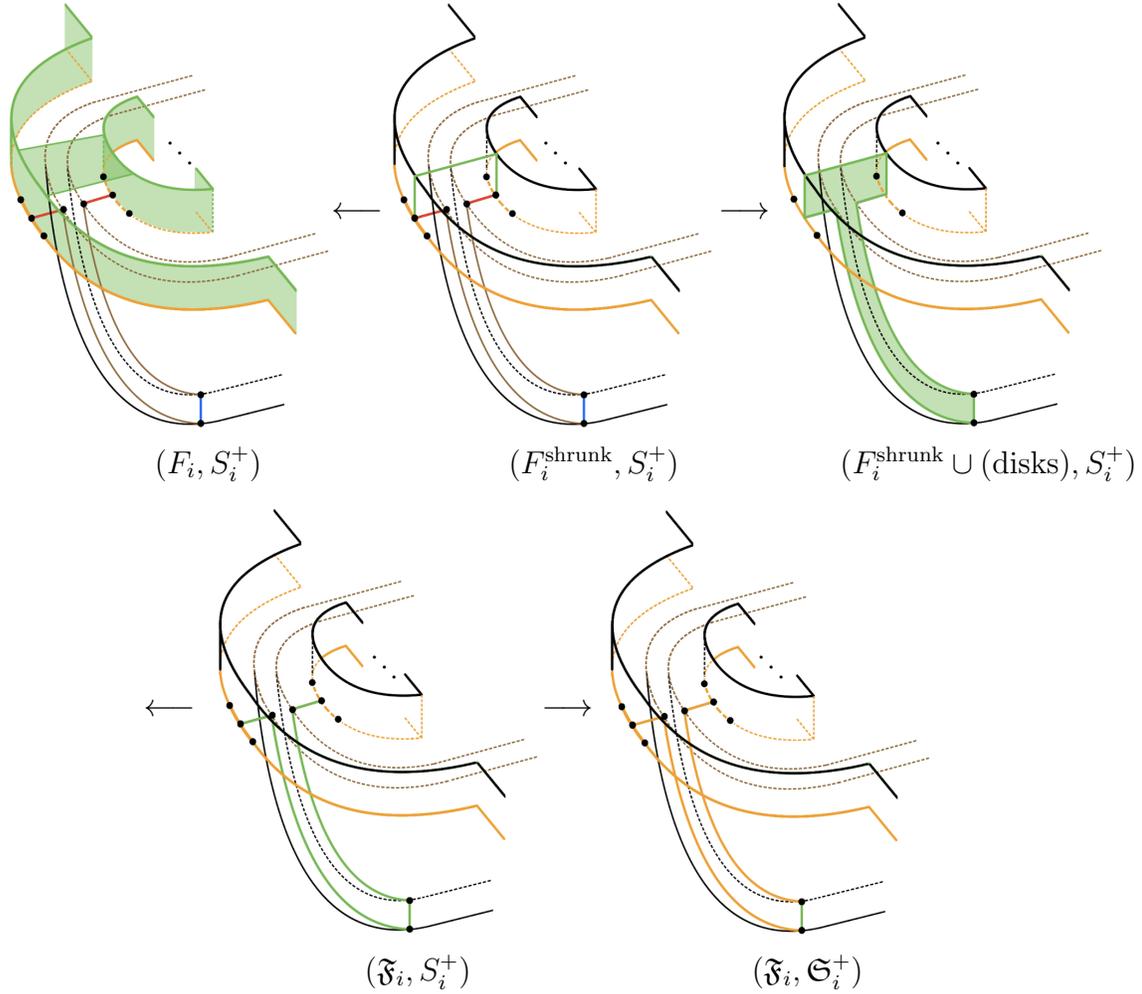
\begin{figure}
    \centering
    \begin{overpic}[width=0.89\textwidth]
    {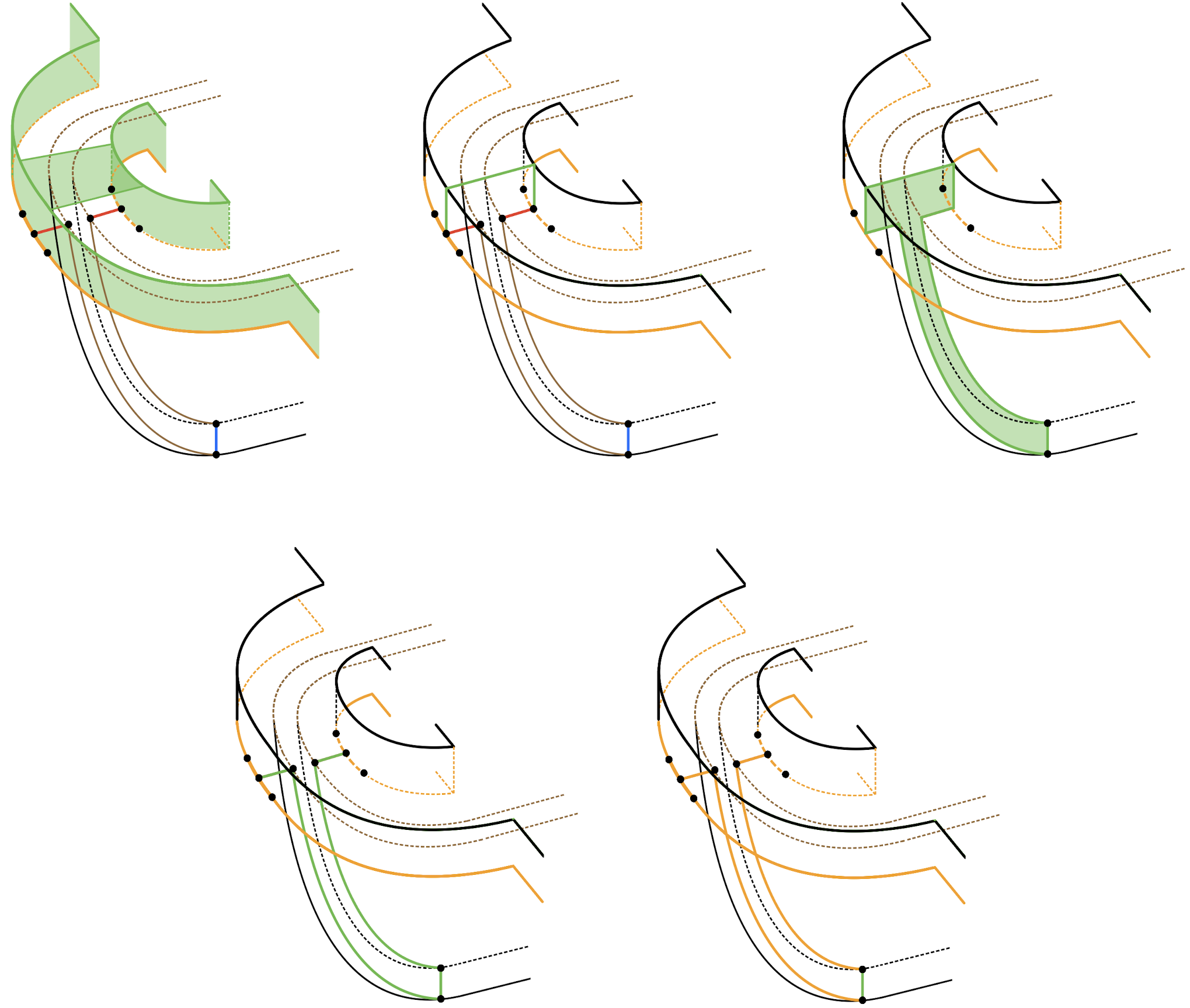}
    \put (14,42) {$(F_i,S^+_i)$}
    \put (30,65) {$\longleftarrow$}
    \put (46,42) {$(F_i^{\mathrm{shrunk}},S^+_i)$}
    \put (65,65) {$\longrightarrow$}
    \put (76,42) {$(F_i^{\mathrm{shrunk}} \cup (\mathrm{disks}),S^+_i)$}
    
    \put (13,20) {$\longleftarrow$}
    \put (33,-4) {$(\mathfrak{F}_i,S^+_i)$}
    \put (49,20) {$\longrightarrow$}
    \put (68,-4) {$(\mathfrak{F}_i,\mathfrak{S}^+_i)$}

    \put (15,72) {\rotatebox{-45}{$\ldots$}}
    \put (49,72) {\rotatebox{-45}{$\ldots$}}
    \put (84.5,72) {\rotatebox{-45}{$\ldots$}}
    \put (33.5,25.5) {\rotatebox{-45}{$\ldots$}}
    \put (68.5,25.5) {\rotatebox{-45}{$\ldots$}}
    \end{overpic}
    \vspace{1em}
    \caption{Inclusions of pairs. In each case, the larger space in the pair is drawn in green (not to be confused with our convention of drawing $\Gamma$ in green), and the smaller space is drawn in orange.}
    \label{fig:FInclusions}
\end{figure}

Consider the inclusions of pairs shown in Figure~\ref{fig:FInclusions}. All of these inclusions induce isomorphisms on relative first homology groups, and we have a commutative diagram

\begin{equation}\label{eq:FSYRCommutativeDiag}
\xymatrix{
H_1(F_i, S^+_i) \ar[r]^{i_*} & H_1(Y,R^+) \\
H_1(F_i^{\,\mathrm{shrunk}},S^+_i) \ar[u]^{i_*} \ar[d]_{i_*} \ar[r]^{i_*} & H_1(Y,R^+) \ar[u]_{=} \ar[d]^{=} \\
H_1(F_i^{\,\mathrm{shrunk}} \cup (\mathrm{disks}), S^+_i) \ar[r]^-{i_*} & H_1(Y,R^+) \\
H_1(\mathfrak{F}_i, S^+_i) \ar[u]^{i_*} \ar[d]_{i_*} \ar[r]^{i_*} & H_1(\mathfrak{Y},\mathfrak{R}^+) \ar[u]_{i_*} \ar[d]^{=} \\
H_1(\mathfrak{F}_i,\mathfrak{S}^+_i) \ar[r]^{i_*} & H_1(\mathfrak{Y},\mathfrak{R}^+)
}
\end{equation}

\begin{corollary}\label{cor:SingCellSquareCommutes}
    The square
    \[
    \xymatrix{
        H_1(F_i, S^+_i) \ar[r]^{i_*}  & H_1(Y, R^+)  \\
        H_1(\mathfrak{F}_i, \mathfrak{S}^+_i) \ar[r]_{i_*} \ar[u]^{\cong} & H_1(\mathfrak{Y}, \mathfrak{R}^+) \ar[u]_{i_*}
        }
    \]
    commutes, where the map on the left is the composition of the isomorphisms on the left side of diagram~\eqref{eq:FSYRCommutativeDiag}.
\end{corollary}

We can describe the composite isomorphism between $H_1(F_i,S^+_i)$ and $H_1(\mathfrak{F}_i,\mathfrak{S}^+_i)$ on the left side of diagram~\eqref{eq:FSYRCommutativeDiag} as follows: our usual basis elements of $H_1(F_i,S^+_i)$ are the classes of the arcs $\gamma$ in $(F_i,S^+_i)$ corresponding to the matching arcs of the diagram $\mathcal{Z}_i$. The map from the top-left of diagram~\eqref{eq:FSYRCommutativeDiag} to the bottom-left of diagram~\eqref{eq:FSYRCommutativeDiag} sends such a basis element $\gamma$ to the homology class of the corresponding one-cell in $H_1(\mathfrak{F}_i,\mathfrak{S}^+_i)$, times an appropriate sign. If we use the orientations on matching arcs of $\Zc_i$ to choose orientations on $\beta$ circles of $\Hc_{\norm}$ as in the below proposition, we can ensure the sign will always be $-1$.

\begin{proposition}\label{prop:ExpandingBasisElts}
    Choose an orientation for each matching arc of $\Zc_0$ and $\Zc_1$. Orient the $\alpha$-arcs of the Heegaard diagram $\Hc_{\norm} = \Hc^{\alpha \beta}_{1/2} \cup_{\Zc_1^*} \Hc' \cup_{\Zc_0^*} \Hc^{\beta \alpha}_{1/2}$ for $(Y,\Gamma)$ such that:
    \begin{itemize}
        \item on the left side of the diagram (outgoing), the $\alpha$ arcs are oriented the same as the matching arcs of $\Zc_1$.
        \item on the right side of the diagram (incoming), the $\alpha$ arcs are oriented oppositely to the matching arcs of $\Zc_0$.
    \end{itemize}
    Also assume that the $\beta$ circles of $\Hc_{\norm}$ are oriented such that
    \begin{itemize}
        \item on the left side of the diagram (outgoing), the local intersections in $\Hc^{\alpha \beta}_{1/2}$ have $\beta \cdot \alpha^{a,\out} = -1$;
        \item on the right side of the diagram (incoming), the local intersections in $\Hc^{\beta \alpha}_{1/2}$ have $\beta \cdot \alpha^{a,\inrm} = +1$.
    \end{itemize}
    Then the basis element $\gamma$ of $H_1(F_i,S^+_i)$ corresponding to an oriented matching arc of $\Zc_i$ gets sent, under the isomorphism $H_1(F_i,S^+_i) \cong H_1(\mathfrak{F}_i,\mathfrak{S}^+_i)$ from the left side of diagram~\eqref{eq:FSYRCommutativeDiag}, to $-1$ times the homology class of the corresponding one-cell of $\mathfrak{F}_i \setminus \mathfrak{S}^+_i$.
\end{proposition}

\begin{proof}
    The argument is similar to the proof of Proposition~\ref{prop:SignsInCellularDiff}, except that instead of a closed $\alpha$ circle we have an $\alpha$ arc. The same diagram, Figure~\ref{fig:CWIntersectionSigns}, shows that on the left/outgoing side of the diagram where local intersections $\beta \cdot \alpha$ are $-1$, the $\alpha$ arcs in the Heegaard diagram are oriented oppositely to the relevant $\beta$ 1-cells. Since the $\alpha$ arcs on this side are oriented the same as the matching arcs of $\Zc_1$, the sign in the statement of the proposition is $-1$ in this case. 
    
    The figure also shows that on the right/incoming side of the diagram where local intersections $\beta \cdot \alpha$ are $+1$, the $\alpha$ arcs in the Heegaard diagram are oriented the same as the relevant $\beta$ 1-cells. Since the $\alpha$ arcs on this side are oriented oppositely to the matching arcs of $\Zc_0$, the sign in the statement of the proposition is also $-1$ in this case. 
\end{proof}

\section{\texorpdfstring{$[\BSDA]$}{[BSDA]} without \texorpdfstring{$\Spinc$}{Spin-c} structures}\label{sec:BSDAwithoutSpinc}

\begin{definition}
    Let $\Hc = (\Sigma, \boldsymbol{\alpha}, \boldsymbol{\beta}, \mathcal{Z}_0, \mathcal{Z}_1, \psi)$ be a  bordered sutured Heegaard diagram. A \emph{generator} for $\Hc$ is a set $\mathbf{x}$ of intersection points in $\boldsymbol{\alpha} \cap \boldsymbol{\beta}$ such that each $\alpha$-circle and $\beta$-circle is occupied exactly once and each $\alpha$-arc and $\beta$-arc is occupied at most once. We call the set of generators $\mathfrak{S}(\Hc)$.
    
    For $\mathbf{x} \in \mathfrak{S}(\Hc)$, we let $o_R(\mathbf{x})$ be the set of incoming arcs ($\alpha$ or $\beta$ as appropriate) of $\Hc$ that are occupied in $x$, and we let $\overline{o}_R(\mathbf{x})$ be the set of incoming arcs that are unoccupied in $\mathbf{x}$. Similarly, we let $o_L(\mathbf{x})$ be the set of outgoing arcs that are occupied in $\mathbf{x}$, and we let $\overline{o}_L(\mathbf{x})$ be the set of outgoing arcs that are unoccupied in $\mathbf{x}$. By a standard abuse of notation, we also use $o_R(\mathbf{x}), \overline{o}_R(\mathbf{x}), o_L(\mathbf{x})$, and $\overline{o}_L(\mathbf{x})$ to denote the corresponding index sets. 
\end{definition}

\begin{remark}\label{rem:HDconstants}
    In the below definition and elsewhere, we will use various constants associated to an $\alpha$-$\alpha$ bordered sutured Heegaard diagram $\Hc$; we list these here for convenience.
    \begin{itemize}
        \item For $i = 0,1$, we let $n_i$ be the rank of $H_1(F_i,S^+_i)$.
        \item For a basis element of $\wedge^* H_1(F_0,S^+_0)$, we let $k$ be the number of wedge factors.
        \item For a basis element of of $\wedge^* H_1(F_1,S^+_1)$, we let $l$ be the number of wedge factors.
        \item We let $a$ and $b$ be the number of $\alpha$ and $\beta$-circles of $\Hc$ respectively.
        \item We let $c = n_1 + \chi(Y,R^+)$; we will show that the map $[\BSDA(Y,\Gamma; \Xi)]^{\Z}_{\mathrm{comb}}$ we define below has degree $c$.
    \end{itemize}
\end{remark}

\begin{definition}\label{def:BSDAZComb}
    Let $(Y,\Gamma)$ be a sutured cobordism from $(F_0,\Lambda_0)$ to $(F_1,\Lambda_1)$. Define a map
    \[[\BSDA(Y,\Gamma; \Xi)]^{\Z}_{\mathrm{comb}} \colon \wedge^* H_1(F_0,S^+_0) \to \wedge^* H_1(F_1,S^+_1)
    \]
    as follows, where $\Xi$ denotes the \emph{set of choices} made below.  
    \begin{itemize}
        \item Choose $\alpha$-arc diagrams $\Zc_i$ representing $(F_i,\Lambda_i)$, and choose an $\alpha$-$\alpha$ bordered sutured Heegaard diagram $\Hc = (\Sigma,\boldsymbol{\alpha},\boldsymbol{\beta},\Zc_0,\Zc_1,\psi)$ representing $(Y,\Gamma)$.
        \item Choose an ordering of the matching arcs of $\Zc_0$, and choose an orientation of each matching arc. Let $\gamma^{\inrm}_i$ denote the element of $H_1(F_0,S^+_0)$ represented by the $i^{th}$ matching arc with its chosen orientation. For each $k$ between $0$ and $n_0$, the elements
        \[
        \gamma^{\inrm}_I = \gamma^{\inrm}_{i_1} \wedge \cdots \wedge \gamma^{\inrm}_{i_k}
        \]
        for $I = (i_1,\ldots,i_k)$, $1 \leq i_1 < \cdots < i_k \leq n_0$, form a basis for $\wedge^k H_1(F_0,S^+_0)$ over $\Z$.
        \item Choose an ordering of the matching arcs of $\Zc_1$, and choose an orientation of each matching arc. Let $\gamma^{\out}_i$ denote the element of $H_1(F_1,S^+_1)$ represented by the $i^{th}$ matching arc with its chosen orientation. For each $l$ between $0$ and $n_1$, the elements
        \[
        \gamma^{\out}_J = \gamma^{\out}_{j_1} \wedge \cdots \wedge \gamma^{\out}_{j_l}
        \]
        for $J = (j_1,\ldots,j_l)$, $1 \leq j_1 < \cdots < j_l \leq n_1$, form a basis for $\wedge^l H_1(F_1,S^+_1)$ over $\Z$.
        \item Choose an ordering and orientations of the closed $\alpha$-circles $\alpha^c_1,\ldots,\alpha^c_a$ of $\Hc$, and similarly for the $\beta$-circles $\beta_1,\ldots,\beta_b$.
        \item As in Proposition~\ref{prop:ExpandingBasisElts}, orient the incoming $\alpha$-arcs $\alpha^{a,\mathrm{in}}_i$ of $\Hc$ oppositely to the matching arcs of $\Zc_0$. Orient the outgoing $\alpha$-arcs $\alpha^{a,\mathrm{out}}_i$ of $\Hc$ the same as the matching arcs of $\Zc_1$.
        \item Order the full set of $\alpha$ curves of $\Hc$ by
        \[ (\alpha^{a,\mathrm{out}}_1,\ldots,\alpha^{a,\mathrm{out}}_{n_1},\alpha^c_1,\ldots,\alpha^c_a, \alpha^{a,\mathrm{in}}_1,\ldots,\alpha^{a,\mathrm{in}}_{n_0})
        \]
        (i.e. outgoing $\alpha$-arcs first, then $\alpha$-circles, then incoming $\alpha$-arcs).
    \end{itemize}
    Based on these choices $\Xi$, which we will later show affect the construction only up to an overall sign, we now define the map $[\BSDA(Y,\Gamma; \Xi)]^{\Z}_{\mathrm{comb}}$ (defined strictly, not up to sign) as follows.    
    \begin{itemize}
        \item Given a generator $\mathbf{x} \in \mathfrak{S}(\Hc)$, define $\sigma_{\mathbf{x}}$ to be the permutation between ordered sets of size $b$ from $\{\beta$-circles of $\Hc\}$ to $\{\alpha$-curves of $\Hc$ occupied in $\mathbf{x}\}$ sending $\beta_i$ to the $\alpha$-curve sharing a point of $\mathbf{x}$ with $\beta_i$.
        \item Let $k$ be the number of occupied incoming $\alpha$-arcs in $\x$ and let $l$ be the number of unoccupied outgoing $\alpha$-arcs in $\x$. 
        \item Let $S_1 \sqcup S_2$ be any partition of $\{1,\ldots,n_i\}$ such that $S_1$ and $S_2$ inherit the induced orderings from $\{1,\ldots,n_i\}$. Then define $\sigma_{S_1,S_2 \leftrightarrow \std}$ as the permutation that puts the elements of $S_1$ in increasing order in positions $1,\ldots,|S_1|$, and puts the elements of $S_2$ in increasing order in positions $|S_2|+1,\ldots,n_i$. In particular, $\sigma_{\overline{o}_L(\mathbf{x}) o_L(\mathbf{x}) \leftrightarrow \std}$ is the permutation of $\{1,\ldots, n_1\}$ sending the indices of unoccupied outgoing $\alpha$-arcs (i.e. $\overline{o}_L(\x)$) in order to $\{1,\ldots,l\}$ and sending the indices of occupied outgoing $\alpha$-arcs (i.e. $o_L(\x)$) in order to $\{l+1,\ldots,n_1\}$. 
        \item For $x \in \mathbf{x}$, let $i(x)$ denote the element of $\Z/2\Z$ defined by 
        \[
        i(x)=\begin{cases}
            0 & \text{if} \quad \beta_x \cdot \alpha_x=1 \\
            1 & \text{if} \quad \beta_x \cdot \alpha_x=-1
        \end{cases}
        \]
        where $\beta_x \cdot \alpha_x$ is the local intersection sign of $x$, which is $1$ if $(\beta_x,\alpha_x)$ gives an oriented basis for the tangent plane of $\Sigma$ at $x$, and is $-1$ if $(\beta_x,\alpha_x)$ gives an anti-oriented basis. Note that $(-1)^{i(x)}$ is the local intersection sign of $x$.
        \item Let $\mathrm{gr}_{\DA}(\mathbf{x}) \in \Z/2\Z$ be defined by
        \begin{equation}\label{eq:DAGrading}
        \mathrm{gr}_{\DA}(\mathbf{x}) =  \sum_{x \in \mathbf{x}} i(x)  + \mathrm{inv} (\sigma_{\mathbf{x}})  + \mathrm{inv}(\sigma_{\overline{o}_L(\x) o_L(\x) \leftrightarrow \std}) + ak + n_1 k  \mod 2
        \end{equation}
        where $\mathrm{inv}$ denotes the number of inversions of each permutation. 
        \item Finally, we let $[\BSDA(Y,\Gamma; \Xi)]^{\Z}_{\comb}$ be the $\Z$-linear map determined by its action on a basis element $\gamma^{\inrm}_I$ of $\wedge^k H_1(F_0,S^+_0)$, given by
        \begin{equation}\label{eq:DefOfBSDA}
        \gamma^{\inrm}_I \mapsto \sum_{\substack{\mathbf{x} \in \mathfrak{S}(\Hc): \\o_R(\mathbf{x}) = I, \\ \overline{o}_L(\mathbf{x}) = J}} (-1)^{\mathrm{gr}_{\DA}(\mathbf{x})} \, \gamma^{\out}_J,
        \end{equation}
        where we identify $o_R(\mathbf{x})$ and $\overline{o}_L(\mathbf{x})$ with subsets of $\{1,\ldots,n_0\}$ and $\{1,\ldots,n_1\}$ respectively. 
    \end{itemize}
\end{definition}

\begin{remark}
    The notation $[\BSDA(Y,\Gamma;\Xi)]^{\Z}_{\comb}$ should not be taken to mean that an underlying bordered Floer bimodule $\BSDA(Y,\Gamma;\Xi)$ has been defined in the literature; such bimodules have only been defined given more restrictive topological assumptions than we impose here.
\end{remark}

\begin{remark}
    Equation~\eqref{eq:DAGrading} should be compared with \cite[Definition 3.20]{HLW}. In Hom--Lidman--Watson's setting, $n_1 = 2g_1$ is always even, so the term $n_1 k$ from \eqref{eq:DAGrading} does not appear in their results. We also note that, in our notation, the term $ak$ corresponds to their term $t(g - k_L - k_R)$. There, $g - k_L - k_R$ counts the number of $\alpha$–circles of $\Hc$, while $t = |\text{Im}(\sigma)\cap A|$ counts the number of occupied right $\alpha$–arcs for any generator $\x$. 
\end{remark}

\begin{remark}
    We could in principle define $i(x)$ by replacing the local intersection sign $\beta_x \cdot \alpha_x$ with its negative $\alpha_x \cdot \beta_x$. Since $|\x|=b$, this change would be harmless up to an overall sign of $(-1)^b$. We choose to adopt the convention $\beta_x\cdot\alpha_x$, motivated partly by the swapping of $\alpha$ and $\beta$ curves that appears in Section~\ref{sec:RelWithSFH} when relating to the conventions of \cite{FJR}. Note that fixing one convention or the other is important to ensure that the definition of $[\BSDA(Y,\Gamma;\Xi)]^{\Z}_{\comb}$ is defined on-the-nose, not just up to sign, once the choices $\Xi$ are made.
\end{remark}

\begin{proposition}\label{prop:DegreeOfBSDA}
    The degree of the map $[\BSDA(Y,\Gamma; \Xi)]^{\Z}_{\comb}$ is $c := n_1 + \chi(Y,R^+)$.
\end{proposition}

\begin{proof}
    Let $\gamma^{\inrm}_I$ be a basis element of $\wedge^* H_1(F_0, S^+_0)$ with $|I| = k$. Let $\x \in \mathfrak{S}(\Hc)$ be any generator with $o_R(\x) = I$, and let $J = \overline{o}_L(\x)$ and $l = |J|$. We will show that $l = k + c$.
    
    The number of occupied outgoing $\alpha$-arcs in $\x$ is $n_1 - l$, and the number of occupied incoming $\alpha$-arcs is $k$. Thus, the total number of occupied $\alpha$-arcs in $\x$ is $n_1 - l + k$. This number of occupied $\alpha$-arcs is the same as $b - a$, the number of $\beta$-circles of $\Hc$ minus the number of $\alpha$-circles. By Section~\ref{sec:CWDecomp}, $\Hc$ gives us a CW decomposition of a space $\mathfrak{Y}$ homotopy equivalent to $Y$, relative to $R^+$, with only 1-cells and 2-cells, where the 1-cells correspond to $\beta$-circles of $\Hc$ and the 2-cells correspond to $\alpha$-circles of $\Hc$. It follows that $\chi(Y,R^+) = a - b$, so
    \[
    n_1 - l + k = -\chi(Y,R^+)
    \]
    and thus $l = k + n_1 + \chi(Y,R^+) = k + c$.
\end{proof}

\begin{remark}
    To match Hom--Lidman--Watson more closely, we could replace $\gamma^{\inrm}_I$ with $(-1)^{|I|} \gamma^{\inrm}_I$ and $\gamma^{\out}_J$ with $(-1)^{|J|} \gamma^{\out}_J$ in equation~\eqref{eq:DefOfBSDA}; see \cite[Proof of Theorem 4.13]{HLW}. By Proposition~\ref{prop:DegreeOfBSDA}, the result would differ by the overall sign $(-1)^c$, and we will disregard overall signs in most situations.
\end{remark}

\begin{example}\label{ex:BSDAForOrdinary}
    Let $(Y,\Gamma)$ be a sutured 3-manifold, i.e. a sutured cobordism from $\emptyset$ to $\emptyset$. The choices $\Xi$ of Definition~\ref{def:BSDAZComb} in this case amount to a sutured Heegaard diagram $\Hc$ (as in Definition~\ref{def:OtherTypesOfHeegaardDiagrams}(c)) representing $(Y,\Gamma)$, an orientation of each $\alpha$ and $\beta$ circle of $\Hc$, and orderings of the sets of $\alpha$ circles and $\beta$ circles of $\Hc$.

    Given such choices $\Xi$, we have a map $[\BSDA(Y,\Gamma; \Xi)]^{\Z}_{\comb}$ from $\Z$ to $\Z$ (identifying $\Z$ with $\wedge^0 (0)$), which we can view as a single integer.  
    This integer is the sum over all $\x \in \mathfrak{S}(\Hc)$ of $(-1)^{\mathrm{gr}_{\DA}(\x)}$, and such $\x$ naturally correspond to bijections between the sets of $\alpha$ circles and $\beta$ circles of $\Hc$. 

    If the numbers $a$ and $b$ of $\alpha$ and $\beta$ circles of $\Hc$ are not equal, then there are no such bijections, so $[\BSDA(Y,\Gamma; \Xi)]^{\Z}_{\comb} = 0$. Otherwise, since we have chosen orderings on the sets of $\alpha$ and $\beta$ circles, such bijections can be identified with permutations on $a$ letters. For $\x \in \mathfrak{S}(\Hc)$, the corresponding permutation is $\sigma_{\x}$ as in Definition~\ref{def:BSDAZComb}. We have
    \[
    \mathrm{gr}_{\DA}(\x) = \sum_{x \in \x} i(x) + \mathrm{inv}(\sigma_{\x})
    \]
    (the other terms in $\mathrm{gr}_{\DA}(\x)$ are zero) and thus
    \[
    (-1)^{\mathrm{gr}_{\DA}(\x)} = \left(\prod_{x \in \x} (-1)^{i(x)} \right) (-1)^{\mathrm{inv}(\sigma_{\x})};
    \]
    note as before that if $x \in \beta_x \cap \alpha_x$ then $(-1)^{i(x)} = \beta_x \cdot \alpha_x$. 
    
    Recall that by Proposition~\ref{prop:SignsInCellularDiff}, the matrix $M_{\Hc}$ representing the cellular differential $\partial_2$ of $C_*(\mathfrak{Y},\mathfrak{R}^+)$ in the ordered bases determined by the choices $\Xi$ has entry in the row corresponding to $\beta_i$ and column corresponding to $\alpha_j$ given by the intersection sign $\beta_i \cdot \alpha_j$. By the permutation expansion of the determinant, we get
    \[
    [\BSDA(Y,\Gamma; \Xi)]^{\Z}_{\comb} = \det M_{\Hc}.
    \]
\end{example}

\subsection{Functoriality of \texorpdfstring{$[\BSDA]$}{[BSDA]} over $\Z$}

We now prove that $[\BSDA(Y,\Gamma;\Xi)]^{\Z}_{\comb}$ respects gluing and identity cobordisms given appropriate choices $\Xi$. Let $(Y,\Gamma)$ be a sutured cobordism from $(F_0,\Lambda_0)$ to $(F_1,\Lambda_1)$ and $(\tilde{Y},\tilde{\Gamma})$ be a sutured cobordism from $(F_1,\Lambda_1)$ to $(F_2,\Lambda_2)$. Fix a set of choices $\tilde{\Xi}$ for $(\tilde{Y},\tilde{\Gamma})$ as in Definition~\ref{def:BSDAZComb}. 

\begin{remark}\label{rem:HDconstantsTilde}
    For bookkeeping purposes, as in Remark~\ref{rem:HDconstants}, we record the constants associated with the sutured surface $(F_2,\Lambda_2)$ and the Heegaard diagram $\tilde{\Hc}$ chosen in $\tilde{\Xi}$.
    \begin{itemize}
        \item We let $n_2$ be the rank of $H_1(F_2,S^+_2)$.
        \item For a basis element of $\wedge^* H_1(F_2,S^+_2)$, we let $m$ be the number of wedge factors. 
    \item We let $\tilde{\gamma}^{\out}_k$ denote the element of $H_1(F_2,S^+_2)$ represented by the $k^{th}$ matching arc of $\Zc_2$ with its chosen orientation, with index $k$ determined by the ordering fixed by $\tilde{\Xi}$. For each $m$ between $0$ and $n_2$, the elements \[ \tilde{\gamma}^{\out}_K = \tilde{\gamma}^{\out}_{k_1} \wedge \cdots \wedge \tilde{\gamma}^{\out}_{k_m} \] for $K=(k_1,\ldots,k_m)$, $1 \leq k_1 < \cdots < k_m \leq n_2$, form a basis for $\wedge^m H_1(F_2,S^+_2)$ over $\Z$ (not to be confused with the constant $K$ appearing in Theorem~\ref{thm:IntroFirstThm}). 
    \item We let $\tilde{a}$ and $\tilde{b}$ be the number of $\alpha$ and $\beta$-circles of $\tilde{\Hc}$, respectively.
    \item We let $d = n_2 + \chi(\tilde{Y},\tilde{R^+})$, which equals the degree of $[\BSDA(\tilde{Y},\tilde{\Gamma}; \tilde{\Xi})]^{\Z}_{\comb}$ by Proposition~\ref{prop:DegreeOfBSDA}.
    \end{itemize}
\end{remark}

\begin{definition}\label{def:ChoicesGluing}
    We say that the sets of choices $\Xi$ for $(Y,\Gamma)$ and $\tilde{\Xi}$ for $(\tilde{Y},\tilde{\Gamma})$ are \emph{composable} if the following restrictions hold. 
    \begin{itemize}
        \item The choice of arc diagram $\Zc_1$ representing $(F_1,\Lambda_1)$ must be the same in $\Xi$ and $\tilde{\Xi}$.
        \item The choices of ordering and orientations of the matching arcs of $\Zc_1$ must be the same in $\Xi$ and $\tilde{\Xi}$. 
    \end{itemize}
    For composable sets of choices $\Xi$ and $\tilde{\Xi}$, we define a specific set of choices, denoted $\tilde{\Xi} \cup \Xi$ for the glued cobordism $(\tilde{Y}\cup_{F_1}Y,\tilde{\Gamma}\cup_{\Lambda_1}\Gamma)$ from $(F_0,\Lambda_0)$ to $(F_2,\Lambda_2)$ as follows.
    \begin{itemize}
        \item Let the choices of arc diagrams $\Zc_i$ representing $(F_i,\Lambda_i)$ for $i=0,2$, together with the choices of ordering and orientations of the arcs in $\Zc_i$, be the same in $\tilde{\Xi} \cup \Xi$ as they are in $\Xi$ and $\tilde{\Xi}$. 
        \item Choose the Heegaard diagram $\tilde{\Hc} \cup_{\Zc_1} \Hc$ as in Definition~\ref{def:BorderedSuturedHDGluing} to represent the cobordism $(\tilde{Y}\cup_{F_1}Y,\tilde{\Gamma}\cup_{\Lambda_1}\Gamma)$, with the additional restriction that the ordering and orientations of the closed $\alpha$- and $\beta$-circles are the same in $\tilde{\Hc}\cup_{\Zc_1}\Hc$ as they are in $\Hc$ and $\tilde{\Hc}$. 
        \item Let $\alpha^{c,\mathrm{new}}$ denote the closed $\alpha$-circles of $\tilde{\Hc} \cup_{\Zc_1} \Hc$ formed by gluing the outgoing (left) arcs $\alpha^{a,\mathrm{out}}$ of $\Hc$ to the incoming (right) arcs $\tilde{\alpha}^{a,\mathrm{in}}$ of $\tilde{\Hc}$ along $\psi(\mathbf{a}_1)$. Give these new $\alpha$-circles the ordering and orientations uniquely determined by this gluing. 
        \item Order the full set of $\alpha$-circles in $\tilde{\Hc} \cup_{\Zc_1} \Hc$ by 
        \[ (\tilde{\alpha}^c_1,\ldots,\tilde{\alpha}^c_{\tilde{a}},\alpha_1^{c,\mathrm{new}},\dots,\alpha_{n_1}^{c,\mathrm{new}},\alpha^c_1,\ldots,\alpha^c_a).
        \]
        It follows that the set of all the $\alpha$-curves in $\tilde{\Hc} \cup_{\Zc_1} \Hc$ has total ordering
        \[ (\tilde{\alpha}^{a,\mathrm{out}}_1,\ldots,\tilde{\alpha}^{a,\mathrm{out}}_{n_2},\tilde{\alpha}^c_1,\ldots,\tilde{\alpha}^c_{\tilde{a}},\alpha_1^{c,\mathrm{new}},\dots,\alpha_{n_1}^{c,\mathrm{new}},\alpha^c_1,\ldots,\alpha^c_a,\alpha^{a,\mathrm{in}}_1,\ldots,\alpha^{a,\mathrm{in}}_{n_0}).
        \]
        \item If the ordering on the $\beta$-circles of $\Hc$ is given by $\Xi$ as $(\beta_1,\ldots,\beta_b)$, and the ordering on the $\beta$-circles of $\tilde{\Hc}$ is given by $\tilde{\Xi}$ as $(\tilde{\beta}_1,\ldots,\tilde{\beta}_{\tilde{b}})$, we order the full set of $\beta$-circles of $\tilde{\Hc} \cup_{\Zc_1} \Hc$ by
        \[
        (\tilde{\beta}_1,\ldots,\tilde{\beta}_{\tilde{b}},\beta_1,\ldots,\beta_b).
        \]
    \end{itemize}
\end{definition}

\begin{proposition}\label{prop:BSDARespectsGluing}
    Let $(Y,\Gamma)$ be a sutured cobordism from $(F_0,\Lambda_0)$ to $(F_1,\Lambda_1)$ and $(\tilde{Y},\tilde{\Gamma})$ be a sutured cobordism from $(F_1,\Lambda_1)$ to $(F_2,\Lambda_2)$. Fix composable sets of choices $\Xi$ and $\tilde{\Xi}$ for $(Y,\Gamma)$ and $(\tilde{Y},\tilde{\Gamma})$, respectively. Given these choices, we can define a map 
    \[
    [\BSDA(\tilde{Y}\cup_{F_1}Y,\tilde{\Gamma}\cup_{\Lambda_1}\Gamma;\tilde{\Xi}\cup \Xi)]^{\Z}_{\comb}\colon \wedge^* H_1(F_0,S^+_0) \to \wedge^* H_1(F_2,S^+_2)
    \]
    as in equation~\eqref{eq:DefOfBSDA}, using the set of choices $\tilde{\Xi} \cup \Xi$. In this case, the composition rule
    \begin{equation}\label{eq:BSDAgluing}
    [\BSDA(\tilde{Y}\cup_{F_1}Y,\tilde{\Gamma}\cup_{\Lambda_1}\Gamma;\tilde{\Xi}\cup \Xi)]^{\Z}_{\comb}=[\BSDA(\tilde{Y},\tilde{\Gamma},\tilde{\Xi})]^{\Z}_{\comb}\circ [\BSDA(Y,\Gamma;\Xi)]^{\Z}_{\comb}
    \end{equation}
    holds up to overall sign. 
\end{proposition}

\begin{proof}
    Let $\gamma_I^{\inrm}$ be an arbitrary basis element of $\wedge^k H_1(F_0,S^+_0)$. The left hand side of equation~\eqref{eq:BSDAgluing} acts on $\gamma_I^{\inrm}$ as
    \begin{equation}\label{eq:BSDAgluingLHS}
        \sum_{\substack{\tilde{\x} \cup \x \in \mathfrak{S}(\tilde{\Hc} \cup_{\Zc_1} \Hc): \\ o_R(\tilde{\x} \cup \x)=I, \\ \bar{o}_L(\tilde{\x} \cup \x)=K}} (-1)^{\mathrm{gr}_{\DA}(\tilde{\x} \cup \x)}\gamma_K^{\out}
    \end{equation}
    The right hand side acts on $\gamma_I^{\inrm}$ as
    \begin{equation}\label{eq:BSDAgluingRHS}
        \sum_{\substack{\tilde{\x} \in \mathfrak{S}(\tilde{\Hc}): \\ o_R(\tilde{\x})=J, \\ \bar{o}_L(\tilde{\x})=K}} \sum_{\substack{\x \in \mathfrak{S}(\Hc): \\ o_R(\x)=I, \\ \bar{o}_L(\x)=J}} (-1)^{\mathrm{gr}_{\DA}(\tilde{\x})} (-1)^{\mathrm{gr}_{\DA}(\x)}\gamma_K^{\out} = \sum_{\substack{\x,\tilde{\x}: \\ o_R(\x)=I, \\ \bar{o}_L(\x)=o_R(\tilde{\x}), \\ \bar{o}_L(\tilde{\x})=K}} (-1)^{\mathrm{gr}_{\DA}(\tilde{\x})+\mathrm{gr}_{\DA}(\x)}\gamma_K^{\out}.        
    \end{equation}
    
    Our goal is to show equality of the formulas \ref{eq:BSDAgluingLHS} and \ref{eq:BSDAgluingRHS} up to an overall factor of $\pm 1$. We first argue that there is a one-to-one correspondence between the index sets
    \[
    \begin{array}{c}
    \{\x,\tilde{\x} \mid 
    o_R(\x)=I,\ 
    \overline{o}_L(\x)=o_R(\tilde{\x}),\ 
    \overline{o}_L(\tilde{\x})=K\} \\[0.6em]
    \updownarrow \\[0.6em]
    \{\tilde{\x}\cup \x \in \mathfrak{S}(\tilde{\Hc}\cup \Hc) \mid
    o_R(\tilde{\x}\cup \x)=I,\ 
    \overline{o}_L(\tilde{\x}\cup \x)=K\}.
    \end{array}
    \]
    This follows from the fact that $o_R(\tilde{\x}\cup \x)=o_R(\x)$ and $\overline{o}_L(\tilde{\x}\cup \x)=\overline{o}_L(\tilde{\x})$, together with the manner by which generators $\tilde{\x}\cup \x$ are formed. Precisely, for each $i$, if the $\alpha$–arc on the $\tilde{\Hc}$ side that forms $\alpha^{c,\mathrm{new}}_i$ (namely, $\tilde{\alpha}^{a,\mathrm{in}}_i$) is occupied in $\tilde{\x}$, then for $\tilde{\x}\cup \x$ to be a generator, the corresponding $\alpha$–arc on the $\Hc$ side (namely, $\alpha^{a,\mathrm{out}}_i$) must be unoccupied, and vice-versa, so the condition $\overline{o}_L(\x)=o_R(\tilde{\x})$ is implicit for each generator $\tilde{\x}\cup \x$. 
    
    Next we will show that up to a $k$-independent sign, $\mathrm{gr}_{\DA}(\tilde{\x} \cup \x)=\mathrm{gr}_{\DA}(\tilde{\x})+\mathrm{gr}_{\DA}(\x)$, rather,  
    \begin{align*}
        &\sum_{\tilde{x} \cup x \in \tilde{\x} \cup \x} i(\tilde{x} \cup x) + \mathrm{inv} (\sigma_{\mathbf{\tilde{x} \cup x}})  + \mathrm{inv}(\sigma_{\overline{o}_L(\tilde{\x} \cup \x) o_L(\tilde{\x} \cup \x) \leftrightarrow \std}) + (a+\tilde{a}+n_1)k + n_2 k \\
        &=\sum_{\tilde{x} \in \tilde{\x}} i(\tilde{x})+ \mathrm{inv} (\sigma_{\tilde{\x}})  + \mathrm{inv}(\sigma_{\overline{o}_L(\tilde{\x}) o_L(\tilde{\x}) \leftrightarrow \std}) + \tilde{a}l + n_2 l \\
        &+\sum_{x \in \x} i(x)+ \mathrm{inv} (\sigma_{\x})  + \mathrm{inv}(\sigma_{\overline{o}_L(\x) o_L(\x) \leftrightarrow \std}) + ak + n_1 k.
    \end{align*}  
    We first notice it follows by construction that
    \[
    \sum_{\tilde{x} \cup x \in \tilde{\x} \cup \x} i(\tilde{x} \cup x)=\sum_{\tilde{x} \in \tilde{\x}} i(\tilde{x})+\sum_{x \in \tilde{\x} \cup \x} i(x).
    \]
    We now want to show that the inversion counts are additive. To do so, we represent the relevant permutations with \emph{crossing strands diagrams}, which are defined as follows. 
    \begin{itemize}
        \item List the indices $\{1,\dots,b\}$ of the $\beta$–circles of $\Hc$ in the order specified by $\Xi$, bottom to top, along the $y$-axis in the plane; list the indices $\{1,\dots,n_0+a+n_1\}$ of the $\alpha$–curves of $\Hc$ in the order specified by $\Xi$, bottom to top, along any line $x=\varepsilon$ in the plane for $\varepsilon>0$. Partition the list of $\alpha$ indices into blocks by type: $\boldsymbol{\alpha}^{a,\mathrm{out}} \sqcup \boldsymbol{\alpha}^{c} \sqcup \boldsymbol{\alpha}^{a,\mathrm{in}}$ with the $\boldsymbol{\alpha}^{a,\mathrm{out}}$ block on bottom.
        
        \item For each $i\in\{1,\dots,b\}$, draw a solid horizontal strand from $i$ on the left to $\sigma_{\x}(i)$ on the right. 
        \item For each $j\in\overline{o}_L(\x) \subset \boldsymbol{\alpha}^{a,\mathrm{out}}$, draw a vertical dotted strand that begins at $j$ on the right and ends at the bottom of the diagram. When drawing dotted strands, ensure they do not cross among themselves.
    \end{itemize}
    Two solid strands $i\to\sigma_{\x}(i)$ and $j\to\sigma_{\x}(j)$ cross if and only if $i<j$ and $\sigma_{\x}(i)>\sigma_{\x}(j)$. Hence the number of ``solid-solid'' crossings equals the number of inversions in $\sigma_{\x}$. A solid strand $i \to \sigma_{\x}(i)$ crosses a dotted strand at $j\in J$ when $\sigma_{\x}(i)<j$, i.e. when the occupied outgoing index $\sigma_{\x}(i)\in J^c$ lies before an unoccupied one $j\in J$. Hence the number of ``solid-dotted'' crossings equals the number of inversions in $\sigma_{J J^c \leftrightarrow \std} = \sigma_{\overline{o}_L(\x)\,o_L(\x)\leftrightarrow \std}$. We think of dotted-solid crossings as the \emph{left idempotent contribution} to $\mathrm{gr}_{\DA}(\x)$. 
    
    Using the ordering from $\tilde{\Xi}$, construct an analogous diagram for $\tilde{\x}$ computing $\mathrm{inv}(\sigma_{\tilde{\x}})+\mathrm{inv}(\sigma_{\overline{o}_L(\tilde{\x})\,o_L(\tilde{\x})\leftrightarrow\id})$, with blocks $\boldsymbol{\tilde{\alpha}}^{a,\mathrm{out}} \sqcup \boldsymbol{\tilde{\alpha}}^{c} \sqcup  \boldsymbol{\tilde{\alpha}}^{a,\mathrm{in}}$ on the right. The crossing strands diagrams for $\x$ and $\tilde{\x}$ are shown in Figure~\ref{fig:permutation diagram gluing}.
    
    Given crossing strands diagrams for $\x$ and $\tilde{\x}$, we form a crossing strands diagram for $\tilde{\x} \cup \x$ as follows. 
    \begin{itemize}
        \item Stack the crossing strands diagram for $\x$ above that for $\tilde{\x}$. Identify the blocks $\boldsymbol{\alpha}^{a,\mathrm{out}}$ and $\boldsymbol{\tilde{\alpha}}^{a,\mathrm{in}}$, relabeling the identified block $\boldsymbol{\alpha}^{c,\mathrm{new}}$. Concatenate the $\beta$–axes of the two diagrams to form the full $\beta$–axis for $\tilde{\x}\cup\x$. The new diagram lists, bottom to top, the $\beta$ (left) and $\alpha$ (right) indices of $\tilde{\mathcal H}\cup\mathcal H$ in the order specified by $\tilde{\Xi} \cup \Xi$.
        \item Under the identification $\boldsymbol{\alpha}^{a,\mathrm{out}}\sim\boldsymbol{\tilde{\alpha}}^{a,\mathrm{in}}$, all the $l$ dotted strands at $j\in\overline{o}_L(\x)$ in the upper diagram become solid strands connecting a subset of $\boldsymbol{\tilde{\beta}}$ to $\boldsymbol{\tilde{\alpha}}^{c,\mathrm{new}}$, representing the new occupancy pattern in the glued generator $\tilde{\x} \cup \x$. 
    \end{itemize}
    The crossing strands diagram for $\tilde{\x} \cup \x$ is shown in Figure~\ref{fig:permutation diagram gluing}. It follows by construction that 
    \begin{align*}
    &\mathrm{inv} (\sigma_{\mathbf{\tilde{x} \cup x}})  + \mathrm{inv}(\sigma_{\overline{o}_L(\tilde{\x} \cup \x) o_L(\tilde{\x} \cup \x) \leftrightarrow \std}) \\
    &=\mathrm{inv} (\sigma_{\mathbf{x}})+ \mathrm{inv} (\sigma_{\tilde{\x}})+\mathrm{inv}(\sigma_{\overline{o}_L(\x) o_L(\x) \leftrightarrow \std})+\mathrm{inv}(\sigma_{\overline{o}_L(\tilde{\x}) o_L(\tilde{\x}) \leftrightarrow \std});
    \end{align*}
    the key point is that the number of crossings between the two strands labeled $n_1 - l$ and $l$ in the rightmost diagram of Figure~\ref{fig:permutation diagram gluing} agrees with the number of crossings between the solid strand labeled $n_1 -l$ and the dotted strand labeled $l$ in the leftmost diagram of Figure~\ref{fig:permutation diagram gluing}.
    
    \begin{figure}
    \centering
    \vspace{0.05in}
    \begin{overpic}[width=\textwidth] 
    {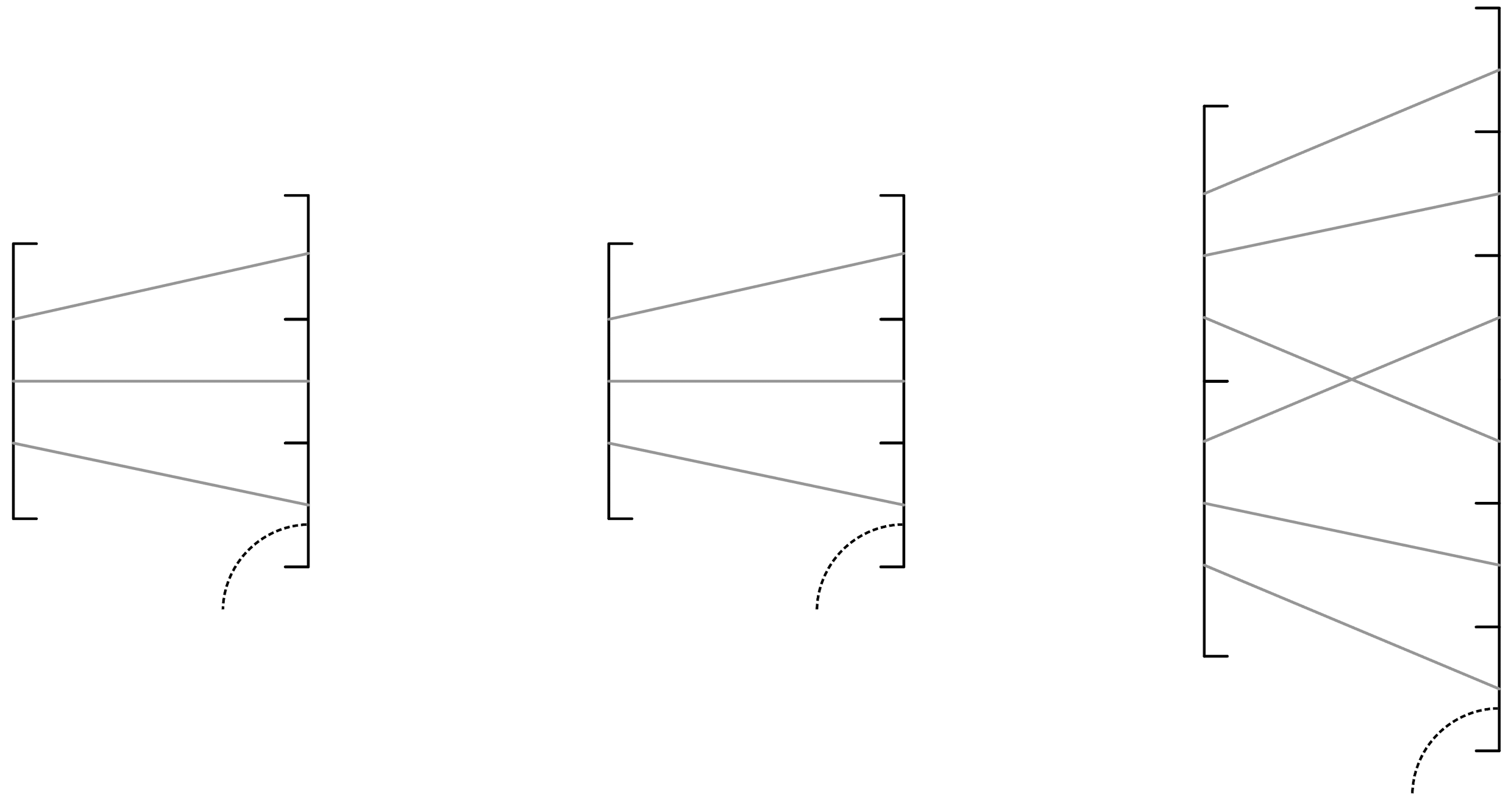}
    \put (-3,27) {$\boldsymbol{\beta}$} 
    \put (22,35) {$\boldsymbol{\alpha}^{a,\mathrm{in}}$}
    \put (22,27) {$\boldsymbol{\alpha}^{c}$}
    \put (22,19) {$\boldsymbol{\alpha}^{a,\mathrm{out}}$}
    \put (10,35) {$k$} 
    \put (10,29) {$a$} 
    \put (6,18) {$n_1-l$} 
    \put (14,10) {$l$} 
    
    \put (36,27) {$\boldsymbol{\tilde{\beta}}$}
    \put (62,35) {$\boldsymbol{\tilde{\alpha}}^{a,\mathrm{in}}$}
    \put (62,27) {$\boldsymbol{\tilde{\alpha}}^{c}$}
    \put (62,19) {$\boldsymbol{\tilde{\alpha}}^{a,\mathrm{out}}$}
    \put (50,35) {$l$} 
    \put (50,29) {$\tilde{a}$} 
    \put (46,18) {$n_2-m$} 
    \put (53,10) {$m$} 

    \put (75,36) {$\boldsymbol{\beta}$} 
    \put (75,19) {$\boldsymbol{\tilde{\beta}}$}
    \put (101,48) {$\boldsymbol{\alpha}^{a,\mathrm{in}}$}
    \put (101,39) {$\boldsymbol{\alpha}^{c}$}
    \put (101,27) {$\boldsymbol{\alpha}^{c,\mathrm{new}}$}
    \put (101,15) {$\boldsymbol{\tilde{\alpha}}^{c}$}
    \put (101,7) {$\boldsymbol{\tilde{\alpha}}^{a,\mathrm{out}}$}
    \put (89,46) {$k$} 
    \put (89,39) {$a$} 
    \put (85,31) {$n_1-l$} 
    \put (86,24) {$l$} 
    \put (89,19) {$\tilde{a}$} 
    \put (84,8) {$n_2-m$} 
    \put (92.5,-2) {$m$} 
    \end{overpic}
    \vspace{0.035in}
    \caption{Left to right: The crossing strands diagrams for $\x$, $\tilde{\x}$, and $\tilde{\x} \cup \x$. The indices indicate clusters of strands that have been grouped to avoid visual clutter. Note that such groups may contain internal solid-solid or solid-dotted crossings that are not shown. Also note that the ordering along the $\beta$-axis is purely schematic: for instance, the $k$ strands ending on $\alpha^{a,\mathrm{in}}$ need not originate from the topmost $\beta$–indices. Lastly, note that the $n_1-l$ and $l$ strands in the rightmost figure need not always all cross. Rather, this is drawn to suggest that there may be new crossings here where there were not previously.}
    \label{fig:permutation diagram gluing}
    \end{figure}
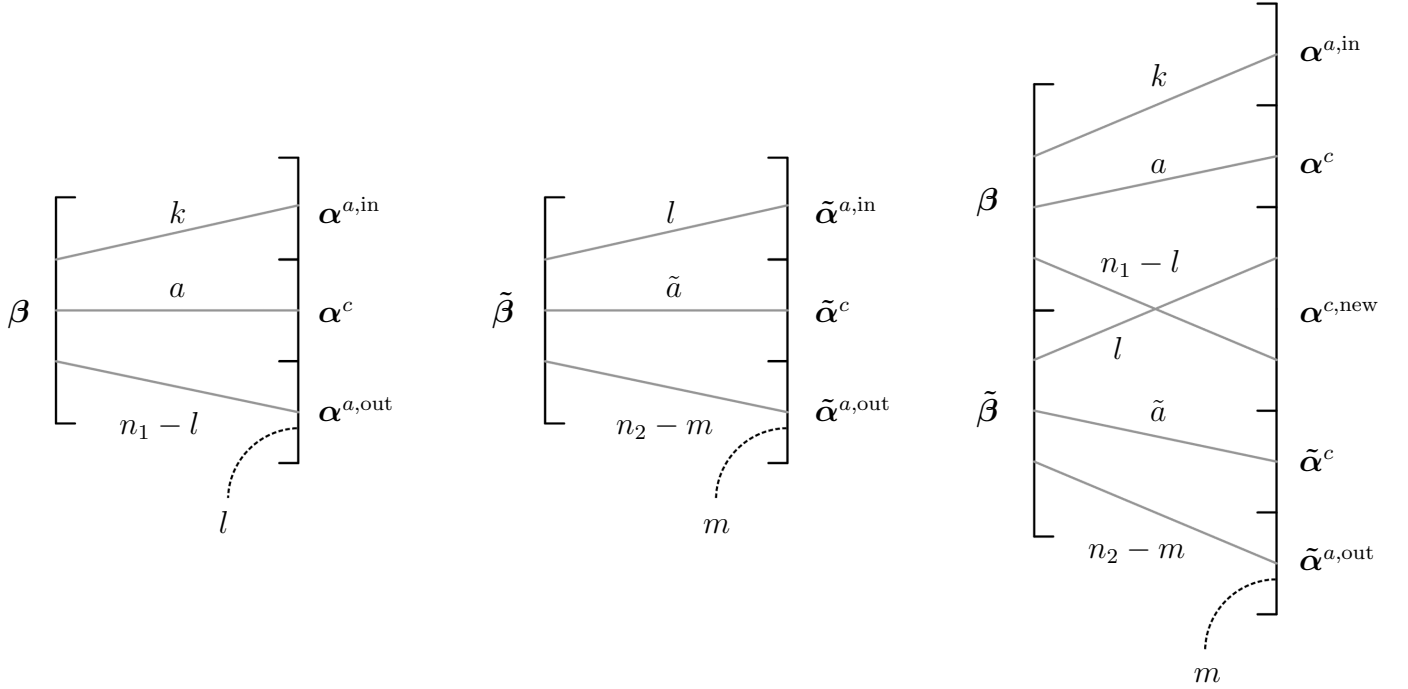

    Lastly, we show that the correction terms satisfy
    \begin{align*}
        (a+\tilde{a}+n_1)k + n_2 k &= \left(\tilde{a}l + n_2 l\right) +\left(  ak + n_1 k\right)  \\
        (a+\tilde{a}+n_1)k + n_2 k +\left(\tilde{a}l + n_2 l\right) +\left(  ak + n_1 k\right) &= 0 \mod 2 
    \end{align*} 
    To do this, we will simplify this expression mod 2, discarding any terms which do not depend on $k$. These such signs ultimately factor out of $[\BSDA(\tilde{Y}\cup_{F_1}Y,\tilde{\Gamma}\cup_{\Lambda_1}\Gamma;\tilde{\Xi}\cup \Xi)]^{\Z}_{\comb}$ as a global sign, which we disregard. In computing, we utilize the relations $l=k+c$ and $m=l+d$ to reduce any factor of $l$ or $m$ to a factor of $k$ modulo terms which do not depend on $k$. We have
    \begin{align*}
        (a+\tilde{a}+n_1)k + n_2 k +\left(\tilde{a}l + n_2 l\right) +\left(  ak + n_1 k\right)
        &= (\tilde{a}+n_2)k + \tilde{a}l + n_2 l \\
        &= (\tilde{a}+n_2)k + \tilde{a}(k+c) + n_2(k+c) \\
        &= 2(\tilde{a}+n_2)k + (\tilde{a}+n_2)c \\
        &= 0 + \text{($k$-independent terms)} \mod 2.
    \end{align*}
\end{proof}

\begin{proposition}\label{prop:BSDARespectsIdentity}
    Write $(\id_{(F,\Lambda)},\Gamma_{\id})$ for the identity sutured cobordism $(F \times [0,1], \Lambda \times [0,1])$ from $(F,\Lambda)$ to itself, as in Definition~\ref{def:SutCobCategory}. We can equip the identity cobordism with a set of choices $\Xi_{\id}$, defined as follows. 
    \begin{itemize}
        \item Choose an $\alpha$-arc diagram $\Zc$ to represent $(F,\Lambda)$, and fix an ordering and orientations of the matching arcs in $\Zc$. Assume that both the incoming and outgoing copy of $\Zc$ have the same ordering and orientation of matching arcs. 
        \item Choose the identity Heegaard diagram $\Hc_{\id}=\Hc^{\alpha \beta}_{1/2} \cup_{\Zc^*} \Hc^{\beta \alpha}_{1/2}$ from Example~\ref{ex:identityHD} to represent $(\id_{(F,\Lambda)},\Gamma_{\id})$. 
        \item Make all the choices for $\Hc_{\id}$ required by Definition~\ref{def:BSDAZComb}, deferring the orderings and orientations of the $\beta$-circles, which we specify below. 
        \item Order the $\beta$-circles of $\Hc_{\id}$ so that the $i^{th}$ $\beta$-circle intersects the $i^{th}$ incoming alpha arc of $\Hc^{\beta \alpha}_{1/2}$ non-trivially. Orient the $\beta$-circles of $\Hc_{\id}$ such that the local intersection signs $\beta \cdot \alpha^{a,\mathrm{in}}$ in $\Hc^{\beta \alpha}_{1/2}$ are always $+1$; this forces all the local intersection signs $\beta \cdot \alpha^{a,\mathrm{out}}$ in $\Hc^{\alpha \beta}_{1/2}$ to be $-1$, since the orientations of the arcs in both copies of $\Zc$ must agree.
    \end{itemize}
    Given these choices $\Xi_{\id}$, the map 
    \[[\BSDA(\id_{(F,\Lambda)},\Gamma_{\id}; \Xi_{\id})]^{\Z}_{\mathrm{comb}} \colon \wedge^* H_1(F,S^+) \to \wedge^* H_1(F,S^+)
    \]
    coincides with the identity map on $\wedge^* H_1(F,S^+)$ up to overall sign. 
\end{proposition}

\begin{proof}
    Associate to $(F,S^+)$ the same constants as in Remark~\ref{rem:HDconstants} assigned to the incoming surface $F_0$, and rename $n_0$ by $n \coloneq \rank H_1(F,S^+)$. Note that 
    \[
    c = n + \chi(F \times [0,1], S^+ \times [0,1]) = n + \chi(F,S^+) = n - n = 0.
    \]
    We want to show that $\mathrm{gr}_{\DA}(\mathbf{x})=0$ mod 2 for each $\x$. First notice in $\Xi_{\id}$ that the local intersection contribution satisfies 
    \begin{align*}
        \sum_{x \in \mathbf{x}} i(x)&=|o_L(\x)|\cdot 1 +|o_R(\x)|\cdot 0 \\
        &=n-l \\
        &= n-k-c \\
        &= k + \text{($k$-independent terms)} \mod 2.
    \end{align*}
    
    \begin{figure}
        \centering
        \vspace{0.035in}
        \begin{overpic}[scale=0.4] 
        {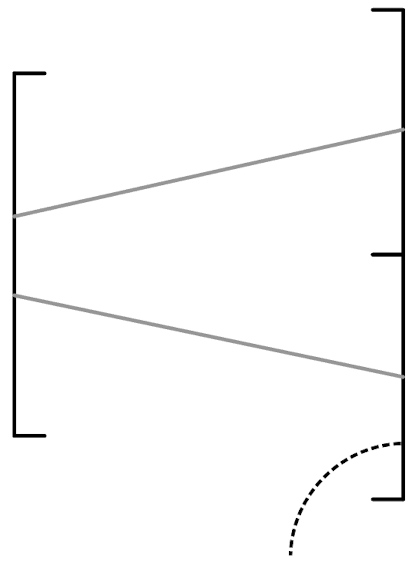}
        \put (-15,50) {$\boldsymbol{\beta}$} 
        \put (80,75) {$\boldsymbol{\alpha}^{a,\mathrm{in}}$} 
        \put (80,30) {$\boldsymbol{\alpha}^{a,\mathrm{out}}$} 
    
        \put (28,75) {$k$} 
        \put (20,25) {$n-k$} 
        \put (48,-10) {$k$} 
        \end{overpic}
        \vspace{0.1in}
        \caption{The crossing–strands diagram for $\x \in \mathfrak{S}(\Hc_{\id})$ under the choices $\Xi_{\id}$ from Proposition~\ref{prop:BSDARespectsIdentity}. Within each block of $k$ or $n-k$ solid strands, all strands are parallel, so no internal crossings occur, consistent with the ordering imposed by $\Xi_{\id}$.}
        \label{fig:permutation diagram identity}
    \end{figure}
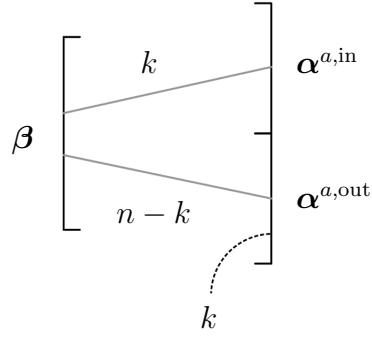

    Next we examine the crossing strands diagrams for $\x$ as in Figure~\ref{fig:permutation diagram identity}. By the ordering conventions in $\Xi_{\id}$, it follows that none of the $2^n$ crossing strands diagrams have internal crossings within the clusters of $k$ and $n-k$ solid strands. Instead, crossings occur only as ``cross-block'' solid-solid or solid-dotted intersections. The solid–solid crossings therefore count the inversions 
    \[
    \mathrm{inv}(\sigma_{\x})=\mathrm{inv}(\sigma_{o_L(\x)\overline{o}_L(\x)\leftrightarrow \std}), 
    \]
    while the solid-dotted crossings count $\mathrm{inv}(\sigma_{\overline{o}_L(\x)\,o_L(\x)\leftrightarrow \std})$. See Figure~\ref{fig:permutation diagram identity}. We can simplify the total inversion count to
    \begin{align*}
        \mathrm{inv}(\sigma_{o_L(\x)\overline{o}_L(\x)\leftrightarrow \std})+
        \mathrm{inv}(\sigma_{\overline{o}_L(\x)\,o_L(\x)\leftrightarrow \std})&=\#\{(i,j):i\in J^c,\ j\in J,i>j\} \\
        &\quad+\#\{(i,j):i\in J^c,\ j\in J,j>i\} \\
        &=\#\{(i,j):i\in J,\ j\in J^c\} \\
        &=|J|\cdot|J^c| \\
        &=(k+c)(n-k-c) \\
        &=k(n-k) \\
        &=kn-k^2 \\
        &=kn-k \mod 2.
    \end{align*}
    For the correction terms, since there are no $\alpha$-circles in $\Hc_{\id}$, we have $ka=0$. Lastly, the $n_1 k$ term is simply $nk$. Hence,
    \[
    \mathrm{gr}_{\DA}(\x)=k+kn-k+nk=0 \mod 2.
    \]
\end{proof}

\begin{remark}
    The orientation conventions in $\Xi_{\id}$ make the above sign verification especially simple and ensure compatibility with Corollary~\ref{cor:CanChooseXiNormalized} below. A global orientation reversal of all the curves in $\Hc_{\id}$ gives another easily checked case: one has $\sum_{x \in \mathbf{x}} i(x)=k$ without the need to mod out by any $k$-independent terms, and hence still $\mathrm{gr}_{\DA}(\x)=0$ for all $\x$. 
    
    In general, given more complicated orientation patterns in $\Hc_{\id}$ or arbitrary Heegaard diagrams, there is likely an argument that works for arbitrary choices. However, since we will later show that $[\BSDA(Y,\Gamma;\Xi)]^{\Z}_{\comb}$ is independent of such choices up to overall sign, we take this simpler approach here.
\end{remark}

\subsection{Choices for $\Hc_{\norm}$ and disjoint unions}

Recall that we will often work with $\alpha$-$\alpha$ bordered sutured Heegaard diagrams of the form $\Hc_{\norm}=\Hc^{\alpha \beta}_{1/2} \cup_{\Zc_1^*} \Hc' \cup_{\Zc_0^*} \Hc^{\beta \alpha}_{1/2}$ where $\Hc'$ is a $\beta$-$\beta$ bordered sutured diagram, as in Corollary~\ref{cor:NormalizedHeegaardDiagrams}. We record here a set of choices which assumes such a decomposition, and assumes the orientation patterns of Proposition~\ref{prop:ExpandingBasisElts}. To conclude the present section, we discuss choices for disjoint unions of cobordisms, which will arise in the proof of Theorem~\ref{Thm:VFNSymmetricMonoidal}. 

\begin{corollary}\label{cor:CanChooseXiNormalized}
    For any set $\Xi$ of choices in Definition~\ref{def:BSDAZComb}, the map $[\BSDA(Y,\Gamma;\Xi)]^{\Z}_{\comb}$ agrees up to sign with the map $[\BSDA(Y,\Gamma;\Xi_{\norm})]^{\Z}_{\comb}$ computed using any set of choices $\Xi_{\norm}$ satisfying the below criteria. 
    \begin{itemize}
        \item As in Corollary~\ref{cor:NormalizedHeegaardDiagrams}, the Heegaard diagram is of the form $\Hc_{\norm}=\Hc^{\alpha \beta}_{1/2} \cup_{\Zc_1^*} \Hc' \cup_{\Zc_0^*} \Hc^{\beta \alpha}_{1/2}$ for some $\beta$-$\beta$ bordered sutured Heegaard diagram $\Hc'$;
        \item the orientations and ordering of the arcs of $\Zc_0$ and $\Zc_1$ are the same in $\Xi_{\norm}$ as they are in $\Xi$;
        \item as in Proposition~\ref{prop:ExpandingBasisElts}, on the left side of the diagram (outgoing), the local intersections in $\Hc^{\alpha \beta}_{1/2}$ have $\beta \cdot \alpha = -1$; on the right side of the diagram (incoming), the local intersections in $\Hc^{\beta \alpha}_{1/2}$ have $\beta \cdot \alpha = +1$;
        \item in the ordering of the $\beta$-circles of $\Hc_{\norm}$, circles glued from $\beta$-arcs of $\Hc^{\alpha \beta}_{1/2}$ and $\Hc'$ are ordered as the arcs of $\Zc_1$ (equivalently, as the arcs of $\Zc_1^*$) and come before the $\beta$-circles of $\Hc'$, while circles glued from $\beta$-arcs of $\Hc'$ and $\Hc^{\beta \alpha}_{1/2}$ are ordered as the arcs of $\Zc_0$ or $\Zc_0^*$ and come after the $\beta$-circles of $\Hc'$.
    \end{itemize} 
\end{corollary}

\begin{proof}
    Given $\Xi$, let $\Xi_{\norm}$ denote the set of choices in Definition~\ref{def:BSDAZComb} where:
    \begin{itemize}
        \item the Heegaard diagram $\Hc_{\norm}$ is 
        \[
        \Hc_{\norm}=\Hc_{\id} \cup_{\Zc_1^*} \Hc \cup_{\Zc_0} \Hc_{\id} = \Hc^{\alpha \beta}_{1/2} \cup_{\Zc_1^*} (\Hc^{\beta \alpha}_{1/2} \cup_{\Zc_1} \Hc \cup_{\Zc_0} \Hc^{\alpha \beta}_{1/2}) \cup_{\Zc_0^*} \Hc^{\beta \alpha}_{1/2},
        \]
        where $\Hc^{\beta \alpha}_{1/2} \cup_{\Zc_1} \Hc \cup_{\Zc_0} \Hc^{\alpha \beta}_{1/2} =: \Hc'$;
        \item the orientations and ordering of the arcs of $\Zc_0$ and $\Zc_1$ are the same in $\Xi_{\norm}$ as they are in $\Xi$;
        \item the orientations of the $\alpha$- and $\beta$-circles of $ \Hc_{\norm}$ that come from $\alpha$- and $\beta$-circles of $\Hc$ are unchanged;
        \item the orientations of the new $\alpha$-circles of $ \Hc_{\norm}$ that are glued from $\alpha$-arcs of $\Hc^{\beta \alpha}_{1/2}$ and $\Hc$ on the left or $\Hc$ and $\Hc^{\alpha \beta}_{1/2}$ on the right are induced from the $\alpha$-arc orientations of $\Hc$;
        \item the orientations of the new $\beta$-circles of $ \Hc_{\norm}$ that are glued from $\beta$-arcs of $\Hc^{\alpha \beta}_{1/2}$ and $\Hc^{\beta \alpha}_{1/2}$ on the left are chosen so that the local intersection signs $\beta \cdot \alpha$ in $\Hc^{\alpha \beta}_{1/2}$ are always $-1$;
        \item the orientations of the new $\beta$-circles of $ \Hc_{\norm}$ that are glued from $\beta$-arcs of $\Hc^{\alpha \beta}_{1/2}$ and $\Hc^{\beta \alpha}_{1/2}$ on the right are chosen so that the local intersection signs $\beta \cdot \alpha$ in $\Hc^{\beta \alpha}_{1/2}$ are always $+1$;
        \item the ordering of the $\alpha$-circles of $ \Hc_{\norm}$ has the new $\alpha$-circles on the left (ordered as the arcs of $\Zc_1$), then the $\alpha$-circles of $\Hc$ (ordered as in $\Xi$), then the new $\alpha$-circles on the right (ordered as the arcs of $\Zc_0$);
        \item the ordering of the $\beta$-circles of $ \Hc_{\norm}$ has the new $\beta$-circles on the left (ordered as the arcs of $\Zc_1$), then the $\beta$-circles of $\Hc$ (ordered as in $\Xi$), then the new $\beta$-circles on the right (ordered as the arcs of $\Zc_0$).
    \end{itemize}
    Then $\Xi_{\norm}$ satisfies the requirements of the statement with $\Hc' = \Hc^{\beta \alpha}_{1/2} \cup_{\Zc_1} \Hc \cup_{\Zc_0} \Hc^{\alpha \beta}_{1/2}$. Furthermore, $\Xi_{\norm}$ is obtained by gluing the composable sets of choices $\Xi_{\id}$ for $\Hc_{\id}$ as defined in Proposition~\ref{prop:BSDARespectsIdentity} to the choices $\Xi$ for $\Hc$ on the left and right, as in Definition~\ref{def:ChoicesGluing}. It follows from Propositions~\ref{prop:BSDARespectsGluing} and \ref{prop:BSDARespectsIdentity} that
    \begin{align*}
    [\BSDA(Y,\Gamma;\Xi_{\norm})]^{\Z}_{\comb} &= [\BSDA(\id;\Xi_{\id})] \circ [\BSDA(Y,\Gamma;\Xi)]^{\Z}_{\comb} \circ [\BSDA(\id;\Xi_{\id})] \\
    &= [\BSDA(Y,\Gamma;\Xi)]^{\Z}_{\comb}
    \end{align*}
    up to sign.
\end{proof}

\begin{definition}\label{def:ChoicesDisjoint}
    Let $(Y,\Gamma)$ be a sutured cobordism from $(F_0,\Lambda_0)$ to $(F_1,\Lambda_1)$ and $(\tilde{Y},\tilde{\Gamma})$ be a sutured cobordism from $(\tilde{F_0},\tilde{\Lambda}_0)$ to $(\tilde{F_1},\tilde{\Lambda}_1)$. Given sets of choices $\Xi$ for $(Y,\Gamma)$ and $\tilde{\Xi}$ for $(\tilde{Y},\tilde{\Gamma})$, we denote by $\Xi \sqcup \tilde{\Xi}$ a set of choices for the sutured cobordism $(Y \sqcup \tilde{Y},\Gamma \sqcup \tilde{\Gamma})$ from $(F_0\sqcup \tilde{F_0},\Lambda_0\sqcup\tilde{\Lambda_0})$ to $(F_1\sqcup\tilde{F_1},\Lambda_1\sqcup\tilde{\Lambda_1})$, defined as follows.
    \begin{itemize}
        \item Choose the arc diagrams $\Zc_i \sqcup \tilde{\Zc}_i$ to represent $(F_i\sqcup \tilde{F_i},\Lambda_i\sqcup\tilde{\Lambda_i})$ for $i=0,1$, where $\Zc_i$ and $\tilde{\Zc}_i$ are the arc diagrams fixed in $\Xi$ and $\tilde{\Xi}$ respectively. 
        \item Assume that the ordering and orientations of matching arcs in $\Zc_i \sqcup \tilde{\Zc}_i$ are the same as they are in $\Xi$ and $\tilde{\Xi}$ respectively. Put the matching arcs of $\Zc_i$ before the matching arcs of $\tilde{\Zc}_i$ in the ordering.
        \item Choose the Heegaard diagram $\Hc \sqcup \tilde{\Hc}$ to represent $(Y \sqcup \tilde{Y},\Gamma \sqcup \tilde{\Gamma})$, where $\Hc$ and $\tilde{\Hc}$ are the Heegaard diagrams fixed in $\Xi$ and $\tilde{\Xi}$ respectively. 
        \item Assume that the ordering and orientations of the $\alpha$-circles and $\beta$-circles in $\Hc$ and $\tilde{\Hc}$ are the same as they are in $\Xi$ and $\tilde{\Xi}$ respectively. Put the $\alpha$- and $\beta$-circles of $\Hc$ before the $\alpha$- and $\beta$-circles of $\tilde{\Hc}$ in the ordering, so that the full set of $\alpha$-circles is ordered by
        \[
        (a^c_1,\ldots,\alpha^c_a, \tilde{\alpha}^c_1,\ldots,\tilde{\alpha}^c_{\tilde{a}})
        \]
        and the full set of $\beta$-circles is ordered by
        \[
        (\beta_1\ldots,\beta_b,\tilde{\beta}_1\ldots,\tilde{\beta}_{\tilde{b}}).
        \]
    \end{itemize}
    Note that the full set of $\alpha$-curves is ordered by
    \[ (\alpha^{a,\mathrm{out}}_1,\ldots,\alpha^{a,\mathrm{out}}_{n_1},\tilde{\alpha}^{a,\mathrm{out}}_1,\ldots,\tilde{\alpha}^{a,\mathrm{out}}_{\tilde{n_1}},\alpha^c_1,\ldots,\alpha^c_a, \tilde{\alpha}^c_1,\ldots,\tilde{\alpha}^c_{\tilde{a}},\alpha^{a,\mathrm{in}}_1,\ldots,\alpha^{a,\mathrm{in}}_{n_0},\tilde{\alpha}^{a,\mathrm{in}}_1,\ldots,\tilde{\alpha}^{a,\mathrm{in}}_{\tilde{n_0}}). 
    \]
\end{definition}

\section{Sutured Alexander functor over \texorpdfstring{$\Z$}{Z}}\label{sec:SuturedAlexanderOverZ}

In this section we construct a sutured, $\Z$-valued version of the Florens--Massuyeau Alexander functor and show that it coincides with the map $[\BSDA(Y,\Gamma;\Xi)]^{\Z}_{\comb}$ up to sign for any set of choices $\Xi$; it will follow that $[\BSDA(Y,\Gamma;\Xi)]^{\Z}_{\comb}$ is an invariant of $(Y,\Gamma)$ up to overall sign.

\subsection{The Alexander function over $\Z$}

Let $(Y,\Gamma)$ be a sutured cobordism from $(F_0,\Lambda_0)$ to $(F_1,\Lambda_1)$. Let $\Xi$ be a set of choices as in Definition~\ref{def:BSDAZComb}, including an $\alpha$-$\alpha$ bordered Heegaard diagram $\Hc$ for $(Y,\Gamma)$. Say the number of $\alpha$-circles of $\Hc$ is $a$ and the number of $\beta$-circles is $b$.

Let $(\mathfrak{Y},\mathfrak{R}^+)$ be the CW pair we associated to $\Hc$ (plus the additional choice of CW decomposition $\mathfrak{R}^+$ of $R^+$) in Section~\ref{sec:CWDecomp}; the orientations on $\alpha$ and $\beta$ circles of $\Hc$ needed to orient the cells come from the set of choices $\Xi$. The cellular chain complex $C_*^{\cell}(\mathfrak{Y},\mathfrak{R}^+)$ computes $H_*(Y,R^+)$ and depends only on $\Hc$ and $\Xi$ (not on the CW decomposition $\mathfrak{R}^+$ of $R^+$); it is isomorphic to $\Z^b$ in degree 1, $\Z^a$ in degree 2, and zero in other degrees. As discussed in Section~\ref{sec:CellOrientations},  the matrix $M_{\Hc}$ representing the cellular differential $\partial_2$ in the bases of 2-cells and 1-cells is a presentation matrix for $H_1(Y,R^+)$. Recall from Proposition~\ref{prop:SignsInCellularDiff} that the entries of $M_{\Hc}$ are algebraic intersection numbers ``$\beta \cdot \alpha$'' of curves in $\Hc$.

The \emph{deficiency}, $d$, of $M_{\Hc}$ (number of generators minus number of relations) is $b-a$, which in this setting equals $-\chi(Y,R^+)$ (cf. proof of Proposition~\ref{prop:DegreeOfBSDA}). The matrix $M_{\Hc}$ is injective (has zero kernel) if and only if $H_2(Y,R^+) = 0$, since the kernel of $M_{\Hc}$ is isomorphic to $H_2(Y,R^+)$. We summarize these facts in the below lemma; the proof is evident. 

\begin{lemma}\label{lem:InjectivePresentationMatrixZ}
    If $H_2(Y,R^+)=0$, then the deficiency-$d$ presentation matrix $M_{\Hc}$ for $H_1(Y,R^+)$ is injective; in particular, $d \geq 0$. If $H_2(Y,R^+) \neq 0$, the matrix $M_{\Hc}$ is never injective.     
\end{lemma}

\begin{definition}\label{def:AlexanderFunctionZ} (cf. \cite[Section 3.1]{LescopSumFormula} and \cite[Definition 2.2]{FMFunctorial})
    Let $d = -\chi(Y,R^+)$. Define the \emph{$\Z$-valued Alexander function}
    \[
    \mathcal{A}^{\Z}_{Y,\Gamma} \colon \wedge^d H_1(Y,R^+) \to \Z,
    \]
    well-defined up to (global) sign, as follows. If $H_2(Y,R^+)$ is nonzero, then $\mathcal{A}^{\Z}_{Y,\Gamma}$ is defined to be the zero function. Otherwise, $d \geq 0$ and $H_1(Y,R^+)$ admits an injective presentation matrix $M$ of deficiency $d$, e.g. the matrix $M_{\Hc}$ associated to a set $\Xi$ of choices as above. Define $\mathcal{A}^{\Z}_{Y,\Gamma}$ by sending an element $u_1 \wedge \cdots \wedge u_d$ of $\wedge^d H_1(Y,R^+)$ to the following element of $\Z$. For each $i$, choose any expression $\overline{u}_i$ of $u_i$ as a $\Z$-linear combination of the generators, form a square matrix from $M$ by adding new columns $\overline{u}_1,\ldots,\overline{u}_d$ to the right, and take the determinant, an element of $\Z$. We define $\mathcal{A}^{\Z}_{Y,\Gamma}(u_1 \wedge \cdots \wedge u_d)$ to be this element of $\Z$.
\end{definition}

\begin{lemma}\label{lem:AlexanderFunctionZWellDefined}
    The Alexander function $\mathcal{A}^{\Z}_{Y,\Gamma}$ of Definition~\ref{def:AlexanderFunctionZ} is independent of the choice of representatives $\overline{u}_1, \ldots, \overline{u}_d$.
\end{lemma}

\begin{proof}
    Fix a presentation matrix $M$ for $H_1(Y,R^+)$ to compute $\mathcal{A}^{\Z}_{Y,\Gamma}$. Any two distinct representatives $\overline{u}_i$ and $\overline{u}_i'$ of $H_1(Y,R^+) \cong \Z^b/\im M$ differ by an element of $\im M$, i.e. $\overline{u}_i=\overline{u}_i'+Mv_i$ for some $v_i\in \Z^a$. But $Mv_i$ is a linear combination of the columns of $M$, and adding a linear combination of columns to another column leaves the determinant unchanged (not merely up to sign), so
    \[
    \det([M|\overline{u}_1 \cdots \overline{u}_i' \cdots \overline{u}_d])=\det([M|\overline{u}_1 \cdots \overline{u}_i \cdots \overline{u}_d]).
    \]
    Applying this to each $i$ proves the claim. 
\end{proof}

\begin{proposition}\label{prop:AlexanderFunctionZIndependentOfM}(cf. \cite[Section 3.1]{LescopSumFormula})
    Up to an overall (global) sign, the Alexander function $\mathcal{A}^{\Z}_{Y,\Gamma}$ of Definition~\ref{def:AlexanderFunctionZ} only depends on the abelian group $H_1(Y,R^+)$ and the deficiency $d$, and does not depend on the specific choice of injective presentation matrix $M$.
\end{proposition}

\begin{proof}
    Suppose that $M\neq \widetilde{M}$ are injective presentation matrices for $H_1(Y,R^+)$ of deficiency $d$. Then by definition, $M$ and $\widetilde{M}$ give rise to short exact sequences 
    \begin{equation}\label{eq:SchanuelMHandMHtildePres}
        \begin{aligned}
        0 \to \mathbb{Z}^a &\xrightarrow{M} \mathbb{Z}^b \xrightarrow{\pi} H_1(Y,R^+) \to 0 \\
        0 \to \mathbb{Z}^{\tilde a} &\xrightarrow{\widetilde{M}} \mathbb{Z}^{\tilde b} \xrightarrow{\tilde{\pi}} H_1(Y,R^+) \to 0
        \end{aligned}
    \end{equation}
    where $d=b-a=\tilde{b}-\tilde{a}$. Let $\mathcal{A}^{\Z}_{Y,\Gamma,M}$ and $\mathcal{A}^{\Z}_{Y,\Gamma,M}$ denote the associated Alexander functions defined using $M$ and $\widetilde{M}$ respectively. Let $u_1 \wedge \cdots \wedge u_d \in \wedge^d H_1(Y,R^+)$. Our goal is to show that 
    \[
    \mathcal{A}^{\Z}_{Y,\Gamma,M}(u_1 \wedge \cdots \wedge u_d)=\mathcal{A}^{\Z}_{Y,\Gamma,\widetilde{M}}(u_1 \wedge \cdots \wedge u_d)
    \]
    up to an overall sign independent of $u_1 \wedge \cdots \wedge u_d$. We follow the strategy underlying the proof of Schanuel’s lemma, passing through auxiliary presentations for $H_1(Y,R^+)$ constructed from $M$ and $\widetilde{M}$. Consider the composition
    \begin{equation}\label{eq:SchanuelSESM1}
        0 \to \mathbb{Z}^{a+\tilde{b}} \xrightarrow{M_1} \Z^{b+\tilde{b}} \xrightarrow{\pi\oplus \tilde{\pi}} H_1\left(Y, R^+\right) \to 0 
    \end{equation}
    where
    \[
    M_1 \coloneq \begin{blockarray}{ccc}
    a & \tilde{b}  \\
    \begin{block}{[cc]c}
    M & -\eta & b  \\
      0 & I & \tilde{b}  \\
    \end{block}
    \end{blockarray}
    \]
    and $\eta:\Z^{\tilde{b}}\to \Z^b$ is any lift of $\tilde{\pi}$ through $\pi$, i.e. $\pi \circ \eta = \tilde{\pi}$. Note that $\eta$ is guaranteed to exist since $\Z^{\tilde{b}}$ is free and hence projective. There is an analogous composition 
    \begin{equation}\label{eq:SchanuelSESM2}
        0 \to \mathbb{Z}^{\tilde{a}+b} \xrightarrow{M_2} \Z^{\tilde{b}+b} \xrightarrow{\tilde{\pi} \oplus \pi} H_1\left(Y, R^+\right) \to 0 
    \end{equation}
    where
    \[
    M_2 \coloneq \begin{blockarray}{ccc}
    \tilde{a} & b & \\
    \begin{block}{[cc]c}
    \widetilde{M} & -\eta' & \tilde{b} \\
    0 & I & b \\
    \end{block}
    \end{blockarray}
    \]
    and $\eta':\Z^{b} \to \Z^{\tilde{b}}$ is a map such that $\tilde{\pi}\circ \eta'=\pi$. We will show that ~\eqref{eq:SchanuelSESM1} and ~\eqref{eq:SchanuelSESM2} are free resolutions for $H_1\left(Y, R^+\right)$, i.e. $M_1$ and $M_2$ are injective presentation matrices for $H_1\left(Y, R^+\right)$ of deficiency $b-a$ and $\tilde{b}-\tilde{a}$ respectively. Specifically, we will show $\im M_1=\ker(\pi\oplus \tilde{\pi})$; the fact that $\im M_2=\ker(\tilde{\pi}\oplus \pi)$ follows analogously. 
    
    We can interpret $M_1$ as a linear map $M_1(x,y)=(M(x)-\eta(y),y)$. Then $\im M_1$ is contained in $\ker(\pi\oplus \tilde{\pi})$ since
    \begin{align*}
        (\pi \oplus \tilde{\pi})\left(M_1(x, y)\right)&=(\pi \oplus \tilde{\pi})\left(M(x)-\eta(y), y\right) \\
        &=\left(\pi(M(x))-\pi(\eta(y))\right)+\tilde{\pi}(y) \\
        &=0-\tilde{\pi}(y)+\tilde{\pi}(y) \\
        &=0
    \end{align*}
    for any $(x,y)$, where in the third equality we are using that $\pi \circ M=0$ by the exactness of the sequence defining $M$. 
    
    For the reverse containment, let $(x,y) \in \ker(\pi\oplus \tilde{\pi})$. We have 
    \begin{align*}
        \pi(x)+\tilde{\pi}(y)&=0 \\
        \pi(x)+\pi(\eta(y))&=0 \\
        \pi(x+\eta(y))&=0 
    \end{align*}
    which implies $x+\eta(y) \in \ker \pi = \im M$. So there exists some $z \in \Z^a$ with 
    \begin{align*}
        M(z)&=x+\eta(y) \\
        M(z)-\eta(y)&=x
    \end{align*}
    and hence $(x,y)=(M(z)-\eta(y),y)=M_1(z,y)$, proving the claim. 

    Equipped with the fact that ~\eqref{eq:SchanuelSESM1} and ~\eqref{eq:SchanuelSESM2} are injective deficiency-$d$ presentations for $H_1\left(Y, R^+\right)$, we now show that the corresponding Alexander functions agree (up to sign) with those computed using $M$ and $\widetilde{M}$, respectively:
    \begin{enumerate}
        \item[(I)] ($\A^{\Z}_{Y,\Gamma,M}=\pm \A^{\Z}_{Y,\Gamma,M_1}$): Choose arbitrary representatives $\overline{u}_i \in \Z^b$ to compute $\mathcal{A}^{\Z}_{Y,\Gamma,M}$. By Lemma~\ref{lem:AlexanderFunctionZWellDefined}, we are free to use the representatives
        \[
        \overline{u}'_i=\left[ \begin{array}{c} \overline{u}_i \\ 0 \end{array} \right]        
        \]
        to compute $\mathcal{A}^{\Z}_{Y,\Gamma,M_1}$. Then we have
        \begin{align*}
            \det([M_1|\overline{u}'_1 \cdots \overline{u}'_d])&=\det \left[
            \begin{array}{c|c|ccc}
            M & -\eta & \overline{u}_1 & \cdots & \overline{u}_d \\
            \hline
            0 & I & & 0 &
            \end{array}
            \right] \\
            &=(-1)^{\tilde{b}d} \det \left[
            \begin{array}{c|ccc|c}
            M &  \overline{u}_1 & \cdots & \overline{u}_d & -\eta \\
            \hline
            0 & & 0 & & I
            \end{array}
            \right] \\
            &= (-1)^{\tilde{b}d} \det([M|\overline{u}_1 \cdots \overline{u}_d]) \det(I) \\
            &=(-1)^{\tilde{b}d} \det([M|\overline{u}_1 \cdots \overline{u}_d]).
        \end{align*}
        Thus $\mathcal{A}^{\Z}_{Y,\Gamma,M}=\mathcal{A}^{\Z}_{Y,\Gamma,M_1}$ up to an overall sign of $(-1)^{\tilde{b}d}$.         
        \item[(II)] ($\A^{\Z}_{Y,\Gamma,M_2}= \pm \A^{\Z}_{Y,\Gamma,\widetilde{M}}$): Follows analogously to the preceding point. In this case the overall sign is $(-1)^{bd}$. 
    \end{enumerate}
    To complete the proof, we can therefore show that $\A^{\Z}_{Y,\Gamma,M_1}=\A^{\Z}_{Y,\Gamma,M_2}$ up to sign. We achieve this by assembling the presentations ~\eqref{eq:SchanuelSESM1} and ~\eqref{eq:SchanuelSESM2} into a larger diagram
    \[
    \xymatrix{
    0 \ar[r] & \mathbb{Z}^{a+\tilde{b}} \ar[r]^{M_1} \ar[d]_{\varphi} & \mathbb{Z}^{b+\tilde{b}} \ar[r]^{\pi \oplus \tilde{\pi} \quad} \ar[d]^{S} & H_1(Y, R^+) \ar[r] \ar[d]^{I} & 0 \\
    0 \ar[r] & \mathbb{Z}^{\tilde{a}+b} \ar[r]^{M_2} & \mathbb{Z}^{\tilde{b}+b} \ar[r]^{\tilde{\pi} \oplus \pi \quad} & H_1(Y, R^+) \ar[r] & 0 
    }
    \]
    where $S \colon \mathbb{Z}^{b+\tilde{b}} \to \mathbb{Z}^{\tilde{b}+b}$ is the swap map $S(x,y)=(y,x)$. The right square commutes automatically. We claim that there exists a unique $\Z$-linear map $\varphi$ such that the left square commutes, and furthermore, that this forces $\varphi$ to be an isomorphism. Let $x \in \Z^{a+\tilde b}$. Since $(\pi \oplus \tilde{\pi}) \circ M_1 = 0$ by exactness of the top row, commutativity of the right square implies 
    \[
    (\tilde{\pi} \oplus \pi)\bigl(S(M_1(x))\bigr)
        = I\bigl((\pi \oplus \tilde{\pi})(M_1(x))\bigr) \\
        = (\pi \oplus \tilde{\pi})(M_1(x)) 
        = 0
    \]
    Hence $S(M_1(x)) \in \ker(\tilde{\pi} \oplus \pi)$. By exactness of the bottom row, $\ker(\tilde{\pi} \oplus \pi) = \operatorname{im} M_2$, so there exists some $\varphi(x) \in \Z^{\tilde a + b}$ such that $M_2(\varphi(x)) = S(M_1(x))$. Since $M_2$ is injective, this element $\varphi(x) \in \Z^{\tilde a + b}$ is unique. We define $\varphi \colon \Z^{a+\tilde{b}} \to \Z^{\tilde{a}+b}$ by this assignment; by construction, the left square commutes. To see that $\varphi$ is necessarily an isomorphism, we check injectivity and surjectivity directly (inspired by the proof of \cite[Lemma~A.2.21]{LescopBook}). 
    
    For injectivity, suppose that $\varphi(x)=0$. Commutativity implies 
    \[
    S(M_1(x))=M_2(\varphi(x))=M_2(0)=0,
    \]
    so $M_1(x)=0$, since $S$ is an isomorphism. The injectivity of $M_1$ implies that we must have $x=0$, so $\varphi$ is injective. 
    
    For surjectivity, let $y \in \Z^{\tilde{a}+b}$. Then $M_2(y) \in \ker (\tilde{\pi} \oplus \pi)$ by exactness of the bottom row. Commutativity of the right square gives $(\tilde{\pi} \oplus \pi) \circ S=I \circ (\pi \oplus \tilde{\pi})$, so $(\tilde{\pi} \oplus \pi) =(\pi \oplus \tilde{\pi}) \circ S^{-1}$. This implies 
    \[
    (\tilde{\pi} \oplus \pi)\bigl(M_2(y)\bigr)=(\pi \oplus \tilde{\pi})\bigl(S^{-1}(M_2(y))\bigr)=0,
    \]
    hence $S^{-1}(M_2(y)) \in \ker(\pi \oplus \tilde{\pi})= \im M_1$ by exactness of the top row. Thus, there exists $x \in \Z^{a+\tilde b}$ such that $M_1(x) = S^{-1}(M_2(y))$, or equivalently, $S(M_1(x)) = M_2(y)$. By definition of $\varphi$, this implies $\varphi(x)=y$, so $\varphi$ is surjective. Now that we have established the map $\varphi$, we may proceed with the proof. 
    \begin{enumerate}
        \item[(III)] ($\A^{\Z}_{Y,\Gamma,M_1}=\pm \A^{\Z}_{Y,\Gamma,M_2}$): Choose arbitrary representatives $\overline{u}'_i=\begin{bmatrix} \overline{u}_i \\ 0 \end{bmatrix} \in \Z^{b+\tilde{b}}$ to compute $\mathcal{A}^{\Z}_{Y,\Gamma,M_1}$. 
        Commutativity gives $S \circ M_1=M_2 \circ \varphi$.  Appending the columns $S\overline{u}'_i = \begin{bmatrix} 0 \\ \overline{u}_i \end{bmatrix}$ to each side of this equation implies that
        \begin{align*}
            [SM_1|S\overline{u}'_1 \cdots S\overline{u}'_d]&=[M_2 \varphi |S\overline{u}'_1 \cdots S\overline{u}'_d]  \\
            S \cdot [M_1|\overline{u}'_1 \cdots \overline{u}'_d] &=  [M_2|S\overline{u}'_1 \cdots S\overline{u}'_d] \cdot \left[ \begin{array}{c|c} \varphi & 0 \\ \hline 0 & I_{d \times d} \end{array} \right]. 
        \end{align*}
        Taking determinants, we see that
        \begin{align*}
            \det(S)\det[M_1|\overline{u}'_1 \cdots \overline{u}'_d]&=\det [M_2|S\overline{u}'_1 \cdots S\overline{u}'_d] \det(\varphi)  \\
            \det[M_1|\overline{u}'_1 \cdots \overline{u}'_d] &= \det(\varphi) \det(S^{-1})  \det [M_2|S\overline{u}'_1 \cdots S\overline{u}'_d] \\
            &=u \cdot \det [M_2|S\overline{u}'_1 \cdots S\overline{u}'_d],
        \end{align*}
        where $u=\det(\varphi)\det(S^{-1})$ is a unit in the ground ring $\Z$. Indeed, both $\varphi$ and $S$ are invertible, so their determinants are units, and hence so is their product.\footnote{In particular, $\det(S^{-1})=(-1)^{b\tilde{b}}$. Thus one may also view this equality as holding up to multiplication by $\pm 1$ times an additional unit $\det(\varphi)$ in the ground ring. This perspective foreshadows the $\Z[G]$ setting, where the analogous Proposition~\ref{prop:AlexanderFunctionZ[G]IndependentOfM} will hold up to multiplication by elements of $\pm G$.} That is,
        \[
        \det[M_1|\overline{u}_1 \cdots \overline{u}_d]= \pm \det [M_2|S\overline{u}_1 \cdots S\overline{u}_d] 
        \]
        since the units in $\Z$ are $\pm 1$. Note that $\det(\varphi)$ does not depend on $u_1 \wedge \cdots \wedge u_d$. 

        Since $(\tilde{\pi}\oplus\pi)\circ S=\pi\oplus\tilde{\pi}$, the representatives $\overline{u}'_i$ mapping to $u_i$ under $\pi\oplus\tilde{\pi}$ also map to $u_i$ under $(\tilde{\pi}\oplus\pi)\circ S$. Thus, $S\overline{u}'_i$ are a valid set of columns to use when computing $\mathcal{A}^{\Z}_{Y,\Gamma,M_2}$, and we are free to choose these by Lemma~\ref{lem:AlexanderFunctionZWellDefined}. Therefore,
        \begin{align*}
            \det[M_1|\overline{u}'_1 \cdots \overline{u}'_d] &= \pm \det [M_2|S\overline{u}'_1 \cdots S\overline{u}'_d] \\
            \mathcal{A}^{\Z}_{Y,\Gamma,M_1}(u_1\wedge \cdots \wedge u_d) &= \pm \mathcal{A}^{\Z}_{Y,\Gamma,M_2}(u_1\wedge \cdots \wedge u_d).
        \end{align*}
    \end{enumerate}
    Putting the equalities (I), (II) and (III) together, we obtain
    \[
    \mathcal{A}^{\Z}_{Y,\Gamma,M}(u_1\wedge \cdots \wedge u_d) = \mathcal{A}^{\Z}_{Y,\Gamma,\widetilde{M}}(u_1\wedge \cdots \wedge u_d)
    \]
    up to an overall sign independent of $u_1 \wedge \cdots \wedge u_d$. 
\end{proof}

\subsection{The Alexander functor over $\Z$}
Now that we have established the well-definedness of $\A^{\Z}_{Y,\Gamma}$, we may use it to define the Alexander functor. Recall that over any commutative ring $R$, if $M$ is a rank $n$ free $R$-module, a \emph{volume form} on $M$ is a choice of $R$-linear isomorphism $\omega \colon \wedge^{n}_RM \to R$; equivalently, a choice of generator for the rank-one free module $\wedge^{n}_RM \cong R$. Any two choices of generators, thus volume forms, differ by a unit in the ground ring $R$. 

Now let $n_i = \mathrm{rank} \, H_1(F_i,S^+_i)$ and $c = n_1+ \chi(Y,R^+)$. Note that $d = -\chi(Y,R^+) = n_1 - c$. Write $i_*$ to denote either of the maps 
\[
H_1(F_i,S^+_i) \to H_1(Y,R^+)
\]
induced by the inclusion of pairs $i:(F_i,S^+_i) \to (Y,R^+)$.

\begin{proposition}\label{prop:AlexanderFunctorZDefn}
    There is a $\Z$-linear map
    \[
    \mathsf{A}_{\Z}(Y,\Gamma) \colon \wedge^* H_1(F_0,S^+_0) \to \wedge^* H_1(F_1,S^+_1),
    \]
    unique up to sign and homogeneous of degree $c$, such that
    \begin{equation}\label{eq:AlexanderFunctorZDefinition}
        \omega(\wedge(\mathsf{A}_{\Z}(Y,\Gamma) \otimes \id)(x \otimes y)) = \A^{\Z}_{Y,\Gamma}(i_* x \wedge i_* y)
    \end{equation}
    where $x \in \wedge^* H_1(F_0,S^+_0)$, $y \in \wedge^{*} H_1(F_1,S^+_1)$, the map
    \[
    \wedge \colon \wedge^{p} H_1(F_1,S^+_1) \otimes \wedge^{q} H_1(F_1,S^+_1) \to \wedge^{n_1} H_1(F_1,S^+_1)
    \]
    sends $z \otimes w$ to $z \wedge w$ when $p+q = n_1$ and is zero otherwise, $\omega$ denotes any volume form on the free $\Z$-module $H_1(F_1,S^+_1)$, the tensor product $\mathsf{A}_{\Z}(Y,\Gamma) \otimes \id$ is computed according to the super-sign rule in Definition~\ref{def:AbZgrCategory} with $|\mathsf{A}_{\Z}(Y,\Gamma)| = c$, and the Alexander functor $\A^{\Z}_{Y,\Gamma}$ is defined to be zero on inputs in $\wedge^{d'} H_1(Y,R^+)$ for $d' \neq d$.
\end{proposition}

\begin{proof}    
    The proof is organized as follows. Fix a basis $\{ \gamma^{\mathrm{in}}_1,\dots,\gamma^{\mathrm{in}}_{n_0}\}$ for $H_1(F_0, S^+_0)$ and a basis $\{ \gamma^{\mathrm{out}}_1,\dots,\gamma^{\mathrm{out}}_{n_1}\}$ for $H_1(F_1, S^+_1)$ as in Definition~\ref{def:BSDAZComb}. We first show that the matrix coefficients of $\mathsf{A}_{\Z}(Y,\Gamma)$ with respect to the incoming basis $\{\gamma^{\inrm}_I\}$ for $\wedge^* H_1(F_0,S^+_0)$ and outgoing basis $\{\gamma^{\out}_J\}$ for $\wedge^* H_1(F_1,S^+_1)$ are uniquely determined by the defining relation~\eqref{eq:AlexanderFunctorZDefinition}. That is, for each pair of basis elements $(\gamma^{\mathrm{in}}_I,\gamma^{\out}_J)$, we require that~\eqref{eq:AlexanderFunctorZDefinition} holds when $x = \gamma^{\inrm}_I$ and $y = \gamma^{\out}_{J^c}$ (where $J^c$ is the complement of $J$). This forces an explicit formula for the coefficient of $\gamma^{\mathrm{out}}_J$ in $\mathsf{A}_{\Z}(Y,\Gamma)(\gamma^{\mathrm{in}}_I)$, so the map is uniquely determined on the chosen basis. We then verify that the defining relation~\eqref{eq:AlexanderFunctorZDefinition} continues to hold for arbitrary $x$ and $y$. Lastly, we establish the degree of $\mathsf{A}_{\Z}(Y,\Gamma)$ and show that $\mathsf{A}_{\Z}(Y,\Gamma)$ is independent of the choice of volume form up to overall sign.
    
    For convenience, fix the volume form $\omega$ such that
    \[
    \omega(\gamma^{\mathrm{out}}_1 \wedge \cdots \wedge \gamma^{\mathrm{out}}_{n_1}) = +1.
    \]
    Note that  
    \[
    \gamma^{\mathrm{out}}_J \wedge \gamma^{\mathrm{out}}_{J^c} = (-1)^{\mathrm{inv}(\sigma_{JJ^c \leftrightarrow \std})} \cdot  \gamma^{\mathrm{out}}_1 \wedge \cdots \wedge \gamma^{\mathrm{out}}_{n_1}
    \]
    where $\sigma_{JJ^c \leftrightarrow \mathrm{std}}$ is given as in Definition~\ref{def:BSDAZComb}. Applying $\omega$ to both sides gives
    \[
    \omega(\gamma^{\mathrm{out}}_J \wedge \gamma^{\mathrm{out}}_{J^c}) = (-1)^{\mathrm{inv}(\sigma_{JJ^c \leftrightarrow \std})}.
    \]
    
    \emph{Uniqueness}: If we take $x = \gamma^{\mathrm{in}}_I$ and $y = \gamma^{\mathrm{out}}_{J^c}$ in equation~\eqref{eq:AlexanderFunctorZDefinition}, after applying the super-sign rule, we obtain
    \begin{align*}
        \A^{\Z}_{Y,\Gamma}(i_* \gamma^{\mathrm{in}}_I \wedge i_* \gamma^{\mathrm{out}}_{J^c})
        &= \omega\!\left(
            \wedge(\mathsf{A}_{\Z}(Y,\Gamma)\otimes\id)
            (\gamma^{\mathrm{in}}_I\otimes \gamma^{\mathrm{out}}_{J^c})
          \right) \\
        &= \omega\!\left(
            (-1)^{c(n_1-l)}\,
            \mathsf{A}_{\Z}(Y,\Gamma)(\gamma^{\mathrm{in}}_I)
            \wedge \gamma^{\mathrm{out}}_{J^c}
          \right) \\
        &= (-1)^{c(n_1-l)}\,\omega\!\left(
            \sum_{J'} a_{I,J'}\,\gamma^{\mathrm{out}}_{J'}\wedge\gamma^{\mathrm{out}}_{J^c}
          \right) \\
        &= (-1)^{c(n_1-l)}\,\sum_{J'} a_{I,J'}\,
           \omega(\gamma^{\mathrm{out}}_{J'}\wedge\gamma^{\mathrm{out}}_{J^c}) \\
        &= (-1)^{c(n_1-l)}\,a_{I,J}\,
           (-1)^{\mathrm{inv}(\sigma_{JJ^c \leftrightarrow \std})},
    \end{align*}
    since $\omega(\gamma^{\mathrm{out}}_{J'} \wedge \gamma^{\mathrm{out}}_{J^c})=0$ unless $J'=J$.  
    Solving for $a_{I,J}$ gives
    \[
    a_{I,J}
        = (-1)^{\mathrm{inv}(\sigma_{JJ^c \leftrightarrow \std})}
          \,(-1)^{c(n_1-l)}\,
          \A^{\Z}_{Y,\Gamma}(i_*\gamma^{\mathrm{in}}_I \wedge i_*\gamma^{\mathrm{out}}_{J^c}).
    \]
    Since $\A^{\Z}_{Y,\Gamma}$ is well-defined (cf. Lemma~\ref{lem:AlexanderFunctionZWellDefined}, Proposition~\ref{prop:AlexanderFunctionZIndependentOfM}), these coefficients $a_{I,J}$ must be uniquely determined up to overall sign. 
    
    \emph{Existence}: There exists a unique (up to overall sign)
    $\Z$-linear map $\mathsf{A}_{\Z}(Y,\Gamma)$ defined by the sum
    \[
    \mathsf{A}_{\Z}(Y,\Gamma)(\gamma^{\mathrm{in}}_I) = \sum_{J'} a_{I,J'} \cdot \gamma^{\mathrm{out}}_{J'}
    \]
    where each $a_{I,J'}$ is defined by the explicit formula ~\eqref{eq:AlexanderFunctorZDefinition}. In other words, we define $\mathsf{A}_{\Z}(Y,\Gamma)$ to be the unique linear map such that equation~\eqref{eq:AlexanderFunctorZDefinition} holds when $x = \gamma^{\mathrm{in}}_I$ and $y = \gamma^{\mathrm{out}}_{J^c}$. 
    
    \emph{Verifying for arbitrary elements}: Let $x=\sum_I a_I \gamma_I^{\inrm} \in \wedge^* H_1\left(F_0, S_0^{+}\right)$ and $y=\sum_J b_{J^c} \gamma_{J^c}^{\out} \in \wedge^* H_1\left(F_1, S_1^{+}\right)$. Note that $i_*$, $\omega$, and $\mathsf{A}_{\Z}(Y,\Gamma) \otimes \id$ are $\Z$-linear, and $\A^{\Z}_{Y,\Gamma}$ is multilinear. After applying equation~\eqref{eq:AlexanderFunctorZDefinition}, we have
    \begin{align*}
        \A^{\Z}_{Y,\Gamma}(i_* x \wedge i_* y)
        &= \A^{\Z}_{Y,\Gamma}\left(i_* \left( \sum_{I} a_I\, \gamma^{\mathrm{in}}_I \right) \wedge i_* \left( \sum_{J} b_{J^c}\, \gamma^{\mathrm{out}}_{J^c} \right) \right)  \\
        &=\A^{\Z}_{Y,\Gamma}\left( \sum_{I}  a_I\, i_*\gamma^{\mathrm{in}}_I  \wedge  \sum_{J} b_{J^c}\, i_*\gamma^{\mathrm{out}}_{J^c} \right) \\
        &= \sum_{I,J} a_I b_{J^c}\, \A^{\Z}_{Y,\Gamma}\left(i_* \gamma^{\mathrm{in}}_I \wedge i_* \gamma^{\mathrm{out}}_{J^c} \right)  \\
        &= \sum_{I,J} a_I b_{J^c}\,
            \omega\left(\wedge(\mathsf{A}_{\Z}(Y,\Gamma) \otimes \id)(\gamma^{\mathrm{in}}_I \otimes \gamma^{\mathrm{out}}_{J^c})\right)  \\
        &= \omega\left(\wedge\left( \sum_{I,J} a_I b_{J^c}\, (\mathsf{A}_{\Z}(Y,\Gamma) \otimes \id)(\gamma^{\mathrm{in}}_I \otimes \gamma^{\mathrm{out}}_{J^c}) \right)\right)  \\
        &= \omega\left(\wedge\left( (\mathsf{A}_{\Z}(Y,\Gamma) \otimes \id)\left( \sum_{I,J} a_I b_{J^c}\, \gamma^{\mathrm{in}}_I \otimes \gamma^{\mathrm{out}}_{J^c} \right) \right)\right) \\
        &= \omega\left( \wedge \left((\mathsf{A}_{\Z}(Y,\Gamma) \otimes \id)\left(\sum_I a_I \gamma_I^{\text {in }}  \otimes \sum_J b_{J^c} \gamma_{J^c}^{\text {out}} \right) \right) \right) \\
        &= \omega\left( \wedge \left((\mathsf{A}_{\Z}(Y,\Gamma) \otimes \id)(x \otimes y) \right)\right).
    \end{align*}
    Therefore, the defining relation~\eqref{eq:AlexanderFunctorZDefinition} holds for arbitrary $x$ and $y$.

    \emph{Degree}: Let $k = |x|$ and let $l = n_1 - |y|$, so that $|y| = n_1 - l$. The right hand side of equation~\eqref{eq:AlexanderFunctorZDefinition} is nonzero only when $i_* x \wedge i_* y$ lies in $\wedge^d H_1(Y, R^+)$, where $d=-\chi(Y,R^+)$, i.e. only when 
    \[
    |i_* x|+|i_* y|=|x|+|y|=k+n_1 - l=d
    \]
   where the first equality holds since $i_*$ is homogeneous. On the left hand side of equation~\eqref{eq:AlexanderFunctorZDefinition}, we apply a map $\omega$ of degree $-n_1$ to an element $\wedge(\mathsf{A}_{\Z}(Y,\Gamma) \otimes \mathrm{id})(x \otimes y)$ of $\wedge^{k'+n_1 - l} H_1(F_1, S_1^+)$, where $k'$ is the degree of $\mathsf{A}_{\Z}(Y,\Gamma)(x)$; the result is nonzero only when $k' - l = 0$, i.e. $k' = l$. Substituting $k'$ for $l$ in $k+n_1 - l = d$ gives
   \[
   k' - k = n_1 - d = c.
   \]
   Thus, $k' = k + c$, so $\mathsf{A}_{\Z}(Y,\Gamma)$ is homogeneous of degree $c$.

    \emph{Sign ambiguity}: Recall that any other choice of volume form will differ from $\omega$ by a unit in the ground ring $\Z$. In other words, there are only two choices of volume form on $H_1(F_1,S^+_1)$: $\omega$ and $-\omega$. If we had instead chosen the other one, $-\omega$, all the coefficients $a_{I,J}$ would carry an extra global factor of $-1$. Hence, changing the volume form does not affect $\mathsf{A}_{\Z}(Y,\Gamma)$ up to overall sign. 
\end{proof}

\begin{remark}\label{rem:AlexanderZSpecialCaseOfFM}
    Note that our sutured Alexander function $\mathcal{A}^{\Z}_{Y,\Gamma}$ recovers the Alexander function of \cite[Definition~2.2]{FMFunctorial} as a special case: If $G={1}$, $Y=M$ is connected with connected boundary, and $R^+$ is a disk, collapsing $R^+$ to a point and viewing this as a basepoint $\star \in \partial M$ identifies our $H_1(Y,R^+)$ with Florens--Massuyeau's $H_1^{\varphi}(M,\star)$, and our $\mathcal{A}^{\Z}_{Y,\Gamma}$ with their $\mathcal{A}^M_{\varphi}$. 
        
    Similarly, our Alexander functor $\mathsf{A}_{\Z}(Y,\Gamma)$ recovers the Alexander functor $\mathsf{A}(M,\varphi)$ of \cite[Section~2.2]{FMFunctorial} as a special case: If $G={1}$ and $(Y=M,\Gamma)$ is a connected cobordism such that each $S_i^+ \subset F_i$ consists of a single interval, collapsing these intervals to basepoints $\star \in \partial F_{\pm g}=S^1$ identifies our $H_1(F_i,S^+_i)$ with Florens--Massuyeau's $H_1^{\varphi}(F_{\pm g},\star)$, and thus our $\mathsf{A}_{\Z}(Y,\Gamma)$ with their $\mathsf{A}(M,\varphi)$.
\end{remark}

\subsection{Identifying $[\BSDA]$ with the Alexander functor over $\Z$}\label{subsec:BSDAequalsAlexanderZ}

\begin{theorem}\label{thm:BSDAAlexanderZ}
    Let $\Xi$ be any set of choices in Definition~\ref{def:BSDAZComb}. We have
    \[
    [\BSDA(Y,\Gamma;\Xi)]^{\Z}_{\comb} = \mathsf{A}_{\Z}(Y,\Gamma)
    \]
    as maps from $\wedge^* H_1(F_0,S^+_0)$ to $\wedge^* H_1(F_1,S^+_1)$ up to sign.
\end{theorem}

\begin{proof}       
    By Corollary~\ref{cor:CanChooseXiNormalized}, we may assume $\Xi$ is of the form $\Xi_{\norm}$. Let $\gamma^{\inrm}_I$ and $\gamma^{\out}_J$ be basis elements of $\wedge^* H_1(F_0,S^+_0)$ and $\wedge^* H_1(F_1,S^+_1)$ respectively. Let $k = |I|$ and $l = |J|$, and assume that $l = k+c$ where $c = n_1 + \chi(Y,R^+)$ as usual.

    As in the proof of Proposition~\ref{prop:AlexanderFunctorZDefn}, fix the volume form $\omega$ such that
    \[
    \omega(\gamma^{\mathrm{out}}_1 \wedge \cdots \wedge \gamma^{\mathrm{out}}_{n_1}) = +1.
    \]
    We will compute
    \begin{equation}\label{eq:TargetQuantity}
    \omega(\wedge([\BSDA(Y,\Gamma;\Xi_{\norm})]^{\Z}_{\comb} \otimes \id)(\gamma^{\inrm}_I \otimes \gamma^{\out}_{J^c}))
    \end{equation}
    and show the result agrees with $\mathcal{A}^{\Z}_{Y,\Gamma}(i_* \gamma^{\inrm}_I \wedge i_* \gamma^{\out}_{J^c})$ up to an overall sign that is independent of $I$ and $J$. 
    
    By Proposition~\ref{prop:AlexanderFunctionZIndependentOfM}, we may use the presentation matrix $M_{\Hc_{\norm}}$ for $H_1(Y,R^+)$ to compute $\mathcal{A}^{\Z}_{Y,\Gamma}(i_* \gamma^{\inrm}_I \wedge i_* \gamma^{\out}_{J^c})$; by the discussion in Section~\ref{sec:CellularChainCx}, we can understand $i_* \gamma^{\inrm}_I$ and $i_* \gamma^{\out}_{J^c}$ in terms of
    \[
    i_* \colon H_1(\mathfrak{F}_i,\mathfrak{S}^+_i) \to H_1(\mathfrak{Y}, \mathfrak{R}^+_i).
    \]
    
    By our super sign convention, \eqref{eq:TargetQuantity} equals
    \[
    (-1)^{c(n_1-k-c)} \omega([\BSDA(Y,\Gamma;\Xi_{\norm})]^{\Z}_{\comb} (\gamma^{\inrm}_I) \wedge \gamma^{\out}_{J^c}),
    \]
    and up to an overall $k$-independent sign, we can replace $(-1)^{c(n_1 - k -c)}$ with $(-1)^{ck}$. Thus, using our specific choice of volume form $\omega$, we see that \eqref{eq:TargetQuantity} equals $(-1)^{ck + \mathrm{inv}(\sigma_{J J^c \leftrightarrow \std})}$ times the matrix entry $E$ of $[\BSDA(Y,\Gamma;\Xi_{\norm})]^{\Z}_{\comb}$ in column $\gamma^{\inrm}_I$ and row $\gamma^{\out}_J$. 
    
    By the permutation expansion of the determinant, we can write $E$ as
    \begin{equation}\label{eq:EAsDeterminantZ}
    (-1)^{\mathrm{inv}(\sigma_{J J^c \leftrightarrow \std}) + ak + n_1 k} 
    \det \kbordermatrix{
      & \alpha^{\out}_{J^c} & \alpha^c & \alpha^{\inrm}_I \\
    \beta^{\out} & -*_{J^c} & * & 0\\
    \beta^c & 0 & * & 0\\
    \beta^{\inrm} & 0 & * & *_I}
    \end{equation}
    where the matrix entries are defined to be the algebraic intersection numbers $\beta \cdot \alpha$ for the relevant $\alpha$ and $\beta$ curves; each column of the block $-*_{J^c}$ is $-1$ in a single row corresponding to an element of $J^C$ and is zero elsewhere, and each column of the block $*_I$ is $1$ in a single row corresponding to an element of $I$ and is zero elsewhere. 

    The proof now breaks up into two cases. Suppose that $H_2(Y,R^+) \neq 0$. We have $\mathsf{A}_{\Z}(Y,\Gamma)=0$, since we defined $\A^{\Z}_{Y,\Gamma}=0$ in this case in Definition~\ref{def:AlexanderFunctionZ}. On the other hand, Lemma~\ref{lem:InjectivePresentationMatrixZ} implies that $M_{\Hc_{\norm}}$ is not injective; equivalently, the columns of $M_{\Hc_{\norm}}$ are $\Z$-linearly dependent. Since the middle $\alpha^c$ column (of three blocks) in the matrix of equation~\eqref{eq:EAsDeterminantZ} is $M_{\Hc_{\norm}}$, the columns of the matrix of equation~\eqref{eq:EAsDeterminantZ} are linearly dependent. It follows that the matrix entry $E$ in equation~\eqref{eq:EAsDeterminantZ} is zero. Thus,
    \[
    [\BSDA(Y,\Gamma;\Xi_{\norm})]^{\Z}_{\comb}=0.
    \]
    
    So suppose that $H_2(Y,R^+) = 0$. Then $M_{\Hc}$ is injective, and up to a sign that is independent of $k = |I|$, we have
    \begin{align*}
        &\omega(\wedge([\BSDA(Y,\Gamma;\Xi_{\norm})]^{\Z}_{\comb} \otimes \id)(\gamma^{\inrm}_I \otimes \gamma^{\out}_{J^c})) \\
        &= (-1)^{\mathrm{inv}(\sigma_{JJ^c \leftrightarrow \std}) + ck} E \\ 
        &= (-1)^{ak + n_1 k + ck} \det \kbordermatrix{
      & \alpha^{\out}_{J^c} & \alpha^c & \alpha^{\inrm}_I \\
    \beta^{\out} & -*_{J^c} & * & 0\\
    \beta^c & 0 & * & 0\\
    \beta^{\inrm} & 0 & * & *_I} \\
        &= \det \kbordermatrix{
       & \alpha^c & \alpha^{\inrm}_I & \alpha^{\out}_{J^c} \\
    \beta^{\out}  & * & 0 & *_{J^c}\\
    \beta^c & * & 0 & 0\\
    \beta^{\inrm} & * & *_I & 0} \\
        &= \A^{\Z}_{Y,\Gamma}(i_* \gamma^{\inrm}_I \wedge i_* \gamma^{\out}_{J^c});
    \end{align*}
    in the second-to-last step, we use that $(-1)^{(n_1-k-c)(a+k) + (n_1 - k - c)} = (-1)^{n_1 k + ak + ck}$ up to a $k$-independent sign. For the last step, note that expanding $\A^{\Z}_{Y,\Gamma}(i_* \gamma^{\inrm}_I \wedge i_* \gamma^{\out}_{J^c})$ directly using Proposition~\ref{prop:ExpandingBasisElts} (whose hypotheses hold given our assumptions about $\Xi_{\norm}$) would give the determinant from the second-to-last step but with $*_I$ and $*_{J^c}$ replaced with $-*_I$ and $-*_{J^c}$ respectively. In total there is a sign discrepancy of $k + (n_1 - k - c) = n_1 - c$ which is independent of $k$. It follows that $[\BSDA(Y,\Gamma;\Xi_{\norm})]^{\Z}_{\comb}$, and thus $[\BSDA(Y,\Gamma;\Xi)]^{\Z}_{\comb}$ for any set of choices $\Xi$, satisfies the equation characterizing $\mathsf{A}_{\Z}(Y,\Gamma)$ up to sign. 
\end{proof}

\subsection{Independence of $[\BSDA]$ over $\Z$}

Theorem~\ref{thm:BSDAAlexanderZ} relates the decategorified invariant $[\BSDA(Y,\Gamma;\Xi)]^{\Z}_{\comb}$ to a quantity that is independent of $\Xi$. Indeed, $\mathsf{A}_{\Z}(Y,\Gamma)$ is well-defined as an up-to-sign map, by Proposition~\ref{prop:AlexanderFunctorZDefn}. This observation allows us to define a version of $[\BSDA(Y,\Gamma)]^{\Z}_{\comb}$ that is independent of the auxiliary data $\Xi$.

\begin{definition}\label{def:BSDAZIndependentofChoices}
    Let $(Y,\Gamma)$ be a sutured cobordism, and let $\Xi$ be any set of choices as in Definition~\ref{def:BSDAZComb}.
    We define
    \[
        [\BSDA(Y,\Gamma)]^{\Z}_{\comb}
    \]
    to be the up-to-sign map
    \[
        [\BSDA(Y,\Gamma;\Xi)]^{\Z}_{\comb}
        \colon \wedge^* H_1(F_0,S^+_0)
        \to
        \wedge^* H_1(F_1,S^+_1),
    \]
    for any choice of $\Xi$. 
\end{definition}

\begin{corollary}\label{cor:BSDAZFunctorial}
    The assignment $\Cob^{\sut}_{2+1} \to \Ab^{\Z\gr}_{\pm 1}$ defined on objects as 
    \[
    (F,\Lambda) \mapsto \wedge^*H_1(F,S^+)
    \]
    and on morphisms as 
    \[
    (Y,\Gamma) \mapsto [\BSDA(Y,\Gamma)]^{\Z}_{\comb}
    \]
    is functorial. 
\end{corollary}

\begin{proof}
    By Proposition~\ref{prop:BSDARespectsIdentity}, we have
    \[
    [\BSDA(\id_{(F,\Lambda)},\Gamma_{\id})]^{\Z}_{\comb}=\id_{\wedge^*H_1(F,S^+)}
    \]
    as up-to-sign maps, where we use the set of choices $\Xi_{\id}$ of Proposition~\ref{prop:BSDARespectsIdentity} to compute $[\BSDA(\id_{(F,\Lambda)},\Gamma_{\id})]^{\Z}_{\comb}$. Similarly, if $(Y,\Gamma)$ and $(\tilde{Y},\tilde{\Gamma})$ are composable sutured cobordisms, Proposition~\ref{prop:BSDARespectsGluing} gives 
    \[
    [\BSDA(\tilde{Y}\cup_{F_1}Y,\tilde{\Gamma}\cup_{\Lambda_1}\Gamma)]^{\Z}_{\comb}=[\BSDA(\tilde{Y},\tilde{\Gamma})]^{\Z}_{\comb}\circ [\BSDA(Y,\Gamma)]^{\Z}_{\comb}
    \]
    as up-to-sign maps, where we use the composable sets of choices $\Xi$ and $\tilde{\Xi}$ and the glued choices $\tilde{\Xi}\cup \Xi$ as in Definition~\ref{def:ChoicesGluing} to compute the maps $[\BSDA(Y,\Gamma)]^{\Z}_{\comb}$, $[\BSDA(\tilde{Y},\tilde{\Gamma})]^{\Z}_{\comb}$, and $[\BSDA(\tilde{Y}\cup_{F_1}Y,\tilde{\Gamma}\cup_{\Lambda_1}\Gamma)]$ respectively.
\end{proof}

\section{Sutured Frohman--Nicas TQFT}\label{sec:SuturedFNTQFT}

Let $(Y,\Gamma)$ be a sutured cobordism from $(F_0,\Lambda_0)$ to $(F_1,\Lambda_1)$. In this section we will relate $[\BSDA(Y,\Gamma)]^{\Z}_{\comb}$ to the sutured Alexander functor defined in the statement of Theorem~\ref{thm:IntroFirstThm}. As in the introduction, let
\[
K = \rank H_1(F_0,S^+_0) + \rank H_1(F_1,S^+_1) + \chi(Y,R^+).
\]
\begin{proposition}\label{prop:RankKerLeqK}
    Assume $H_2(Y,R^-) = 0$. We have $\rank \ker(i_*) \leq K$ with equality if and only if $H_2(Y,R^+) = 0$.
\end{proposition}

    \begin{proof}
    We have exact sequences
    \[
    0 \to \ker i_* \to H_1(F_0 \sqcup F_1, S^+_0 \sqcup S^+_1) \xrightarrow{i_*} H_1(Y,R^+) \to H_1(Y,R^+)/\im(i_*) \to 0
    \]
    and
    \begin{align*}
    \cdots &\to H_1(F_1 \cup R^+ \cup F_0, R^+) \xrightarrow{i_*} H_1(Y, R^+) \to H_1(Y, F_1 \cup R^+ \cup F_0) \\
    &\to H_0(F_1 \cup R^+ \cup F_0, R^+) = 0.
    \end{align*}
    By excision, we have a canonical identification $H_1(F_0 \sqcup F_1, S^+_0 \sqcup S^+_1) \cong H_1(F_1 \cup R^+ \cup F_0, R^+)$, under which the maps $i_*$ to $H_1(Y,R^+)$ agree. It follows that $H_1(Y, F_1 \cup R^+ \cup F_0) \cong H_1(Y, R^+) / \im (i_*)$. On the other hand, the free part of $H_1(Y, F_1 \cup R^+ \cup F_0)$ is $H^1(Y, F_1 \cup R^+ \cup F_0) \cong H_2(Y,R^-)$, which we have assumed to be zero. It follows that $H_1(Y, F_1 \cup R^+ \cup F_0)$, and thus $H_1(Y, R^+) / \im (i_*)$, is finite and has rank zero. Then the first sequence gives
    \begin{align*}
    \rank \ker(i_*) &= \rank H_1(F_0,S^+_0) + \rank H_1(F_1,S^+_1) - \rank H_1(Y,R^+) \\
    &= \rank H_1(F_0,S^+_0) + \rank H_1(F_1,S^+_1) + \chi(Y,R^+) - \rank H_2(Y,R^+) \\
    &= K - \rank H_2(Y,R^+).
    \end{align*}
    It follows that $\rank \ker(i_*) \leq K$ with equality if and only if $\rank H_2(Y,R^+) = 0$. But $H_2(Y,R^+) \cong H^1(Y, F_1 \cup R^- \cup F_0)$ is free abelian, so we have equality if and only if $H_2(Y,R^+) = 0$.
\end{proof}

We now prove Theorem~\ref{thm:IntroFirstThm} from the introduction, starting with the proof of Lemma~\ref{lem:IntroMainLemma}. Given Theorem~\ref{thm:BSDAAlexanderZ}, this lemma can be viewed as a sutured generalization of Florens--Massuyeau's result ``Alexander functor for $G = \{1\}$ recovers the Frohman--Nicas TQFT'' (see \cite[Section 8.3]{FMFunctorial}) for the case of one-sided sutured cobordisms.

\begin{proof}[Proof of Lemma~\ref{lem:IntroMainLemma}]
    By Theorem~\ref{thm:BSDAAlexanderZ}, it suffices to show $\mathsf{A}_{\Z}(Y,\Gamma) = |K_{Y,\Gamma}|$ up to sign, where we are viewing $\mathsf{A}_{\Z}(Y,\Gamma)$ as an element of $\wedge^* H_1(F,S^+)$ up to sign. Writing
    \[
    \mathsf{A}_{\Z}(Y,\Gamma) = \sum_J a_{Y,\Gamma,J} \gamma_J
    \]
    where $a_{Y,\Gamma,J} \in \Z$, we have
    \[
    a_{Y,\Gamma,J} = (-1)^{|\mathsf{A}_{\Z}(Y,\Gamma)||\gamma_{J^c}| + \mathrm{inv}(\sigma_{J J^c \leftrightarrow \std})} \A^{\Z}_{Y,\Gamma}(i_* \gamma_{J^c})
    \]
    by definition of $\mathsf{A}_{\Z}(Y,\Gamma)$. Since $|\mathsf{A}_{\Z}(Y,\Gamma)| = c$ and $|\gamma_{J^c}| = |J^c| = n-c$, the sign $(-1)^{|\mathsf{A}_{\Z}(Y,\Gamma)||\gamma_{J^c}|}$ is independent of $J$ and can be disregarded. 
    
    Make choices $\Xi_{\norm}$ for $(Y,\Gamma)$ satisfying the requirements of Corollary~\ref{cor:CanChooseXiNormalized}. Let $\Hc_{\norm}$ be the Heegaard diagram chosen as part of $\Xi_{\norm}$ and let $M_{\Hc_{\norm}}$ be the corresponding presentation matrix for $H_1(Y,R^+)$; by Proposition~\ref{prop:AlexanderFunctionZIndependentOfM}, we are free to choose $M_{\Hc_{\norm}}$ to compute $\A^{\Z}_{Y,\Gamma}(i_* \gamma_{J^c})$.
    
    Recall by Lemma~\ref{lem:InjectivePresentationMatrixZ}, the presentation matrix $M_{\Hc_{\norm}}$ is injective if and only if $H_2(Y,R^+)=0$. This breaks the proof up into two cases. By definition, if $H_2(Y,R^+) \neq 0$, $\A^{\Z}_{Y,\Gamma}$ is the zero function, implying $\mathsf{A}_{\Z}(Y,\Gamma)=0$. On the other hand, from Proposition~\ref{prop:RankKerLeqK}, if $H_2(Y,R^+) \neq 0$ then we would have either $H_2(Y,R^-) \neq 0$ or $\rank \ker(i_*) < K$. In both of these cases $|K_{Y,\Gamma}|=0$ by the definition of $|K_{Y,\Gamma}|$ in Theorem~\ref{thm:IntroFirstThm}. Thus, the lemma holds in this case. 
        
    So assume that $H_2(Y,R^+) =0$ and thus $M_{\Hc_{\norm}}$ is injective. We can write $M_{\Hc_{\norm}}$ as $\kbordermatrix{ 
     & \alpha^c  \\
     \beta^{\out} & * \\
     \beta^c & * }$. If the bottom block $\kbordermatrix{ 
     & \alpha^c  \\
     \beta^c & * }$ of $M_{\Hc_{\norm}}$ has more rows than columns, then no generators $\x'$ for $\Hc'$ or $\x$ for $\Hc$ can exist, so $[\BSD(Y,\Gamma,\Xi_{\norm})]^{\Z}_{\comb} = 0$. On the other hand, by Proposition~\ref{prop:SubmatrixPresentation}, this bottom block is a presentation matrix for $H_1(Y, F \cup R^+)$, and ``more rows than columns'' in a presentation matrix forces $H_1(Y,F \cup R^+)$ to be infinite. Thus $H_2(Y,R^-) \neq 0$, so $|K_{Y,\Gamma}| = 0$ by definition and the lemma follows. We can hence assume without loss of generality that the bottom block $\kbordermatrix{ 
     & \alpha^c  \\
     \beta^c & * }$ of $M_{\Hc_{\norm}}$ has at least as many columns as rows.
     
     Note that if we apply row and column operations (over $\Z$) to $M_{\Hc_{\norm}}$ or its bottom block, we get other valid presentation matrices for $H_1(Y,R^+)$ or $H_1(Y,F \cup R^+)$ respectively. By applying row and column operations  over $\Z$ to the bottom block $\kbordermatrix{ 
     & \alpha^c  \\
     \beta^c & * }$ of $M_{\Hc_{\norm}}$, we can put it in the form $\kbordermatrix{
     & \widetilde{\alpha}^c_1 & \widetilde{\alpha}^c_2 \\
     \widetilde{\beta}^c & 0 & *_3
     }$ where $*_3$ is a square matrix in Smith normal form. It follows that $*_3$ is a presentation matrix for $H_1(Y,F \cup R^+)$, so we have $\det(*_3) = |H_1(Y,F \cup R^+)|$ (interpreted as zero if $H_1(Y,F \cup R^+)$ is infinite).
     
     Applying the same row and column operations to the larger matrix $M_{\Hc_{\norm}}$, we obtain a presentation matrix $\widetilde{M}_{\Hc_{\norm}}$ for $H_1(Y,R^+)$ of the form
     \[
     \widetilde{M}_{\Hc_{\norm}} = \kbordermatrix{
     & \widetilde{\alpha}^c_1 & \widetilde{\alpha}^c_2 \\
     \beta^{\out} & *_1 & *_2 \\
     \widetilde{\beta}^c & 0 & *_3
     };
     \]
     note that $\widetilde{M}_{\Hc_{\norm}}$ is injective because $M_{\Hc_{\norm}}$ is injective. For convenience, define $ (\widetilde{M}_{\Hc_{\norm}})_i$ for $i \in \{1,2\}$ by
     \[
     \widetilde{M}_{\Hc_{\norm}} = \kbordermatrix{
     & \widetilde{\alpha}^c \\
     \beta^{\out} & (\widetilde{M}_{\Hc_{\norm}})_1 \\
     \widetilde{\beta}^c & (\widetilde{M}_{\Hc_{\norm}})_2
     }.
     \]
     We have 
    \[
    a_{Y,\Gamma,J} = (-1)^{\mathrm{inv}(\sigma_{J J^c \leftrightarrow \std})} \det \kbordermatrix{ 
     & \widetilde{\alpha}^c & \alpha^{\out}_{J^c} \\
     \beta^{\out} & (\widetilde{M}_{\Hc_{\norm}})_1  & *_{J^c} \\
     \widetilde{\beta}^c & (\widetilde{M}_{\Hc_{\norm}})_2 & 0 }
    \]
    where the columns of $*_{J^c}$ each have $1$ in a single row and $0$ elsewhere (the rows with a $1$ are the rows corresponding to $J^c$). We can swap columns to get
    \[
    a_{Y,\Gamma,J} = (-1)^{\mathrm{inv}(\sigma_{J J^c \leftrightarrow \std})} \det \kbordermatrix{ 
     & \alpha^{\out}_{J^c} & \widetilde{\alpha}^c  \\
     \beta^{\out} & *_{J^c} & (\widetilde{M}_{\Hc_{\norm}})_1  \\
     \widetilde{\beta}^c & 0  & (\widetilde{M}_{\Hc_{\norm}})_2 }
    \]
    up to an overall sign. Swapping rows so that the top-left block is the identity matrix, we get 
    \[
    a_{Y,\Gamma,J} = (-1)^{\mathrm{inv}(\sigma_{J J^c \leftrightarrow \std}) + \mathrm{inv}(\sigma_{J^c J \leftrightarrow \std})} \det \kbordermatrix{ 
     & \alpha^{\out}_{J^c} & \widetilde{\alpha}^c  \\
     \beta^{\out}_{J^c} & I & (\widetilde{M}_{\Hc_{\norm}})_{1,J^c}  \\
     \beta^{\out}_J & 0 & (\widetilde{M}_{\Hc_{\norm}})_{1,J} \\
     \widetilde{\beta}^c & 0  & (\widetilde{M}_{\Hc_{\norm}})_2 }
    \]
    We have $|J| = c$ $(:= n + \chi(Y,R^+))$ and $|J^c| = n - c$, so $(-1)^{\mathrm{inv}(\sigma_{J J^c \leftrightarrow \std}) + \mathrm{inv}(\sigma_{J^c J \leftrightarrow \std})}$ amounts to $(-1)^{c(n-c)}$ which is a $J$-independent sign that can be disregarded. Thus, up to a $J$-independent sign, we have 
    \[
    a_{Y,\Gamma,J} =  \det \kbordermatrix{ 
     & \widetilde{\alpha}^c  \\
     \beta^{\out}_J & (\widetilde{M}_{\Hc_{\norm}})_{1,J} \\
     \widetilde{\beta}^c & (\widetilde{M}_{\Hc_{\norm}})_2 },
    \]
    the determinant of the deletion of the chosen presentation matrix $ \widetilde{M}_{\Hc_{\norm}}= \kbordermatrix{ 
     & \widetilde{\alpha}^c  \\
     \beta^{\out} & (\widetilde{M}_{\Hc_{\norm}})_1 \\
     \widetilde{\beta}^c & (\widetilde{M}_{\Hc_{\norm}})_2 }$ in which the rows corresponding to $J^c$ have been deleted.

     In other words, $a_{Y,\Gamma,J}$ is the determinant of the deletion of $\kbordermatrix{
     & \widetilde{\alpha}^c_1 & \widetilde{\alpha}^c_2 \\
     \beta^{\out} & *_1 & *_2 \\
     \widetilde{\beta}^c & 0 & *_3
     }$ in which the rows corresponding to $J^c$ (all part of the top set $\beta^{\out}$ of rows) have been deleted. Equivalently, $a_{Y,\Gamma,J}$ is $\det(*_3)$ times the determinant of the deletion of $*_1$ in which the rows corresponding to $J^c$ have been deleted.  If $\det(*_3) = 0$, we see that $a_{Y,\Gamma,J} = 0$ for all $J$, and also $|K_{Y,\Gamma}| = 0$ by definition since $H_1(Y,F \cup R^+)$ is infinite, so the lemma follows in this case.
     
     Assume that $\det(*_3) \neq 0$. We claim that, up to sign, $a_{Y,\Gamma,J}$ is $\det(*_3)$ times the coefficient of $\wedge^K \ker i_*$ on the basis element $\gamma_J$ of $\wedge^* H_1(F,S^+)$; this claim would prove the lemma.

     To prove the claim, note that by the discussion of Section~\ref{sec:CellularChainCx}, each basis element $\gamma_i$ of $H_1(F,S^+)$ corresponds to the negative of some basis one-cell for $C_1^{\cell}(\mathfrak{Y},\mathfrak{R}^+)$ (this one-cell is one of the ``$\beta^{\out}$'' basis vectors labeling rows of $M_{\Hc_{\norm}}$ and $\widetilde{M}_{\Hc_{\norm}}$). By commutativity of the square in Corollary~\ref{cor:SingCellSquareCommutes}, a linear combination of the $\gamma_i$ is in $\ker i_*$ if and only if the corresponding linear combination of negative basis one-cells ($\beta^{\out}$ basis vectors) for $C_1^{\cell}(\mathfrak{Y},\mathfrak{R}^+)$ is zero in $H_1(\mathfrak{Y},\mathfrak{R}^+)$, i.e. is in the span of the columns of $\widetilde{M}_{\Hc_{\norm}}$. 
     
     Now if some linear combination of $\beta^{\out}$ basis vectors is in the span of the columns of $\widetilde{M}_{\Hc_{\norm}}$ (indexed by $\tilde{\alpha}^c_1$ and $\tilde{\alpha}^c_2$), then actually this linear combination must be in the span of the columns indexed by $\tilde{\alpha}^c_1$ only. Indeed, suppose the coefficients on the $\tilde{\alpha}^c_2$ columns are $c_i$. Passing to the quotient of the span of all $\beta$ basis vectors by the $\beta^c$ basis vectors, we get that the linear combination of the columns of $*_3$ with coefficients $c_i$ is zero. Since $*_3$ is invertible, we conclude that all $c_i$ are zero as desired.

     It follows that the (negatives of) the columns of $*_1$, interpreted as coefficients for linear combinations of the basis vectors $\gamma_i$ of $H_1(F,S^+)$, span $\ker i_*$. Since $\widetilde{M}_{\Hc_{\norm}}$ is injective, these columns are also linearly independent, so they form a basis for $\ker i_*$. We can compute $\wedge^K \ker i_*$ by wedging these columns together; the coefficient of the the result on $\gamma_J$ is the determinant of the submatrix of $*_1$ in which the rows corresponding to $J^c$ have been deleted. The claim (and thus the lemma) follows.
\end{proof}

\begin{lemma}\label{lem:BSDAvsBSD}
    Let $(Y,\Gamma)$ be a sutured cobordism from $(F_0,\Lambda_0)$ to $(F_1,\Lambda_1)$, and use any set of choices $\Xi$ to define $[\BSDA(Y,\Gamma)]^{\Z}_{\comb}$ as in Definition~\ref{def:BSDAZIndependentofChoices}. Viewing $(Y,\Gamma)$ instead as a one-sided cobordism from $(\emptyset,\emptyset)$ to $(-F_0\sqcup F_1,\Lambda_0\sqcup \Lambda_1)$, the same construction yields a map $[\BSDD(Y,\Gamma)]^{\Z}_{\comb}$ (defined up to sign), which we can view as an up-to-sign element of
    \[ 
    \mathrm{Hom}_{\Z}(\Z,\wedge^* H_1(-F_0\sqcup F_1,S^+_0\sqcup S^+_1)) \cong \wedge^* H_1(F_0,S^+_0) \otimes \wedge^* H_1(F_1, S^+_1).
    \]
    We claim that the map $[\BSDA(Y,\Gamma)]^{\Z}_{\comb}$ agrees up to overall sign with the composition appearing in Theorem~\ref{thm:IntroFirstThm}, wherein we substitute the element $[\BSDD(Y,\Gamma)]^{\Z}_{\comb}$ for $|K_{Y,\Gamma}|$.
\end{lemma}

\begin{proof}
    We compute the composition $(\varepsilon \otimes \id) \circ (\id \otimes [\widehat{BSDD}(Y,\Gamma)]^{\Z}_{\comb})$ directly and compare it with $[\BSDA(Y,\Gamma)]^{\Z}_{\comb}$. 
    
    To do this, we need choices $\Xi_{\DD}$ compatible with viewing $(Y,\Gamma)$ as a one-sided cobordism. We will choose $\Xi_{\DD}$ so that all the orientations of curves in the Heegaard diagram (and thus all local intersection signs) end up the same as when we view $(Y,\Gamma)$ as two-sided and use the choices $\Xi$. To achieve this, we take the arc diagram representing $-F_0 \sqcup F_1$ in the choices $\Xi_{\DD}$ to be $-\Zc_0 \sqcup \Zc_1$, with the matching arcs of $-\Zc_0$ oriented oppositely to the matching arcs of $\Zc_0$ in $\Xi$ and the matching arcs of $\Zc_1$ oriented the same as in $\Xi$. We take the Heegaard diagram $\Hc_{\DD}$ representing the one-sided cobordism $(Y,\Gamma)$ to be the original Heegaard diagram $\Hc$ chosen as part of $\Xi$, reinterpreted as a one-sided Heegaard diagram. Since all $\alpha$ arcs of $\Hc_{\DD}$ are outgoing, they should be oriented the same as the matching arcs of $-\Zc_0 \sqcup \Zc_1$ in the choices $\Xi_{\DD}$, i.e. the $\alpha$ arcs from $\Zc_0$ should be oriented oppositely to the matching arcs of $\Zc_0$ in $\Xi$, and the $\alpha$ arcs of $\Zc_1$ should be oriented the same as the matching arcs of $\Zc_1$ in $\Xi$. Thus the orientations of curves in $\Hc_{\DD}$ coming from $\Xi_{\DD}$ are the same as the orientations of curves in $\Hc$ coming from $\Xi$. We also order the matching arcs of $-\Zc_0 \sqcup \Zc_1$ in $\Xi_{\DD}$ so that the arcs of $-\Zc_0$ come before the arcs of $\Zc_1$. A schematic depicting our reinterpretation of $(Y,\Gamma)$ as a one-sided cobordism is shown below. 
    
    \begin{center}
    \hspace*{-2cm}
    \begin{overpic}[width=.65\textwidth]{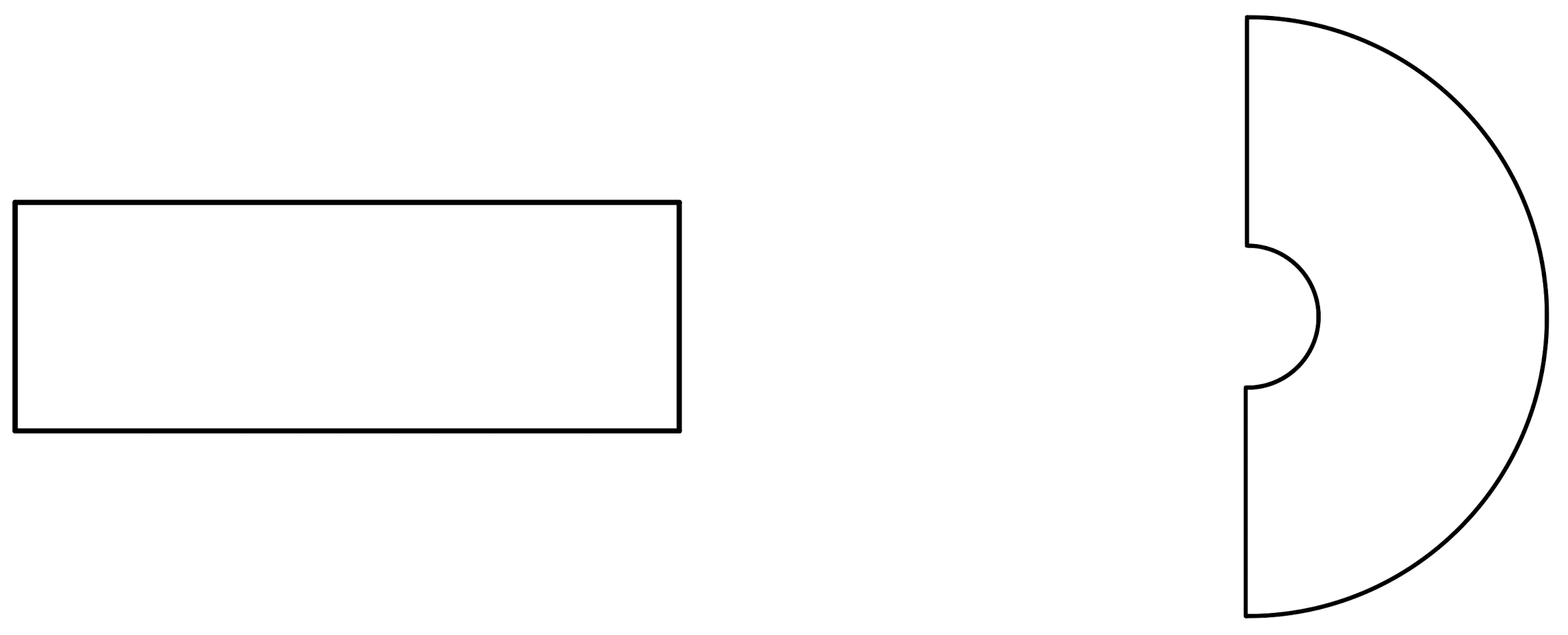}
        \put (-5,19) {$F_1$}
        \put (45.5,19) {$-F_0$}
        \put (17.5,19) {$(Y,\Gamma)$}

        \put (60,19) {\scalebox{1.2}{$\longrightarrow$}}

        \put (70,8) {$-F_0$}
        \put (72.5,31) {$F_1$}
        \put (87,19) {\scalebox{0.9}{$(Y,\Gamma)$}}
        \put (110,8) {$\emptyset$}
        \put (110,31) {$\emptyset$}
    \end{overpic}
    \end{center}
    
    Now, to view $[\BSDD(Y,\Gamma,\Xi_{\DD})]^{\Z}_{\comb}$ as an element of 
    $\wedge^* H_1(F_0,S^+_0) \otimes \wedge^* H_1(F_1,S^+_1)$, we first describe its action as a $\Z$-linear map. It sends the basis element $1 \in \Z$ to
    \[
    \sum_{\x}(-1)^{\mathrm{gr}_{\DD}(\x)} ((-1)^{|\overline{o}_R(\x)|}\gamma^{\inrm}_{\overline{o}_R(\x)})\otimes \gamma^{\out}_{\overline{o}_L(\x)}. 
    \]
    The extra sign $|\overline{o}_R(\x)| = n_0 - k$ arises because the basis elements of $H_1(F_0,S^+_0)$ coming from the orientations on matching arcs of $-\Zc_0$ in the choices $\Xi_{\DD}$ are $-1$ times the basis elements $\gamma^{\inrm}_i$ of $H_1(F_0,S^+_0)$ coming from the orientations on matching arcs of $\Zc_0$ in the choices $\Xi$. We write this extra sign $n_0-k$ as just $k$, since we can ignore any $k$-independent constants. That is, 
    \[
    [\BSDD(Y,\Gamma,\Xi_{\DD})]^{\Z}_{\comb}(1)=\sum_{\x}(-1)^{\mathrm{gr}_{\DD}(\x)+k} \gamma^{\inrm}_{\overline{o}_R(\x)}\otimes \gamma^{\out}_{\overline{o}_L(\x)}.
    \]
    
    Now we would like to compute $\mathrm{gr}_{\DD}(\x)$. Below, the familiar crossing strands diagram used to compute the inversion terms in $\mathrm{gr}_{\DA}(\x)$ is pictured on the left; on the right is the diagram computing $\mathrm{gr}_{\DD}(\x)$.

    \begin{center}
    \begin{overpic}[width=.67\textwidth]{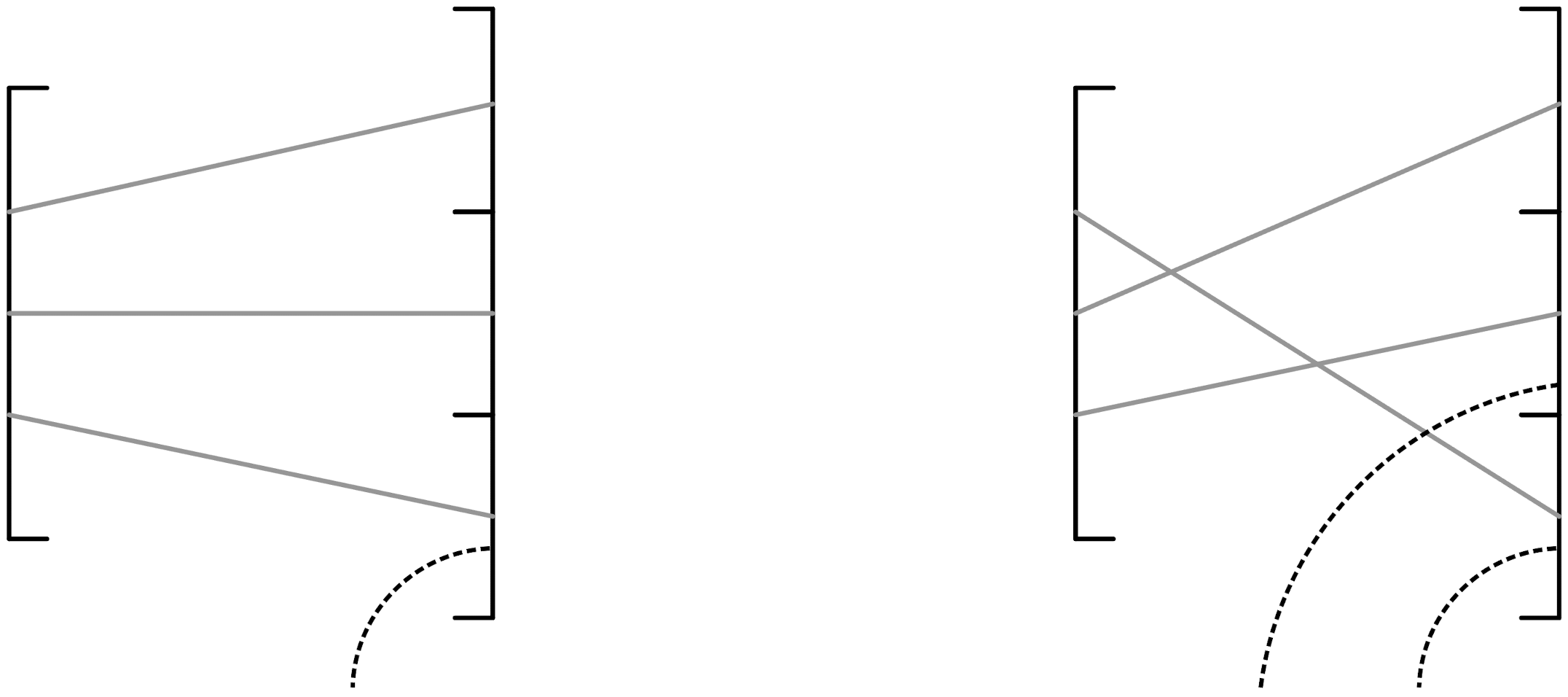}
        \put (-4,24) {$\boldsymbol{\beta}$}
        \put (6,9) {$n_1-l$}
        \put (15,26) {$a$}
        \put (15,36) {$k$}
        \put (33,10) {$\boldsymbol{\alpha}^{a,out}$}
        \put (33,24) {$\boldsymbol{\alpha}^{c}$}
        \put (33,35) {$\boldsymbol{\alpha}^{a,in}$}
        \put (22,-3) {$l$}

        \put (-4,24) {$\boldsymbol{\beta}$}
        \put (72,15) {$n_1-l$}
        \put (71,31) {$k$}
        \put (85,34) {$a$}
        \put (103,10) {$\boldsymbol{\alpha}^{a,in}$}
        \put (103,24) {$\boldsymbol{\alpha}^{a,out}$}
        \put (103,35) {$\boldsymbol{\alpha}^{c}$}
        \put (79.5,-3) {$l$}
        \put (87,-3) {$n_0-k$}
    \end{overpic}
    \end{center}
    In this rightmost diagram, since all arcs $\boldsymbol{\alpha}^{a,out}$ and $\boldsymbol{\alpha}^{a,in}$ have been reinterpreted as outgoing, according to $\Xi_{\DD}$ they all must come before $\boldsymbol{\alpha}^{c}$ in the block-ordering. This produces a total difference in the inversion terms of $\mathrm{gr}_{\DD}(\x)$ and $\mathrm{gr}_{\DA}(\x)$ of
    \[
    ak + (n_1-l)k + lk+\mathrm{inv}(\sigma_{I^c I \leftrightarrow \std})=ak + n_1 k + \mathrm{inv}(\sigma_{I^c I \leftrightarrow \std}).
    \]
    Note that the local intersection contributions of $\mathrm{gr}_{\DD}(\x)$ agree with the corresponding terms of $\mathrm{gr}_{\DA}(\x)$. Additionally, no correction terms appear in $\mathrm{gr}_{\DD}(\x)$, so comparing with 
    equation~\eqref{eq:DAGrading}, we find
    \[
    \mathrm{gr}_{\DD}(\x)=\mathrm{gr}_{\DA}(\x)+\mathrm{inv}(\sigma_{I^c I \leftrightarrow \std}). 
    \]

    Now we compute the composition in Theorem~\ref{thm:IntroFirstThm} with  $|K_{Y,\Gamma}|$ replaced by $[\BSDD(Y,\Gamma)]^{\Z}_{\comb}$; we identify $\wedge^* H_1(-F_0\sqcup F_1,S^+_0\sqcup S^+_1)$ with $\wedge^* H_1(F_0,S^+_0) \otimes \wedge^* H_1(F_1, S^+_1)$ in the usual way. Under these identifications, the first map 
    \[
    \wedge^* H_1(F_0,S^+_0) \xrightarrow{\id \otimes [\widehat{BSDD}(Y,\Gamma)]^{\Z}_{\comb}} \wedge^* H_1(F_0,S^+_0) \otimes \wedge^* H_1(F_1, S^+_1)
    \]
    takes a basis element $\gamma_I^{\inrm}$ of $\wedge^* H_1(F_0,S^+_0)$ to 
    \[
    \sum_{\x}(-1)^{\mathrm{gr}_{\DD}(\x)+k} \gamma_I^{\inrm} \otimes (\gamma_{\overline{o}_R(\x)}^{\inrm}\otimes \gamma_{\overline{o}_L(\x)}^{\out})
    \]
    Next, using the sign rule for morphisms in Definition~\ref{def:AbZgrCategory} and the definition of $\varepsilon$, this is sent under $\varepsilon \otimes \id$ to
    \begin{align*}
        &\sum_{\substack{\x}}(-1)^{\mathrm{gr}_{\DD}(\x)+k+|\varepsilon| l} \varepsilon(\gamma_I^{\inrm} \otimes \gamma_{\overline{o}_R(\x)}^{\inrm}) \otimes \gamma_{\overline{o}_L(\x)}^{\out} \\
        &= \sum_{\substack{\x : \\ \overline{o}_R(\x)=I^c}}(-1)^{\mathrm{gr}_{\DD}(\x)+k+n_0 l} \varepsilon(\gamma_I^{\inrm} \otimes \gamma_{I^c}^{\inrm}) \otimes \gamma_{\overline{o}_L(\x)}^{\out} \\
        &= \sum_{\substack{\x : \\ \overline{o}_R(\x)=I^c}}(-1)^{\mathrm{gr}_{\DD}(\x)+k+n_0 l} (-1)^{\mathrm{inv}(\sigma_{I I^c \leftrightarrow \mathrm{std}})}  \gamma_{\overline{o}_L(\x)}^{\out}.
        \\
        &= \sum_{\substack{\x : \\ \overline{o}_R(\x)=I^c}}(-1)^{\left(\mathrm{gr}_{\DA}(\x)+\mathrm{inv}(\sigma_{I^cI \leftrightarrow \mathrm{std}}) \right)+k+n_0 l+\mathrm{inv}(\sigma_{II^c \leftrightarrow \mathrm{std}})}  \gamma^{\out}_{\overline{o}_L(\x)}. 
    \end{align*}
    However, since $\mathrm{inv}(\sigma_{II^c \leftrightarrow \mathrm{std}})+\mathrm{inv}(\sigma_{I^c I \leftrightarrow \mathrm{std}})=k(n_0-k)$, the above sign simplifies to
    \begin{align*}
        \mathrm{gr}_{\DA}(\x)+k(n_0-k)+k+n_0 l&=\mathrm{gr}_{\DA}(\x)+n_0k-k^2+k+n_0(k+c) \\
        &=\mathrm{gr}_{\DA}(\x)+n_0k+k+k+n_0k+n_0c \mod 2 \\
        &= \mathrm{gr}_{\DA}(\x)+\text{($k$-independent terms)}.
    \end{align*}
    It follows that the composition agrees with $[\BSDA(Y,\Gamma)]^{\Z}_{\comb}$ up to overall sign. 
\end{proof}

Equipped with this result, the proof of Theorem~\ref{thm:IntroFirstThm} follows as an immediate consequence. 

\begin{proof}[Proof of Theorem~\ref{thm:IntroFirstThm}]
    By Lemma~\ref{lem:BSDAvsBSD}, it suffices to show that $[\BSDD(Y,\Gamma)]^{\Z}_{\comb}$ agrees with $|K_{Y,\Gamma}|$ as an up-to-sign element of $\wedge^* H_1(-F_0 \sqcup F_1, S^+_0 \sqcup S^+_1)$, which follows from Lemma~\ref{lem:IntroMainLemma}.
\end{proof}

We now show that the sutured Frohman--Nicas TQFT specializes to the classical Frohman--Nicas TQFT for cobordisms of closed oriented surfaces. Recall from the introduction, given such a cobordism $W \colon \Sigma_0 \to \Sigma_1$, there is an associated sutured cobordism $(Y,\Gamma)=\Sut(W,\gamma)$ formed by choosing an auxiliary acyclic ribbon graph $\gamma$, and this construction admits canonical identifications $H_1(F_i,S^+_i) \cong H_1(\Sigma_i)$. Proposition~\ref{prop:OrdinaryFNFromSuturedFN} from the introduction, which we now prove, says that after these identifications, $\Vc^{\FN}_{\sut}$ recovers the classical Frohman--Nicas TQFT $\Vc^{\FN}$.

\begin{proof}[Proof of Proposition~\ref{prop:OrdinaryFNFromSuturedFN}]
    First note that we have a commutative diagram
    \[
    \xymatrix{
    H_1(F_0 \sqcup F_1, S^+_0 \sqcup S^+_1;\Q) \ar[rr]^{i_*} & & H_1(Y, R^+;\Q) \\
    H_1(F_0 \sqcup F_1; \Q) \ar[rr]^{i_*} \ar[u]^{i_*} \ar[d]_{i_*} & & H_1(Y;\Q) \ar[u]_{i_*} \ar[d]^{i_*} \\
    H_1(\Sigma_0 \sqcup \Sigma_1;\Q) \ar[rr]^{i_*} & & H_1(W;\Q)
    }
    \]
    where the vertical maps are isomorphisms. Indeed, the top-right vertical map is an isomorphism because $H_1(R^+;\Q) = 0$ and $H_0(R^+;\Q) \xrightarrow{i_*} H_0(Y;\Q)$ is an isomorphism (topologically, $R^+$ consists of one 2-disk in each component of $Y$). To see that the bottom-right vertical map is an isomorphism, we may consider each component of $W$ separately and assume without loss of generality that $W$ is connected. First note that in the degenerate cases where $\gamma$ is a single point, the map is an isomorphism. If $\gamma$ is not a single point, consider the Mayer--Vietoris sequence associated to the decomposition of $W$ as $W = Y \cup_{C} \overline{\nb(\gamma)}$, where $C$ is $\partial(\overline{\nb(\gamma)})$ minus a neighborhood of each endpoint of $\gamma$. We have
    \[
    \cdots \to H_1(C) \to H_1(\overline{\nb(\gamma)}) \oplus H_1(Y) \to H_1(W) \to 0,
    \]
    since the map from $H_0(C)$ to $H_0(\overline{\nb(\gamma)}) \oplus H_0(Y)$ is injective. We have $H_1(\overline{\nb(\gamma)}) = 0$, and the map from $H_1(C)$ to $H_1(Y)$ is the zero map. Indeed, while $C$ is topologically a sphere $S^2$ with one puncture for each endpoint of $\gamma$, and thus has a basis for $H_1$ given by loops around all the punctures but one, each of these loops bounds a surface in $Y$ (namely the corresponding component of $F_0$ or $F_1$) and so maps to zero in $H_1(Y)$. It follows that the bottom-right map in the diagram is an isomorphism.

    By definition, $W$ is rationally homologically trivial if and only if the bottom map is surjective, and we see this is true if and only if the top map is surjective. Furthermore, by excision we can identify the top map with
    \[
    H_1(F_1 \cup R^+ \cup F_0, R^+;\Q) \xrightarrow{i_*} H_1(Y,R^+;\Q).
    \]
    We have an exact sequence
    \[
    \cdots \to H_1(F_1 \cup R^+ \cup F_0, R^+;\Q) \xrightarrow{i_*} H_1(Y,R^+;\Q) \to H_1(Y, F_1 \cup R^+ \cup F_0;\Q) \to 0
    \]
    because $H_0(F_1 \cup R^+ \cup F_0, R^+) = 0$ (by assumption, all components of $F_i$ intersect $S^+$ and thus $R^+$ nontrivially). Hence $W$ is rationally homologically trivial if and only if
    \[
    H_1(Y,F_1 \cup R^+ \cup F_0;\Q) = 0,
    \]
    or equivalently if and only if $H_1(Y, F_1 \cup R^+ \cup F_0)$ is finite.

    Now, Poincar{\'e}--Lefschetz duality gives $H_1(Y, F_1 \cup R^+ \cup F_0) \cong H^2(Y, R^-)$. Since $R^-$ is topologically a 2-disk in the boundary of $Y$, we have $H^2(Y,R^-) \cong H^2(Y)$. By the universal coefficient theorem, the free part of $H^2(Y)$ is the free part of $H_2(Y)$, and the torsion part of $H^2(Y)$ is the torsion part of $H_1(Y)$, which we saw above is the torsion part of $H_1(W)$. It follows that when $W$ is rationally homologically trivial, $H_1(Y, F_1 \cup R^+ \cup F_0) \cong \tors(H_1(W))$, so the prefactors in the definitions of $|K_W|$ and $|K_{Y,\Gamma}|$ agree.

    For the integer $K$ in Theorem~\ref{thm:IntroFirstThm}, note that $\rank(H_1(F_i,S^+_i)) = 2g_i$. We also have
    \begin{align*}
        \chi(Y,R^+) &= -\rank(H_1(Y,R^+)) + \rank(H_2(Y,R^+)) \\
        &=-\rank(H_1(Y)) + \rank(H_2(Y)) \\
        &=-\rank(H_1(Y)) + \rank(H^1(Y,\partial Y)) \\
        &= -\rank(H_1(Y)) + \rank(H_1(Y,\partial Y)).
    \end{align*}
    There is an exact sequence
    \[
    0 \to \im(i_* \colon H_1(\partial Y) \to H_1(Y)) \to H_1(Y) \to H_1(Y, \partial Y) \to 0
    \]
    (the last zero is because each component of $Y$ has connected nonempty boundary). The ``half lives, half dies'' theorem implies that $\im(i_* \colon H_1(\partial Y) \to H_1(Y))$ has rank $g_0 + g_1$ (half of $\rank(H_1(\partial Y)) = 2g_0 + 2g_1)$). It follows that
    \[
    \rank(H_1(Y,\partial Y)) = \rank(H_1(Y)) - g_0 - g_1;
    \]
    combining with the above, we get
    \[
    \chi(Y,R^+) = -g_0 - g_1.
    \]
    Thus, the integer $K$ in Theorem~\ref{thm:IntroFirstThm} is $2g_0 + 2g_1 - g_0 - g_1 = g_0 + g_1$, as in the Frohman--Nicas TQFT. Note that the half lives, half dies theorem gives $\rank(\ker i_*) = K$ in this case; the phenomenon $\rank(\ker i_*) < K$ can only occur in the sutured generalization of the Frohman--Nicas TQFT.

    With these numerics in mind, the commutative diagram at the beginning of the proof implies that $\wedge^{g_0 + g_1} \ker i_*$ from the Frohman--Nicas TQFT agrees with $\wedge^K \ker i_*$ from Theorem~\ref{thm:IntroFirstThm} under our identifications $H_1(\Sigma_i) \cong H_1(F_i,S^+_i)$. Putting in the prefactors, we see that $|K_W|$ agrees with $|K_{Y,\Gamma}|$ under these identifications. The integer $n_0$ in Theorem~\ref{thm:IntroFirstThm} is equal to $2g_0$ and thus is even; the proposition follows.
\end{proof}

\section{Symmetric monoidal structure}\label{sec:SymmetricMonoidal}

The goal of this section is to prove Theorem~\ref{Thm:VFNSymmetricMonoidal}, that $\Vc^{\FN}_{\sut}$ is a symmetric monoidal functor. The results of Section~\ref{sec:BSDAwithoutSpinc} combined with the relationship to the sutured Alexander functor established that $[\BSDA(Y,\Gamma)]^{\Z}_{\comb}$ is functorial (cf. Corollary~\ref{cor:BSDAZFunctorial}). Having established Theorem~\ref{thm:IntroFirstThm} in the previous section, we obtain the following consequence.

\begin{corollary}\label{V^FN_sutFunctorial}
    Suppose that $\Vc^{\FN}_{\sut}$ assigns to each object $(F,\Lambda)[d]$ of $\Cob^{\sut}_{2+1}$ 
    the graded abelian group $\wedge^* H_1(F,S^+)$ whose summand in degree $d$ is $\wedge^{d+k} H_1(F,S^+)$, 
    and to each morphism $(Y,\Gamma)\colon (F_0,S^+_0)[d_0] \to (F_1,S^+_1)[d_1]$ 
    the up-to-sign map $\wedge^* H_1(F_0,S^+_0) \to \wedge^* H_1(F_1,S^+_1)$ defined by the composition in Theorem~\ref{thm:IntroFirstThm} involving $|K_{Y,\Gamma}|$. Then this assignment is a functor $\Vc^{\FN}_{\sut}$ from $\Cob^{\sut}_{2+1}$ to $\Ab^{\Z\gr}_{\pm 1}$.
    \end{corollary}

\subsection{The categories $\Cob^{\sut}_{2+1}$ and $\Ab^{\Z\gr}_{\pm 1}$}

However, to prove that $\Vc^{\FN}_{\sut}$ is not only a functor, but a symmetric monoidal functor, we first need to verify that both the source and target categories are symmetric monoidal (see \cite[Definition 8.1.12]{EGNO}). We begin with the source category, $\Cob^{\sut}_{2+1}$, whose symmetric monoidal structure is relatively intuitive. Nonetheless, we verify all the axioms in full detail, hoping to make clear that all the necessary checks can be done directly and do not depend on any external results. 

\begin{theorem}\label{thm:CobSutSymmetricMonoidal}
    As in Definition~\ref{def:SutCobCategory}, define $\Cob^{\sut}_{2+1}$ to be the category with: 
    \begin{itemize}
        \item Objects: $(F,\Lambda)[d]$, where $(F,\Lambda)$ is a sutured surface and $d \in \Z$ is a formal degree shift.\footnote{At this stage, the shift $[d]$ is purely symbolic; it plays no role until we apply the functor $\Vc^{\FN}_{\sut}$, which will interpret $(F,\Lambda)[d]$ by placing the summand $\wedge^k H_1(F,S^+)$ in degree $k-d$.}    
        
        \item Morphisms from $(F_0,\Lambda_0)[d_0]$ to $(F_1,\Lambda_1)[d_1]$: Sutured cobordisms $(Y,\Gamma)$ from $(F_0,\Lambda_0)$ to $(F_1,\Lambda_1)$, defined up to isomorphism of sutured cobordisms. 
        
        \item Composition: Given morphisms $(Y,\Gamma)$ from $(F_0,\Lambda_0)[d_0]$ to $(F_1,\Lambda_1)[d_1]$ and $(Y',\Gamma')$ from $(F_1,\Lambda_1)[d_1]$ to $(F_2,\Lambda_2)[d_2]$, we define the composite morphism to be the sutured cobordism
        \[
        (Y',\Gamma')\circ (Y,\Gamma) \coloneq (Y' \cup_{F_1} Y,\Gamma' \cup_{\Lambda_1} \Gamma),
        \]
        from $(F_0,\Lambda_0)[d_0]$ to $(F_2,\Lambda_2)[d_2]$ given by gluing $(Y,\Gamma)$ and $(Y',\Gamma')$ along the common boundary $F_1$, as in Definition~\ref{def:SuturedCobGluing}.
        \item Identity Morphism on $(F,\Lambda)[d]$: Write $(\id_{(F,\Lambda)[d]},\Gamma_{\id})$ for the identity sutured cobordism 
        \[
        (F \times [0,1], \Lambda \times [0,1])
        \]
        oriented so that $F \times \{0\} = -F$ and $F \times \{1\} = F$.
    \end{itemize}
    Equipped with:
    \begin{itemize}
        \item Monoidal structure: A bifunctor $\amalg: \Cob^{\sut}_{2+1} \times \Cob^{\sut}_{2+1} \to \Cob^{\sut}_{2+1}$ which we will commonly denote using infix notation, writing $(F,\Lambda) \amalg (F',\Lambda')$ for $\amalg((F,\Lambda), (F',\Lambda'))$ and $(Y,\Gamma) \amalg (Y',\Gamma')$ for $\amalg((Y,\Gamma),(Y',\Gamma'))$. It is defined as follows. 
        \begin{itemize}
            \item On objects: We have 
            \[
            (F,\Lambda)[d] \amalg (F',\Lambda')[d']=(F \sqcup F',\Lambda \sqcup \Lambda')[d+d'].
            \]
            \item On morphisms: $(Y,\Gamma)\amalg (Y',\Gamma')=(Y\sqcup Y',\Gamma \sqcup \Gamma')$. 
        \end{itemize}
        \item Monoidal unit: The empty cobordism with no degree shift, $(\emptyset,\emptyset)[0]$. 
        \item Associator and unitor: The evident identity cobordisms.
        \item Braiding: A natural isomorphism whose $((F,\Lambda)[d],(F',\Lambda')[d'])$-component
        \[
        \gamma_{(F,\Lambda)[d],(F',\Lambda')[d']}:(F,\Lambda)[d] \amalg (F',\Lambda')[d'] \to (F',\Lambda')[d'] \amalg (F,\Lambda)[d]
        \]
        is the identity cobordism on the disjoint union, $((F \sqcup F')\times [0,1],(\Lambda \sqcup \Lambda')\times [0,1])$, but with the boundary identifications twisted so that:
        \begin{itemize}
            \item the incoming boundary is identified with $- (F \sqcup F') = -F \sqcup -F'$, and
            \item the outgoing boundary is identified with $F' \sqcup F$.
        \end{itemize}
    \end{itemize}
    Then $(\Cob^{\sut}_{2+1},\amalg,\gamma_{(F,\Lambda)[d],(F',\Lambda')[d']})$ forms a symmetric monoidal category. 
\end{theorem}

\begin{proof}
    The fact that $\Cob^{\sut}_{2+1}$ is a category follows from the general framework of cobordism categories. The only additional feature here is the presence of the sutures and distinguished subsets $R^\pm$, but Definition~\ref{def:SuturedCobGluing} shows that gluing preserves this structure. The usual arguments show that composition is associative up to isomorphism of sutured cobordisms. Thus $\Cob^{\sut}_{2+1}$ forms a category.

    We will appeal to a simple graphical calculus in order to carry out computations in $\Cob^{\sut}_{2+1}$. It is required to follow the below rules: 
    \begin{itemize}
        \item A sutured surface is represented by a vertical interval. When depicting disjoint unions $F \sqcup F'$, draw $F'$ above $F$.
        \item To depict a cobordism $(Y,\Gamma)$ from $F$ to $F'$, draw $-F$ on the right and $F'$ on the left (matching our incoming/outgoing conventions) and draw a strip connecting $F$ to $F'$, labeled $(Y,\Gamma)$. Unlabeled cobordisms are assumed to be identity cobordisms. 
        \item Composition is given by horizontal gluing along the parallel intervals. 
    \end{itemize}

    We first verify that $\amalg$ is a bifunctor. Fix morphisms 
    \begin{align*}
        (Y_1,\Gamma_1): (F_0,\Lambda_0)[d_0]&\to(F_1,\Lambda_1)[d_1] \\
        (Y_2,\Gamma_2): (F_1,\Lambda_1)[d_1]&\to(F_2,\Lambda_2)[d_2] \\
        (Y_1',\Gamma_1'): (F_0',\Lambda_0')[d'_0]&\to(F_1',\Lambda_1')[d'_1] \\
        (Y_2',\Gamma_2'): (F_1',\Lambda_1')[d'_1]&\to(F_2',\Lambda_2')[d'_2].
    \end{align*}
    The interchange law holds, since disjoint union and gluing commute: 
    \begin{align*}
        &\big((Y_2,\Gamma_2)\amalg(Y_2',\Gamma_2')\big)\circ\big((Y_1,\Gamma_1)\amalg(Y_1',\Gamma_1')\big) \\ 
        &= (Y_2\sqcup Y_2', \Gamma_2\sqcup\Gamma_2') \cup_{F_1\sqcup F_1'} (Y_1\sqcup Y_1', \Gamma_1\sqcup\Gamma_1') \\
        &= \big((Y_2\cup_{F_1}Y_1)\sqcup(Y_2'\cup_{F_1'}Y_1'), (\Gamma_2\cup_{\Lambda_1}\Gamma_1)\sqcup(\Gamma_2'\cup_{\Lambda_1'}\Gamma_1')\big) \\ 
        &= \big((Y_2,\Gamma_2)\circ(Y_1,\Gamma_1)\big) \amalg \big((Y_2',\Gamma_2')\circ(Y_1',\Gamma_1')\big).
    \end{align*}
    Diagramatically, before gluing, both sides of the interchange law look like
    \begin{center}
    \vspace{0.1in}
    \begin{overpic}[width=.55\textwidth]{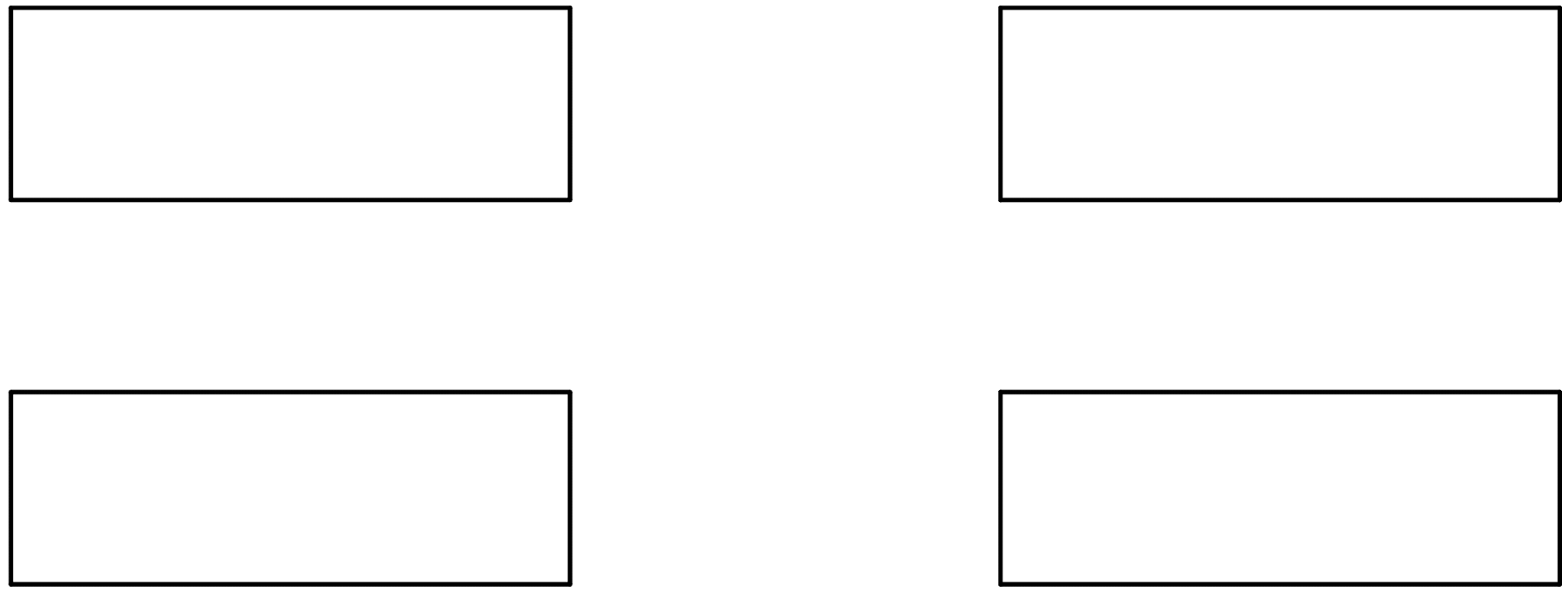}
        \put (-6,5) {\scalebox{0.9}{$F_2$}}
        \put (-6,29.5) {\scalebox{0.9}{$F_2'$}}
        \put (12,5) {\scalebox{0.9}{$(Y_2,\Gamma_2)$}}
        \put (12,29.5) {\scalebox{0.9}{$(Y_2',\Gamma_2')$}}
        \put (38,5) {\scalebox{0.9}{$-F_1$}}
        \put (38,29.5) {\scalebox{0.9}{$-F_1'$}}
        \put (57,5) {\scalebox{0.9}{$F_1$}}
        \put (57,29.5) {\scalebox{0.9}{$F_1'$}}
        \put (75,5) {\scalebox{0.9}{$(Y_1,\Gamma_1)$}}
        \put (75,29.5) {\scalebox{0.9}{$(Y_1',\Gamma_1')$}}
        \put (101,5) {\scalebox{0.9}{$-F_0$}}
        \put (101,29.5) {\scalebox{0.9}{$-F_0'$}}
    \end{overpic}
    \end{center}
    The diagram representing the left-hand side of the interchange law is obtained from this one by gluing horizontally along the top and bottom intervals simultaneously, while the diagram representing the right-hand side is obtained by gluing horizontally along the top and bottom intervals separately; the two diagrams represent isomorphic sutured cobordisms. 
    
    We also see that $\amalg$ respects identity morphisms, since 
    \begin{align*}
        \id_{(F,\Lambda)[d]} \amalg \id_{(F',\Lambda')[d']}
        &= \left((F\times[0,1]) \sqcup (F'\times[0,1]),\,(\Lambda\times[0,1]) \sqcup (\Lambda'\times[0,1])\right) \\
        &\cong \left((F\sqcup F')\times[0,1],\,(\Lambda\sqcup\Lambda')\times[0,1]\right) \\
        &= \id_{(F\sqcup F',\,\Lambda\sqcup\Lambda')[\,d+d']} \\
        &= \id_{(F,\Lambda)[d]\;\amalg\;(F',\Lambda')[d']}.
    \end{align*}
    where $\cong$ denotes isomorphic sutured cobordisms. We omit the trivial graphical diagram. 
    
    The associator in $\Cob^{\sut}_{2+1}$ has component 
    \begin{align*}
        \alpha_{(F,\Lambda)[d],\,(F',\Lambda')[d'],\,(F'',\Lambda'')[d'']} &\colon ((F,\Lambda)[d] \amalg (F',\Lambda')[d']) \amalg (F'',\Lambda'')[d'']  \\
        &\quad \to (F,\Lambda)[d] \amalg ((F',\Lambda')[d'] \amalg (F'',\Lambda'')[d''])
    \end{align*}
    defined to be the identity cobordism on the triple (symmetric) disjoint union
    \[
    (F\sqcup F'\sqcup F'',\Lambda \sqcup \Lambda' \sqcup \Lambda'')
    \]
    with its domain and codomain (nested disjoint unions, parenthesized differently) identified with the symmetric disjoint union in the natural way. Our graphical diagram for the associator is
    \begin{center}
    \vspace{0.1in}
    \begin{overpic}[width=.2\textwidth]{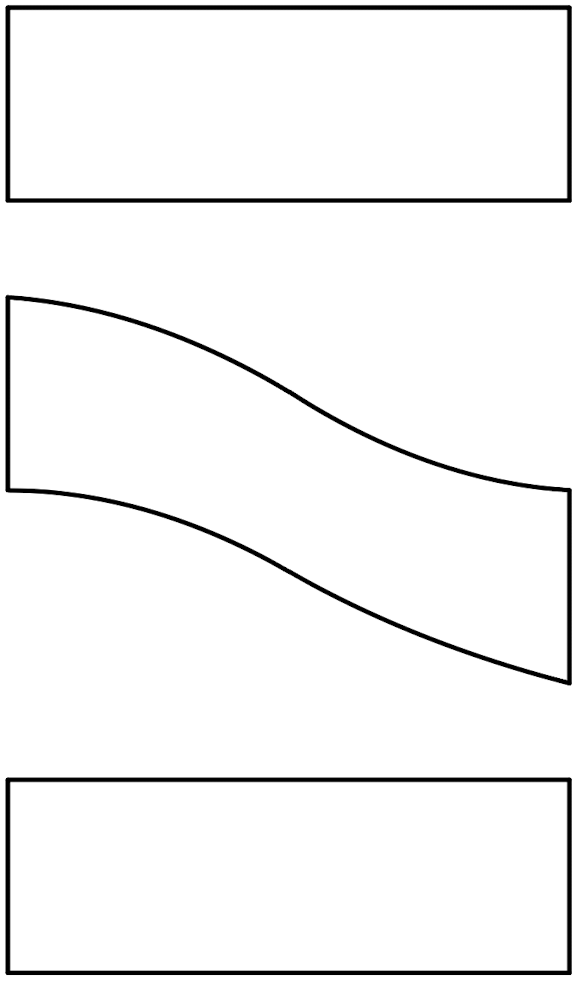}
        \put (-10,8) {\scalebox{0.9}{$F$}}
        \put (-10,58) {\scalebox{0.9}{$F'$}}
        \put (-10,88) {\scalebox{0.9}{$F''$}}
        \put (64,8) {\scalebox{0.9}{$-F$}}
        \put (64,38) {\scalebox{0.9}{$-F'$}}
        \put (64,88) {\scalebox{0.9}{$-F''$}}
    \end{overpic}
    \end{center}
    From this it is clear that the coherence pentagon \cite[Equation 2.2]{EGNO} commutes: after gluing, both directions give the $4$-fold identity cobordism, differing only by how the domain and codomain are parenthesized. 
    
    The left and right unitors in $\Cob^{\sut}_{2+1}$ are the natural isomorphisms with components 
    \begin{align*}
        &\lambda_{(F, \Lambda)[d]}:(\emptyset, \emptyset)[0] \amalg(F, \Lambda)[d] \cong(F, \Lambda)[d], \\ & \rho_{(F, \Lambda)[d]}:(F, \Lambda)[d] \amalg(\emptyset, \emptyset)[0] \cong(F, \Lambda)[d] 
    \end{align*}
    defined as the usual identity cobordisms. Our graphical diagrams for the unitors are 
    \begin{center}
    \begin{overpic}[width=.55\textwidth]{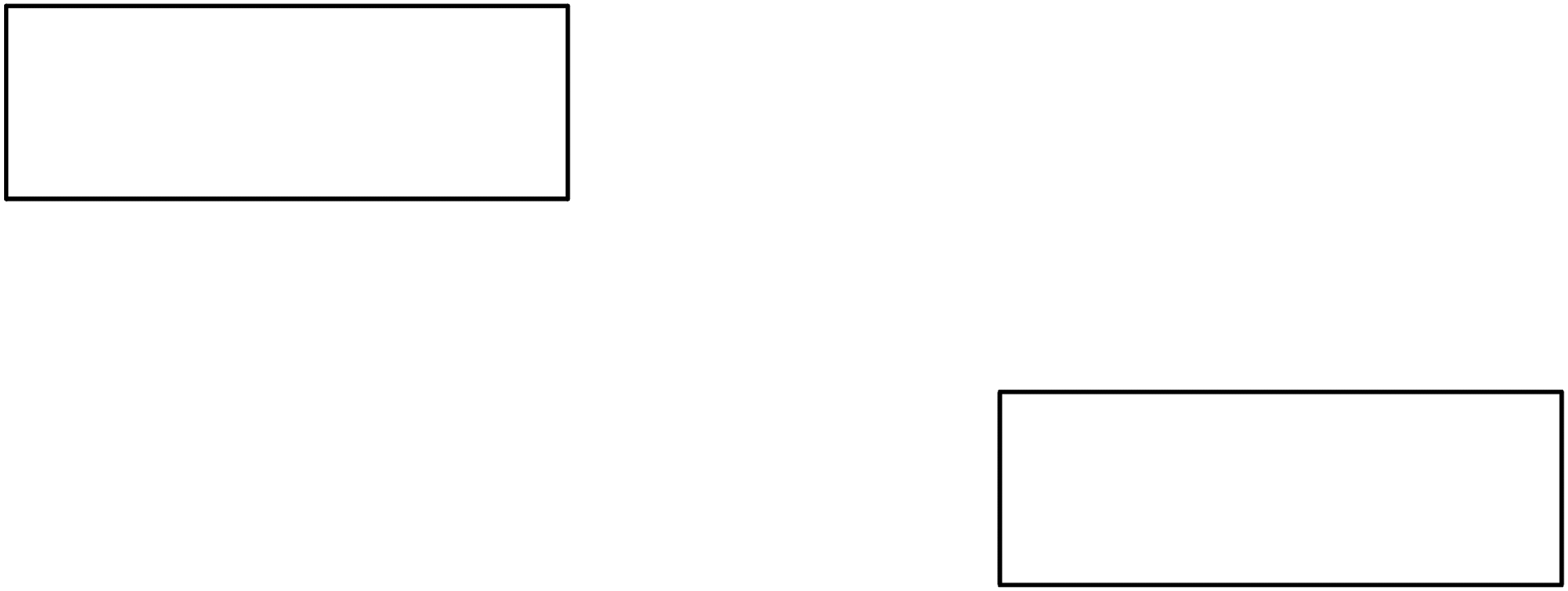}
        \put (-6,30) {\scalebox{0.9}{$F$}}
        \put (40,30) {\scalebox{0.9}{$-F$}}
        \put (41,5) {\scalebox{0.9}{$\emptyset$}}
        \put (58,5) {\scalebox{0.9}{$F$}}
        \put (102,5) {\scalebox{0.9}{$-F$}}
        \put (105,30) {\scalebox{0.9}{$\emptyset$}}
        \put (15,-5) {\scalebox{0.9}{$\lambda_{(F, \Lambda)[d]}$}}
        \put (77,-5) {\scalebox{0.9}{$\rho_{(F, \Lambda)[d]}$}}
    \end{overpic}
    \vspace{0.1in}
    \end{center}
    From this it is clear that the monoidal coherence triangle \cite[Equation 2.10]{EGNO} for the unitors commutes: after gluing, both directions give the $2$-fold identity cobordism, differing only by the placement of the empty component in the disjoint union. 

    Our graphical diagram for the component of the braiding on $((F,\Lambda)[d],(F',\Lambda')[d]')$ is
    \begin{center}
    \vspace{0.1in}
    \begin{overpic}[width=.2\textwidth]{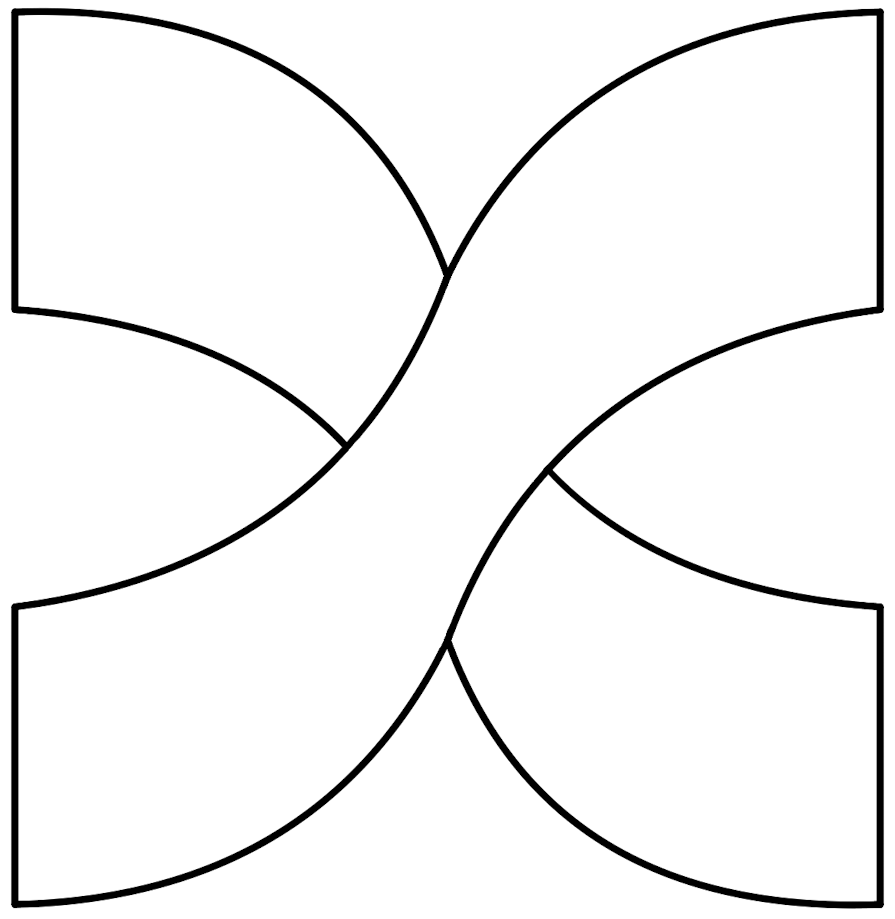}
        \put (-15,14) {\scalebox{0.9}{$F'$}}
        \put (-15,80) {\scalebox{0.9}{$F$}}
        \put (104,14) {\scalebox{0.9}{$-F$}}
        \put (104,80) {\scalebox{0.9}{$-F'$}}
    \end{overpic}
    \end{center}
    
    We now check all the standard properties for the braiding, starting with its naturality in $(F,\Lambda)[d],$ and $(F',\Lambda')[d']$. We expand the value of $\amalg$ on surfaces and cobordisms here for concreteness. 
    \[
    \xymatrix@C=6em{
      (F_0 \sqcup F_0', \Lambda_0 \sqcup \Lambda_0')[d_0+d_0'] \ar[r]^{\gamma_{(F_0,\Lambda_0),(F_0',\Lambda_0')}} \ar[d]_{(Y,\Gamma)\sqcup(Y',\Gamma')} &
      (F_0' \sqcup F_0, \Lambda_0' \sqcup \Lambda_0)[d_0+d_0'] \ar[d]^{(Y',\Gamma')\sqcup(Y,\Gamma)} \\
      (F_1 \sqcup F_1', \Lambda_1 \sqcup \Lambda_1')[d_1+d_1'] \ar[r]^{\gamma_{(F_1,\Lambda_1),(F_1',\Lambda_1')}} &
      (F_1' \sqcup F_1, \Lambda_1' \sqcup \Lambda_1)[d_1+d_1']
    }
    \]
    The graphical diagram establishing naturality is shown below. The top figure corresponds to the clockwise composition; the bottom figure shows the counterclockwise composition. 
    
    \begin{center}
    \begin{overpic}[width=0.55\linewidth]{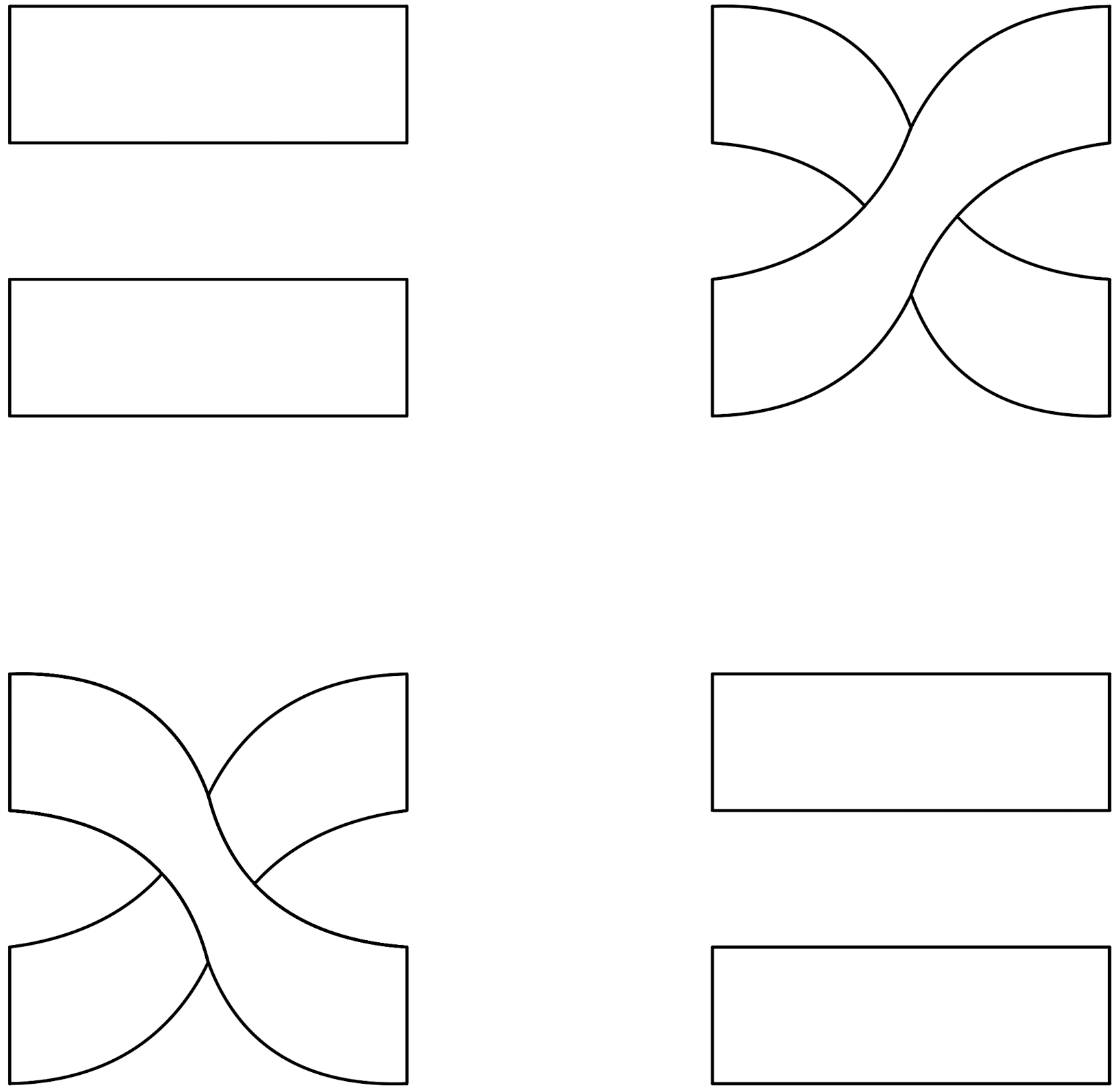}
        \put (-7,5) {\scalebox{0.9}{$F_1'$}} 
        \put (-7,29) {\scalebox{0.9}{$F_1$}} 
        \put (-7,65) {\scalebox{0.9}{$F_1'$}} 
        \put (-7,90) {\scalebox{0.9}{$F_1$}} 
        \put (12.5,65) {\scalebox{0.9}{$(Y',\Gamma')$}} 
        \put (13.5,90) {\scalebox{0.9}{$(Y,\Gamma)$}} 
        \put (38,5) {\scalebox{0.9}{$-F_1$}}
        \put (38,30) {\scalebox{0.9}{$-F_1'$}}
        \put (38,65) {\scalebox{0.9}{$-F_0'$}}
        \put (38,90) {\scalebox{0.9}{$-F_0$}}
        \put (58,5) {\scalebox{0.9}{$F_1$}}
        \put (58,30) {\scalebox{0.9}{$F_1'$}}
        \put (58,65) {\scalebox{0.9}{$F_0'$}}
        \put (58,90) {\scalebox{0.9}{$F_0$}}
        \put (75.6,5.5) {\scalebox{0.9}{$(Y,\Gamma)$}} 
        \put (76.7,30.3) {\scalebox{0.9}{$(Y',\Gamma')$}} 
        \put (100,5) {\scalebox{0.9}{$-F_0$}}
        \put (100,30) {\scalebox{0.9}{$-F_0'$}}
        \put (100,65) {\scalebox{0.9}{$-F_0$}}
        \put (100,90) {\scalebox{0.9}{$-F_0'$}}
        \put (23,-10) {\scalebox{0.9}{$\gamma_{(F_1,\Lambda_1),(F_1',\Lambda_1')} \circ ((Y,\Gamma)\sqcup(Y',\Gamma'))$}}
        \put (23,50) {\scalebox{0.9}{$((Y',\Gamma')\sqcup(Y,\Gamma)) \circ \gamma_{(F_0,\Lambda_0),(F_0',\Lambda_0')}$}}
    \end{overpic}
    \vspace{1em}
    \end{center}
    
    After gluing, the two diagrams represent isomorphic sutured cobordisms, which are equal in $\Cob^{\sut}_{2+1}$. Thus, the braiding is natural. 
    
    The hexagon coherence diagram for the braiding (\cite[diagram (8.1)]{EGNO}) is given by 
    \begin{equation*}
        \scalebox{0.9}{
    \xymatrix@C=6em@R=4em{
    ((F,\Lambda)[d] \amalg (F',\Lambda')[d']) \amalg (F'',\Lambda'')[d''] 
      \ar[r]^{\raise2ex\hbox{$\scriptstyle \gamma_{(F,\Lambda)[d],\,(F',\Lambda')[d']} \amalg \; \id_{(F'',\Lambda'')[d'']}$}} 
      \ar[d]_{\scriptstyle \alpha_{(F,\Lambda)[d],\,(F',\Lambda')[d'],\,(F'',\Lambda'')[d'']}} 
    & ((F',\Lambda')[d'] \amalg (F,\Lambda)[d]) \amalg (F'',\Lambda'')[d''] 
      \ar[d]^{\scriptstyle \alpha_{(F',\Lambda')[d'],\,(F,\Lambda)[d],\,(F'',\Lambda'')[d'']}} \\
    (F,\Lambda)[d] \amalg ((F',\Lambda')[d'] \amalg (F'',\Lambda'')[d'']) 
      \ar[d]_{\scriptstyle \gamma_{(F,\Lambda)[d],\,(F',\Lambda')[d'] \amalg (F'',\Lambda'')[d'']}} 
    & (F',\Lambda')[d'] \amalg ((F,\Lambda)[d] \amalg (F'',\Lambda'')[d'']) 
      \ar[d]^{\scriptstyle \id_{(F',\Lambda')[d']} \amalg \; \gamma_{(F,\Lambda)[d],\,(F'',\Lambda'')[d'']}} \\
    ((F',\Lambda')[d'] \amalg (F'',\Lambda'')[d'']) \amalg (F,\Lambda)[d] 
      \ar[r]_{\lower2ex\hbox{$\scriptstyle \alpha_{(F',\Lambda')[d'],\,(F'',\Lambda'')[d''],\,(F,\Lambda)[d]}$}} 
    & (F',\Lambda')[d'] \amalg ((F'',\Lambda'')[d''] \amalg (F,\Lambda)[d])
    }
    }
    \end{equation*}
    We show that this commutes graphically below, with the counterclockwise path on top and the clockwise path on bottom.

    \begin{center}
    \begin{overpic}[width=.8\textwidth]{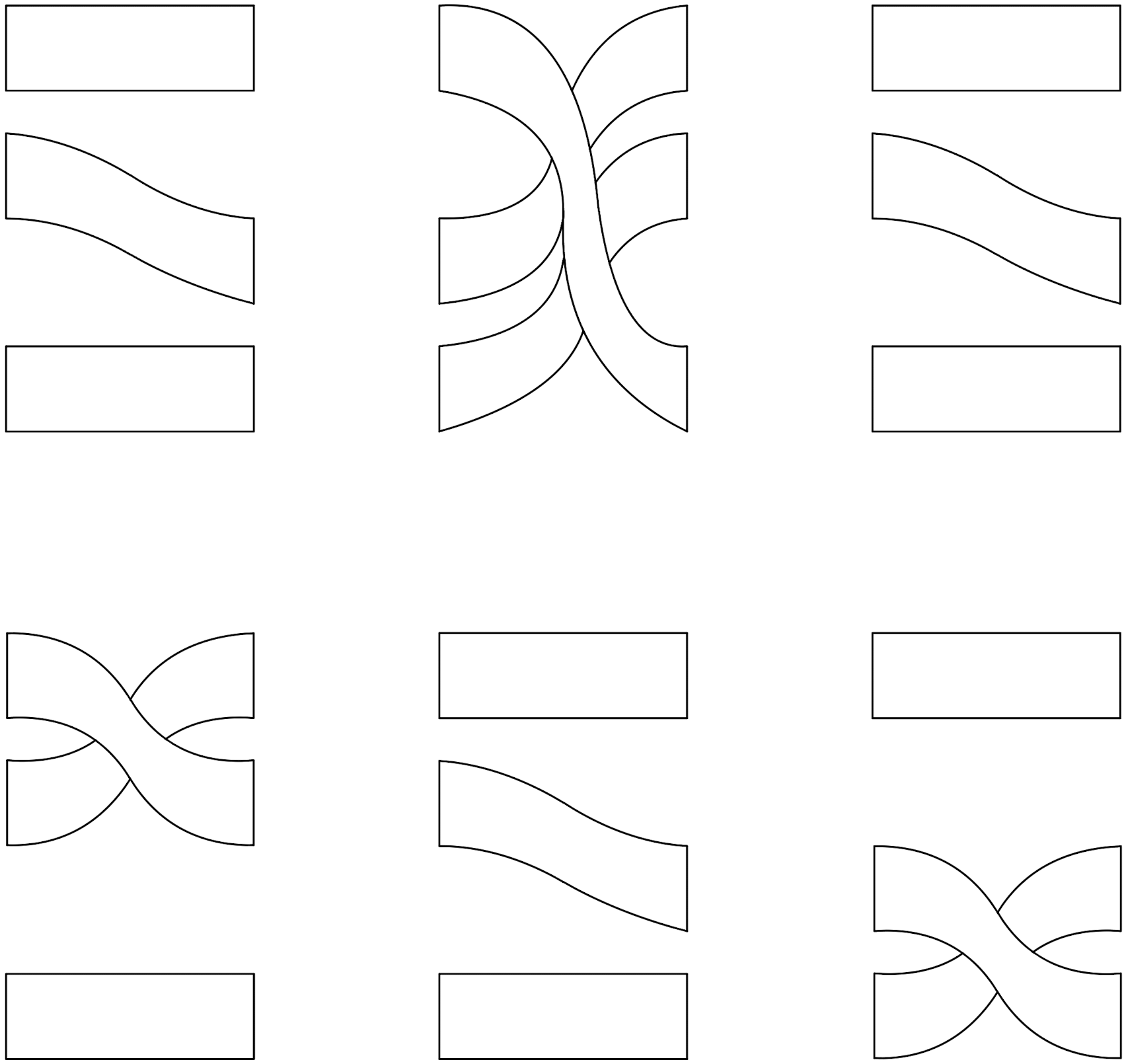}
        \put (-4,3) {\scalebox{0.9}{$F'$}}
        \put (-4,22) {\scalebox{0.9}{$F''$}}
        \put (-4,33) {\scalebox{0.9}{$F$}}
        \put (-4,59) {\scalebox{0.9}{$F'$}}
        \put (-4,77) {\scalebox{0.9}{$F''$}}
        \put (-4,89) {\scalebox{0.9}{$F$}}
        \put (23,3) {\scalebox{0.9}{$-F'$}}
        \put (23,22) {\scalebox{0.9}{$-F$}}
        \put (23,33) {\scalebox{0.9}{$-F''$}}
        \put (23,59) {\scalebox{0.9}{$-F'$}}
        \put (23,70) {\scalebox{0.9}{$-F''$}}
        \put (23,89) {\scalebox{0.9}{$-F$}}
        \put (35,3) {\scalebox{0.9}{$F'$}}
        \put (35,22) {\scalebox{0.9}{$F$}}
        \put (35,33) {\scalebox{0.9}{$F''$}}
        \put (35,59) {\scalebox{0.9}{$F'$}}
        \put (35,70) {\scalebox{0.9}{$F''$}}
        \put (35,89) {\scalebox{0.9}{$F$}}
        \put (62,3) {\scalebox{0.9}{$-F'$}}
        \put (62,15) {\scalebox{0.9}{$-F$}}
        \put (62,33) {\scalebox{0.9}{$-F''$}}
        \put (62,59) {\scalebox{0.9}{$-F$}}
        \put (62,77) {\scalebox{0.9}{$-F'$}}
        \put (62,89) {\scalebox{0.9}{$-F''$}}
        \put (73.5,3) {\scalebox{0.9}{$F'$}}
        \put (73.5,15) {\scalebox{0.9}{$F$}}
        \put (73.5,33) {\scalebox{0.9}{$F''$}}
        \put (73.5,59) {\scalebox{0.9}{$F$}}
        \put (73.5,77) {\scalebox{0.9}{$F'$}}
        \put (73.5,89) {\scalebox{0.9}{$F''$}}
        \put (101,3) {\scalebox{0.9}{$-F$}}
        \put (101,15) {\scalebox{0.9}{$-F'$}}
        \put (101,33) {\scalebox{0.9}{$-F''$}}
        \put (101,59) {\scalebox{0.9}{$-F$}}
        \put (101,70) {\scalebox{0.9}{$-F'$}}
        \put (101,89) {\scalebox{0.9}{$-F''$}}
        \put (4,50) {\scalebox{0.85}{$
        \alpha_{(F',\Lambda')[d'],\,(F'',\Lambda'')[d''],\,(F,\Lambda)[d]} 
        \circ 
        \gamma_{(F,\Lambda)[d],\,(F',\Lambda')[d'] \amalg (F'',\Lambda'')[d'']}
        \circ 
        \alpha_{(F,\Lambda')[d],\,(F',\Lambda')[d'],\,(F'',\Lambda'')[d'']}
        $}}
        \put (-1.5,-6) {\scalebox{0.85}{$
        (\id_{(F',\Lambda')[d']} 
        \amalg 
        \gamma_{(F,\Lambda)[d],\,(F'',\Lambda'')[d'']}) 
        \circ 
        \alpha_{(F',\Lambda')[d'],\,(F,\Lambda)[d],\,(F'',\Lambda'')[d'']} 
        \circ 
        (\gamma_{(F,\Lambda)[d],\,(F',\Lambda')[d']} 
        \amalg 
        \id_{(F'',\Lambda'')[d'']})
        $}}
    \end{overpic}
    \end{center}
    \vspace{1em}
    
    Once again, the top and bottom are give isomorphic sutured cobordisms after gluing. Since our braiding is symmetric (verified below), commutativity for the other hexagon coherence diagram for the braiding (\cite[diagram (8.2)]{EGNO}) follows automatically.
    
    For the braiding--unitor coherence triangle given by
    \[
    \xymatrix@C=7em{
      (F,\Lambda)[d] \amalg (\emptyset,\emptyset)[0]
        \ar[r]^{\gamma_{(F,\Lambda)[d],(\emptyset,\emptyset)[0]}}
        \ar[dr]_{\rho_{(F,\Lambda)[d]}} &
      (\emptyset,\emptyset)[0] \amalg (F,\Lambda)[d]
        \ar[d]^{\lambda_{(F,\Lambda)[d]}} \\
      & (F,\Lambda)[d]
    }
    \]
    we argue commutativity graphically: the top and bottom representations below represent isomorphic sutured cobordisms after gluing. 

    \vspace{1em}
    \begin{center}
    \begin{overpic}[width=.58\textwidth]{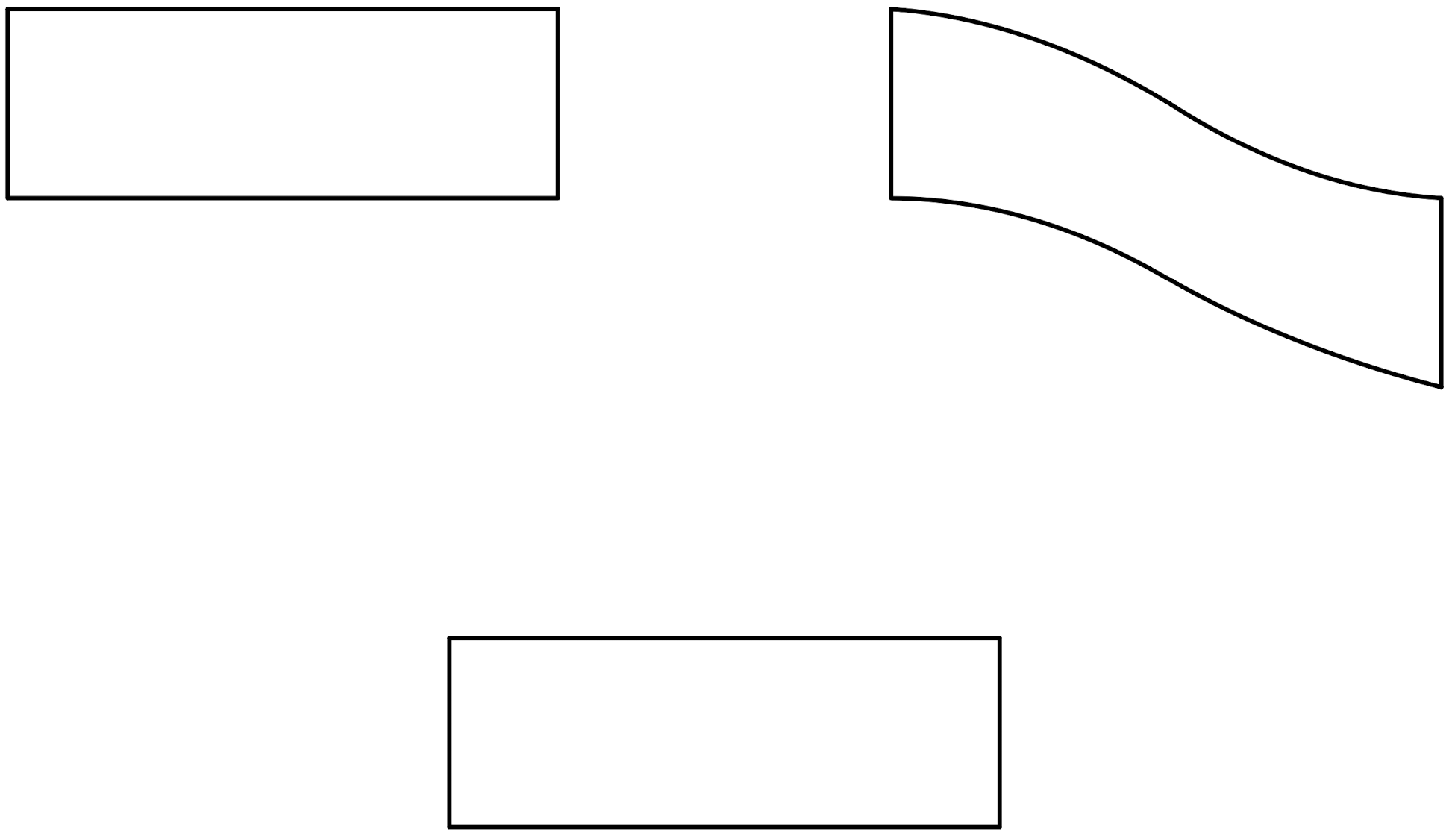}
        \put (-5,49) {\scalebox{0.9}{$F$}}
        \put (40,49) {\scalebox{0.9}{$-F$}}
        \put (42,35) {\scalebox{0.9}{$\emptyset$}}
        \put (56,49) {\scalebox{0.9}{$F$}}
        \put (56,35) {\scalebox{0.9}{$\emptyset$}}
        \put (104,50) {\scalebox{0.9}{$\emptyset$}}
        \put (102,35) {\scalebox{0.9}{$-F$}}
        \put (26,6) {\scalebox{0.9}{$F$}}
        \put (71,6) {\scalebox{0.9}{$-F$}}
        \put (73,19) {\scalebox{0.9}{$\emptyset$}}
        \put (31,27.5) {\scalebox{0.9}{$\lambda_{(F,\Lambda)[d]} \circ \gamma_{(F,\Lambda)[d],(\emptyset,\emptyset)[0]}$}}
        \put (45,-5) {\scalebox{0.9}{$\rho_{(F,\Lambda)[d]}$}}
    \end{overpic}
    \end{center}
    \vspace{1em}
    
    Lastly, to show the braiding is symmetric, we can see that the triangle
    \[
    \xymatrix@C=7em{
      (F,\Lambda)[d] \amalg (F',\Lambda')[d'] 
        \ar[rr]^{\gamma_{(F,\Lambda)[d], (F',\Lambda')[d']}} 
        \ar[rrd]_{\id_{(F\sqcup F',\Lambda \sqcup \Lambda')[d+d']}} 
      && 
      (F',\Lambda')[d'] \amalg (F,\Lambda)[d] 
        \ar[d]^{\gamma_{(F',\Lambda')[d'], (F,\Lambda)[d]}} \\
      && 
      (F,\Lambda)[d] \amalg (F',\Lambda')[d'] 
    }
    \]
    commutes by observing that the below diagram is equivalent to the one for the identity cobordism on $(F,\Lambda)[d] \amalg (F',\Lambda')[d']$. 
    
    \vspace{1em}
    \begin{center}
    \begin{overpic}[width=.55\textwidth]{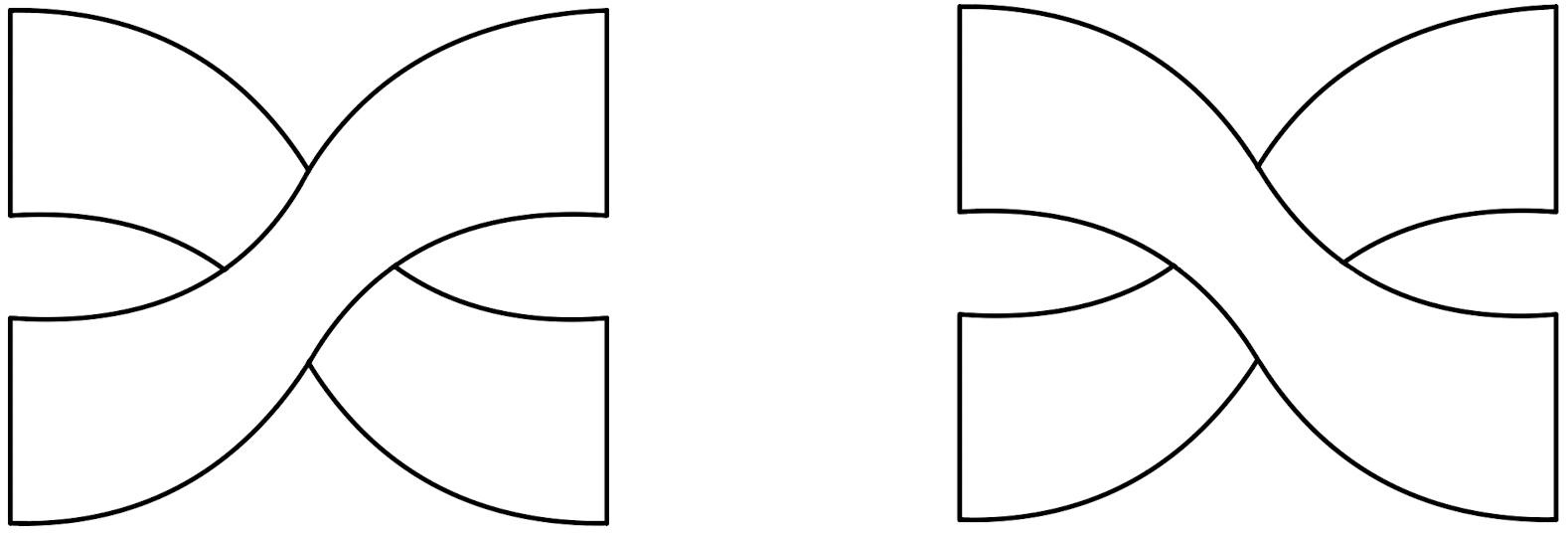}
        \put (-5,6) {\scalebox{0.9}{$F$}}
        \put (-5,26) {\scalebox{0.9}{$F'$}}
        \put (41,6) {\scalebox{0.9}{$-F'$}}
        \put (41,26) {\scalebox{0.9}{$-F$}}
        \put (55,6) {\scalebox{0.9}{$F'$}}
        \put (55,26) {\scalebox{0.9}{$F$}}
        \put (102,6) {\scalebox{0.9}{$-F$}}
        \put (102,26) {\scalebox{0.9}{$-F'$}}
        \put (7,-7) {\scalebox{0.9}{$\gamma_{(F',\Lambda')[d'], (F,\Lambda)[d]}$}}
        \put (67,-7) {\scalebox{0.9}{$\gamma_{(F,\Lambda)[d], (F',\Lambda')[d']}$}}
    \end{overpic}
    \end{center}
    \vspace{1em}
\end{proof}

Having established the topological framework, we now turn to the algebraic side of the theory. We will continue to carry out all the necessary verifications in full detail. 

\begin{theorem}\label{thm:AbZSymmetricMonoidal}
    As in Definition~\ref{def:AbZgrCategory}, define $\Ab^{\Z\gr}_{\pm 1}$ to be the category with: 
    \begin{itemize}
        \item Objects: $\Z$-graded abelian groups $V_* = \bigoplus_{k \in \Z} V_k$
        \item Morphisms: Homogeneous $\Z$-linear maps from $f \colon V_* \to W_*$ of arbitrary degree, modulo multiplication by $\pm 1$, with the evident composition and identity morphism. 
    \end{itemize}
    Equipped with:
    \begin{itemize}
        \item Monoidal structure: A bifunctor $\boldotimes: \Ab^{\Z\gr}_{\pm 1} \times \Ab^{\Z\gr}_{\pm 1} \to \Ab^{\Z\gr}_{\pm 1}$, which we will commonly denote by $V_* \boldotimes W_*$ for $\boldotimes(V_*,W_*)$ and $f \boldotimes g$ for $\boldotimes(f,g)$, defined as follows. 
        \begin{itemize}
            \item On objects: $\boldotimes$ coincides with the ordinary tensor product over the ground ring $\Z$ of $\Z$-graded abelian groups, with grading determined by 
            \[
            \deg(v \otimes w) = \deg(v)+\deg(w)
            \]
            That is, if $V_* = \bigoplus_{i \in \Z} V_i$ and $W_* = \bigoplus_{j \in \Z} W_j$, the tensor product is defined by 
            \[ 
            V_* \boldotimes W_* = \bigoplus_{k \in \Z} (V_* \boldotimes W_*)_k, 
            \] 
            where 
            \[ (V_* \boldotimes W_*)_k \coloneq \bigoplus_{i+j=k} V_i \otimes W_j. 
            \] 
            \item On morphisms: $\boldotimes$ is the signed tensor product of $\Z$-linear maps defined as
            \[
            (f \boldotimes g)(v \otimes w) \coloneq (-1)^{|f||w|} f(v) \otimes g(w).
            \]
        \end{itemize}
        \item Monoidal unit: $(\Z)_0$, the graded abelian group with $(\Z)_0=\Z$ and $(\Z)_k=0$ for $k \neq 0$.
        \item Associator and unitor: The usual maps. 
        \item Braiding: A natural isomorphism given by
        \begin{align*}
            \gamma_{V_*,W_*}: V_* \boldotimes W_* &\to W_*\boldotimes V_* \\
            v \otimes w &\mapsto (-1)^{|v||w|} w \otimes v
        \end{align*}        
    \end{itemize}
    Then $(\Ab^{\Z\gr}_{\pm 1},\boldotimes,\gamma_{V_*,W_*})$ forms a symmetric monoidal category. 
\end{theorem}

\begin{proof}
    To verify that $\boldotimes$ is a bifunctor, consider morphisms $f$ and $g$ from $V_*$ to $W_*$ and $f'$ and $g'$ from $V'_*$ to $W'_*$. For all sign checks in this proof, we will assume the domain elements are homogeneous; the same results extend to non-homogeneous generators by linearity. We first check the up-to-sign interchange law: On the right-hand side, we have
    \begin{align*}
        ((f' \circ f)\boldotimes(g' \circ g))(v \otimes w)
        &= (-1)^{|f' \circ f||w|}(f'\circ f)(v) \otimes (g'\circ g)(w) \\
        &= (-1)^{(|f'|+|f|)|w|} f'(f(v)) \otimes g'(g(w)) \\
        &= (-1)^{|f'||w| + |f||w|} f'(f(v)) \otimes g'(g(w)),
    \end{align*}
    whereas on the left-hand side we have
    \begin{align*}
        ((f'\boldotimes g')\circ (f\boldotimes g))(v\otimes w)
        &= (f'\boldotimes g')\big((-1)^{|f||w|} f(v) \otimes g(w)\big) \\
        &= (-1)^{|f||w|} (f'\boldotimes g')\big(f(v) \otimes g(w)\big) \\
        &= (-1)^{|f||w|} (-1)^{|f'||g(w)|} f'(f(v)) \otimes g'(g(w)) \\
        &= (-1)^{|f||w|} (-1)^{|f'|(|g|+|w|)} f'(f(v)) \otimes g'(g(w)) \\
        &= (-1)^{|f||w| + |f'||g| + |f'||w|} f'(f(v)) \otimes g'(g(w)).
    \end{align*}
    The maps differ by an overall sign $(-1)^{|f'||g|}$, which we can ignore in $\Ab^{\Z\gr}_{\pm 1}$.\footnote{As in Remark~\ref{rem:MotivationForUpToSign}, the necessity of treating morphisms only up to sign in our target category is on display here. When $f'$ and $g$ are both odd, this super-interchange law does not hold on-the-nose.} Furthermore, the identity morphism is respected exactly (without the need for sign corrections) since it is an even map: 
    \begin{align*}
        (\id_V\boldotimes\id_W)(v\otimes w)
        &= (-1)^{|\id_V||w|}\id_V(v)\otimes \id_W(w) \\
        &= v\otimes w\\
        &= \id_{V\boldotimes W}(v\otimes w).
    \end{align*}
    Thus, it follows that $\boldotimes$ is a bifunctor. 
    
    We now check that the associator is natural by proving that the following diagram commutes.
    \[
    \xymatrix{
      (V_* \boldotimes W_*) \boldotimes X_* \ar[rr]^{\alpha_{V_*, W_*, X_*}} \ar[d]_{(f \boldotimes g) \boldotimes h} &&
      V_* \boldotimes (W_* \boldotimes X_*) \ar[d]^{f \boldotimes (g \boldotimes h)} \\
      (V'_* \boldotimes W'_*) \boldotimes X'_* \ar[rr]_{\alpha_{V'_*, W'_*, X'_*}} &&
      V'_* \boldotimes (W'_* \boldotimes X'_*)
    }
    \]
    Indeed, the clockwise composition gives
    \begin{align*}
        (v \otimes w) \otimes x &\xmapsto{\alpha} v \otimes (w \otimes x) \\
        &\xmapsto{f \boldotimes (g \boldotimes h)} (-1)^{|f||w \otimes x|} f(v) \otimes ((g \boldotimes h))(w \otimes x)) \\
        &= (-1)^{|f|(|w| + |x|)} f(v) \otimes \left( (-1)^{|g||x|} g(w) \otimes h(x) \right) \\
        &= (-1)^{|f||w| + |f||x| + |g||x|} f(v) \otimes (g(w) \otimes h(x))
    \end{align*}
    whereas the counterclockwise path gives
    \begin{align*}
        (v \otimes w) \otimes x &\xmapsto{(f \boldotimes g) \boldotimes h} (-1)^{|f \boldotimes g||x|} ((f \boldotimes g)(v \otimes w)) \otimes h(x) \\
        &= (-1)^{(|f| + |g|)|x|} ((-1)^{|f||w|} f(v) \otimes g(w)) \otimes h(x) \\
        &\xmapsto{\alpha} (-1)^{|f||w| + |f||x| + |g||x|} f(v) \otimes (g(w) \otimes h(x)).
    \end{align*}
    
    The left and right unitors are 
    \[
    \begin{aligned}
      \lambda_{V_*} : \Z \boldotimes V_* &\to V_* &\quad \rho_{V_*} : V_* \boldotimes \Z &\to V_* \\
      n \otimes v &\mapsto nv &\quad v \otimes n &\mapsto nv.
    \end{aligned}
    \]
    These maps are natural in $V_*$, since we will show that the below squares commute. 
    \[
    \begin{array}{cc}
    \xymatrix{
    \Z \boldotimes V_* \ar[rr]^{\lambda_{V_*}} \ar[d]_{\id_{\Z} \boldotimes f} &&
    V_* \ar[d]^f \\
    \Z \boldotimes V'_* \ar[rr]_{\lambda_{V'_*}} &&
    V'_*
    }
    &
    \xymatrix{
    V_* \boldotimes \Z \ar[rr]^{\rho_{V_*}} \ar[d]_{f \boldotimes \id_{\Z}} &&
    V_* \ar[d]^f \\
    V'_* \boldotimes \Z \ar[rr]_{\rho_{V'_*}} &&
    V'_*
    }
    \end{array}
    \]
    For the the left unitor, clockwise gives
    \[
    n \otimes v \xmapsto{\lambda_{V_*}} nv \xmapsto{f} n f(v)
    \]
    while counterclockwise gives 
    \begin{align*}
    (\id_{\Z} \boldotimes f)(n \otimes v)
    &= (-1)^{|\id_{\Z}||v|} \id_{\Z}(n) \otimes f(v) \\
    &= (-1)^{0 \cdot |v|} n \otimes f(v) \\
    &= n \otimes f(v),
    \end{align*}
    then applying $\lambda_{V'_*}$ to this gives $nf(v)$. The right unitor is similar: the clockwise composition takes $v \otimes n$ to $nf(v)$, whereas 
    \begin{align*}
        (f \boldotimes \id_{\Z})(v \otimes n)
        &= (-1)^{|f||n|} f(v) \otimes \id_{\Z}(n) \\
        &= (-1)^{|f| \cdot 0} f(v) \otimes n \\
        &= f(v) \otimes n,
    \end{align*}
    since $|n|=0$ for all $n\in \Z$. Applying $\rho_{V'_*}$ to this shows that the counterclockwise composite map is also $nf(v)$. 
    
    The monoidal pentagon and triangle coherence diagrams involving $\alpha_{V_*, W_*, X_*}$, $\lambda_{V_*}$, and $\rho_{V_*}$ commute without sign ambiguity. Since these are the standard coherence maps in a symmetric monoidal category of free modules with tensor product as the monoidal structure, we do not repeat the verification here. 
    
    We conclude the proof by showing that $\gamma$ is natural in $V_*$ and $W_*$ up to sign, and satisfies the rest of the necessary coherence properties on the nose. The naturality square is written
    \[
    \xymatrix@C=6em{
      V_* \boldotimes W_* \ar[r]^{\gamma_{V_*,W_*}} \ar[d]_{f \boldotimes g} &
      W_* \boldotimes V_* \ar[d]^{g \boldotimes f} \\
      V'_* \boldotimes W'_* \ar[r]_{\gamma_{V'_*,W'_*}} &
      W'_* \boldotimes V'_*
    }
    \]
    Clockwise gives
    \begin{align*}
    ((g \boldotimes f)\circ \gamma_{V_*,W_*})(v \otimes w)
    &= (g \boldotimes f)\big((-1)^{|v||w|} w \otimes v\big) \\
    &= (-1)^{|v||w|} (-1)^{|g||v|} g(w) \otimes f(v) \\
    &= (-1)^{|v||w|+|g||v|} g(w) \otimes f(v).
    \end{align*}
    Whereas counterclockwise gives
    \begin{align*}
    \gamma_{V'_*,W'_*}(f \boldotimes g)(v \otimes w)
    &= \gamma_{V'_*,W'_*}\big((-1)^{|f||w|} f(v) \otimes g(w)\big) \\
    &= (-1)^{|f||w|} (-1)^{|f(v)||g(w)|} g(w) \otimes f(v) \\
    &= (-1)^{|f||w| + (|f|+|v|)(|g|+|w|)} g(w) \otimes f(v) \\
    &= (-1)^{|f||g| + 2|f||w| + |v||g| + |v||w|} g(w) \otimes f(v) \\
    &= (-1)^{|v||w|+|v||g|+|f||g|} g(w) \otimes f(v).
    \end{align*}
    These differ by an overall sign $(-1)^{|f||g|}$. 
    
    The coherence hexagon \cite[diagram (8.1)]{EGNO} is given by 
    \[
    \xymatrix@C=6em@R=4em{
    ((U_* \boldotimes V_*) \boldotimes W_*) 
      \ar[r]^{\raise1ex\hbox{$\scriptstyle \gamma_{U_*,V_*} \boldotimes \id_{W_*}$}} 
      \ar[d]_{\scriptstyle \alpha_{U_*,V_*,W_*}} 
    & ((V_* \boldotimes U_*) \boldotimes W_*) 
      \ar[d]^{\scriptstyle \alpha_{V_*,U_*,W_*}} \\
    U_* \boldotimes (V_* \boldotimes W_*) 
      \ar[d]_{\scriptstyle \gamma_{U_*,V_* \boldotimes W_*}} 
    & V_* \boldotimes (U_* \boldotimes W_*) 
      \ar[d]^{\scriptstyle \id_{V_*} \boldotimes \gamma_{U_*,W_*}} \\
    (V_* \boldotimes W_*) \boldotimes U_* 
      \ar[r]_{\lower1ex\hbox{$\scriptstyle \alpha_{V_*,W_*,U_*}$}} 
    & V_* \boldotimes (W_* \boldotimes U_*)
    }
    \]
    Clockwise gives
    \begin{align*}
        &((v \otimes w) \otimes x) \xmapsto{\gamma_{V_*,W_*} \boldotimes \id} 
        (-1)^{|v||w|} (w \otimes v) \otimes x \\
        &\xmapsto{\alpha} 
        (-1)^{|v||w|} w \otimes (v \otimes x) \\
        &\xmapsto{\id \boldotimes \gamma} 
        (-1)^{|v||w|} (-1)^{|v||x|} w \otimes (x \otimes v) \\
        &= (-1)^{|v||w| + |v||x|} w \otimes (x \otimes v).
    \end{align*}
    Whereas counterclockwise gives
    \begin{align*}
        &((v \otimes w) \otimes x) \xmapsto{\alpha} 
        v \otimes (w \otimes x) \\
        &\xmapsto{\gamma \boldotimes \id} 
        (-1)^{|v||w \otimes x|} (w \otimes x) \otimes v \\
        &\xmapsto{\alpha} 
        (-1)^{|v||w| + |v||x|} w \otimes (x \otimes v).
    \end{align*}
    The other coherence hexagon \cite[diagram (8.2)]{EGNO} follows automatically from symmetry of the braiding (verified below). The braiding--unitor coherence triangle is
    \[
    \xymatrix@C=7em{
      (V_* \boldotimes \Z)
        \ar[r]^{\gamma_{V_*,\,\Z}}
        \ar[dr]_{\rho_{V_*}} &
      (\Z \boldotimes V_*)
        \ar[d]^{\lambda_{V_*}} \\
      & V_*
    }
    \]
    Counterclockwise gives $v \otimes n \xmapsto{\rho} nv$, while traveling clockwise gives
    \begin{align*}
        v \otimes n &\xmapsto{\gamma} (-1)^{|v||n|} n \otimes v  \\
        &= n \otimes v  \\
        &\xmapsto{\lambda} nv
    \end{align*}
    where the second line holds since $|n|=0$.
    
    Finally, we show the braiding is symmetric in $V_*$ and $W_*$: The diagram
    \[
    \xymatrix@C=6em{
      V_* \boldotimes W_* 
        \ar[rr]^{\gamma_{V_*,W_*}} 
        \ar[rrd]_{\mathrm{id}} 
      && 
      W_* \boldotimes V_* 
        \ar[d]^{\gamma_{W_*,V_*}} \\
      && 
      V_* \otimes W_* 
    }
    \]
    commutes since 
    \begin{align*}
        \gamma_{W_*,V_*} \circ \gamma_{V_*,W_*}(v \otimes w)
        &= \gamma_{W_*,V_*}\left( (-1)^{|v||w|} w \otimes v \right) \\
        &= (-1)^{|v||w|} (-1)^{|w||v|} v \otimes w \\
        &= (-1)^{2|v||w|} v \otimes w \\
        &= v \otimes w.
    \end{align*}
\end{proof}

\subsection{The functor $\Vc^{\FN}_{\sut}$}

We now verify that $\Vc^{\FN}_{\sut}$ is an honest symmetric monoidal functor, establishing a proof of Theorem~\ref{Thm:VFNSymmetricMonoidal}. The definition of a \emph{symmetric monoidal functor} is that of a braided monoidal functor between symmetric monoidal categories. We begin the proof by showing that $\Vc^{\FN}_{\sut}$ is a monoidal functor; for this, it is not necessary to specify a unit compatibility map for $\Vc^{\FN}_{\sut}$. This is consistent with the framework of Etingof--Gelaki--Nikshych--Ostrik: their definition of monoidal functor \cite[Definition 2.4.1]{EGNO} does not build in unit compatibility, but it is equivalent to a definition which does, \cite[Definition 2.4.5]{EGNO}. Next, to prove that $\Vc^{\FN}_{\sut}$ is a braided monoidal functor, we follow \cite[Definition 8.1.7]{EGNO}, which defines a braided monoidal functor as a monoidal functor satisfying a braiding coherence square \cite[Equation 8.5]{EGNO} that we will check directly.  

\begin{theorem}[cf. Theorem \ref{Thm:VFNSymmetricMonoidal}]
The data below defines a symmetric monoidal functor $\Vc^{\FN}_{\sut}$. 
\begin{itemize}
    \item $\Vc^{\FN}_{\sut}$ acts on objects as 
    \[
    (F,\Lambda)[d] \mapsto (\wedge^* H_1(F,S^+))[d]
    \]
    where $[d]$ denotes a degree shift downward by $d$, so that 
    \[
    (\Vc^{\FN}_{\sut}((F, \Lambda)[d]))_k=\wedge^{k+d} H_1(F, S^{+}). 
    \]
    \item $\Vc^{\FN}_{\sut}$ acts on morphisms $(Y,\Gamma)$ from $(F_0,\Lambda_0)[d_0]$ to $(F_1,\Lambda_1)][d_1]$  as the up-to-sign composition 
    \[
    (\varepsilon \boldotimes \id)\circ(\id \boldotimes |K_{Y,\Gamma}|):\wedge^* H_1(F_0,S^+_0) \to \wedge^* H_1(F_1,S^+_1)
    \]
    defined in Theorem~\ref{thm:IntroFirstThm}; equivalently, as the up-to-sign map
    \[
    \Vc^{\FN}_{\sut}(Y, \Gamma)=[\BSDA(Y,\Gamma)]^{\Z}_{\comb}
    \]
    by Theorem~\ref{thm:IntroFirstThm}. As a morphism from $\wedge^* H_1(F_0,S^+_0)[d_0]$ to $\wedge^* H_1(F_1,S^+_1)[d_1]$, the map $\Vc^{\FN}_{\sut}(Y,\Gamma)$ has degree $c + d_0 - d_1$.
    \item $\Vc^{\FN}_{\sut}$ has monoidal structure map determined by
    \begin{align*}
        \Phi_{(F,\Lambda)[d],(F',\Lambda')[d']}: \Vc^{\FN}_{\sut}((F,\Lambda)[d]) \boldotimes \Vc^{\FN}_{\sut}((F',\Lambda')[d']) &\to \Vc^{\FN}_{\sut}((F,\Lambda)[d] \amalg (F',\Lambda')[d']) \\
        \gamma_I \otimes \gamma_{I'} &\mapsto (-1)^{k'd}\gamma_I \wedge \gamma_{I'}
    \end{align*}
    where $\gamma_I$ and $\gamma_{I'}$ are basis elements of $\wedge^k H_1(F,S^+)$ and $\wedge^k H_1(F',(S^+)')$ respectively, as defined in Definition~\ref{def:BSDAZComb}. 
    \end{itemize}
\end{theorem}

\begin{proof}
    Throughout the proof, we expand $\boldotimes$ and $\amalg$ on objects for clarity. The first property to check is naturality of $\Phi_{(F,\Lambda)[d],(F',\Lambda')[d']}$ in $(F,\Lambda)[d],$ and $(F',\Lambda')[d']$, which holds if the below square commutes. 
    \[
    \xymatrix@C=6em@R=5em{
        \wedge^* H_1(F, S^+)[d] \otimes \wedge^* H_1(F', (S^+)')[d']
            \ar[r]^-{\Phi_{(F,\Lambda)[d], (F',\Lambda')[d']}}
            \ar[d]_-{\Vc^{\FN}_{\sut}(Y,\Gamma) \boldotimes \Vc^{\FN}_{\sut}(Y',\Gamma')}
        &
        \wedge^* H_1(F \sqcup F', S^+ \sqcup (S^+)')[d+d']
            \ar[d]^-{\Vc^{\FN}_{\sut}(Y \sqcup Y', \Gamma \sqcup \Gamma')} \\
        \wedge^* H_1(\widetilde{F}, \widetilde{S}^+)[\tilde{d}]
            \otimes
            \wedge^* H_1(\widetilde{F}', \widetilde{(S^+)'})[\tilde{d}']
            \ar[r]_-{\Phi_{(\widetilde{F},\widetilde{\Lambda})[\tilde{d}], (\widetilde{F}',\widetilde{\Lambda}')[\tilde{d}']}}
        &
        \wedge^* H_1(\widetilde{F} \sqcup \widetilde{F}', \widetilde{S}^+ \sqcup \widetilde{(S^+)'})[\tilde{d} + \tilde{d}']
    }
    \]
    Let $\gamma_I$ and $\gamma_{I'}$ be basis elements of $\wedge^* H_1(F, S^+)[d]$ and $\wedge^* H_1(F', (S^+)')[d']$. We first compute the image of the counterclockwise composition on the basis element 
    \[
    \gamma_I \otimes \gamma_{I'} \in \wedge^* H_1(F, S^+)[d] \otimes \wedge^* H_1(F', (S^+)')[d'].
    \]
    We have
    \begin{align*}
        (\Vc^{\FN}_{\sut}(Y,\Gamma) \boldotimes \Vc^{\FN}_{\sut}(Y',\Gamma'))(\gamma_I \otimes \gamma_{I'})&=(-1)^{|\Vc^{\FN}_{\sut}(Y,\Gamma)||\gamma_{I'}|}\Vc^{\FN}_{\sut}(Y,\Gamma)(\gamma_{I})\otimes \Vc^{\FN}_{\sut}(Y',\Gamma')(\gamma_{I'}) \\
        &=(-1)^{(c+d-\tilde{d})(k'-d')}\Vc^{\FN}_{\sut}(Y,\Gamma)(\gamma_{I})\otimes \Vc^{\FN}_{\sut}(Y',\Gamma')(\gamma_{I'}).
    \end{align*}
    The map $\Phi_{(\widetilde{F},\widetilde{\Lambda})[\tilde{d}], (\widetilde{F}',\widetilde{\Lambda}')[\tilde{d}']}$ then sends this to 
    \begin{equation}\label{eq:V^FNNaturalitySign}
        (-1)^{(c+d-\tilde{d})(k'-d')}(-1)^{\tilde{d}(k'+c')}\Vc^{\FN}_{\sut}(Y,\Gamma)(\gamma_{I})\wedge \Vc^{\FN}_{\sut}(Y',\Gamma')(\gamma_{I'}).
    \end{equation}
    After simplifying, this map has a total sign of 
    \begin{align*}
        (c+d-\tilde{d})(k'-d')+\tilde{d}(k'+c')&=c k'-c d' +k'd - d d' - k'\tilde{d} + \tilde{d} d'+k'\tilde{d} +\tilde{d} c' \\
        &=ck'+k'd+(\text{$k$ independent terms}).
    \end{align*}
    
    On the other hand, along the clockwise path the structure map first acts as defined:
    \[
    \Phi_{(F,\Lambda)[d],(F',\Lambda')[d']}(\gamma_I \otimes \gamma_{I'})=(-1)^{k'd}\gamma_I \wedge \gamma_{I'}. 
    \]
    To find the image of this element under the next map, first interpret $\Vc^{\FN}_{\sut}(Y \sqcup Y', \Gamma \sqcup \Gamma')$ as $[\BSDA(Y \sqcup Y', \Gamma \sqcup \Gamma')]^{\Z}_{\comb}$. Then to compute $[\BSDA(Y \sqcup Y', \Gamma \sqcup \Gamma')]^{\Z}_{\comb}$, take sets of choices $\Xi$ and $\Xi'$ for $(Y,\Gamma)$ and $(Y',\Gamma')$ respectively, and use these to define a set of choices $\Xi \sqcup \Xi'$ for the cobordism $(Y \sqcup Y', \Gamma \sqcup \Gamma')$ as in Definition~\ref{def:ChoicesDisjoint}. The resulting map $[\BSDA(Y \sqcup Y', \Gamma \sqcup \Gamma';\Xi \sqcup \Xi')]^{\Z}_{\comb}$ takes the element $(-1)^{k'd}\gamma_I \wedge \gamma_{I'}$ to
    \[
    \sum_{\substack{\x,\x' : \\ o_R(\x)=I, \\ \bar{o}_R(\x')=I'}} (-1)^{\mathrm{gr}_{\DA}(\x \sqcup \x')+k'd}\gamma_{\overline{o}_L(\x \sqcup \x')}
    \]
    In order to compare this expression with ~\eqref{eq:V^FNNaturalitySign}, we would like to expand $[\BSDA]^{\Z}_{\comb}$ over the disjoint union. Since $\x$ and $\x'$ are disjoint, and arcs of $\Zc_1$ come before arcs of $\Zc_1'$ in the ordering of matching arcs on $\Zc_1 \sqcup \Zc_1'$ in the choices $\Xi \sqcup \Xi'$, we know that
    \[
    \gamma_{\overline{o}_L(\x \sqcup \x')}=\gamma_{\overline{o}_L(x)} \wedge \gamma_{\overline{o}_L(x')}.
    \]
    However, we will observe that 
    \[
    \mathrm{gr}_{\DA}(\x \sqcup \x') \neq \mathrm{gr}_{\DA}(\x) + \mathrm{gr}_{\DA}(\x').
    \]
    The crossing strands diagram for a generator $\x \sqcup \x'$ of the chosen $\Hc \sqcup \Hc'$ takes the form pictured below. For clarity, we separate the $\alpha$-arcs and $\beta$-curves from $\Hc$ and $\Hc'$ into sub-blocks. 
    \begin{center}
    \begin{overpic}[scale=.28]{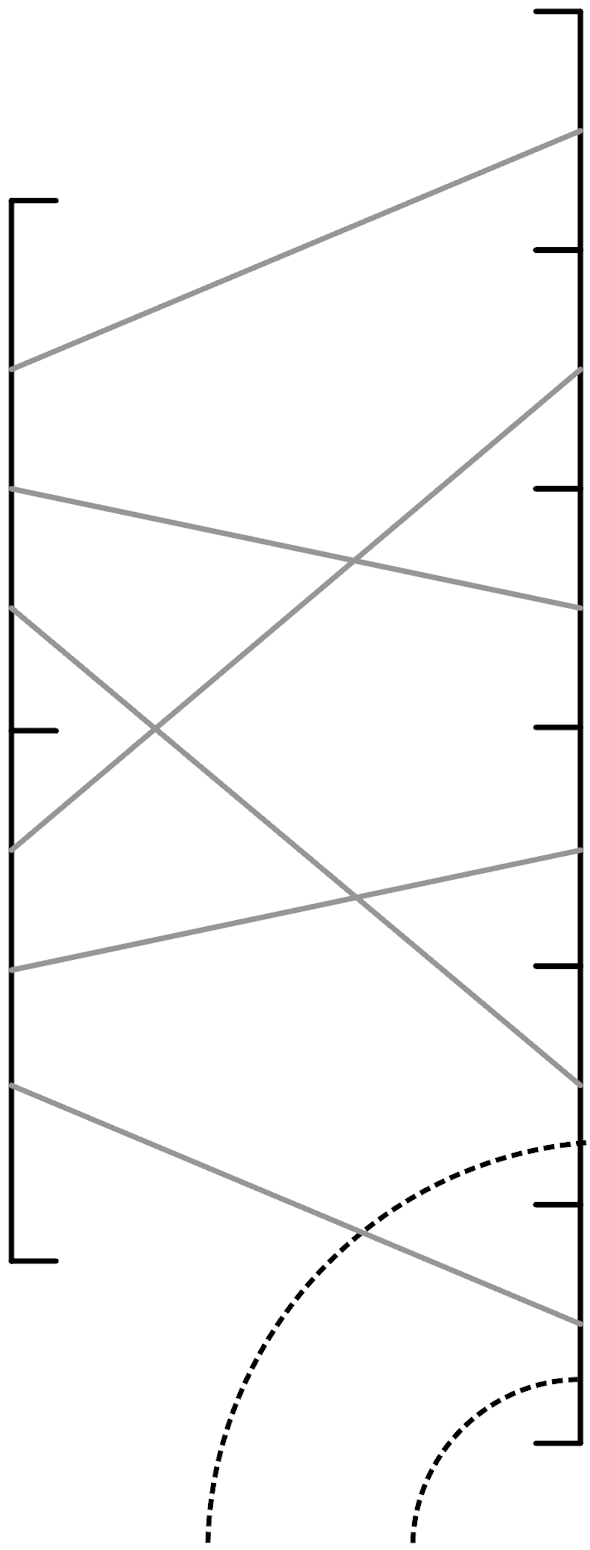}
        \put (-6,33) {$\boldsymbol{\beta}$}
        \put (-6,68) {$\boldsymbol{\beta}'$}
        \put (11.5,-5) {$l'$}
        \put (26,-5) {$l$}
        \put (6.5,19) {\scalebox{0.75}{$n_1-l$}}
        \put (17,32) {\scalebox{0.8}{$n'_1-l'$}}
        \put (28,45) {\scalebox{0.8}{$a$}}
        \put (10,69) {\scalebox{0.8}{$a'$}}
        \put (26,73) {\scalebox{0.8}{$k$}}
        \put (13,88) {\scalebox{0.8}{$k'$}}
        \put (42,15) {$\boldsymbol{\alpha}^{a,\mathrm{out}}$}
        \put (42,30) {$(\boldsymbol{\alpha}^{a,\mathrm{out}})'$}
        \put (42,45) {$\boldsymbol{\alpha}^{c}$}
        \put (42,60) {$(\boldsymbol{\alpha}^{c})'$}
        \put (42,75) {$\boldsymbol{\alpha}^{a,\mathrm{in}}$}
        \put (42,90) {$(\boldsymbol{\alpha}^{a,\mathrm{in}})'$}
        \vspace{0.1in}
    \end{overpic}
    \end{center}
    The ordering conventions imposed by $\Xi \sqcup \tilde{\Xi}$ interleave the indices of $\x$ and $\x'$, producing new crossings. These cause an inversion sign count difference between $\mathrm{gr}_{\DA}(\x \sqcup \x')$ and $\mathrm{gr}_{\DA}(\x) +\mathrm{gr}_{\DA}(\x')$ of 
    \[
    (n_1-l)l'+a(n_1'-l')+k(n_1'-l')+ka'.
    \]
    The intersection contributions of $\mathrm{gr}_{\DA}(\x \sqcup \x')$ and $\mathrm{gr}_{\DA}(\x) + \mathrm{gr}_{\DA}(\x')$ clearly agree. The ``$ak$'' correction term for $\mathrm{gr}_{\DA}(\x \sqcup \x')$ is $(a+a')(k+k')$, so the difference with the sum of ``$ak$'' correction terms for $\mathrm{gr}_{\DA}(\x)$ and $\mathrm{gr}_{\DA}(\x')$ is $ka' + k'a$. The ``$n_1 k$'' correction term $\mathrm{gr}_{\DA}(\x \sqcup \x')$ is $(n_1 + n_1')(k + k')$ so the difference with the sum of ``$n_1 k$'' correction terms for $\mathrm{gr}_{\DA}(\x)$ and $\mathrm{gr}_{\DA}(\x')$ is $n_1' k + n_1 k'$. After simplifying, we find a total sign difference of 
    \begin{align*}
       &(n_1-l)l'+a(n_1'-l')+k(n_1'-l')+ka'+ka'+k'a+n_1' k+n_1 k' \\
       &\quad =k'c +\text{($k$ independent terms)} \mod 2. 
    \end{align*}
    Thus, 
    \[
    \mathrm{gr}_{\DA}(\x \sqcup \x')=\mathrm{gr}_{\DA}(\x) + \mathrm{gr}_{\DA}(\x')+k'c+\text{($k$ independent terms)}
    \]
    So we may expand 
    \[
    \sum_{\substack{\x,\x' : \\ o_R(\x)=I, \\ \bar{o}_R(\x')=I'}} (-1)^{\mathrm{gr}_{\DA}(\x \sqcup \x')+k'd}\gamma_{\overline{o}_L(\x \sqcup \x')}=\sum_{\substack{\x, \x' : \\ o_R(\x) = I, \, o_R(\x') = I'}}
    (-1)^{\mathrm{gr}_{\DA}(\x) + \mathrm{gr}_{\DA}(\x') + k'c + k'd}
    \, \gamma_{\overline{o}_L(\x)} \wedge \gamma_{\overline{o}_L(\x')}
    \]
    up to overall sign. By multilinearity, this equals 
    \begin{align*}
        &(-1)^{k'c + k'd}
        \left( 
            \sum_{\substack{\x \in \mathfrak{S}(\Hc) : \\ o_R(\x) = I}}
            (-1)^{\mathrm{gr}_{\DA}(\x)} \gamma_{\overline{o}_L(\x)}
        \right)
        \wedge
        \left(
            \sum_{\substack{\x' \in \mathfrak{S}(\Hc') : \\ o_R(\x') = I'}}
            (-1)^{\mathrm{gr}_{\DA}(\x')} \gamma_{\overline{o}_L(\x')}
        \right) \\
        &\quad = (-1)^{k'c + k'd}
        [\BSDA(Y,\Gamma)]^{\Z}_{\comb}(\gamma_I)
        \wedge
        [\BSDA(Y',\Gamma')]^{\Z}_{\comb}(\gamma_{I'}) \\
        &\quad = (-1)^{k'c + k'd}
        \Vc^{\FN}(Y,\Gamma)(\gamma_I)
        \wedge
        \Vc^{\FN}(Y',\Gamma')(\gamma_{I'}).
    \end{align*}
    Since the sign $(-1)^{k'c + k'd}$ agrees with the sign in equation~\eqref{eq:V^FNNaturalitySign}, $\Phi_{(F,\Lambda)[d],(F',\Lambda')[d']}$ is a natural transformation. 

    Next, we will prove the coherence hexagon for the monoidal structure map is commutative. 
    \[
    \makebox[\textwidth][c]{
    \scalebox{0.83}{
    \xymatrix@C=6em@R=4em{
    \left( \wedge^* H_1(F,S^+)[d] \otimes \wedge^* H_1(F',(S^+)')[d'] \right)
    \otimes \wedge^* H_1(F'',(S^+)')[d'']
    \ar[r]^-{\raisebox{3ex}{\hbox{$
    \alpha_{\wedge^* H_1(F,S^+)[d],\,
    \wedge^* H_1(F',(S^+)')[d'],\,
    \wedge^* H_1(F'',(S^+)')[d'']}
    $}}}
    \ar[d]_{\scriptstyle
    \Phi_{(F,\Lambda)[d],(F',\Lambda')[d']} \,\boldotimes\, \id}
    &
    \wedge^* H_1(F,S^+)[d] \otimes
    \left( \wedge^* H_1(F',(S^+)')[d'] \otimes
    \wedge^* H_1(F'',(S^+)')[d''] \right)
    \ar[d]^{\scriptstyle
    \id \,\boldotimes\,
    \Phi_{(F',\Lambda')[d'],(F'',\Lambda'')[d'']}}
    \\
    \wedge^* H_1(F \sqcup F', S^+ \sqcup (S^+)')[d+d']
    \otimes \wedge^* H_1(F'',(S^+)')[d'']
    \ar[d]_{\scriptstyle
    \Phi_{(F \sqcup F', \Lambda \sqcup \Lambda')[d+d'],
    (F'',\Lambda'')[d'']}}
    &
    \wedge^* H_1(F,S^+)[d] \otimes
    \wedge^* H_1(F' \sqcup F'', (S^+)' \sqcup (S^+)')[d'+d'']
    \ar[d]^{\scriptstyle
    \Phi_{(F,\Lambda)[d],
    (F' \sqcup F'', \Lambda' \sqcup \Lambda'')[d'+d'']}}
    \\
    \wedge^* H_1((F \sqcup F') \sqcup F'',
    (S^+ \sqcup (S^+)') \sqcup (S^+)')[d+d'+d'']
    \ar[r]_-{\raisebox{-4ex}{\hbox{$
    \Vc^{\FN}_{\sut}\!\left(
    \alpha_{(F,\Lambda)[d],\,
    (F',\Lambda')[d'],\,
    (F'',\Lambda'')[d'']}
    \right)
    $}}}
    &
    \wedge^* H_1(F \sqcup (F' \sqcup F''),
    S^+ \sqcup ((S^+)' \sqcup (S^+)'))[d+d'+d'']
    }
    }
    }
    \]
    The counterclockwise path acts on a basis element $(\gamma_I \otimes \gamma_{I'}) \otimes \gamma_{I''}$ of the source by
    \begin{align*}
        (\gamma_I \otimes \gamma_{I'}) \otimes \gamma_{I''}
            &\xmapsto{\Phi \boldotimes \id}
            (-1)^{k'd} (\gamma_I \wedge \gamma_{I'}) \otimes \gamma_{I''} \\
        &\quad \xmapsto{\Phi}
            (-1)^{k'd + k''(d+d')} \, ((\gamma_I \wedge \gamma_{I'}) \wedge \gamma_{I''}) \\
        &\quad \xmapsto{\Vc^{\FN}_{\sut}(\alpha)}
            (-1)^{k'd + k''d + k''d'} (\gamma_I \wedge (\gamma_{I'} \wedge \gamma_{I''}))
    \end{align*}
    The clockwise path acts on this same basis element by
    \begin{align*}
        (\gamma_I \otimes \gamma_{I'}) \otimes \gamma_{I''}
            &\xmapsto{\alpha}
            \gamma_I \otimes (\gamma_{I'} \otimes \gamma_{I''}) \\
        &\quad \xmapsto{\id \boldotimes \Phi}
             \gamma_I \otimes ((-1)^{k''d'}\gamma_{I'} \wedge \gamma_{I''}) \\
        &\quad \xmapsto{\Phi}
            (-1)^{k''d'+(k'+k'')d} \, (\gamma_I \wedge (\gamma_{I'} \wedge \gamma_{I''})).
    \end{align*}
    The two images agree on the nose, so the hexagon commutes. 
    
    The proof up until this point says that $\Vc^{\FN}_{\sut}$ is a monoidal functor. To establish that it is a symmetric monoidal functor, we need the following diagram to commute: 
    \begin{equation*}
    \makebox[\textwidth][c]{
    \scalebox{0.95}{
    \xymatrix@C=10em@R=4em{
    \wedge^* H_1(F, S^+)[d] \otimes \wedge^* H_1(F', (S^+)')[d'] 
        \ar[r]^-{\gamma_{\wedge^* H_1(F, S^+)[d], \wedge^* H_1(F', (S^+)')[d'] }} 
        \ar[d]_{\Phi_{(F,\Lambda)[d],\,(F',\Lambda')[d']}} 
    & 
    \wedge^* H_1(F', (S^+)')[d'] \otimes \wedge^* H_1(F, S^+)[d]
        \ar[d]^{\Phi_{(F',\Lambda')[d'],\,(F,\Lambda)[d]}} \\
    \wedge^* H_1(F \sqcup F', S^+ \sqcup (S^+)')[d + d'] 
        \ar[r]_-{\Vc^{\FN}_{\sut}(\gamma_{(F,\Lambda)[d],\,(F',\Lambda')[d']})} 
    & 
    \wedge^* H_1(F' \sqcup F, (S^+)' \sqcup S^+)[d' + d]
    }
    }
    }
    \end{equation*}
    Note that $\gamma_I$ is a homogeneous element of degree $k$ in $\wedge^* H_1(F, S^+)$, so in $\wedge^* H_1(F, S^+)[d]$, we have $|\gamma_I|=k-d$. Thus, following the clockwise path:
    \begin{align*}
        \gamma_I \otimes \gamma_{I'}
            &\xmapsto{\gamma}
            (-1)^{(k-d)(k'-d')}\gamma_{I'} \otimes \gamma_I 
            \\
            &\xmapsto{\Phi}
            (-1)^{(k-d)(k'-d')}(-1)^{kd'}\gamma_{I'} \wedge \gamma_I.
    \end{align*}
    This sign amounts to a total of
    \begin{align*}
        (k-d)(k'-d')+kd'&=kk'-kd'-k'd-dd'+kd' \\
        &=kk'+k'd+\text{($k$ independent terms)} \mod 2.
    \end{align*}
    
    For the counterclockwise path, we compute $[\BSDA]^{\Z}_{\mathrm{comb}}$ of the braiding cobordism $\gamma_{(F,\Lambda)[d],\,(F',\Lambda')[d']}$, which we will then apply to $(-1)^{k'd}\gamma_I \wedge \gamma_{I'}$. The manifold $\gamma_{(F,\Lambda)[d],\,(F',\Lambda')[d']}$ is homeomorphic to
    \[
     (\id_{(F \sqcup F',\Lambda \sqcup \Lambda')[d+d']},\Gamma_{\id} \sqcup \Gamma'_{\id})=(\id_{(F,\Lambda)[d]},\Gamma_{\id}) \amalg (\id_{(F',\Lambda')[d]},\Gamma'_{\id}). 
    \]
    To compute $[\BSDA(\id_{(F \sqcup F',\Lambda \sqcup \Lambda')[d+d']},\Gamma_{\id} \sqcup \Gamma'_{\id})]^{\Z}_{\mathrm{comb}}$, use choices $\Xi_{\gamma}$ that are the same as $\Xi_{\id} \sqcup \Xi'_{\id}$ in Definition~\ref{def:ChoicesDisjoint} (where $\Xi_{\id}$ and $\Xi'_{\id}$ are defined as in Proposition~\ref{prop:BSDARespectsIdentity}), except that we use the arc diagram $\Zc' \sqcup \Zc$ (rather than $\Zc \sqcup \Zc'$) to represent the outgoing surface $F' \sqcup F$, and correspondingly, in the ordering of arcs of $\Zc' \sqcup \Zc$, we put the arcs of $\Zc'$ before the arcs of $\Zc$. The Heegaard diagram chosen in $\Xi_{\gamma}$ is still $\Hc_{\id} \sqcup \Hc_{\id}'$, and for a generator $\x \sqcup \x'$ of $\Hc_{\id} \sqcup \Hc_{\id}'$, the intersection patterns assumed by $\Xi_{\gamma}$ imply that
    \[
    \sum_{x \sqcup x' \in \x \sqcup \x'}i(x \sqcup x')=n-k+n'-k'=k+k'
    \]
    up to a $k$- and $k'$-independent sign. For the inversion counts, we appeal to the crossing strands diagram for $\x \sqcup \x'$. 
    \begin{center}
    \begin{overpic}[scale=.25]{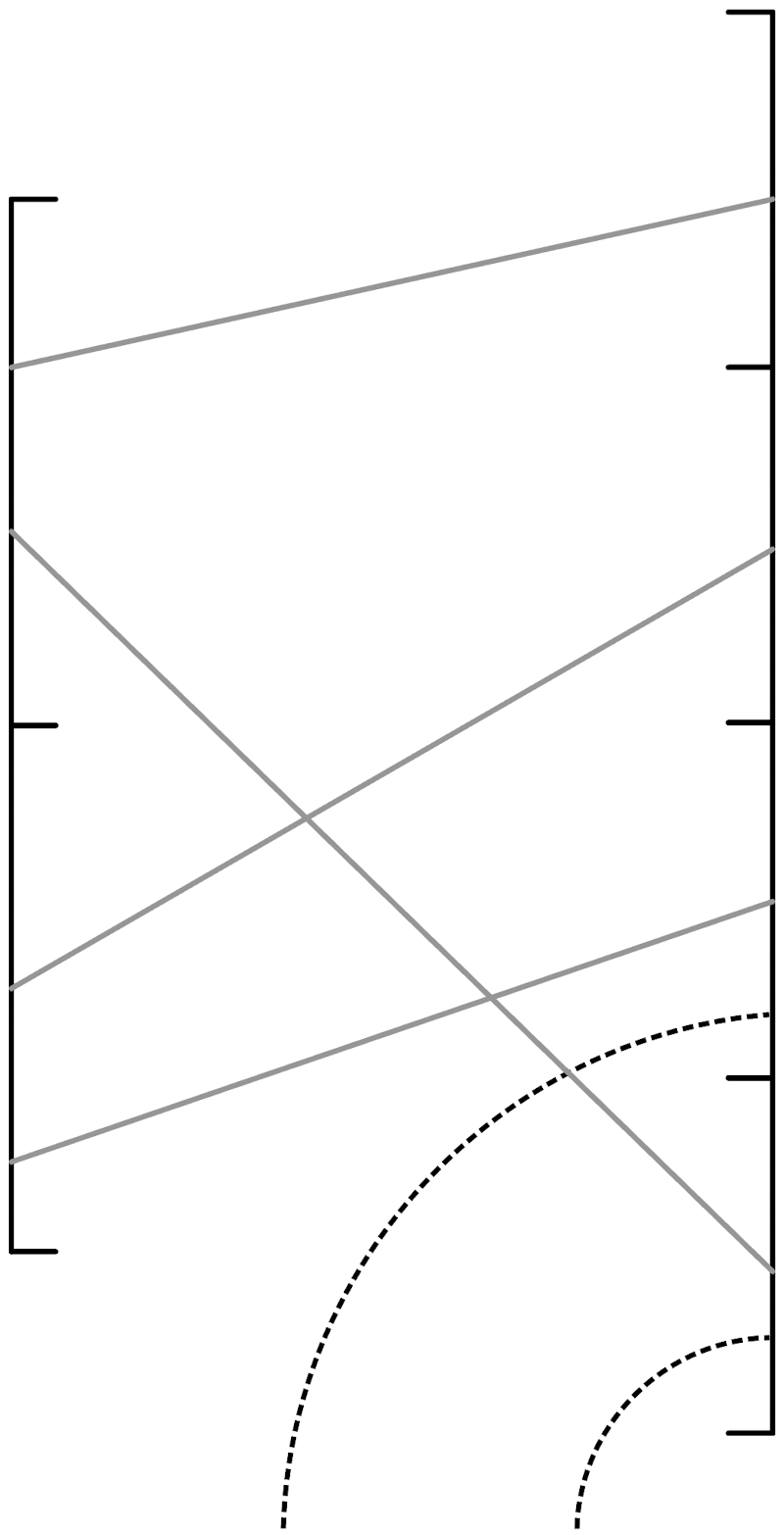}
        \put (-8,35) {$\boldsymbol{\beta}$}
        \put (-8,67) {$\boldsymbol{\beta}'$}
        \put (55,18) {$(\boldsymbol{\alpha}^{a,\mathrm{out}})'$}
        \put (55,40) {$\boldsymbol{\alpha}^{a,\mathrm{out}}$}
        \put (55,63) {$\boldsymbol{\alpha}^{a,\mathrm{in}}$}
        \put (55,88) {$(\boldsymbol{\alpha}^{a,\mathrm{in}})'$}
        \put (25,84) {$k'$}
        \put (8,62) {$n'-k'$}
        \put (35,49) {$k$}
        \put (7,21) {$n-k$}
        \put (17,-6) {$k$}
        \put (36,-6) {$k'$}    
    \end{overpic}
    \end{center}
    The diagram counts
    \[
    \mathrm{inv}(\sigma_{II^c \leftrightarrow \mathrm{std}})+\mathrm{inv}(\sigma_{I'(I')^c \leftrightarrow \mathrm{std}})+\mathrm{inv}(\sigma_{I^c I \leftrightarrow \mathrm{std}})+\mathrm{inv}(\sigma_{(I')^c I' \leftrightarrow \mathrm{std}})=k(n-k)+k'(n'-k');
    \]
    plus an additional
    \[
    k(n'-k')+(n-k)(n'-k')+k(n'-k')=(n-k)(n'-k') \mod 2.
    \]
    Lastly, the ``$ak$'' term is zero, since $a=0$, and the ``$n_1k$'' term is $(n+n')(k+k')$. Therefore, we have a total of
    \begin{align*}
        \mathrm{gr}_{\DA}(\x \sqcup \x')&=k+k'+k(n-k)+k'(n'-k')+(n-k)(n'-k')+(n+n')(k+k') \\
        &=kk' + \text{($k$-independent terms)} \mod 2. 
    \end{align*}
    Thus, 
    \begin{align*}
        \Vc^{\FN}_{\sut}(\gamma_{(F,\Lambda)[d],\,(F',\Lambda')[d']})((-1)^{k'd}\gamma_{I} \wedge \gamma_{I'})=(-1)^{k'd+kk'}\gamma_{I'} \wedge \gamma_I,
    \end{align*}
    which agrees with the other path on the nose. This proves that $\Vc^{\FN}_{\sut}$ is a symmetric monoidal functor. 
\end{proof}

\section{Relationship with SFH}\label{sec:RelationshipWithSFH}

In \cite{joiningandgluing}, Zarev outlines a general philosophy that bordered sutured Floer invariants should be understandable in terms of ordinary (non-bordered) sutured Floer invariants; idempotents in the bordered theory (for us: basis elements $\gamma^{\inrm}_I$ and $\gamma^{\out}_J$) correspond to different ways of extending the sutures on a sutured cobordism $(Y,\Gamma) \colon (F_0,\Lambda_0) \to (F_1,\Lambda_1)$ over $F_i$ to get an ordinary sutured 3-manifold. This philosophy suggests that, given basis elements $\gamma^{\inrm}_I$ and $\gamma^{\out}_J$, the matrix coefficient of $[\BSDA(Y,\Gamma;\Xi)]^{\Z}_{\comb}$ for these basis elements should be the Euler characteristic of a certain ordinary sutured Floer group. We verify this prediction here; below this analysis will help us in defining the $\Spinc$ version $[\BSDA(Y,\Gamma;\Xi)]^{\Z[H]}_{\comb}$ of $[\BSDA(Y,\Gamma;\Xi)]^{\Z}_{\comb}$.

\subsection{Zarev caps}\label{sec:ZarevCaps}

We now define a special class of bordered sutured manifolds---along with Heegaard diagrams for them---that give a standard way to build an ordinary sutured manifold from a bordered sutured one. 

\begin{definition}\label{def:ZarevCaps}
    Let $(Y,\Gamma)$ be a sutured cobordism from $(F_0,\Lambda_0)$ to $(F_1,\Lambda_1)$. Choose $\alpha$-arc diagrams $\Zc_i$ representing $(F_i,\Lambda_i)$. Let $S_0 \subset \mathbf{a}_0$ be a subset of the set of arcs of $\Zc_0$ and let $S_1 \subset \mathbf{a}_1$ be a subset of the set of arcs of $\Zc_1$. The \emph{right Zarev cap for $(F_0,\Lambda_0)$ with respect to $S_0$} (resp. \emph{left Zarev cap for $(F_1,\Lambda_1)$ with respect to $S_1$}, denoted $(Y^L_{S_1},\Gamma^L_{S_1})$), denoted $(Y^R_{S_0},\Gamma^R_{S_0})$, is the one-sided sutured cobordism from $(\emptyset,\emptyset)$ to $(F_0,\Lambda_0)$ (resp. $(F_1,\Lambda_1)$ to $(\emptyset,\emptyset)$) defined as follows.
    \begin{itemize}
        \item Let $(Y^R_{S_0})'$ (resp. $(Y^L_{S_1})'$) be the identity cobordism from $(F_0,\Lambda_0)$ (resp. $(F_1,\Lambda_1)$) to itself, i.e. the product manifold $(F_0 \times [0,1],\Lambda_0 \times [0,1])$, viewed with $F_0 \times \{0\}$ (resp. $F_1 \times \{0\}$) on the right and $F_0 \times \{1\}$ (resp. $F_0 \times \{0\}$) on the left. 
        \item Start by extending the sutures of the identity cobordism to sutures on $F_0 \times \{0\}$ (resp. $F_1 \times \{1\}$) by taking the sutures in $F_0 \times \{0\}$ (resp. $F_1 \times \{1\}$) to be a small parallel pushoff of $S^+_0 \times \{0\}$ into $F_0 \times \{0\}$ (resp. $S^+_1 \times \{1\}$ into $F_1 \times \{1\}$). These sutures define a sutured cobordism $((Y^R_{S_0})',(\Gamma^R_{S_0})')$ from $(\emptyset, \emptyset)$ to $(F_0,\Lambda_0)$ (resp. $((Y^L_{S_1})',(\Gamma^L_{S_1})')$ from $(F_1, \Lambda_1)$ to $(\emptyset, \emptyset)$) whose $R^+$ region, which we will call $(R^+_0)'$ (resp. $(R^+_1)'$), is such that $(R^+_0)' \cap (F_0 \times \{0\})$ is a thickening of $S^+_0$ into $F_0 \times \{0\}$ (resp. $(R^+_1)' \cap (F_1 \times \{1\})$ is a thickening of $S^+_1$ into $F_1 \times \{1\}$). The $R^-$ region covers all of $F_0 \times \{0\}$ (resp. $F_1 \times \{0\}$) except for this thickened neighborhood of $S^+_0$ (resp. $S^+_1$).
        \item For each arc $a_i \in S_0$ (resp. $a_i \in S_1$), modify the sutures $(\Gamma^R_{S_0})'$ (resp. $(\Gamma^L_{S_1})'$) so they follow along parallel to $a_i$, and a strip parallel to $a_i$ that was previously in the $R^-$ region is now in the $R^+$ region. The rightmost portion of Figure~\ref{fig:ZarevCapsGeneral} depicts a schematic for these modifications to $R^+$ in the case of a right Zarev cap; the left Zarev cap case is analogous. 
        \item Define $(Y^R_{S_0},\Gamma^R_{S_0})$ (resp. $(Y^L_{S_1},\Gamma^L_{S_1})$) as the sutured 3-manifold $((Y^R_{S_0})',(\Gamma^R_{S_0})')$ (resp. $((Y^L_{S_1})',(\Gamma^L_{S_1})')$) after these modifications are made to the sutures $(\Gamma^R_{S_0})'$ (resp. $(\Gamma^L_{S_1})'$). 
        \end{itemize}
    Let $(Y,\Gamma_{S_1,S_0})$ be the ordinary sutured 3-manifold defined as 
    \[
    (Y^L_{S_1},\Gamma^L_{S_1}) \circ (Y,\Gamma) \circ (Y^R_{S_0}, \Gamma^R_{S_0})=(Y^L_{S_1} \cup_{F_1} Y \cup_{F_1} Y^R_{S_0}, \Gamma^L_{S_1} \cup_{\Lambda_1} \Gamma \cup_{\Lambda_0} \Gamma^R_{S_0}).
    \]
    Let $R^+_{S_1,S_0}$ denote the $R^+$ subset of the boundary of $(Y,\Gamma_{S_1,S_0})$.
\end{definition}

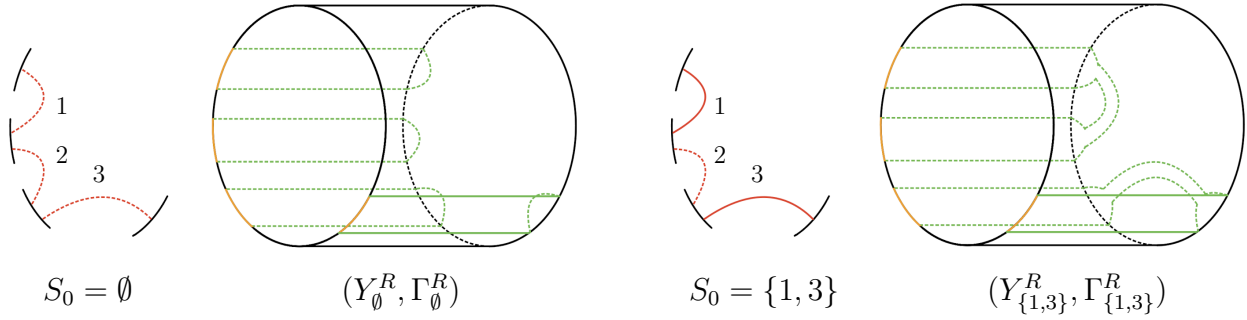
\begin{figure}
    \centering
    \vspace{1.2em}
    \begin{overpic}[width=\textwidth]{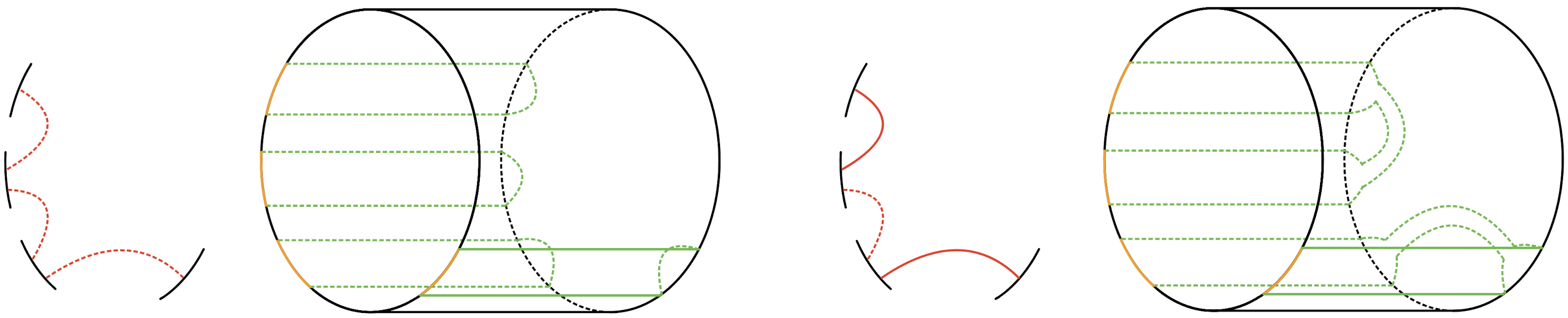}
        \put (4,11) {$\scalebox{0.8}{1}$}
        \put (4,7) {$\scalebox{0.8}{2}$}
        \put (7,5.5) {$\scalebox{0.8}{3}$}
        \put (3,-4) {$S_0=\emptyset$}
        \put (27,-4) {$(Y^R_{\emptyset},\Gamma^R_{\emptyset})$}
        \put (57,11) {$\scalebox{0.8}{1}$}
        \put (57,7) {$\scalebox{0.8}{2}$}
        \put (60,5.5) {$\scalebox{0.8}{3}$}
        \put (55,-4) {$S_0=\{1,3\}$}
        \put (79,-4) {$(Y^R_{\{1,3\}},\Gamma^R_{\{1,3\}})$}
    \end{overpic}
    \vspace{0.1in}
    \caption{Two right Zarev caps for the sutured surface $(F_0,\Lambda_0)$ parameterized by a fixed arc diagram $\Zc_0$ shown in the first and third columns, as in Definition~\ref{def:ZarevCaps}. The two caps correspond to different choices of the subset of arcs $S_0$. The cap on the left may also be viewed as the underlying manifold from which all nine right Zarev caps associated to this parameterization of the surface $(F_0,\Lambda_0)$ arise under the general construction of Definition~\ref{def:ZarevCaps}. As in Figure~\ref{fig:CobordismOrientations}, $R^+$ is left unshaded to avoid visual clutter.}
    \label{fig:ZarevCapsGeneral}
\end{figure}

There are standard Heegaard diagrams that can be used to represent Zarev caps. In our setting, it is most convenient to use $\beta$-bordered diagrams for Zarev caps; $\alpha$-bordered versions can be defined analogously.

\begin{definition}
    Let $S_1$ (resp. $S_0$) be a subset of the set of arcs $\mathbf{a}_1$ (resp. $\mathbf{a}_0$) of $\Zc_1$ (resp. $\Zc_0$). We define one-sided bordered sutured Heegaard diagrams $\Hc^L_{S_1}$ for left Zarev caps $(Y^L_{S_1},\Gamma^L_{S_1})$ (resp. $\Hc^R_{S_0}$ for right Zarev caps $(Y^R_{S_0},\Gamma^R_{S_0})$), $\beta$-bordered by $\Zc_1^*$ on the right (resp. $\Zc_0^*$ on the left), as follows: 
    \begin{itemize}
        \item Start with the half-identity diagram $\Hc^{\alpha \beta}_{1/2}$ for $\Zc_1$ (resp. $\Hc^{\beta \alpha}_{1/2}$ for $\Zc_0$), as in Example~\ref{ex:HalfIdentityDiags}, but with all $\alpha$ arcs deleted. We view $\Hc^{\alpha \beta}_{1/2}$ and $\Hc^{\beta \alpha}_{1/2}$ as having handle decompositions as in the second and fourth pictures of Figure~\ref{fig:Half identity diagram}, where the cores of the handles are the $\beta$ arcs of the diagram. In other words, start with a Heegaard surface $\Sigma \coloneq F(\Zc_1^*)$ (resp. $\Sigma \coloneq F(\Zc_0^*)$), which has one $2$-dimensional $1$-handle for each type-$(t=\beta)$ matching-arc in $\Zc_1^*$ (resp. $\Zc_0^*$), then declare the set $\boldsymbol{\beta}^a \subset \Hc^L_{S_1}$ (resp. $\boldsymbol{\beta}^a \subset \Hc^R_{S_0}$) of $\beta$-arcs to be the cores of these $1$-handles. The leftmost portion of the middle row of Figure~\ref{fig:ZarevCaps} depicts this initial diagram schematically. 
        \item Every type-$(t=\alpha)$ matching arc $a_i\in S_1$ (resp. $a_i\in S_0$) has a corresponding $2$-dimensional $1$-handle $h_i$ in $\Sigma$: the one whose core is $a_i^*$, the type-$(t=\beta)$ matching arc in $\mathbf{a}_1^*$ (resp. $\mathbf{a}_0^*$) corresponding to $a_i$. Now, for each $i$, attach an additional $2$-dimensional $1$-handle to $h_i \subset \Sigma$ in such a way that the attaching spheres of the new handle coincide with the endpoints of the co-core of $h_i$, as pictured on the rightmost portion of the middle row of Figure~\ref{fig:ZarevCaps}. 
        \item For each $i$, add an $\alpha$-circle running once along $h_i$, intersecting the corresponding $\beta$-arc exactly once.
    \end{itemize}
\end{definition}

It follows by construction that the diagrams $\Hc^L_{S_1}$ and $\Hc^R_{S_0}$ represent $(Y^L_{S_1},\Gamma^L_{S_1})$ and $(Y^R_{S_0},\Gamma^R_{S_0})$ respectively. Figure~\ref{fig:ZarevCaps} shows the diagram $\Hc^R_{S_0}$ representing a right Zarev cap.  

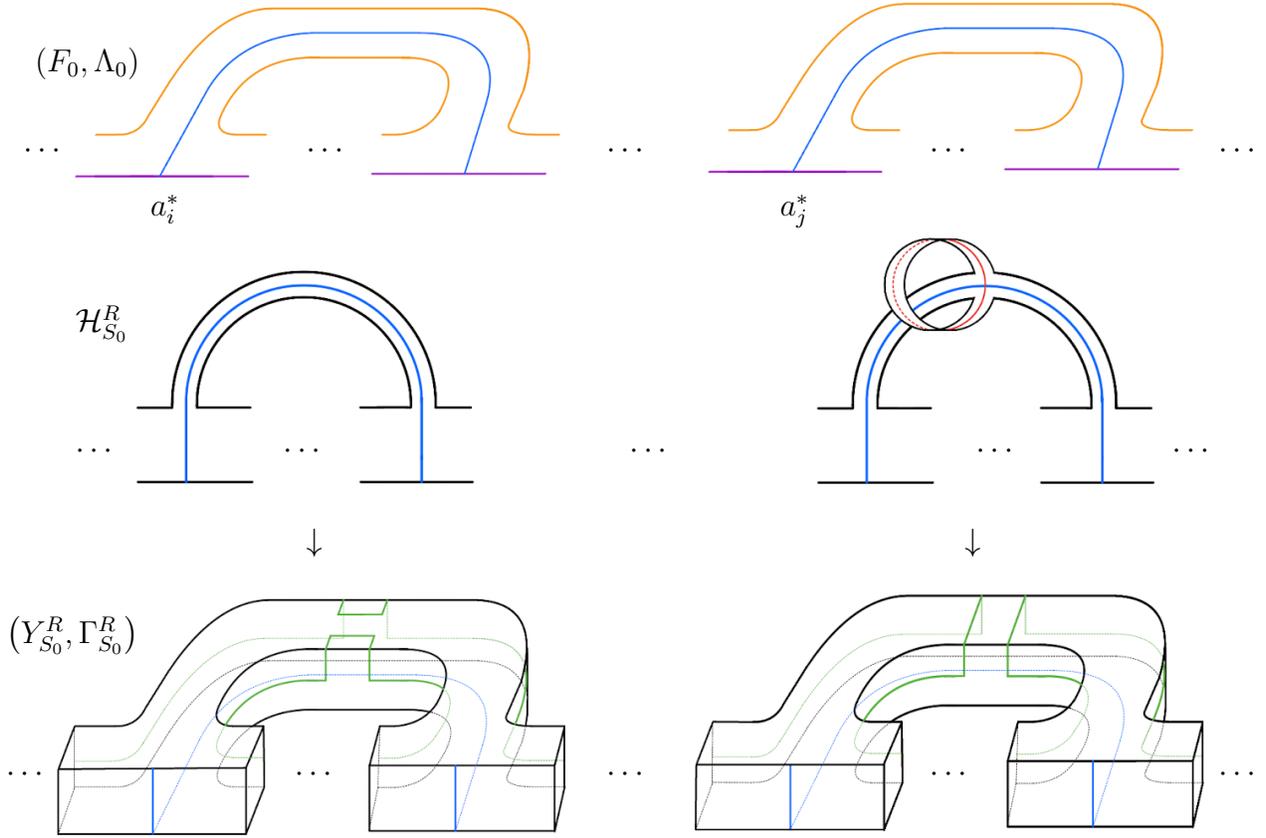
\begin{figure}
    \centering
    \vspace{1.2em}
    \begin{overpic}[width=0.93\textwidth]{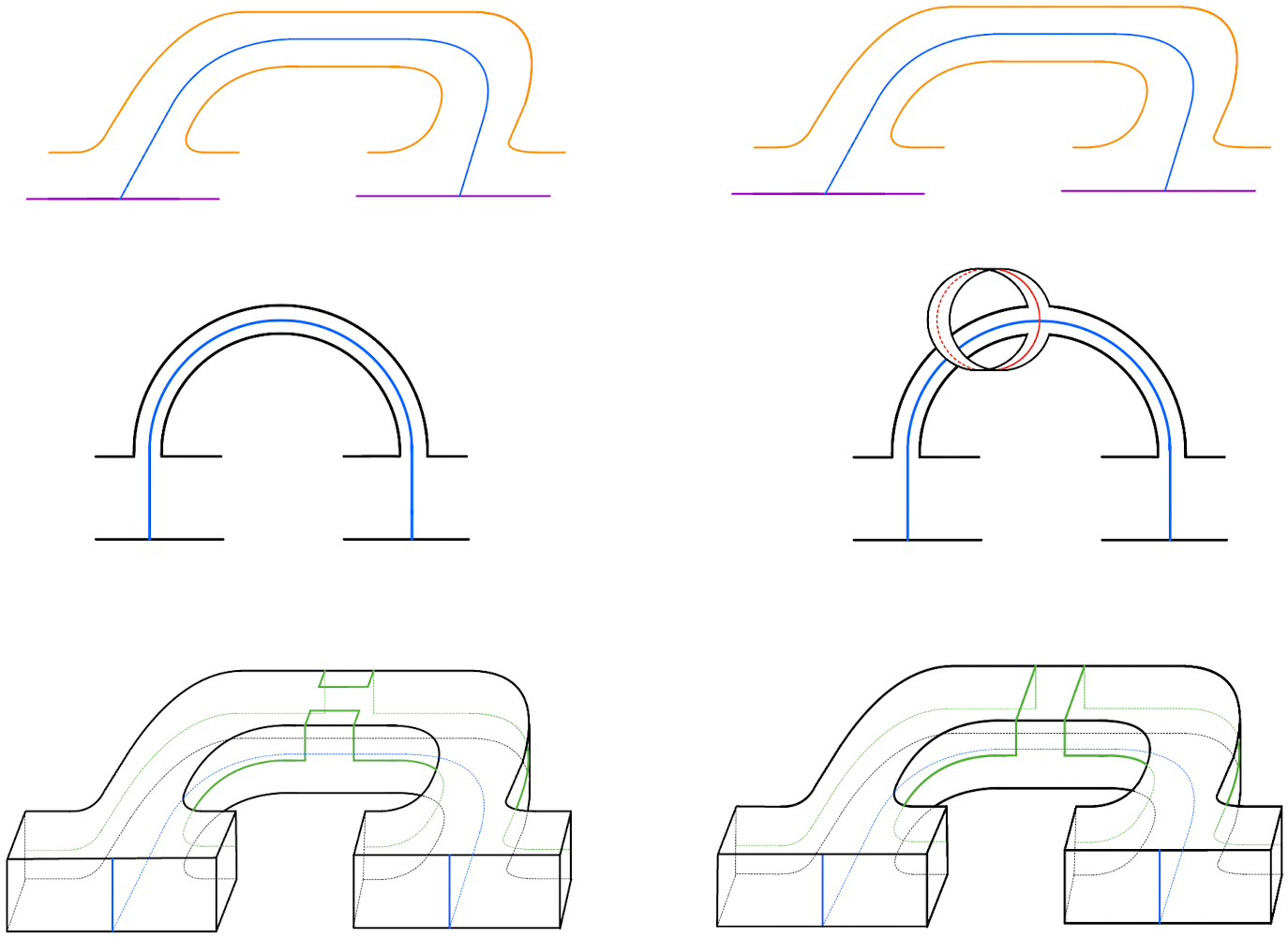}
        \put (-2.5,59) {$\cdots$}
        \put (22,59) {$\cdots$}
        \put (48,59) {$\cdots$}
        \put (76,59) {$\cdots$}
        \put (101,59) {$\cdots$}
        \put (8.5,54) {$a_i^*$}
        \put (63,54) {$a_j^*$}
        \put (-1.5,66.5) {$(F_0,\Lambda_0)$}
        \put (2,44) {$\Hc_{S_0}^{R}$}
        \put (-4,17) {$\left(Y^R_{S_0},\Gamma^R_{S_0}\right)$}
        \put (2,33) {$\cdots$}
        \put (20,33) {$\cdots$}  
        \put (50,33) {$\cdots$}
        \put (80,33) {$\cdots$}  
        \put (97,33) {$\cdots$}
        \put (22,25) {$\downarrow$}  
        \put (79,25) {$\downarrow$}
        \put (-4,5) {$\cdots$}
        \put (21,5) {$\cdots$}
        \put (48,5) {$\cdots$}
        \put (76,5) {$\cdots$}
        \put (101,5) {$\cdots$}
    \end{overpic}
    \vspace{0.1in}
    \caption{Top row: a sutured surface $(F_0,\Lambda_0)$ built from a type-$(t=\beta)$ arc diagram $\Zc_0^*$. Bottom row: the left Zarev cap $(Y^R_{S_0},\Gamma^R_{S_0})$ for $(F_0,\Lambda_0)$ with respect to $S_0$. Middle row: the Heegaard diagram $\Hc_{S_0}^{R}$ representing the Zarev cap in the bottom row. An additional $1$-handle and associated $\alpha$-circle as described in Definition~\ref{def:CappedHeegaardDiagram} are drawn on the far right, indicating that the matching arc $a_j$ belongs to the subset $S_0$; the absence of such a $1$-handle and $\alpha$-circle on the far left imply that $a_i \notin S_0$. Similar to Figure~\ref{fig:classification of surfaces}, $(F_0,\Lambda_0)$ and $\Hc_{S_0}^{R}$ have been rotated to conserve space. Again, we do not shade in $R^+$ to avoid visual clutter.}
    \label{fig:ZarevCaps}
\end{figure}

\begin{definition}\label{def:CappedHeegaardDiagram}
    Let $(Y,\Gamma)$ be a sutured cobordism from $(F_0,\Lambda_0)$ to $(F_1,\Lambda_1)$. By Corollary~\ref{cor:NormalizedHeegaardDiagrams}, there is an $\alpha$-$\alpha$ bordered sutured Heegaard diagram representing $(Y,\Gamma)$ of the form $\Hc_{\norm} = \Hc^{\alpha \beta}_{1/2} \cup_{\Zc_1^*} \Hc' \cup_{\Zc_0^*} \Hc^{\beta \alpha}_{1/2}$ where $\Hc'$ is a $\beta$-$\beta$ bordered sutured Heegaard diagram representing $(Y,\Gamma)$. Let $S_0$ (resp. $S_1$) be a subset of the set of arcs $\mathbf{a}_0$ (resp. $\mathbf{a}_1$) of $\Zc_0$ (resp. $\Zc_1$). Define the \emph{$(S_1,S_0)$–capped Heegaard diagram}
    \[
    \Hc_{S_1,S_0} := \Hc^L_{S_1} \cup_{\Zc_1^*} \Hc' \cup_{\Zc_0^*} \Hc^R_{S_0}.
    \]
\end{definition}

By construction, we have the following corollary. 
\begin{corollary}
    The diagram $\Hc_{S_1,S_0}$ of Definition~\ref{def:CappedHeegaardDiagram} is a sutured Heegaard diagram as in Definition~\ref{def:OtherTypesOfHeegaardDiagrams}(c) and represents the sutured manifold $(Y,\Gamma_{S_1,S_0})$ for any $S_0$ and $S_1$. 
\end{corollary}

\subsection{Idempotent-dependent CW complexes}\label{sec:IdempotentDepCWComplexes}

Given the setup of Definition~\ref{def:CappedHeegaardDiagram}, with subsets $S_i$ of the set of arcs $\mathbf{a}_i$ of $\Zc_i$, the constructions of Section~\ref{sec:CWDecomp} applied to the diagram $\Hc_{S_1,S_0}$ for $(Y,\Gamma_{S_1,S_0})$ give us a CW pair $(\mathfrak{Y}_{S_1,S_0}, \mathfrak{R}^+_{S_1,S_0})$, where $\mathfrak{Y}_{S_1,S_0}$ is homotopy equivalent to $Y$, depending on a choice of CW decomposition $\mathfrak{R}^+_{S_1,S_0}$ of $R^+_{S_1,S_0}$ (plus choices of orientations for $\alpha$ and $\beta$ circles of $\Hc_{S_1,S_0}$ that we will discuss below). Similarly, the diagram $\Hc_{\norm}$ for $(Y,\Gamma)$ (plus additional choices) gives a CW pair $(\mathfrak{Y},\mathfrak{R}^+)$ where $\mathfrak{Y}$ is homotopy equivalent to $Y$. We will make the following assumptions on the choices; these assumptions will let us relate $(\mathfrak{Y}_{S_1,S_0}, \mathfrak{R}^+_{S_1,S_0})$ and $(\mathfrak{Y},\mathfrak{R}^+)$.

First, note that $R^+_{S_1,S_0}$ can be viewed as $R^+$ with a strip (rectangle) glued on for each element of $S_1 \sqcup S_0$ as in the rightmost portion of Figure~\ref{fig:ZarevCapsGeneral} and the bottom right of Figure~\ref{fig:ZarevCaps}. Let $R^{+,\mathrm{strips}}_{S_1,S_0}$ denote the union of these strips, so that $R^+_{S_1,S_0} = R^+ \cup R^{+,\mathrm{strips}}_{S_1,S_0}$. The intersection $R^+ \cap R^{+,\mathrm{strips}}_{S_1,S_0}$ consists of two intervals for each element of $S_1 \sqcup S_0$.

\begin{figure}
    \centering
    \vspace{.5em}
    \begin{overpic}[width=0.85\textwidth]{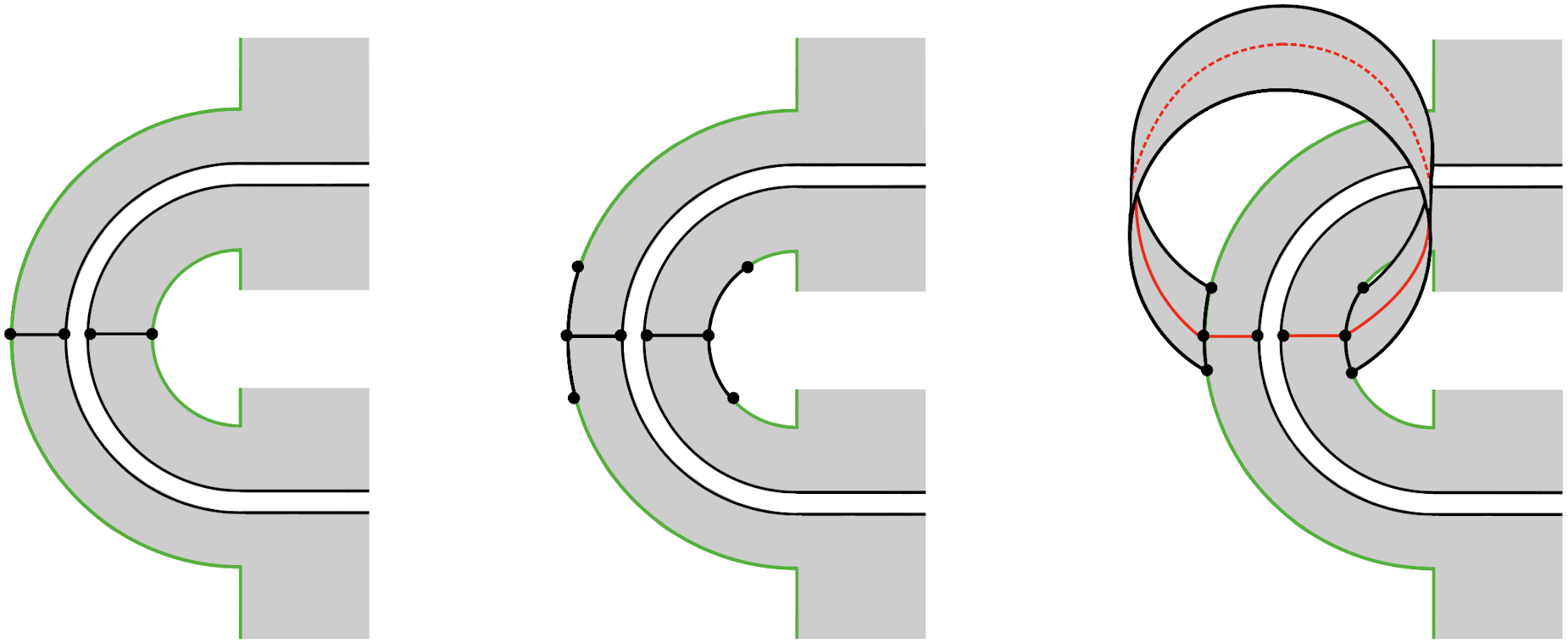}
        \put (19,18) {$\vdots$}
        \put (55,18) {$\vdots$}
        \put (95,18) {$\vdots$}
        \put (9,-6) {$\cup$ disks}
        \put (44,-6) {$\cup$ disks}
        \put (83,-6) {$\cup$ disks}
    \end{overpic}
    \vspace{2em}
    \caption{Left: CW decomposition $\mathfrak{R}^+$ of $R^+$. Middle: Adding extra cells to this CW decomposition. Right: adding $\mathfrak{R}^{+,\mathrm{strips}}_{S_1,S_0}$ to get $\mathfrak{R}^+_{S_1,S_0}$.}
    \label{fig:IdemDependentRPlus}
\end{figure}
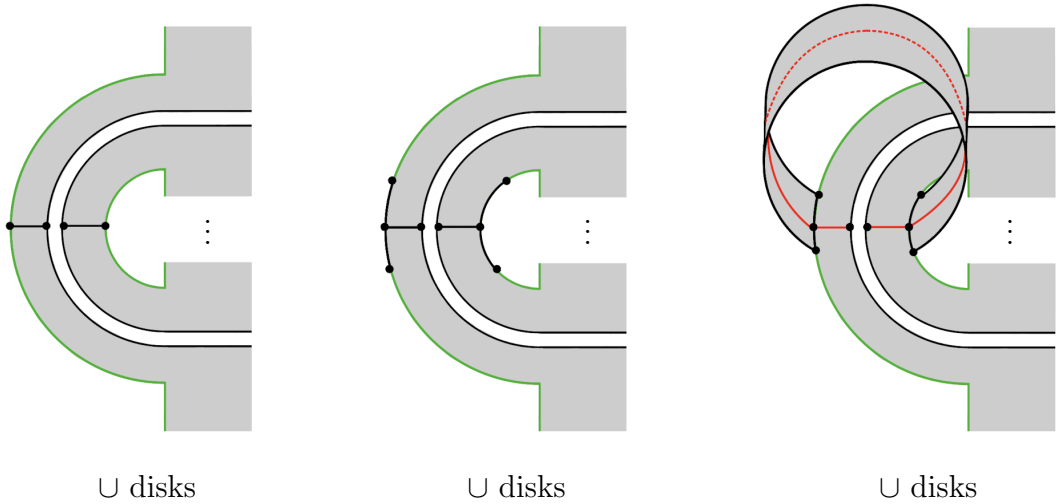

\begin{definition}
    When choosing the CW decomposition $\mathfrak{R}^+$ of $R^+$ (see the left side of Figure~\ref{fig:IdemDependentRPlus}), include additional 0-cells and 1-cells as in the middle picture of Figure~\ref{fig:IdemDependentRPlus}, so that we have a CW decomposition of $R^+ \cap R^{+,\mathrm{strips}}_{S_1,S_0}$. Extend $\mathfrak{R}^+ \cap R^{+,\mathrm{strips}}_{S_1,S_0}$ to a CW decomposition $\mathfrak{R}^{+,\mathrm{strips}}_{S_1,S_0}$ of $R^{+,\mathrm{strips}}_{S_1,S_0}$ by adding three new 1-cells and two new 2-cells for each strip as in the right picture of Figure~\ref{fig:IdemDependentRPlus}; we have a CW decomposition $\mathfrak{R}^+_{S_1,S_0} = \mathfrak{R}^+ \cup \mathfrak{R}^{+,\mathrm{strips}}_{S_1,S_0}$ of $R^+_{S_1,S_0}$. Use this particular CW decomposition in defining the CW pair $(\mathfrak{Y}_{S_1,S_0}, \mathfrak{R}^+_{S_1,S_0})$.
\end{definition}

We have inclusions of subcomplexes
\[
\mathfrak{Y} \subset \mathfrak{Y} \cup \mathfrak{R}^{+,\mathrm{strips}}_{S_1,S_0} \subset \mathfrak{Y}_{S_1,S_0};
\]
note that $\mathfrak{Y}_{S_1,S_0}$ is obtained from $\mathfrak{Y} \cup \mathfrak{R}^{+,\mathrm{strips}}_{S_1,S_0}$ by adding one 2-cell for each element of $S_1 \sqcup S_0$. Specifically, if $\Hc_{\norm}$ has $a$ $\alpha$-circles and $b$ $\beta$ circles, then $\mathfrak{Y}_{S_1,S_0} \setminus \mathfrak{R}^+_{S_1,S_0}$ has $b$ 1-cells corresponding to the $\beta$-circles of $\mathcal{H}_{S_1,S_0}$, $a$ 2-cells corresponding to the $\alpha$-circles of $\Hc'$, and $|S_0|+|S_1|$ 2-cells corresponding to the $\alpha$-circles of $\Hc^L_{S_1} \sqcup \Hc^R_{S_0}$. These last $|S_0|+|S_1|$ 2-cells are the cells of $\mathfrak{Y}_{S_1,S_0} \setminus (\mathfrak{Y} \cup \mathfrak{R}^{+,\mathrm{strips}}_{S_1,S_0})$. We fix orientation choices as follows.

\begin{definition}\label{def:OrientationsForHJcI}
    Choose orientations of the $\alpha$ and $\beta$ circles of $\Hc_{\norm}$ satisfying the assumptions of Proposition~\ref{prop:ExpandingBasisElts}; note that the statement of this proposition also gives us orientations of the $\alpha$ arcs of $\Hc_{\norm}$. The additional orientation choices needed for $\Hc_{S_1,S_0}$ are a choice of orientation for the $\alpha$ circles corresponding to elements of $S_1 \sqcup S_0$, each of which has a corresponding $\alpha$ arc in $\Hc_{\norm}$. Assume that the orientations of these $\alpha$ circles of $\Hc_{S_1,S_0}$ are chosen compatibly with the orientations of the corresponding $\alpha$ arcs of $\Hc_{\norm}$. Equivalently, orient these $\alpha$ circles so that on the left (outgoing) side of $\Hc_{S_1,S_0}$, the local intersections in the Zarev cap diagram $\Hc^L_{S_1}$ have $\beta \cdot \alpha = -1$, while on the right (incoming) side of $\Hc_{S_1,S_0}$, the local intersections in the Zarev cap diagram $\Hc^R_{S_0}$ have $\beta \cdot \alpha = +1$.
\end{definition}

While the above orientation choices will be useful later, the following corollary holds for arbitrary choices of orientations.

\begin{corollary}\label{cor:IdemDependentCellularCxes}
    The cellular chain complex
    \[
    C_*^{\cell}(\mathfrak{Y} \cup \mathfrak{R}^{+,\mathrm{strips}}_{S_1,S_0}, \mathfrak{R}^+_{S_1,S_0})
    \]
    computes $H_*(Y,R^+)$. The cellular chain complex
    \[
    C_*^{\cell}(\mathfrak{Y}_{S_1,S_0}, \mathfrak{R}^+_{S_1,S_0})
    \]
    computes $H_*(Y,R^+_{S_1,S_0})$.
\end{corollary}

\begin{proof}
    The second claim follows from the general setup of Section~\ref{sec:CWDecomp}. For the first claim, because $\mathfrak{R}^+_{S_1,S_0} = \mathfrak{R}^+ \cup \mathfrak{R}^{+,\mathrm{strips}}_{S_1,S_0}$ and $\mathfrak{Y} \cap \mathfrak{R}^{+,\mathrm{strips}}_{S_1,S_0} = \mathfrak{R}^+$, the claim follows from excision for CW pairs plus the fact from Section~\ref{sec:CWDecomp} that $C_*^{\cell}(\mathfrak{Y},\mathfrak{R}^+)$ computes $H_*(Y,R^+)$.
\end{proof}

\subsubsection{Canonical ambient CW complex}\label{sec:AmbientCWComplex}
Later it will be useful to have one CW complex that has all the various $\mathfrak{Y}_{J^c,I}$ as subcomplexes. Consider $I=\mathbf{a}_0 \subset \mathbf{a}_0$ and $J=\emptyset \subset \mathbf{a}_1$. These subsets maximize $|J^c|+|I|$, which equals the number of new $\alpha$-circles added to $\Hc^{\alpha \beta}_{1/2} \sqcup \Hc^{\beta \alpha}_{1/2}$ in the process of capping off. Thus, the associated CW complex $\mathfrak{Y}_{J^c,I}=\mathfrak{Y}_{\mathbf{a}_1,\mathbf{a}_0}$ contains the maximum possible number of extra $2$-cells, so that for any other $I,J$, we see that
\[
\mathfrak{Y}_{J^c,I} \subset \mathfrak{Y}_{\mathbf{a}_1,\mathbf{a}_0}.
\]
One could take this as an alternative definition of $\mathfrak{Y}_{J^c,I}$: start with the ambient CW complex $\mathfrak{Y}_{\mathbf{a}_1,\mathbf{a}_0}$ with $(n_1+n_0)$-many extra $2$-cells, and obtain $\mathfrak{Y}_{J^c,I}$ by omitting any $2$-cells corresponding to $\alpha$-circles not in $\Hc_{J^c,I}$. Likewise, we have inclusions
\[
\mathfrak{Y} \subset \mathfrak{Y}_{J^c,I} \subset \mathfrak{Y}_{\mathbf{a}_1,\mathbf{a}_0}
\]
of CW complexes.

\subsection{Matrix entries}\label{sec:MatrixEntries}
    
Let $(Y,\Gamma)$ be a sutured cobordism from $(F_0,\Lambda_0)$ to $(F_1,\Lambda_1)$. Fix a set of choices $\Xi_{\norm}$ as in Corollary~\ref{cor:CanChooseXiNormalized}, so that the Heegaard diagram for $(Y,\Gamma)$ is $\Hc_{\norm}=\Hc^{\alpha \beta}_{1/2} \cup_{\Zc_1^*} \Hc' \cup_{\Zc_0^*} \Hc^{\beta \alpha}_{1/2}$ where $\Hc'$ is a $\beta$-$\beta$ bordered diagram for $(Y,\Gamma)$. For basis elements $\gamma^{\inrm}_I$ of $\wedge^* H_1(F_0,S^+_0)$ and $\gamma^{\out}_J$ of $\wedge^* H_1(F_1,S^+_1)$, view $I$ and $J$ as sets of arcs of $\Zc_0$ and $\Zc_1$ respectively; let $J^c$ be the complement of $J$. The ordinary sutured Heegaard diagram $ \Hc_{J^c,I} \coloneq \Hc^L_{J^c} \cup_{\Zc_1^*} \Hc' \cup_{\Zc_0^*} \Hc^R_I$ represents $(Y,\Gamma_{J^c,I})$.

We note that the following lemma holds; the proof is evident.

\begin{lemma}\label{lem:GeneratorBijectionH-HLH'HR}
    The $\alpha$- and $\beta$-circles of $\Hc_{J^c,I}$ are in natural bijection with a subset of the $\alpha$- and $\beta$-curves of $\Hc_{\norm}$, namely those that are occupied in any generator $\x \in \mathfrak{S}(\Hc_{\norm})$ with $\overline{o}_L(\x) = J$ and $o_R(\x) = I$. 
\end{lemma}

\begin{definition}\label{def:Xi_J^c,I}
    Given a set of choices $\Xi_{\norm}$ for $(Y,\Gamma)$, we can obtain a set of choices $\Xi_{J^c,I}$ for the purely sutured manifold $(Y,\Gamma_{J^c,I})$ as follows. 
    \begin{itemize}
        \item Use the Heegaard diagram $\Hc_{J^c,I}$ to represent $(Y,\Gamma_{J^c,I})$, as in Definition~\ref{def:CappedHeegaardDiagram}. 
        \item Let the choices of orientations of the $\alpha$- and $\beta$-circles of $\Hc_{J^c,I}$ be those chosen in Definition~\ref{def:OrientationsForHJcI} given the orientation choices included in $\Xi_{\norm}$.
        \item Let the $\beta$-circles of $\Hc_{J^c,I}$ have the same ordering as they have in $\Xi_{\norm}$. Order the $\alpha$-circles of $\Hc_{J^c,I}$ by putting the new $\alpha$ circles on the outgoing side first (ordered according to $J^c$), then the $\alpha$ circles of $\Hc_{\norm}$, then the new $\alpha$ circles on the incoming side (ordered according to $I$). 
    \end{itemize}
\end{definition}

\begin{corollary}\label{cor:FirstMatrixEntryComp}
    Let $E \in \Z$ denote the matrix entry of $[\BSDA(Y,\Gamma;\Xi_{\norm})]^{\Z}_{\comb}$ in column $\gamma^{\inrm}_I$ and row $\gamma^{\out}_J$. We have
    \[
    E = (-1)^{\mathrm{sgn}(\sigma_{JJ^c \leftrightarrow \std}) + ak + n_1 k} [\BSDA(Y,\Gamma_{J^c,I};\Xi_{J^c,I})]^{\Z}_{\comb}
    \]
    where $k = |I|$, $a$ is the number of $\alpha$-circles of $\Hc_{\mathrm{norm}}$, $n_1$ is the number of arcs of $\Zc_1$, and $\sigma_{JJ^c \leftrightarrow \std}$ is the permutation of $\{1,\ldots,n_1\}$ sending the elements of $J$ to $\{1,\ldots,|J|\}$ in order and sending the elements of $J^c$ to $\{|J|+1,\ldots,n_1\}$ in order.
\end{corollary}

\begin{proof}
    By Lemma~\ref{lem:GeneratorBijectionH-HLH'HR} and the discussion of Example~\ref{ex:BSDAForOrdinary}, we can write
    \[
    [\BSDA(Y,\Gamma_{J^c,I};\Xi_{J^c,I})]^{\Z}_{\comb} = \sum_{\substack{\x \in \mathfrak{S}(\Hc): \\ \overline{o}_L(\x) = J, \\ o_R(\x) = I}} (-1)^{\left(\sum_{x \in \x} \beta_x \cdot \alpha_x \right) + \mathrm{sgn}(\sigma_{\x})};
    \]
    the terms $ak$ and $n_1 k$ in equation~\eqref{eq:DAGrading} do not appear in the exponent of the above equation because the relevant value of ``$k$'' would be zero. Plugging this formula into the statement of the corollary, we get the definition of $E$.
\end{proof}

\subsection{Balanced sutured manifolds} 

Now we relate the above discussion to sutured Floer homology, which is defined only for balanced sutured manifolds. We start by reviewing this condition.

\begin{definition}
    Let $(Y,\Gamma, R^+, R^-)$ be an ordinary (non-bordered) sutured 3-manifold. We say $(Y,\Gamma)$ is \emph{weakly balanced} if $\chi(R^+) = \chi(R^-)$ in each component of $Y$. We say $(Y,\Gamma)$ is \emph{balanced} if $(Y,\Gamma)$ is weakly balanced, has no closed components, and each component of $\partial Y$ intersects both $R^+$ and $R^-$ nontrivially (equivalently, intersects $\Gamma$ nontrivially).
\end{definition}

\begin{proposition}\label{prop:SameNumberAlphaBeta}
    For an ordinary (non-bordered) sutured Heegaard diagram $\Hc$ as in Definition~\ref{def:OtherTypesOfHeegaardDiagrams}(c), the corresponding sutured 3-manifold $(Y,\Gamma)$ is weakly balanced if and only if, on each component of the Heegaard surface, the number of $\alpha$-circles equals the number of $\beta$-circles. If $(Y,\Gamma)$ is weakly balanced, then it is balanced if and only if the Heegaard surface $\Sigma$ has no closed components and the $\alpha$ and $\beta$-circles of $\Hc$ each form linearly independent subsets of $H_1(\Sigma)$.
\end{proposition}

\begin{proof}
    For the first statement, we may assume without loss of generality that $\Sigma$ is connected, in which case the result follows from the first part of the proof of \cite[Proposition 2.9]{Juhasz}. In more detail, if we let $(Y_0 = \Sigma \times [0,1],\Gamma, R^+_0,R^-_0)$ be the product sutured manifold from which $(Y,\Gamma,R^+,R^-)$ is built by handle addition along the $\alpha$ and $\beta$-circles, then $R^+_0$ and $R^-_0$ are each homeomorphic to $\Sigma$, so $\chi(R^+_0) = \chi(R^-_0)$. Each $\alpha$ handle addition gives a surgery on $R^+_0$ that increases its Euler characteristic by $2$, and each $\beta$ handle addition gives a surgery on $R^-_0$ that increases its Euler characteristic by $2$; the first statement follows.

    For the second statement, we may also assume that $\Sigma$ is connected. In this case the result follows from the statement of \cite[Proposition 2.9]{Juhasz}; note that the $\alpha$- and $\beta$-circles of $\Hc$ form linearly independent subsets of $H_1(\Sigma)$ if and only if they form linearly independent subsets of $H_1(\Sigma;\Q)$.
\end{proof}

We can view an ordinary sutured 3-manifold $(Y,\Gamma)$ as a sutured cobordism from $(\emptyset, \emptyset)$ to $(\emptyset,\emptyset)$. Thus, as in Example~\ref{ex:BSDAForOrdinary}, given choices $\Xi$ we have a $\Z$-linear map $[\BSDA(Y,\Gamma;\Xi)]^{\Z}_{\comb}$ from $\wedge^0 H_1(\emptyset,\emptyset) \cong \Z$ to itself, i.e. an integer
\[
[\BSDA(Y,\Gamma;\Xi)]^{\Z}_{\comb} \in \Z.
\]

\begin{proposition}\label{prop:BSDAZeroUnlessBalanced}
    For an ordinary sutured 3-manifold $(Y,\Gamma)$, we have 
    \[
    [\BSDA(Y,\Gamma;\Xi)]^{\Z}_{\comb} = 0
    \]
    unless $(Y,\Gamma)$ is weakly balanced.
\end{proposition}

\begin{proof}
    Let $\Hc$ be a Heegaard diagram for $(Y,\Gamma)$. If $(Y,\Gamma)$ is not weakly balanced, Proposition~\ref{prop:SameNumberAlphaBeta} implies that the set of generators $\mathfrak{S}(\Hc)$ is empty, so $[\BSDA(Y,\Gamma;\Xi)]^{\Z}_{\comb} = 0$.
\end{proof}

\begin{remark}
    Suppose $(Y,\Gamma)$ is weakly balanced but not balanced, and let $\Hc$ be a sutured Heegaard diagram (with Heegaard surface $\Sigma$) representing $(Y,\Gamma)$. By Proposition~\ref{prop:SameNumberAlphaBeta}, either $\Sigma$ has some closed components or the set of $\alpha$ circles or $\beta$ circles is not linearly independent in $H_1(\Sigma)$. In the second case it will follow that $[\BSDA(Y,\Gamma;\Xi)]^{\Z}_{\comb} = 0$. 
    
    However, in the first case where $\Sigma$ has closed components, $[\BSDA(Y,\Gamma;\Xi)]^{\Z}_{\comb}$ may be nonzero. For example, take $\Sigma = \Sigma_g$ to be a closed orientable surface of genus $g$, and take $\Hc$ to have no $\alpha$ or $\beta$ circles. Then $Y = \Sigma \times [0,1]$ with no sutures, $R^+ = \Sigma \times \{0\}$, and $R^- = \Sigma \times \{1\}$. No data other than $\Sigma$ is needed for the choices $\Xi$, and $\mathfrak{S}(\Hc)$ has a unique element (the empty generator $\x$). It follows that $[\BSDA(Y,\Gamma;\Xi)]^{\Z}_{\comb} = 1 \in \Z$.
\end{remark}

Now let $(Y,\Gamma)$ be a sutured cobordism from $(F_0,\Lambda_0)$ to $(F_1,\Lambda_1)$. For arc diagrams $\Zc_i$ representing $(F_i,\Lambda_i)$ and basis elements $\gamma^{\inrm}_I$ for $\wedge^* H_1(F_0,S^+_0)$ and $\gamma^{\out}_J$ for $\wedge^* H_1(F_1,S^+_1)$, we consider the ordinary sutured manifold $(Y,\Gamma_{J^c,I})$.

\begin{corollary}\label{cor:J^c,IWeaklyBalChiEqualsZero}
    If $Y$ is connected, we have $\chi(Y,R^+_{J^c,I})=0$ if and only if $(Y,\Gamma_{J^c,I})$ is weakly balanced. If $Y$ is disconnected, then $(Y,\Gamma_{J^c,I})$ is weakly balanced if and only if this equation holds for each connected component of $Y$.
\end{corollary}

\begin{proof}
    Without loss of generality assume that $Y$ is connected, and take a Heegaard diagram $\Hc_{J^c,I}$ as in Section~\ref{sec:MatrixEntries} to represent $(Y,\Gamma_{J^c,I})$. In Proposition~\ref{prop:DegreeOfBSDA} we argued that $\chi(Y,R^+_{J^c,I})$ equals the difference between the number of $\alpha$-circles and $\beta$-circles of $\Hc_{J^c,I}$, but by Proposition~\ref{prop:SameNumberAlphaBeta}, $(Y,\Gamma_{J^c,I})$ is weakly balanced if and only if these are equal, so it is weakly balanced if and only if $\chi(Y,R^+_{J^c,I})$ vanishes.
\end{proof}

Another way to formulate the weakly balanced condition for $(Y,\Gamma_{J^c,I})$ is as follows.

\begin{corollary}\label{cor:YSuturedWeaklyBalancedCond}
    If $(Y,\Gamma_{J^c,I})$ is connected, then $(Y,\Gamma_{J^c,I})$ is weakly balanced if and only if $|J| = |I| + c$ where $c := n_1 + \chi(Y,R^+)$. If $(Y,\Gamma_{J^c,I})$ is disconnected, then $(Y,\Gamma_{J^c,I})$ is weakly balanced if and only if the above equation holds for each connected component of $Y$.
\end{corollary} 
 
\begin{proof}
    We assume without loss of generality that $(Y,\Gamma_{J^c,I})$ is connected. Let $a'$ and $b'$ be the number of $\alpha$ and $\beta$-circles of $\Hc'$. There are $|J^c| = n_1 - |J|$ $\alpha$-circles in $\Hc^L_{J^c}$ and $|I|$ $\alpha$-circles in $\Hc^R_I$. There are also $n_0 + n_1$ $\beta$-circles of $\Hc_{J^c,I}$ that are glued from $\beta$-arcs.
        
    Thus, by Proposition~\ref{prop:SameNumberAlphaBeta}, $(Y,\Gamma_{J^c,I})$ is weakly balanced if and only if
    \[
    a' + n_1 - |J| + |I| = b' + n_0 + n_1,
    \]
    i.e.
    \begin{align*}
    |J| &= |I| + a' - b' - n_0 \\
        &= |I| + n_1 + (a' -b' - n_0 - n_1).
    \end{align*}
    We claim that $\chi(Y,R^+) = a' - b' - n_0 - n_1$. Indeed, $\chi(Y,R^+)$ can be computed as the number of 2-cells of $\mathfrak{Y} \setminus \mathfrak{R}^+$ minus the number of 1-cells of $\mathfrak{Y} \setminus \mathfrak{R}^+$. The number of 1-cells is $b' + n_0 + n_1$ and the number of 2-cells is $a'$, proving the claim.
\end{proof}

\subsection{Relating \texorpdfstring{$[\BSDA]$}{BSDA} to \texorpdfstring{$\chi(\SFH)$}{chi(SFH)}}\label{sec:RelWithSFH}

\subsubsection{Ordinary sutured 3-manifolds}\label{sec:RelWithSFHOrdinarySutured}

Suppose $(Y,\Gamma)$ is an ordinary (non-bordered) sutured 3-manifold that is weakly balanced; then by Juhasz \cite{Juhasz}, when $(Y,\Gamma)$ is balanced there is a sutured Floer homology group $\SFH(Y,\Gamma)$. We will use the conventions of Friedl--Juhasz--Rasmussen \cite{FJR} when discussing sutured Floer homology. For reasons discussed in Remark~\ref{rem:ChangeOfConventions}, when we make choices $\Xi$ (including a Heegaard diagram $\Hc$), we will be most interested in the relationship between the integer $[\BSDA(Y,\Gamma;\Xi)]^{\Z}_{\comb}$ and (in the balanced case) $\SFH$ of the sutured 3-manifold represented in the conventions of \cite{FJR} by the diagram $\Hc_{\alpha \leftrightarrow \beta}$ in which the roles of the $\alpha$ and $\beta$ curves of $\Hc$ have been reversed. Comparing Definition~\ref{def:AssocBorderedSuturedCob} with the second paragraph below Definition 2.4 of \cite{FJR}, this sutured 3-manifold is $(Y_{R^+ \leftrightarrow R^-},\Gamma_{R^+ \leftrightarrow R^-})$ where $Y_{R^+ \leftrightarrow R^-}$ is $Y$ with the same orientation (i.e. $Y_{R^+ \leftrightarrow R^-} = Y$) and $\Gamma_{R^+ \leftrightarrow R^-}$ is such that the roles of $R^+$ and $R^-$ are the reverse in $(Y,\Gamma_{R^+ \leftrightarrow R^-})$ of what they are in $(Y,\Gamma)$. 

We order and orient the $\alpha$ and $\beta$ curves of $\Hc_{\alpha \leftrightarrow \beta}$ as the $\beta$ and $\alpha$ curves of $\Hc$ respectively. By \cite[Section 2.4]{FJR} we get an orientation $\omega_{\Xi}$ of the real vector space $H_*(Y, R^-_{R^+ \leftrightarrow R^-};\R)$, or in other words of $H_*(Y,R^+;\R)$. In the special case where $\Hc$ has no $\alpha$ or $\beta$ curves, so $H_*(Y,R^+;\R)$ is the zero vector space, we will take $\omega_{\Xi}$ to be the positive orientation of the zero vector space.

By \cite[Definition 2.9]{FJR}, if $(Y,\Gamma_{R^+ \leftrightarrow R^-})$ is balanced (this holds if and only if $(Y,\Gamma)$ is balanced), there is a finitely generated abelian group $\SFH(Y, \Gamma_{R^+ \leftrightarrow R^-}, \omega_{\Xi})$, computed using the diagram $\Hc_{\alpha \leftrightarrow \beta}$, with a $\Z_2$ grading that is defined on the nose (not just up to overall shift).

\begin{remark}\label{rem:ChiSFHWhenNotBalancedNoSpinc}
    Below we will refer to $\chi(\SFH(Y,\Gamma_{R^+ \leftrightarrow R^-},\omega_{\Xi}))$ even when $(Y,\Gamma)$ is weakly balanced but not balanced. In this case there is no sutured Floer group $\SFH(Y,\Gamma_{R^+ \leftrightarrow R^-},\omega_{\Xi})$; instead we will interpret $\chi(\SFH(Y,\Gamma_{R^+ \leftrightarrow R^-},\omega_{\Xi}))$ as the sum over generators $\x$ of $\Hc_{\alpha \leftrightarrow \beta}$ of the quantity $(-1)^{b_1(Y,R^-_{R^+ \leftrightarrow R^-}) }m(\x)$ where $m(\x)$ is the sign specified by the formula in \cite[Lemma 2.8]{FJR}. Equivalently, we interpret $\chi(\SFH(Y,\Gamma_{R^+ \leftrightarrow R^-},\omega_{\Xi}))$ as the sum over generators $\x$ of $\Hc$ of the quantity $(-1)^{b_1(Y,R^+) + \sum_{x \in \x} i(x) + \mathrm{inv}(\sigma_{\x})}$; note that our intersection signs $\beta_x \cdot \alpha_x$ become ``$\alpha \cdot \beta$'' intersection signs in $\Hc_{\alpha \leftrightarrow \beta}$.
\end{remark}

\begin{proposition}\label{prop:BSDAandSFH}
    Given the above setup with $(Y,\Gamma)$ weakly balanced, we have an equality of integers
    \[
    [\BSDA(Y,\Gamma;\Xi)]^{\Z}_{\comb} = (-1)^{b_1(Y,R^+)}\chi(\SFH(Y_{R^+ \leftrightarrow R^-}, \Gamma_{R^+ \leftrightarrow R^-}, \omega_{\Xi}).
    \]
\end{proposition}

\begin{proof}
    This follows by comparing \cite[Lemma 2.8]{FJR} with our Definition~\ref{def:BSDAZComb} (in the balanced case) or Remark~\ref{rem:ChiSFHWhenNotBalancedNoSpinc} (in the weakly balanced but not balanced case). 
\end{proof}

\subsubsection{Sutured cobordisms}

Now let $(Y,\Gamma)$ be a sutured cobordism from $(F_0,\Lambda_0)$ to $(F_1,\Lambda_1)$, and fix a set of choices $\Xi_{\norm}$ as in Corollary~\ref{cor:CanChooseXiNormalized}. Let $\gamma^{\inrm}_I$ and $\gamma^{\out}_J$ be basis elements of $\wedge^* H_1(F_0,S^+_0)$ and $\wedge^* H_1(F_1,S^+_1)$ respectively. 

As in Corollary~\ref{cor:FirstMatrixEntryComp}, let $E$ be the matrix entry of the map $[\BSDA(Y,\Gamma;\Xi_{\norm})]^{\Z}_{\comb}$ in column $\gamma^{\inrm}_I$ and row $\gamma^{\out}_J$.

\begin{corollary}\label{cor:MatrixEntryAndSFHOverZ}
    If $(Y,\Gamma_{J^c,I})$ is not weakly balanced, then $E = 0$. Otherwise, we have
    \[
    E = (-1)^{b_1(Y,R^+_{J^c,I}) + \mathrm{inv}(\sigma_{JJ^c \leftrightarrow \std}) + ak + n_1 k} \chi(\SFH(Y, \Gamma_{J^c,I;R^+ \leftrightarrow R^-}, \omega_{\Xi_{J^c,I}}))
    \]
    where $(Y,\Gamma_{J^c,I;R^+ \leftrightarrow R^-})$ is obtained from $(Y,\Gamma_{J^c,I})$ by reversing the roles of $R^+$ and $R^-$ while keeping the orientation of $Y$ unchanged.
\end{corollary}

\begin{proof}
    This follows from Corollary~\ref{cor:FirstMatrixEntryComp}, Proposition~\ref{prop:BSDAZeroUnlessBalanced}, and Proposition~\ref{prop:BSDAandSFH}.
\end{proof}

We can also interpret $\A^{\Z}_{Y,\Gamma}(i_* \gamma^{\inrm}_I \wedge i_* \gamma^{\out}_{J^c})$ directly as a sign times the Euler characteristic of a sutured Floer homology group. Assume that $|J| = |I| + n_1 + \chi(Y,R^+)$, so that $|I| + |J^c| = -\chi(Y,R^+)$ and $\A^{\Z}_{Y,\Gamma}(i_* \gamma^{\inrm}_I \wedge i_* \gamma^{\out}_{J^c})$ makes sense.

\begin{corollary}\label{cor:SFHInterpOfAlexFunctor}
    If $(Y,\Gamma_{J^c,I})$ is not weakly balanced, then $\A^{\Z}_{Y,\Gamma}(i_* \gamma^{\inrm}_I \wedge i_* \gamma^{\out}_{J^c}) = 0$. Otherwise, we have
    \[
    \A^{\Z}_{Y,\Gamma}(i_* \gamma^{\inrm}_I \wedge i_* \gamma^{\out}_{J^c}) = (-1)^{b_1(Y,R^+_{J^c,I}) + ak + n_1k + ck} \chi(\SFH(Y, \Gamma_{J^c,I;R^+ \leftrightarrow R^-}, \omega_{\Xi_{J^c,I}}))
    \]
    up to the global (independent of $I$, $J$, and $k$) sign ambiguity in the definition of $\A^{\Z}_{Y,\Gamma}$.
\end{corollary}

\begin{proof}
    If $E$ is defined as above, then by the proof of Theorem~\ref{thm:BSDAAlexanderZ} we have
    \[
    \A^{\Z}_{Y,\Gamma}(i_* \gamma^{\inrm}_I \wedge i_* \gamma^{\out}_{J^c}) = (-1)^{\mathrm{inv}(\sigma_{JJ^c \leftrightarrow \std}) + ck} E
    \]
    Thus, by Corollary~\ref{cor:MatrixEntryAndSFHOverZ}, if $(Y,\Gamma_{J^c,I})$ is not weakly balanced then $\A^{\Z}_{Y,\Gamma}(i_* \gamma^{\inrm}_I \wedge i_* \gamma^{\out}_{J^c}) = 0$. If $(Y,\Gamma_{J^c,I})$ is weakly balanced, then 
    \begin{align*}
        \A^{\Z}_{Y,\Gamma}(i_* \gamma^{\inrm}_I \wedge i_* \gamma^{\out}_{J^c}) &= (-1)^{\mathrm{inv}(\sigma_{JJ^c \leftrightarrow \std}) + ck} E\\
        &= (-1)^{b_1(Y,R^+_{J^c,I}) + ak + n_1k + ck} \chi(\SFH(Y, \Gamma_{J^c,I;R^+ \leftrightarrow R^-}, \omega_{\Xi_{J^c,I}}));
    \end{align*}
    the second equality follows from Corollary~\ref{cor:MatrixEntryAndSFHOverZ}. 
         
    We also note that we can prove this corollary directly; as in the proof of Theorem~\ref{thm:BSDAAlexanderZ}, the definition of $\A^{\Z}_{Y,\Gamma}(i_* \gamma^{\inrm}_I \wedge i_* \gamma^{\out}_{J^c})$ gives us
    \[
    \A^{\Z}_{Y,\Gamma}(i_* \gamma^{\inrm}_I \wedge i_* \gamma^{\out}_{J^c})
    =\det \kbordermatrix{
    & \alpha^c & \alpha^{\inrm}_I & \alpha^{\out}_{J^c} \\
    \beta^{\out}  & * & 0 & -*_{J^c}\\
    \beta^c & * & 0 & 0\\
    \beta^{\inrm} & * & -*_I & 0}
    \]
    whereas computing $\chi(\SFH(Y, \Gamma_{J^c,I;R^+ \leftrightarrow R^-}, \omega_{\Xi_{J^c,I}}))$ directly from the definitions (Remark~\ref{rem:ChiSFHWhenNotBalancedNoSpinc} in the weakly balanced but not balanced case) gives
    \[
    \chi(\SFH(Y, \Gamma_{J^c,I;R^+ \leftrightarrow R^-}, \omega_{\Xi_{J^c,I}})) = (-1)^{b_1(Y,R^+_{J^c,I})} \det \kbordermatrix{
    & \alpha^{\out}_{J^c} & \alpha^c & \alpha^{\inrm}_I  \\
    \beta^{\out} & -*_{J^c}  & * & 0 \\
    \beta^c & 0 & * & 0\\
    \beta^{\inrm} & 0 & * & *_I} 
    \]
    The first determinant is $(-1)^{k + (n_1 - k - c)(a + k)}$ times the second determinant, and up to global terms that are independent of $k$, the sign is equal to $(-1)^{ak + n_1k + ck}$.
\end{proof}

\section{\texorpdfstring{$\Spinc$}{Spin-c} structures}\label{sec:SpincStructures}

\subsection{Relative \texorpdfstring{$\Spinc$}{Spin-c} structures for generators}

Let $(Y,\Gamma)$ be a sutured cobordism from $(F_0,\Lambda_0)$ to $(F_1,\Lambda_1)$ and fix the set of choices $\Xi_{\norm}$ as in Corollary~\ref{cor:CanChooseXiNormalized}, so that the Heegaard diagram representing $(Y,\Gamma)$ is $\Hc_{\norm}=\Hc^{\alpha \beta}_{1/2} \cup_{\Zc_1^*} \Hc' \cup_{\Zc_0^*} \Hc^{\beta \alpha}_{1/2}$ where $\Hc'$ is a $\beta$-$\beta$ bordered diagram for $(Y,\Gamma)$. On the boundary of the purely sutured manifold $(Y,\Gamma_{J^c,I})$, there is a canonical vector field $v_{J^c,I}$ (a section of $TY$ restricted to the boundary), defined up to a contractible space of choices as discussed in \cite[Notation 4.1]{Juhasz}. This vector field points out of $Y$ along $R^+$ and into $Y$ along $R^-$. Along the sutures $\Gamma$, it points tangentially to $Y$, directed from the $R^-$ region to the $R^+$ region.

By \cite[Definition 4.2]{Juhasz}, there is a set $\Spinc(Y, \partial Y, v_{J^c,I})$ of $\Spinc$ structures on $Y$ relative to the vector field $v_{J^c,I}$ on the boundary. By \cite[Lemma 2.3]{FJR}, this set is nonempty if and only if $(Y,\Gamma_{J^c,I})$ is weakly balanced, in which case it is an affine space for $H^2(Y, \partial Y) \cong H_1(Y)$. 

By \cite[Definition 4.5]{Juhasz} (see also \cite[Definition 2.2]{FJR}), each generator of the diagram $\Hc_{J^c,I}=\Hc^L_{J^c} \cup_{\Zc_1^*} \Hc' \cup_{\Zc_0^*} \Hc^R_I$ for $(Y,\Gamma_{J^c,I})$ gives us an element of $\Spinc(Y,\partial Y, v_{J^c,I})$. Under the bijection outlined in Lemma~\ref{lem:GeneratorBijectionH-HLH'HR}, we get the following association of $\Spinc$ structures to generators of $\Hc_{\norm}$.

\begin{definition}\label{def:SpincStrsOfGens}
    For a generator $\x$ of $\Hc_{\norm}=\Hc^{\alpha \beta}_{1/2} \cup_{\Zc_1^*} \Hc' \cup_{\Zc_0^*} \Hc^{\beta \alpha}_{1/2}$ with $o_L(\x) = J^c$ and $o_R(\x) = I$, and corresponding generator $\x_L \cup \x' \cup \x_R$ of $\Hc_{J^c,I}=\Hc^L_{J^c} \cup_{\Zc_1^*} \Hc' \cup_{\Zc_0^*} \Hc^R_I$, the \emph{relative $\Spinc$ structure} $\mathfrak{s}^{\rel}(\x)$ \emph{of} $\x$ is defined to be the element of $\Spinc(Y,\partial Y, v_{J^c,I})$ that is associated to $\x_L \cup \x' \cup \x_R$ by \cite[Definition 4.5]{Juhasz}.
\end{definition}

\begin{remark}\label{rem:MorseFunctionConventions}
    In \cite[Definition 4.5]{Juhasz}, the $\Spinc$ structure associated to a generator of a Heegaard diagram $\Hc_{J^c,I}$ for a sutured 3-manifold $(Y,\Gamma_{J^c,I})$ is obtained by choosing a Morse function $f$ on $Y$, compatible with the handle decomposition of $Y$ coming from $\Hc_{J^c,I}$, such that the $\alpha$ handles give index-1 critical points of $f$ and the $\beta$ handles give index-2 critical points for $f$. Then, for $\x \in \mathfrak{S}(\Hc_{J^c,I})$, the vector field is obtained by modifying the upward gradient flow of $f$ in a neighborhood of the gradient flow lines determined by $\x$ between index 1 and 2 critical points.

    By contrast, as in Remark~\ref{rem:ChangeOfConventions}, we will choose a Morse function $f$ on $Y$ such that the $\beta$ handles give index-1 critical points and the $\alpha$ handles give index-2 critical points. For $\x \in \mathfrak{S}(\Hc_{J^c,I})$, our vector field is obtained by modifying the downward gradient flow of $f$ rather than the upward gradient flow. Thus, despite our difference in conventions for whether $\alpha$ or $\beta$ handles correspond to index $1$ or $2$ critical points, the relative $\Spinc$ structures we will consider are the same as those in \cite{Juhasz}.
\end{remark}

\begin{remark}
    The canonical vector field on the boundary of $(Y,\Gamma_{J^c,I;R^+ \leftrightarrow R^-})$ is $-v_{J^c,I}$, where we multiply the vectors in $v_{J^c,I}$ by $-1$. If we view $\x_L \cup \x' \cup \x_R$ as a generator of the diagram $(\Hc_{J^c,I})_{\alpha \leftrightarrow \beta}$ for $(Y,\Gamma_{J^c,I;R^+ \leftrightarrow R^-})$, then its corresponding element of $\Spinc(Y,\partial Y, -v_{J^c,I})$ is $-\mathfrak{s}^{\rel}(\x)$ where the negative sign means ``multiply the nonvanishing vector field by $-1$.''
\end{remark}

\begin{remark}
    The above constructions assign, to each generator $\x$ of the bordered sutured Heegaard diagram $\Hc$, a homology class of nonvanishing vector fields on $Y$ with certain suture-determined behavior on $\partial Y$, including on the ``bordered'' parts $F_i$ of $\partial Y$. It is interesting to compare with \cite[Section 4.3]{LOT} and \cite[Section 4.5]{Zarev}, as well as \cite{GrippYangBordered}, where a generator $\x$ of $\Hc$ is also assigned a homology class (or homotopy class in \cite{GrippYangBordered}) of nonvanishing vector fields on $Y$ with certain behavior on $\partial Y$. In these references, the vector field on the boundary is slightly different than the one we use here. When modifying the gradient vector field to make it nonvanishing, the modifications in $F_i \subset \partial Y$ are performed only in a neighborhood of each of the boundary critical points of the Morse function $f$ on $Y$. Outside these neighborhoods, the vector field is still tangent to the boundary of $Y$ everywhere on $F_i$. By contrast, the vector fields we consider are only tangent to the boundary of $Y$ along the 1-dimensional sutures $\Gamma_{J^c,I}$.
\end{remark}

\subsection{\texorpdfstring{$\Spinc$}{Spin-c} version of \texorpdfstring{$[\BSDA]$}{[BSDA]} depending on reference \texorpdfstring{$\Spinc$}{Spin-c} structures}

For each choice of $I,J$ such that $(Y,\Gamma_{J^c,I})$ is weakly balanced (i.e. $|J| = |I| + c$ in each component of $Y$), fix a reference element $\mathfrak{s}_{J^c,I}$ of $\Spinc(Y, \partial Y, v_{J^c,I})$, which we call a \emph{reference $\Spinc$ structure}. We can use $-\mathfrak{s}_{J^c,I}$ as a reference element of $\Spinc(Y, \partial Y, -v_{J^c,I})$; we will let $\s_{\mathrm{ref}}$ denote this element, so that $\mathfrak{s}_{J^c,I} = -\s_{\mathrm{ref}} \in \Spinc(Y,\partial Y, v_{J^c,I})$. Given these choices, we can make the following definition. Let $H = H_1(Y)$.

\begin{definition}\label{def:IntrinsicDegrees}
 Define a grading by $H$ on generators $\x$ of $\Hc_{\norm}$ with $o_L(\x) = J^c$ and $o_R(\x) = I$ by comparing $\mathfrak{s}^{\rel}(\x)$ with the reference element $\mathfrak{s}_{J^c,I}$ of $\Spinc(Y, \partial Y, v_{J^c,I})$, or in other words
\[
\mathrm{gr}_{H} \x := \mathfrak{s}^{\rel}(\x) -\mathfrak{s}_{J^c,I} \in H.
\] 
\end{definition}

\begin{remark}
    In Definition~\ref{def:IntrinsicDegrees} we are writing $H$ additively; in the context of $\Z[H]$, by contrast, we write $H$ multiplicatively.
\end{remark}

\begin{definition}\label{def:BSDADepOnRefSpinc}
    Let $\Xi_{\mathfrak{s}}$ denote the following set of data: first, take any set of choices $\Xi_{\norm}$ satisfying the criteria of Corollary~\ref{cor:CanChooseXiNormalized}. Next, for each choice of $I,J$ such that $(Y,\Gamma_{J^c,I})$ is weakly balanced, make a choice of reference element $\mathfrak{s}_{J^c,I}$ of $\Spinc(Y, \partial Y, v_{J^c,I})$. Given these choices, let
    \[\BSDA(Y,\Gamma;\Xi_{\mathfrak{s}})]^{\Z[H]}_{\comb;\, \mathrm{concrete}} \colon \wedge^* H_1(F_0,S^+_0) \otimes_{\Z} \Z[H] \to \wedge^* H_1(F_1,S^+_1) \otimes_{\Z} \Z[H]
    \]
    denote the $\Z[H]$-linear map determined by its action on a basis element $\gamma_I^{\inrm}$ of $\wedge^* H_1(F_0,S^+_0)$, given by
    \[
    \gamma^{\inrm}_I \mapsto \sum_{\substack{\x \in \mathfrak{S}(\Hc_{\norm}): \\ \overline{o}_L(\x) = J, \\ o_R(\x) = I}} (-1)^{\mathrm{gr}_{\DA}(\x)} \cdot \left( \mathrm{gr}_H(\x) \in H \right) \cdot \gamma^{\out}_J.
    \]
    Note that for any $I,J$ such that some $\x \in \mathfrak{S}(\Hc_{\norm})$ exists with $\overline{o}_L(\x) = J$ and $o_R(\x) = I$, we have that $(Y,\Gamma_{J^c,I})$ is weakly balanced by Proposition~\ref{prop:SameNumberAlphaBeta}, so $\mathrm{gr}_H(\x)$ makes sense in the above formula.
\end{definition}

For any choice of reference element $\mathfrak{s}_{J^c,I}$ of $\Spinc(Y, \partial Y, v_{J^c,I})$, we now observe that the map $[\BSDA(Y,\Gamma;\Xi_{\mathfrak{s}})]^{\Z[H]}_{\comb,\mathrm{concrete}}$ specializes to $[\BSDA(Y,\Gamma;\Xi_{\norm})]^{\Z}_{\mathrm{comb}}$. Denote by 
\[
\varepsilon \colon \Z[H] \to \Z
\]
the augmentation map that sends $h \mapsto 1$ for all $h \in H$. 

\begin{corollary}\label{lem:BSDASpecializesToBSDAZ}
    There is an equality of $\Z$-linear maps (without any sign ambiguity) given by
    \[
    [\BSDA(Y,\Gamma;\Xi_{\mathfrak{s}})]^{\Z[H]}_{\comb,\mathrm{concrete}} \otimes_{\Z[H]} \varepsilon = [\BSDA(Y,\Gamma;\Xi_{\norm})]^{\Z}_{\comb}
    \]
\end{corollary}

\begin{proof}
    When we tensor over $\Z[H]$ with $\varepsilon$, we have $\mathrm{gr}_{H}(\x)=1$ for all $\x$. The result follows by comparing Definition~\ref{def:BSDADepOnRefSpinc} in this case with Definition~\ref{def:BSDAZComb} of $[\BSDA(Y,\Gamma;\Xi_{\norm})]^{\Z}_{\mathrm{comb}}$. 
\end{proof}

\subsection{Torsion interpretation of $[\BSDA(Y,\Gamma;\Xi_{\mathfrak{s}})]^{\Z[H]}_{\comb;\, \mathrm{concrete}}$}

We now relate the map from Definition~\ref{def:BSDADepOnRefSpinc} to the constructions of Friedl--Juhász--Rasmussen. In particular, we connect it to the $\Z[H]$-valued Euler characteristic of $\SFH$ \cite[below Definition 2.13]{FJR}, Friedl--Juhász--Rasmussen's $\Z$-valued torsion function $T_{M,\gamma,\omega}\colon \Spinc \to \Z$ \cite[above Remark 3.19]{FJR}, and their torsion invariant $\tau(M,\gamma,\s,\omega) \in \Z[H]$ \cite[Definition 3.18]{FJR}. 

This perspective leads to the proofs of our main results over $\Z[G]$ and $\Z[H]$ when $H$ has torsion: the proof of Corollary~\ref{cor:MatrixEntryAndCellBdryDet} relies on interpreting $[\BSDA(Y,\Gamma;\Xi_{\mathfrak{s}})]^{\Z[H]}_{\comb;\,\mathrm{concrete}}$ via Friedl--Juhász--Rasmussen's torsion function $\tau(M,\gamma,\s,\omega)$, and this corollary is a key input for the proofs of our main results in Section~\ref{sec:AlexFunctorsInOtherSettings}.

\begin{remark}\label{rem:ChiSFHWhenNotBalanced}
    Below we will refer to $\chi(\SFH(Y,\Gamma_{J^c,I;R^+ \leftrightarrow R^-},\s,\omega_{\Xi_{J^c,I}}))$; as in Remark~\ref{rem:ChiSFHWhenNotBalancedNoSpinc}, this expression only makes sense when $(Y,\Gamma_{J^c,I})$ is balanced. If $(Y,\Gamma_{J^c,I})$ is weakly balanced but not balanced, we will interpret $\chi(\SFH(Y,\Gamma_{J^c,I;R^+ \leftrightarrow R^-},\s,\omega_{\Xi_{J^c,I}}))$ as the sum over generators $\x$ of $\Hc_{J^c,I}$ with $-\s^{\mathrm{rel}}(\x) = \s$ of the quantity $(-1)^{b_1(Y,R^+_{J^c,I}) + \sum_{x \in \x} i(x) + \mathrm{inv}(\sigma_{\x})}$ as in Remark~\ref{rem:ChiSFHWhenNotBalancedNoSpinc}.
\end{remark}

\begin{lemma}\label{lem:MatrixEntrySFHSpinc}
    If $(Y,\Gamma_{J^c,I})$ is weakly balanced, then the matrix entry $E$ of the map 
    \[
    [\BSDA(Y,\Gamma;\Xi_{\mathfrak{s}})]^{\Z[H]}_{\comb;\, \mathrm{concrete}}
    \]
    in column $\gamma^{\inrm}_I \otimes 1$ and row $\gamma^{\out}_J \otimes 1$ is 
    \[
    E = (-1)^{b_1(Y,R^+_{J^c,I}) + \mathrm{inv}(\sigma_{JJ^c \leftrightarrow \std}) + n_1k + ak} \sum_{h \in H} \left(\chi(\SFH(Y,\Gamma_{J^c,I;R^+ \leftrightarrow R^-}, h^{-1} \cdot \s_{\mathrm{ref}}, \omega_{\Xi_{J^c,I}}))\right) \cdot h,
    \]
    where the right-hand side of the equation is interpreted as in Remark~\ref{rem:ChiSFHWhenNotBalanced} when $(Y,\Gamma_{J^c,I})$ is weakly balanced but not balanced.
\end{lemma}

\begin{proof}
    This is similar to Corollary~\ref{cor:MatrixEntryAndSFHOverZ}, which can be proved directly. For a generator $\x$ of $\Hc$, we have $\mathrm{gr}_H(\x) = h$ if and only if $\s^{\rel}(\x) - \s_{J^c,I} = h$, which holds if and only if $-\s^{\rel}(\x) -(-\s_{J^c,I}) = -h$. We can write this equation as $-\s^{\rel}(\x) = -h + \s_{\mathrm{ref}}$, or (in multiplicative notation) $-\s^{\rel}(\x) = h^{-1} \cdot \s_{\mathrm{ref}}$. Note that $-\s^{\rel}(\x)$ is the $\Spinc$ structure associated to $\x$ as a generator of $\Hc_{J^c,I;\alpha \leftrightarrow \beta}$. It follows that $\mathrm{gr}_H(\x) = h$ if and only if $\x$ contributes to $\chi(\SFH(Y,\Gamma_{J^c,I;R^+ \leftrightarrow R^-}, h^{-1} \cdot \s_{\mathrm{ref}}, \omega_{\Xi_{J^c,I}}))$. The lemma, including the sign out front, follows as in Corollary~\ref{cor:MatrixEntryAndSFHOverZ}.
\end{proof}

We now assume $Y$ is connected for simplicity.

\begin{corollary}\label{cor:MatrixEntryAndTFunction}
    If $Y$ is connected and $(Y,\Gamma_{J^c,I})$ is weakly balanced, then the matrix entry $E$ of the map 
    \[
    [\BSDA(Y,\Gamma;\Xi_{\mathfrak{s}})]^{\Z[H]}_{\comb;\, \mathrm{concrete}}
    \]
    in column $\gamma^{\inrm}_I \otimes 1$ and row $\gamma^{\out}_J \otimes 1$ is 
    \[
    E = (-1)^{b_1(Y,R^+_{J^c,I}) + \mathrm{inv}(\sigma_{JJ^c \leftrightarrow \std}) + n_1k + ak} \sum_{h \in H} \left( T_{(Y,\Gamma_{J^c,I;R^+\leftrightarrow R^-},\omega_{\Xi_{J^c,I}})}(h^{-1} \cdot \s_{\mathrm{ref}})\right) \cdot h
    \]
    where 
    \[
    T_{(Y,\Gamma_{J^c,I;R^+\leftrightarrow R^-},\omega_{\Xi_{J^c,I}})} \colon \Spinc(Y,\partial Y, -v_{J^c,I}) \to \Z
    \]
    is defined as in \cite[above Remark 3.19]{FJR} (note that the definition of $T_{(Y,\Gamma_{J^c,I;R^+\leftrightarrow R^-},\omega_{\Xi_{J^c,I}})}$ requires that $(Y,\Gamma_{J^c,I;R^+ \leftrightarrow R^-})$ has nonempty $R^+$ and $R^-$ boundary on each component of $Y$; this follows from the topological assumptions we imposed on $Y$ in Definition~\ref{def:SuturedCob}).
\end{corollary}

\begin{proof}
    When $(Y,\Gamma_{J^c,I})$ is balanced, this follows from Lemma~\ref{lem:MatrixEntrySFHSpinc} and \cite[Theorem 1]{FJR}. When $(Y,\Gamma_{J^c,I})$ is weakly balanced but not balanced, we note that the proof of \cite[Theorem 1]{FJR} generalizes to weakly balanced sutured 3-manifolds with $\chi(\SFH)$ defined as in Remark~\ref{rem:ChiSFHWhenNotBalanced}. Indeed, the proof of \cite[Theorem 1]{FJR} uses balancedness only in the proof of \cite[Proposition 4.2]{FJR}, where balancedness is used to argue that the combinatorial sum in Remark~\ref{rem:ChiSFHWhenNotBalanced} can be interpreted as $\chi(\SFH)$.
\end{proof}

\begin{corollary}\label{cor:MatrixEntryAndTauTorsion}
    If $Y$ is connected and $(Y,\Gamma_{J^c,I})$ is weakly balanced, then the matrix entry $E$ of the map 
    \[
    [\BSDA(Y,\Gamma;\Xi_{\mathfrak{s}})]^{\Z[H]}_{\comb;\, \mathrm{concrete}}
    \]
    in column $\gamma^{\inrm}_I \otimes 1$ and row $\gamma^{\out}_J \otimes 1$ is 
    \[
    E = (-1)^{b_1(Y,R^+_{J^c,I}) + \mathrm{inv}(\sigma_{JJ^c \leftrightarrow \std}) + n_1k + ak}  \tau(Y,\Gamma_{J^c,I;R^+\leftrightarrow R^-}, \s_{\mathrm{ref}},\omega_{\Xi_{J^c,I}})
    \]
    where $\tau(Y,\Gamma_{J^c,I;R^+\leftrightarrow R^-}, \s_{\mathrm{ref}},\omega_{\Xi_{J^c,I}}) \in \Z[H]$ is defined as in \cite[Definition 3.18]{FJR} (note that \emph{a priori} $\tau(Y,\Gamma_{J^c,I;R^+\leftrightarrow R^-}, \s_{\mathrm{ref}},\omega_{\Xi_{J^c,I}})$ is an element of a larger ring $Q(H)$ containing $\Z[H]$, but by \cite[discussion below Definition 3.18]{FJR}, since $(Y,\Gamma_{J^c,I;R^+ \leftrightarrow R^-})$ has nonempty $R^+$ and $R^-$ boundary on the unique component of $Y$, we have $\tau(Y,\Gamma_{J^c,I;R^+\leftrightarrow R^-}, \s_{\mathrm{ref}},\omega_{\Xi_{J^c,I}}) \in \Z[H]$).    
\end{corollary}

\begin{proof}
    The definition of $T_{(M,\gamma,\omega)}(\s)$ above Remark 3.19 in \cite{FJR} is that $T_{(M,\gamma,\omega)}(\s)$ is defined to be $\tau(M,\gamma,\s,\omega)_1$, the coefficient of $\tau(M,\gamma,\s,\omega) \in \Z[H]$ on $1 \in H$. Friedl--Juhasz--Rasmussen then note that one can use \cite[equation (2)]{FJR} to recover all of $\tau(M,\gamma,\s,\omega) \in \Z[H]$ from the function $T_{(M,\gamma,\omega)}$. This equation says that, for the related CW-pair torsion function $\tau(X,Y,\mathfrak{l},\omega)$ from \cite[Section 3.4]{FJR}, we have $\tau(X,Y,h \cdot \mathfrak{l}, \pm \omega) = \pm h \cdot \tau(X,Y,\mathfrak{l},\omega)$ for every $h \in H$. As long as $\Spinc(M,\gamma) \cong \Lift(X,Y)$ (i.e. $X \neq Y$), this equation, together with \cite[Definition 3.17]{FJR}, \cite[Remark 3.11]{FJR}, and \cite[above Lemma 3.4]{FJR}, gives $\tau(M,\gamma,h \cdot \s,\pm \omega) = \pm h \cdot \tau(M,\gamma, \s,\omega)$ for every $h \in H$. 
    
    When $X = Y$, then the definition of $\tau(X,Y,\mathfrak{e},\omega)$ in \cite[above Lemma 3.4]{FJR} is $+h \in \Z[H]$ (assuming $\omega$ is the positive orientation of the zero vector space, which holds for our orientation choices by the discussion of Section~\ref{sec:RelWithSFHOrdinarySutured}) where $\mathfrak{e} \in \mathrm{Eul}(X,Y)$ corresponds to $h \in H$ under the canonical identification $\mathrm{Eul}(X,Y) \cong H$. Thus we have $\tau(M,\gamma,h \cdot \s_{\mathrm{ref}},\omega) = h$, so the equation $\tau(M,\gamma,h \cdot \s,\pm \omega) = \pm h \cdot \tau(M,\gamma, \s,\omega)$ still holds for every $h \in H$.

    It follows that the coefficient of $\tau(M,\gamma,\s,\omega)$ on $h \in H$, or in other words the coefficient of $(h^{-1}\tau(M,\gamma,\s,\omega))$ on $1 \in \Z[H]$, is the same as the coefficient of $\tau(M,\gamma,h^{-1} \cdot \s,\omega))$ on $1 \in H$. By definition of $T_{(M,\gamma,\omega)}$, this coefficient is $T_{(M,\gamma,\omega)}(h^{-1} \cdot \s)$. We get 
    \[
    \tau(M,\gamma,\s,\omega) = \sum_{h \in H} ( T_{(M,\gamma,\omega)}(h^{-1} \cdot \s) ) \cdot h
    \]
    so this corollary follows from Corollary~\ref{cor:MatrixEntryAndTFunction}.
\end{proof}

\subsection{Covers and lifts}\label{sec:CoversAndLifts} 

\subsubsection{General choices}\label{sec:CWCoveringSpaceGeneral}

Let $(Y,\Gamma)$ be a connected sutured cobordism from $(F_0,\Lambda_0)$ to $(F_1,\Lambda_1)$, and let $\Xi$ be a set of general (non-normalized) choices as in Definition~\ref{def:BSDAZComb}. Section~\ref{sec:CWDecomp} gives us CW decompositions of spaces $\mathfrak{Y}$ and $\mathfrak{R}^+$ homotopy equivalent (or identical in the case of $\mathfrak{R}^+)$ to $Y$ and $R^+$. Recall that $\Xi$ includes the data of orderings and orientations on the closed $\alpha$ and $\beta$ circles of $\Hc$; these determine orientations of the cells of $\mathfrak{Y} \setminus \mathfrak{R}^+$ following the discussion in the beginning of Section~\ref{sec:CellOrientations}.

\begin{definition}
    Let $H = H_1(Y)$. Let $\pi  \colon \widehat{Y} \to Y$ denote the \emph{maximal abelian cover} of $Y$, so that the deck group of $\pi$ is $H$. For any subspace $X \subset Y$, define 
    \[
    \widehat{X} \coloneq \pi^{-1}(X).
    \]
    In particular, taking $X = \mathfrak{Y}$ and $X = \mathfrak{R}^+$, we have spaces $\widehat{X} = \widehat{\mathfrak{Y}}$ and $\widehat{\mathfrak{R}}^+$, which are CW complexes. Note that the orientations of the cells of $\mathfrak{Y} \setminus \mathfrak{R}^+$ naturally give us orientations of the cells of $\widehat{\mathfrak{Y}} \setminus \widehat{\mathfrak{R}}^+$ such that the projection map $\pi$ is orientation-preserving. 
    
    Let $G=H/\tors(H)$. Given $\widehat{Y}$, for the quotient by the free action of $\tors(H)$, we write
    \[
    \overline{Y} \coloneq \widehat{Y}/\tors(H).
    \]
    Then $p: \overline{Y} \to Y$ is the \emph{maximal free abelian cover} of $Y$, with deck group $G$. For any subspace $X \subset Y$, define 
    \[
    \overline{X} \coloneq p^{-1}(X).
    \]
    In particular, taking $X = \mathfrak{Y}$ and $X = \mathfrak{R}^+$, we have spaces $\overline{X} = \overline{\mathfrak{Y}}$ and $\overline{\mathfrak{R}}^+$, which are CW complexes. The orientations of the cells of $\mathfrak{Y} \setminus \mathfrak{R}^+$ give us orientations of the cells of $\overline{\mathfrak{Y}} \setminus \overline{\mathfrak{R}}^+$ such that the projection map $p$ is orientation-preserving.
\end{definition}

Let $a$ and $b$ be the numbers of $\alpha$ and $\beta$ circles of $\Hc$. Choose a lift of each cell of $\mathfrak{Y}$ not contained in $\mathfrak{R}^+$ to a cell of $\widehat{\mathfrak{Y}}$. We write $\widehat{e}_1,\dots,\widehat{e}_b$ for these such lifts of 1-cells $e_1,\dots,e_b$ and $\widehat{e}_{b+1},\dots,\widehat{e}_{b+a}$ for these such lifts of 2-cells $e_{b+1},\dots,e_{b+a}$. Then we have ordered $\Z[H]$-bases $\widehat{e}_1,\dots,\widehat{e}_b$ for $C^{\cell}_1(\widehat{\mathfrak{Y}},\widehat{\mathfrak{R}}^+)$ and $\widehat{e}_{b+1},\dots,\widehat{e}_{b+a}$ for $C^{\cell}_2(\widehat{\mathfrak{Y}},\widehat{\mathfrak{R}}^+)$. The cell groups $C^{\cell}_i(\widehat{\mathfrak{Y}},\widehat{\mathfrak{R}}^+)$ are zero for $i \neq 1,2$.

\begin{definition}
    Let $\widehat{M}_{\Hc}$ be the matrix for the cellular differential
    \[
    \widehat{\partial}_2 \colon C^{\cell}_2(\widehat{\mathfrak{Y}},\widehat{\mathfrak{R}}^+)\to C^{\cell}_1(\widehat{\mathfrak{Y}},\widehat{\mathfrak{R}}^+)
    \]
    in the above ordered bases.
\end{definition}

The matrix $\widehat{M}_{\Hc}$ has cokernel $H_1(\widehat{Y},\widehat{R}^+)$, and is thus a presentation matrix for $H_1(\widehat{Y},\widehat{R}^+)$ of deficiency $d = a - b = -\chi(Y,R^+)$. Note that the kernel of $\widehat{M}_{\Hc}$ is $H_2(\widehat{Y},\widehat{R}^+)$.

Now pass to the quotient $\overline{Y} =\widehat{Y}/\tors(H)$. We can pass our chosen bases to this quotient via the identification
\[
C^{\cell}_1(\widehat{\mathfrak{Y}},\widehat{\mathfrak{R}}^+) \otimes_{\Z[H]} \Z[G] \stackrel{\cong}{\longrightarrow} C^{\cell}_1(\overline{\mathfrak{Y}},\overline{\mathfrak{R}}^+) 
\]
determined by 
\[
\widehat{e}_i \otimes 1 \mapsto \overline{e}_i
\]
and similarly for $C^{\cell}_2(\widehat{\mathfrak{Y}},\widehat{\mathfrak{R}}^+)$, where we view $\Z[G]$ as a $\Z[H]$-module via the ring homomorphism $\Z[H]\to \Z[G]$ induced by the quotient map $H \to G$. 

\begin{definition}
    Let $\overline{M}_{\Hc}$ be the matrix for the cellular differential 
    \[
    \overline{\partial}_2 \colon C^{\cell}_2(\overline{\mathfrak{Y}},\overline{\mathfrak{R}}^+)\to C^{\cell}_1(\overline{\mathfrak{Y}},\overline{\mathfrak{R}}^+)
    \]
    in the above ordered bases. 
\end{definition}

The matrix $\overline{M}_{\Hc}$ has cokernel $H_1(\overline{Y},\overline{R}^+)$, and is thus a presentation matrix for $H_1(\overline{Y},\overline{R}^+)$ of deficiency $d = a - b = -\chi(Y,R^+)$. Note that the kernel of $\overline{M}_{\Hc}$ is $H_2(\overline{Y},\overline{R}^+)$.

\subsubsection{Normalized choices}\label{sec:CWCoveringSpaceNormalized}

Now assume that $\Xi$ is of the form $\Xi_{\norm}$ as in Corollary~\ref{cor:NormalizedHeegaardDiagrams}, so that the Heegaard diagram is of the form $\Hc_{\norm}=\Hc^{\alpha \beta}_{1/2} \cup_{\Zc_1^*} \Hc' \cup_{\Zc_0^*} \Hc^{\beta \alpha}_{1/2}$. Section~\ref{sec:CWDecomp} gives us CW subcomplexes $\mathfrak{F}_i$ and $\mathfrak{S}^+_i$ of $\mathfrak{Y}$ with $\mathfrak{F}_i \cap \mathfrak{R}^+ = \mathfrak{S}^+_i$. Letting $\pi \colon \widehat{Y} \to Y$ be the maximal abelian cover of $Y$, we have $\widehat{R}^+ = \widehat{\mathfrak{R}}^+ \subset \widehat{\mathfrak{Y}} \subset Y$ by taking preimages as above. We set $\widehat{F}_i = \pi^{-1}(F_i)$ and $\widehat{S}^+_i = \pi^{-1}(S^+_i) = \widehat{F}_i \cap \widehat{R}^+_i$. We also set $\widehat{\mathfrak{F}}_i = \pi^{-1}(\mathfrak{F}_i)$ and $\widehat{\mathfrak{S}}^+_i = \pi^{-1}(\mathfrak{S}^+_i) = \widehat{\mathfrak{F}}_i \cap \widehat{\mathfrak{R}}^+$. As in Corollary~\ref{cor:SingCellSquareCommutes}, the diagram 
    \[
    \xymatrix{
        H_1(\widehat{F}_i, \widehat{S}^+_i) \ar[r]^{i_*}  & H_1(\widehat{Y}, \widehat{R}^+)  \\
        H_1(\widehat{\mathfrak{F}}_i, \widehat{\mathfrak{S}}^+_i) \ar[r]_{i_*} \ar[u]^{\cong} & H_1(\widehat{\mathfrak{Y}}, \widehat{\mathfrak{R}}^+) \ar[u]_{i_*}
        }
    \]
commutes, where the map on the left is the hat-version of the composition of isomorphisms on the left of diagram~\eqref{eq:FSYRCommutativeDiag}.

Now, Section~\ref{sec:IdempotentDepCWComplexes} gives us larger idempotent-dependent CW pairs $(\mathfrak{Y}_{J^c,I}, \mathfrak{R}^+_{J^c,I})$ for each $I,J$, with each $\mathfrak{Y}_{J^c,I}$ homotopy equivalent to $Y$; orderings and orientations of cells for these CW complexes are chosen as in $\Xi_{J^c,I}$ from Definition~\ref{def:Xi_J^c,I}. As in Section~\ref{sec:AmbientCWComplex}, we will view $\mathfrak{Y}$ and all the complexes $\mathfrak{Y}_{J^c,I}$ as subcomplexes of one largest CW complex $\mathfrak{Y}_{\mathbf{a}_1,\mathbf{a}_0}$; note that the orderings and orientations encoded in $\Xi_{\norm}$ and $\Xi_{J^c,I}$ are the same as those inherited from $\Xi_{\mathbf{a}_1,\mathbf{a}_0}$.

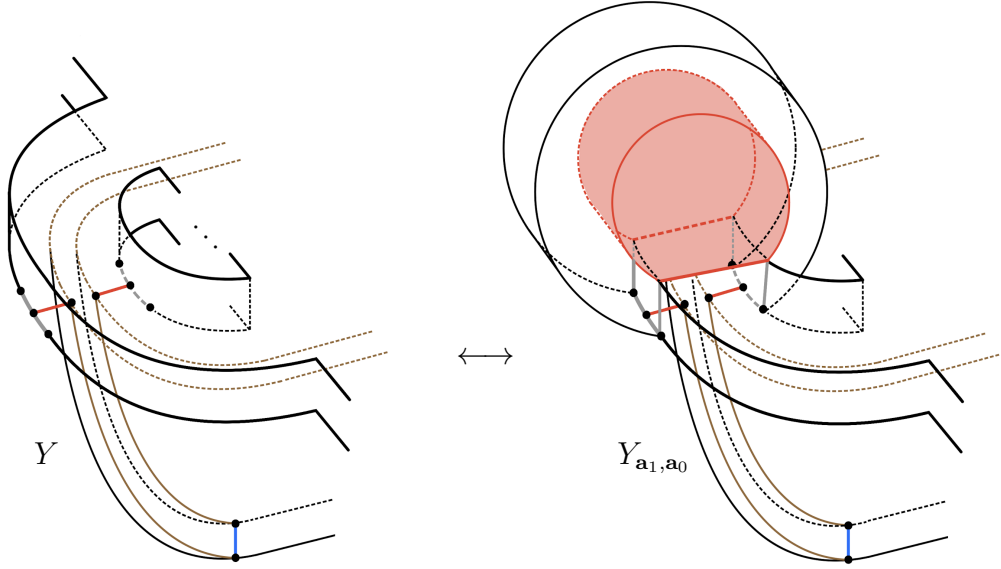
\begin{figure}
    \centering
    \begin{overpic}[width=0.81\textwidth]
    {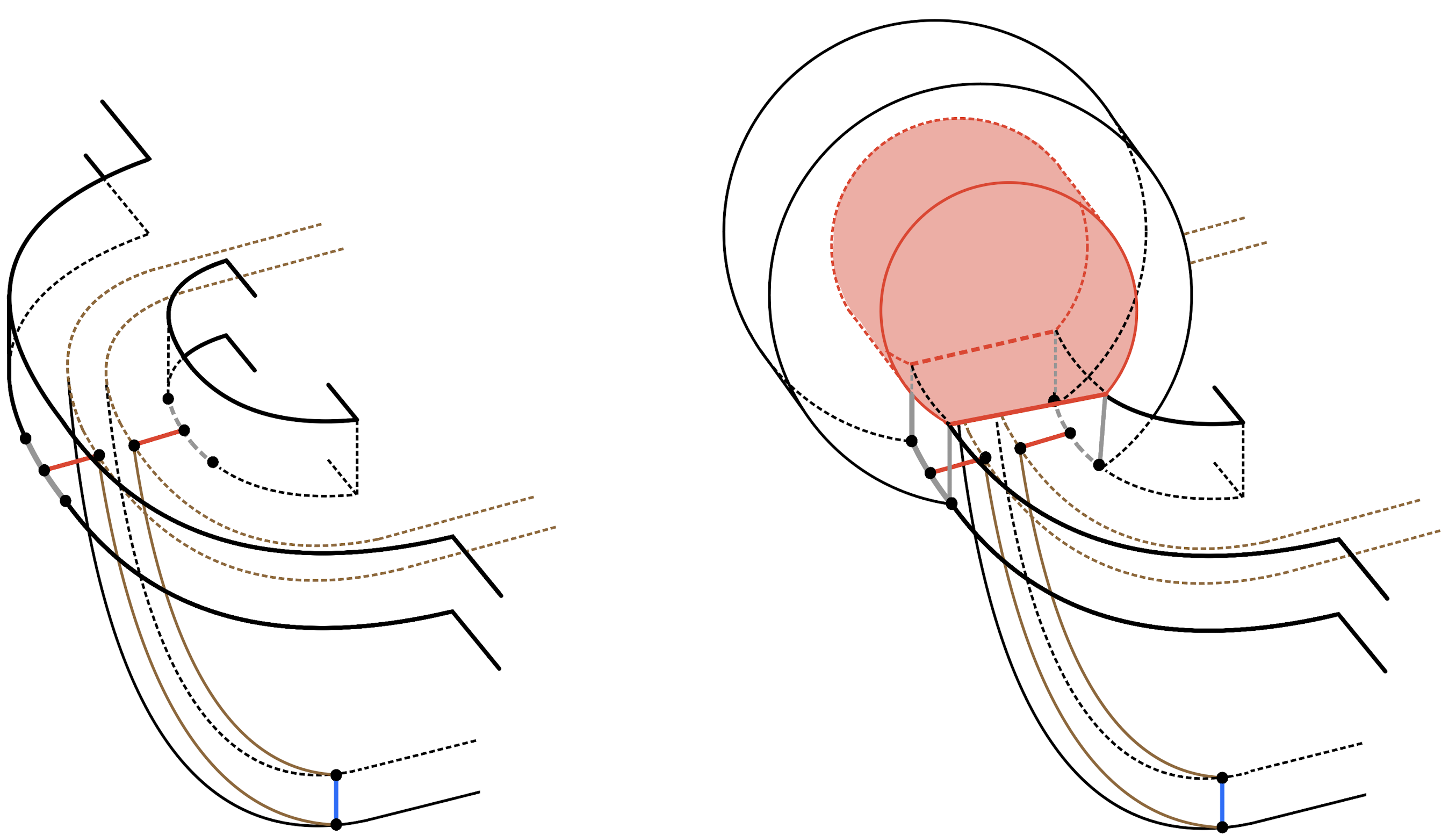}
    \put (3,10) {$Y$}
    \put (45,20) {$\longleftrightarrow$}
    \put (61,10) {$Y_{\mathbf{a}_1,\mathbf{a}_0}$}
    \put (18.7,33.8) {\rotatebox{-45}{$\ldots$}}
    \end{overpic}
    \vspace{1em}
    \caption{The homeomorphic 3-manifolds $Y$ (left) and $Y_{\mathbf{a}_1,\mathbf{a}_0}$ (right). In general, $Y_{\mathbf{a}_1,\mathbf{a}_0}$ differs from $Y$ by having extra ``blobs'' as shown in the picture, and the homeomorphism from $Y_{\mathbf{a}_1,\mathbf{a}_0}$ to $Y$ shrinks the blob into the bulk of $Y$. The inverse homeomorphism re-expands the blob. For clarity, on the right we omit portions of the diagram lying behind the foreground; otherwise these would appear as additional dotted curves and clutter the figure.}
    \label{fig:HomeoYYa1a0}
\end{figure}

Let $Y_{\mathbf{a}_1,\mathbf{a}_0}$ be the topological space associated to the sutured Heegaard diagram $\Hc_{\mathbf{a}_1,\mathbf{a}_0}$. We have a canonical homeomorphism $\varphi$ between $Y$, the topological space associated to the bordered sutured Heegaard diagram $\Hc_{\norm}$, and $Y_{\mathbf{a}_1,\mathbf{a}_0}$ as shown in Figure~\ref{fig:HomeoYYa1a0}. This homeomorphism $\varphi$ is the evident identity map on $\mathfrak{Y} \subset Y$, which gets sent to the corresponding subset $\mathfrak{Y} \subset  \mathfrak{Y}_{\mathbf{a}_1,\mathbf{a}_0}$ of $Y_{\mathbf{a}_1,\mathbf{a}_0}$ (we will call this latter subset $\mathfrak{Y}'$ rather than $\mathfrak{Y}$ temporarily). Given our chosen model $\pi \colon \widehat{Y} \to Y$ for the maximal abelian cover of $Y$, we use
\[
\pi_{\mathbf{a}_1,\mathbf{a}_0} \coloneq \widehat{Y} \xrightarrow{\pi} Y \xrightarrow{\varphi} Y_{\mathbf{a}_1,\mathbf{a}_0}
\]
as our model for the maximal abelian cover $\widehat{Y}_{\mathbf{a}_1,\mathbf{a}_0}$ of $Y_{\mathbf{a}_1,\mathbf{a}_0}$ (so we have $\widehat{Y}_{\mathbf{a}_1,\mathbf{a}_0} = \widehat{Y}$ as topological spaces). If we let $\widehat{\mathfrak{Y}}' = \pi_{\mathbf{a}_1,\mathbf{a}_0}^{-1}(\mathfrak{Y}')$, it follows from the fact that $\varphi|_{\mathfrak{Y}}$ is the evident identification of $\mathfrak{Y}$ with $\mathfrak{Y}'$ that we have a canonical identification $\widehat{\mathfrak{Y}} \cong \widehat{\mathfrak{Y}}'$. In other words, we can view $\widehat{\mathfrak{Y}}$ and its subcomplexes $\widehat{\mathfrak{R}}^+$, $\widehat{\mathfrak{F}}_i$, and $\widehat{\mathfrak{S}}^+_i$ as all being subcomplexes of $\widehat{\mathfrak{Y}}_{\mathbf{a}_1,\mathbf{a}_0}$, such that the map
\[
\widehat{i}_* \colon H_1(\widehat{\mathfrak{F}}_i,\widehat{\mathfrak{S}}^+_i) \to H_1(\widehat{\mathfrak{Y}},\widehat{\mathfrak{R}}^+)
\]
comes from the usual inclusion of subcomplexes of $\widehat{\mathfrak{Y}}_{\mathbf{a}_1,\mathbf{a}_0}$. 

We also have subcomplexes
\[
\widehat{\mathfrak{Y}} \subset \widehat{\mathfrak{Y}}_{J^c,I} \subset \widehat{\mathfrak{Y}}_{\mathbf{a}_1,\mathbf{a}_0}
\]
and
\[
\widehat{\mathfrak{R}}^+ \subset \widehat{\mathfrak{R}}^+_{J^c,I} \subset \widehat{\mathfrak{R}}^+_{\mathbf{a}_1,\mathbf{a}_0}
\]
for all $I$ and $J$, where $\widehat{\mathfrak{Y}}_{J^c,I} \coloneq \pi_{\mathbf{a}_1,\mathbf{a}_0}^{-1}(\mathfrak{Y}_{J^c,I})$ and $\widehat{\mathfrak{R}}^+_{J^c,I} \coloneq \pi_{\mathbf{a}_1,\mathbf{a}_0}^{-1}(\mathfrak{R}^+_{J^c,I})$. We note that
\[
\pi_{J^c,I} \coloneq \pi_{\mathbf{a}_1,\mathbf{a}_0} \colon \widehat{\mathfrak{Y}}_{J^c,I} \to \mathfrak{Y}_{J^c,I}
\]
is (a model for) the maximal abelian cover of $\mathfrak{Y}_{J^c,I}$, since $\mathfrak{Y}_{J^c,I}$ is a strong deformation retract of $Y_{\mathbf{a}_1,\mathbf{a}_0}$. Furthermore, $\widehat{\mathfrak{R}}^+_{J^c,I}$ is the preimage of $\mathfrak{R}^+_{J^c,I} \subset \mathfrak{Y}_{J^c,I}$ under the covering map $\pi_{J^c,I}$.

The above all has analogues for maximal free abelian covers. With $\overline{Y} \coloneq \widehat{Y} / \tors(H)$, we have $p \colon \overline{Y} \to Y$ as the maximal free abelian cover of $Y$, and we set $\overline{F}_i = p^{-1}(F_i)$, $\overline{S}^+_i = p^{-1}(S^+_i) = \overline{F}_i \cap \overline{R}^+_i$, $\overline{\mathfrak{F}}_i = p^{-1}(\mathfrak{F}_i)$, and $\overline{\mathfrak{S}}^+_i = p^{-1}(\mathfrak{S}^+_i) = \overline{\mathfrak{F}}_i \cap \overline{\mathfrak{R}}^+$. The diagram 
    \[
    \xymatrix{
        H_1(\overline{F}_i, \overline{S}^+_i) \ar[r]^{i_*}  & H_1(\overline{Y}, \overline{R}^+)  \\
        H_1(\overline{\mathfrak{F}}_i, \overline{\mathfrak{S}}^+_i) \ar[r]_{i_*} \ar[u]^{\cong} & H_1(\overline{\mathfrak{Y}}, \overline{\mathfrak{R}}^+) \ar[u]_{i_*}
        }
    \]
commutes. We use $\overline{Y}_{\mathbf{a}_1,\mathbf{a}_0} \coloneq \widehat{Y}_{\mathbf{a}_1,\mathbf{a}_0} / \tors(H)$ as our model for the maximal free abelian cover 
\[
p_{\mathbf{a}_1,\mathbf{a}_0} \colon \overline{Y}_{\mathbf{a}_1,\mathbf{a}_0} \to Y_{\mathbf{a}_1,\mathbf{a}_0}
\]
of $Y_{\mathbf{a}_1,\mathbf{a}_0}$. Let $\overline{\mathfrak{Y}}_{\mathbf{a}_1,\mathbf{a}_0} = p_{\mathbf{a}_1,\mathbf{a}_0}^{-1}(\mathfrak{Y}_{\mathbf{a}_1,\mathbf{a}_0})$; we can view $\overline{\mathfrak{Y}}$ and its subcomplexes $\overline{\mathfrak{R}}^+$, $\overline{\mathfrak{F}}_i$, and $\overline{\mathfrak{S}}^+_i$ as all being subcomplexes of $\overline{\mathfrak{Y}}_{\mathbf{a}_1,\mathbf{a}_0}$, such that the map
\[
\overline{i}_* \colon H_1(\overline{\mathfrak{F}}_i,\overline{\mathfrak{S}}^+_i) \to H_1(\overline{\mathfrak{Y}},\overline{\mathfrak{R}}^+)
\]
comes from the usual inclusion of subcomplexes of $\overline{\mathfrak{Y}}_{\mathbf{a}_1,\mathbf{a}_0}$. We have subcomplexes
\[
\overline{\mathfrak{Y}} \subset \overline{\mathfrak{Y}}_{J^c,I} \subset \overline{\mathfrak{Y}}_{\mathbf{a}_1,\mathbf{a}_0}
\]
and
\[
\overline{\mathfrak{R}}^+ \subset \overline{\mathfrak{R}}^+_{J^c,I} \subset \overline{\mathfrak{R}}^+_{\mathbf{a}_1,\mathbf{a}_0}
\]
for all $I$ and $J$, where $\overline{\mathfrak{Y}}_{J^c,I} \coloneq p_{\mathbf{a}_1,\mathbf{a}_0}^{-1}(\mathfrak{Y}_{J^c,I})$ and $\overline{\mathfrak{R}}^+_{J^c,I} \coloneq p_{\mathbf{a}_1,\mathbf{a}_0}^{-1}(\mathfrak{R}^+_{J^c,I})$. We note that
\[
p_{J^c,I} \coloneq p_{\mathbf{a}_1,\mathbf{a}_0} \colon \overline{\mathfrak{Y}}_{J^c,I} \to \mathfrak{Y}_{J^c,I}
\]
is (a model for) the maximal free abelian cover of $\mathfrak{Y}_{J^c,I}$ and that $\overline{\mathfrak{R}}^+_{J^c,I}$ is the preimage of $\mathfrak{R}^+_{J^c,I} \subset \mathfrak{Y}_{J^c,I}$ under the covering map $p_{J^c,I}$.

\subsubsection{Homology of lifted surfaces}\label{sec:LiftedSurfacesHomology}

Continue to assume that $\Xi$ is of the form $\Xi_{\norm}$ as in Corollary~\ref{cor:NormalizedHeegaardDiagrams}.
We record here the structure of $H_1(\widehat{F}_i,\widehat{S}^+_i)$ for later use; we study the outgoing case $i=1$ (the incoming $i=0$ case follows analogously). Recall that $H_1(F_1,S_1^+)$ is freely generated over $\Z$ by the matching arcs $\gamma_i^{\out}$ of $\Zc_1$ with chosen orientations coming from $\Xi_{\norm}$. As in Section~\ref{sec:CellularChainCx}, the identification $H_1(F_1,S_1^+) \cong H_1(\mathfrak{F}_1,\mathfrak{S}^+_1)$ sends $\gamma^{\out}_i$ to $-1$ times the 1-cell of $\mathfrak{F}_1 \setminus \mathfrak{S}^+_1$ which corresponding to the unique $\beta$ circle intersecting the outgoing $\alpha$ arc corresponding to $\gamma^{\out}_i$ nontrivially. Thus $H_1(\mathfrak{F}_1,\mathfrak{S}^+_1)\cong C^{\cell}_1(\mathfrak{F}_1,\mathfrak{S}^+_1)$ is freely generated by the $n_1$ relative $1$-cells of $(\mathfrak{F}_1,\mathfrak{S}^+_1)$ corresponding to $\beta$-circles that intersect $\Hc^{\alpha\beta}_{1/2}$ nontrivially; denote these by $e_i^{\out} \in \{e_1,\dots,e_b\}$. Thus, a $\Z[H]$ basis for $H_1(\widehat{F}_1,\widehat{S}^+_1)$ is given by choosing $n_1$ lifts $\widehat{e}_i^{\out}$ of each such $e_i^{\out}$ to a cell of $\widehat{\mathfrak{F}}_1$; as above, the orientations for the cells of $\mathfrak{F}_1$ naturally give us orientations of the cells of $\widehat{\mathfrak{F}}_1$ such that the projection map $\pi$ is orientation-preserving (recall that $\Xi_{\norm}$ assumes the orientation conventions of Proposition~\ref{prop:ExpandingBasisElts}). 

In Section~\ref{sec:BSDAIndofRefSpinc}, we will appeal to the following identification ($\Phi_0$ is defined similarly): 
\[
\Phi_1 \colon \wedge^* H_1(F_1,S^+_1) \otimes_{\Z} \Z[H] \stackrel{\cong}{\longrightarrow} \wedge^*_{\Z[H]} H_1(\widehat{F}_1, \widehat{S}^+_1).
\]
This map depends on a choice of $\Z[H]$-basis defined as above. Specifically, for a generator $\gamma_i^{\out} \in H_1(F_1,S_1^+)$, Proposition~\ref{prop:ExpandingBasisElts} established a correspondence that first sends
\begin{align*}
    H_1(F_1,S_1^+) &\stackrel{\cong}{\longrightarrow} H_1(\mathfrak{F}_1,\mathfrak{S}^+_1) \\
    \gamma_i^{\out} &\mapsto -e_i^{\out}
\end{align*}
then extends by $\Z$-linearity (note that we also showed there, $e_i^{\inrm}$ affords the same sign in the incoming case). Now, since $H_1(\widehat{F}_1,\widehat{S}^+_1)$ is freely generated over $\Z[H]$, for each $0 \leq l \leq n_1$, the elements
\[
\widehat{e}^{\out}_J = \widehat{e}_{j_1} \wedge \cdots \wedge \widehat{e}^{\out}_{j_l}
\]
for $J = (j_1,\ldots,j_l)$, $1 \leq j_1 < \cdots < j_l \leq n_1$, form a basis for $\wedge^l H_1(\widehat{F}_1,\widehat{S}^+_1)$ over $\Z[H]$. Then given a generator $\gamma^{\mathrm{out}}_J \in \wedge^* H_1(F_1,S^+_1)$, the identification $\Phi_1$ is characterized by sending
\[
\Phi_1(\gamma^{\out}_J \otimes 1_{\Z[H]})=-\widehat{e}_J^{\out}
\]
then extending by $\Z[H]$-linearity and extending over wedge products, where the sign appears since we have assumed $\widehat{e}_i^{\out} \stackrel{\pi}{\mapsto} e_i^{\out}$ is orientation-preserving for every $1\leq i \leq b$. 

Now pass to the quotient $\overline{Y} =\widehat{Y}/\tors(H)$. Focusing attention to the outgoing sutured surface (the incoming case follows analogously), we can pass our chosen bases to this quotient via the identification 
\[
C^{\cell}_1(\widehat{\mathfrak{F}}_1,\widehat{\mathfrak{S}}^+_1) \otimes_{\Z[H]} \Z[G] \stackrel{\cong}{\longrightarrow} C^{\cell}_1(\overline{\mathfrak{F}}_1,\overline{\mathfrak{S}}^+_1) 
\]
determined by 
\[
\widehat{e}_i^{\out} \otimes 1 \mapsto \overline{e}_i^{\out}.
\]
Thus, there is a module $H_1(\overline{F}_1,\overline{S}^+_1)$ freely generated over $\Z[G]$ by lifts of 1-cells $\widehat{e}_i^{\out}$ which have passed through the quotient; this implies identifications
\[
H_1(\widehat{F}_1, \widehat{S}^+_1) \otimes_{\Z[H]} \Z[G] \cong H_1(\overline{F}_1, \overline{S}^+_1),
\]
that are canonical (do not only hold up to deck transformation). We will write $\overline{e}_J^{\out}$ for a basis element of $\wedge^*_{\Z[G]} H_1(\overline{F}_1, \overline{S}^+_1)$, and likewise for the incoming case.

\subsubsection{Identifying lifts with relative $\Spinc$ structures}

We now apply the framework of \cite{FJR} to the CW pair $(X,Y)=(\mathfrak{Y}_{J^c,I},\mathfrak{R}^+_{J^c,I})$. Although the notions of lifts and Euler structures notions originate in the work of Turaev, we borrow Friedl--Juhász--Rasmussen's definitions to streamline the exposition. In particular, we appeal directly to the composite bijection established in \cite{FJR} and omit the intermediate constructions. We continue to assume that $Y$ is connected.

\begin{proposition}\label{prop:SpincAndLifts}
    For every $I,J$ such that $(Y,\Gamma_{J^c,I})$ is weakly balanced and $\mathfrak{Y}_{J^c,I} \neq \mathfrak{R}^+_{J^c,I}$, there is a $H$-equivariant bijection 
    \[
    \Spinc(Y,\partial Y, -v_{J^c,I}) \cong \Lift(\mathfrak{Y}_{J^c,I},\mathfrak{R}^+_{J^c,I})
    \]
    where $\Lift(X,Y)$ is defined for a CW pair $(X,Y)$ as in \cite[Definition 3.1]{FJR}. Note that by the discussion in Section~\ref{sec:CWCoveringSpaceNormalized}, we can understand $\Lift(\mathfrak{Y}_{J^c,I},\mathfrak{R}^+_{J^c,I})$ using $\widehat{\mathfrak{Y}}_{J^c,I}$ as our model for the maximal abelian cover of $\mathfrak{Y}_{J^c,I}$.
\end{proposition}

\begin{proof}
    Recall that our definition of $\Spinc(Y,\partial Y, -v_{J^c,I})$ coincides with the definition of $\Spinc(Y,\Gamma_{J^c,I;R^+ \leftrightarrow R^-})$ from \cite[Definition 2.2]{FJR}, so the result follows by composing various bijections established in \cite{FJR}. The pair $Z_{J^c,I} = (\mathcal{A}_{J^c,I}, d)$, where $\mathcal{A}_{J^c,I}$ is the sutured handle complex associated to $(Y,\Gamma_{J^c,I;R^+ \leftrightarrow R^-})$ as in Section~\ref{sec:CWDecomp} and $d$ is the identity diffeomorphism on $Y$, satisfies the definition of a \emph{handle decomposition} of $(Y,\Gamma_{J^c,I;R^+ \leftrightarrow R^-})$ in the sense of \cite[Definition~3.12]{FJR}. Then \cite[Definition 3.13]{FJR} gives an affine $H$-isomorphism 
    \[
    s_{Z_{J^c,I}} \colon \Spinc(Y,\partial Y, -v_{J^c,I}) \to  \Eul(\A_{J^c,I}), 
    \]
    where $\Eul(\A_{J^c,I})$ is defined as in \cite[Definition 3.10]{FJR}. Next, by \cite[Remark 3.11]{FJR} there is a canonical bijection $\Eul(\A_{J^c,I})\cong\Eul(\mathfrak{Y}_{J^c,I},\mathfrak{R}^+_{J^c,I})$, where $\Eul(\mathfrak{Y},\mathfrak{R}^+_{J^c,I})$ is defined as in \cite[Definition 3.2]{FJR}. Finally, whenever $\chi(\mathfrak{Y}_{J^c,I},\mathfrak{R}^+_{J^c,I})=0$ and $\mathfrak{Y}_{J^c,I} \neq \mathfrak{R}^+_{J^c,I}$, \cite[Definition 3.3]{FJR} implies that 
    \[
    \Lift(\mathfrak{Y}_{J^c,I},\mathfrak{R}^+_{J^c,I}) \to \Eul(\mathfrak{Y}_{J^c,I},\mathfrak{R}^+_{J^c,I})
    \]
    is an $H$-equivariant bijection. Since we have assumed that $(Y,\Gamma_{J^c,I})$ is weakly balanced, Corollary~\ref{cor:J^c,IWeaklyBalChiEqualsZero} together with the homotopy equivalence of $\mathfrak{Y}_{J^c,I}$ relative $\mathfrak{R}^+_{J^c,I}$ with $Y$ relative $R^+_{J^c,I}$ gives us $\chi(Y,R^+_{J^c,I})=\chi(\mathfrak{Y}_{J^c,I},\mathfrak{R}^+_{J^c,I})=0$. We have also assumed $\mathfrak{Y}_{J^c,I} \neq \mathfrak{R}^+_{J^c,I}$, so the bijection holds. 
\end{proof}

\begin{remark}\label{rem:SpecialCaseLiftSpinc}
    The condition $\mathfrak{Y}_{J^c,I} \neq \mathfrak{R}^+_{J^c,I}$ in Proposition~\ref{prop:SpincAndLifts} always holds for all $I,J$ except in the degenerate case where the Heegaard surface $\Sigma$ is a closed surface $\Sigma_g$ of genus $g$ with no $\alpha$ or $\beta$ circles. In this case, we must have $I = J = \emptyset$, with $J^c = \emptyset$ as well; we have $\Spinc(\Sigma_g \times [0,1], \partial(\Sigma_g \times [0,1]), -v_{\emptyset,\emptyset}) \cong H_1(\Sigma_g) \cong \Z^{2g}$ while the corresponding set of lifts has only the trivial (empty) lift. There is a canonical element $\frac{\partial}{\partial t} \in \Spinc(\Sigma_g \times [0,1], \partial(\Sigma_g \times [0,1]), -v_{\emptyset,\emptyset})$ and we will use $\frac{\partial}{\partial t}$ in choosing the reference $\Spinc$ structure $\s_{\mathrm{ref}} = -\s_{\emptyset,\emptyset}$ (i.e. we will use $-\frac{\partial}{\partial t}$ in choosing $\s_{\emptyset,\emptyset}$). As discussed in \cite[below Definition 3.3]{FJR}, there is still a map from the one-point set of lifts into $\Spinc(\Sigma_g \times [0,1], \partial(\Sigma_g \times [0,1]), -v_{\emptyset,\emptyset})$, and its image is $\{\frac{\partial}{\partial t}\}$. Thus, even in this degenerate case, it holds that for all $I,J$ such that $(Y,\Gamma_{J^c,I})$ is weakly balanced, we have an element of $\Lift(\mathfrak{Y}_{J^c,I},\mathfrak{R}^+_{J^c,I})$ corresponding to $\s_{\mathrm{ref}} = -\s_{J^c,I} \in \Spinc(Y,\partial Y, -v_{J^c,I})$. This fact will appear in the statement of the below corollary.
\end{remark}

\begin{corollary}\label{cor:MatrixEntryAndCellBdryDet}
    If $(Y,\Gamma_{J^c,I})$ is weakly balanced, then the matrix entry $E$ of the map 
    \[
    [\BSDA(Y,\Gamma;\Xi_{\mathfrak{s}})]^{\Z[H]}_{\comb;\, \mathrm{concrete}}
    \]
    in column $\gamma^{\inrm}_I \otimes 1$ and row $\gamma^{\out}_J \otimes 1$ is 
    \[
    E = (-1)^{\mathrm{inv}(\sigma_{JJ^c \leftrightarrow \std}) + n_1 k + ak} \det \widehat{A}
    \]
    where $\widehat{A}$ is the matrix for the boundary map $\partial_2$ of the cellular chain complex of $(\widehat{\mathfrak{Y}}_{J^c,I},\widehat{\mathfrak{R}}^+_{J^c,I})$ in the bases for the chain groups $C_2$ and $C_1$ determined by any choice of lifts of the 2-cells and 1-cells of $\mathfrak{Y}_{J^c,I} \setminus \mathfrak{R}^+_{J^c,I}$ to 2-cells and 1-cells of $\widehat{\mathfrak{Y}}_{J^c,I} \setminus \widehat{\mathfrak{R}}^+_{J^c,I}$ compatible with the element $\mathfrak{l}$ of $\Lift(\mathfrak{Y}_{J^c,I},\mathfrak{R}^+_{J^c,I})$ corresponding to $\s_{\mathrm{ref}} = -\s_{J^c,I} \in \Spinc(Y,\partial Y, -v_{J^c,I})$. Note that if the Heegaard surface $\Sigma_{J^c,I}$ of $\Hc_{J^c,I;\alpha \leftrightarrow \beta}$ is closed and connected with no $\alpha$ or $\beta$ circles, so $\Lift(\mathfrak{Y}_{J^c,I},\mathfrak{R}^+_{J^c,I})$ has a unique element $\mathfrak{l}$, then the sense in which $\mathfrak{l}$ corresponds to $\s_{\mathrm{ref}}$ should be interpreted as in Remark~\ref{rem:SpecialCaseLiftSpinc}. We also define the determinant of the empty (size $0 \times 0$) matrix to be $1$. 
\end{corollary}

\begin{proof}
    If the Heegaard surface $\Sigma_{J^c,I}$ is closed with no $\alpha$ or $\beta$ circles, the equation in the corollary is $1 = 1$, so assume we are not in this degenerate case. By \cite[Lemma 3.6]{FJR} and Corollary~\ref{cor:MatrixEntryAndTauTorsion}, it suffices to show that
    \[
    \tau(Y,\Gamma_{J^c,I;R^+ \leftrightarrow R^-}, \s_{\mathrm{ref}}, \omega_{\Xi_{J^c,I}}) = \tau(\mathfrak{Y}_{J^c,I}, \mathfrak{R}^+_{J^c,I}, \mathfrak{l}, \omega_{\Xi_{J^c,I}}),
    \]
    where $\tau(\mathfrak{Y}_{J^c,I}, \mathfrak{R}^+_{J^c,I}, \mathfrak{l}, \omega_{\Xi_{J^c,I}})$ is the CW-pair torsion defined in \cite[Section 3.4]{FJR}. This follows from \cite[Definition 3.17]{FJR}, \cite[Remark 3.11]{FJR}, and \cite[above Lemma 3.4]{FJR}.
\end{proof}

\subsection{\texorpdfstring{$\Spinc$}{Spin-c} version of \texorpdfstring{$[\BSDA]$}{[BSDA]} independent of reference \texorpdfstring{$\Spinc$}{Spin-c} structures}\label{sec:BSDAIndofRefSpinc}

Continue to assume $Y$ is connected. We require throughout this section that $(Y,\Gamma)$ be equipped with a set of normalized choices $\Xi_{\norm}$ as in Corollary~\ref{cor:CanChooseXiNormalized}, so that we choose a Heegaard diagram of the form $\Hc_{\norm}=\Hc^{\alpha \beta}_{1/2} \cup_{\Zc_1^*} \Hc' \cup_{\Zc_0^*} \Hc^{\beta \alpha}_{1/2}$ for some $\beta$-$\beta$ bordered sutured Heegaard diagram $\Hc'$. We want to use the map 
\[
[\BSDA(Y,\Gamma;\Xi_{\mathfrak{s}})]^{\Z[H]}_{\comb,\mathrm{concrete}}
\]
of Definition~\ref{def:BSDADepOnRefSpinc}, depending on reference $\Spinc$ structures, to define a map that is independent of reference $\Spinc$ structures. 

\begin{definition}\label{def:BSDAIndOfRefSpinc}
    Define 
    \[              
    [\BSDA(Y,\Gamma;\Xi_{\norm})]^{\Z[H]}_{\comb} \colon \wedge^*_{\Z[H]} H_1(\widehat{F}_0, \widehat{S}^+_0) \to \wedge^*_{\Z[H]} H_1(\widehat{F}_1, \widehat{S}^+_1)
    \]
    as follows; first, for each $\beta$ circle and $\alpha$ circle of $\Hc_{\norm} = \Hc^{\alpha \beta}_{1/2} \cup_{\Zc_1^*} \Hc' \cup_{\Zc_0^*} \Hc^{\beta \alpha}_{1/2}$, corresponding to the 1-cells and 2-cells of $\mathfrak{Y}$ relative to $\mathfrak{R}^+$, choose a lifted 1-cell or 2-cell of $\widehat{\mathfrak{Y}}$. Now, for each $I$ and $J$, consider the relative 1-cells and 2-cells of the CW pair $(\mathfrak{Y}_{J^c I}, \mathfrak{R}^+_{J^c I})$, corresponding to the $\beta$ and $\alpha$ circles of the diagram $\Hc_{J^c,I}$. We have an inclusion of CW pairs $(\mathfrak{Y},\mathfrak{R}^+) \subset (\mathfrak{Y}_{J^c,I}, \mathfrak{R}^+_{J^c,I})$; under this inclusion, all 1-cells of $\mathfrak{Y}_{J^c,I}$ relative to $\mathfrak{R}^+_{J^c,I}$ come from 1-cells of $\mathfrak{Y}$ relative to $\mathfrak{R}^+$. The 2-cells of $\mathfrak{Y}_{J^c,I}$ relative to $\mathfrak{R}^+_{J^c,I}$ that do not come from 2-cells of $\mathfrak{Y}$ relative to $\mathfrak{R}^+$ are in bijection with the $\alpha$ circles of $\Hc_{J^c,I}$ that correspond to $\alpha$ arcs of $\Hc_{\norm}$, i.e. they are in bijection with the elements of $J^c \sqcup I$. Each such $\alpha$ circle of $\Hc_{J^c,I}$ has a partner $\beta$ circle of $\Hc_{J^c,I}$ corresponding to the same element of $J^c \sqcup I$; the partner $\beta$ circle corresponds to a 1-cell of $\mathfrak{Y}$ relative to $\mathfrak{R}^+$, so we have already chosen a lift of this 1-cell to $\widehat{\mathfrak{Y}} \subset \widehat{\mathfrak{Y}}_{J^c,I}$. There is a unique lift of the corresponding 2-cell of $\mathfrak{Y}_{J^c,I}$ to a 2-cell of $\widehat{\mathfrak{Y}}_{J^c,I}$ such that in the cellular differential $d_2$ for the CW pair $(\widehat{\mathfrak{Y}}_{J^c,I}, \widehat{\mathfrak{R}}^+_{J^c,I})$, the matrix entry from the 2-cell to the 1-cell is $-1$ (if the cells are on the outgoing / left side) or $+1$ (if the cells are on the incoming / right side). 

    The chosen lifts, viewed up to equivalence, give us an element of $\mathrm{Lift}(\mathfrak{Y}_{J^c,I}, \mathfrak{R}^+_{J^c,I})$ and thus, by Proposition~\ref{prop:SpincAndLifts}, an element $\s_{\mathrm{ref}}$ of $\Spinc(Y,\partial Y, -v_{J^c,I})$. Let $\s_{J^c,I} := -\s_{\mathrm{ref}} \in \Spinc(Y,\partial Y, v_{J^c,I})$. These reference elements $\s_{J^c,I}$ give us, by Definition~\ref{def:BSDADepOnRefSpinc}, a map
    \[ [\BSDA(Y,\Gamma;\Xi_{\mathfrak{s}})]^{\Z[H]}_{\comb;\, \mathrm{concrete}} \colon \wedge^* H_1(F_0,S^+_0) \otimes_{\Z} \Z[H] \to \wedge^* H_1(F_1,S^+_1) \otimes_{\Z} \Z[H]
    \]
    The chosen lifts (not viewed up to equivalence) also give us identifications
    \[
    \Phi_i \colon \wedge^* H_1(F_i,S^+_i) \otimes_{\Z} \Z[H] \stackrel{\cong}{\longrightarrow} \wedge^*_{\Z[H]} H_1(\widehat{F}_i, \widehat{S}^+_i).
    \]
    as in Section~\ref{sec:LiftedSurfacesHomology}. We let 
    \[              [\BSDA(Y,\Gamma;\Xi_{\norm})]^{\Z[H]}_{\comb} \colon \wedge^*_{\Z[H]} H_1(\widehat{F}_0, \widehat{S}^+_0) \to \wedge^*_{\Z[H]} H_1(\widehat{F}_1, \widehat{S}^+_1)
    \]
    be the map obtained from $[\BSDA(Y,\Gamma;\Xi_{\mathfrak{s}})]^{\Z[H]}_{\comb;\, \mathrm{concrete}}$ by pre- and post-composing with these identifications. 
\end{definition}

\begin{proposition}
    The map $[\BSDA(Y,\Gamma;\Xi_{\norm})]^{\Z[H]}_{\comb}$ of the previous definition is independent of the choices of lifts for cells of $\mathfrak{Y}$ relative to $\mathfrak{R}^+$, up to multiplication by an element of $\pm H$. 
\end{proposition}

\begin{proof}
   The data of $[\BSDA(Y,\Gamma;\Xi_{\norm})]^{\Z[H]}_{\comb}$ that relies on our chosen lifts are the reference $\Spinc$ structures $\s_{J^c,I}$, which enable us to define the map $[\BSDA(Y,\Gamma;\Xi_{\mathfrak{s}})]^{\Z[H]}_{\comb,\mathrm{concrete}}$, and the $\Z[H]$ bases for $H_1(\widehat F_i,\widehat S_i^+)$, which enable the identifications $\Phi_i$ per Section~\ref{sec:LiftedSurfacesHomology}. By \cite[Definition~3.1]{FJR}, if $\widehat{e}$ is a particular choice of lift (not viewed up to equivalence) of a cell $e$, any other lift differs by a unique element of $H$. It suffices to analyze the effects of changing the lift of only one cell; this cell can either be (1) a 2-cell corresponding to an $\alpha$-circle of $\Hc$, (2) a 1-cell corresponding to a $\beta$-circle of $\Hc'$, (3) a 1-cell corresponding to a $\beta$-circle of $\Hc$ glued from an outgoing $\beta$-arc of $\Hc'$, or (4) a 1-cell corresponding to a $\beta$-circle of $\Hc$ glued from an incoming $\beta$-arc of $\Hc'$. Write $\Phi_i'$ and $\mathfrak{s}'_{J^c,I}$ for the new data determined by changing a particular chosen lift in each case, and let $[\BSDA(Y,\Gamma;\Xi_{\norm})]^{\Z[H]}_{\comb;\,\mathrm{new}}$ be the resulting map.

   \emph{Case (1): 2-cells.}: First consider lifts $\widehat{e}$ of 2-cells $e$ which come from $\alpha$ circles of $\Hc$. By Section~\ref{sec:LiftedSurfacesHomology}, these lifts do not interact with the $\Z[H]$-basis for $H_1(\widehat F_i,\widehat S_i^+)$, so $\Phi'_i=\Phi_i$ is automatic. Furthermore, if we replace $\widehat{e}$ with $h \widehat{e}$ for some fixed $h \in H$, then for all $I,J$, the resulting element of $\Lift(\mathfrak{Y}_{J^c,I},\mathfrak{R}^+_{J^c,I})$ gets multiplied by $h$. Thus the reference element $\s_{\mathrm{ref}}$ of $\Spinc(Y,\partial Y, -v_{J^c,I})$ gets multiplied by $h$, so the reference element $\s_{J^c,I}$ of $\Spinc(Y,\partial Y, v_{J^c,I})$ gets multiplied by $h^{-1}$. In additive notation, we have $\s'_{J^c,I} = \s_{J^c,I} - h$, so for each generator $\x$ of $\Hc$, we have $\mathrm{gr}'_H(\x) = \s^{\rel}(\x) - \s'_{J^c,I} = \s^{\rel(\x)} - (\s_{J^c,I} - h) = \mathrm{gr}_H(\x) + h$. Switching back to multiplicative notation, we get
   \[ [\BSDA(Y,\Gamma;\Xi_{\mathfrak{s}'})]^{\Z[H]}_{\comb;\,\mathrm{concrete}}=h \cdot [\BSDA(Y,\Gamma;\Xi_{\mathfrak{s}})]^{\Z[H]}_{\comb;\,\mathrm{concrete}},
   \]
   and since $\Phi'_i = \Phi_i$ we get 
   \[ [\BSDA(Y,\Gamma;\Xi_{\norm})]^{\Z[H]}_{\comb;\,\mathrm{new}}=h \cdot [\BSDA(Y,\Gamma;\Xi_{\norm})]^{\Z[H]}_{\comb}.
   \]   

   \emph{Case (2): 1-cells corresponding to $\beta$-circles of $\Hc'$.} This case is analogous to case (1), except that if we replace $\widehat{e}$ with $h\widehat{e}$ for some $h \in H$, then for all $I,J$, the resulting element of $\Lift(\mathfrak{Y}_{J^c,I},\mathfrak{R}^+_{J^c,I})$ gets multiplied by $h^{-1}$. We get
   \[ [\BSDA(Y,\Gamma;\Xi_{\norm})]^{\Z[H]}_{\comb;\,\mathrm{new}}=h^{-1} \cdot [\BSDA(Y,\Gamma;\Xi_{\norm})]^{\Z[H]}_{\comb}.
   \] 

   \emph{Case (3): 1-cells corresponding to outgoing $\beta$-arcs of $\Hc'$.} In this case, we consider two cases for $J$: either $J$ contains the outgoing $\beta$-arc, or it does not. First suppose it does not, so the outgoing $\beta$-arc is in $J^c$. In this case we have $\Phi'_1(\gamma^{\out}_{J} \otimes 1_{\Z[H]}) = \Phi_1(\gamma^{\out}_{J} \otimes 1_{\Z[H]})$, and we still have $\Phi'_0 = \Phi_0$. For any $I$, there is also a 2-cell $e'$ of $\mathfrak{Y}_{J^c,I} \setminus \mathfrak{R}^+_{J^c,I}$ corresponding to an $\alpha$-circle of $\Hc_{J^c,I}$ that intersects the $\beta$-circle once in the Zarev cap piece of the diagram. The lift $\widehat{e'}$ that we use for $e'$ is determined by the lift $\widehat{e}$ for $e$ and the condition that the matrix entry from $\widehat{e'}$ to $\widehat{e}$ in the cellular boundary map $\partial_2$ for $C_*(\mathfrak{Y}_{J^c,I},\mathfrak{R}^+_{J^c,I})$ is $-1$. Thus, since we are changing $\widehat{e}$ to $h\widehat{e}$, we are also changing $\widehat{e'}$ to $h\widehat{e'}$. The reference element of $\Lift(\mathfrak{Y}_{J^c,I},\mathfrak{R}^+_{J^c,I})$ is unchanged (multiplying the lift of a 2-cell by $h$ multiplies the element of $\Lift(\mathfrak{Y}_{J^c,I},\mathfrak{R}^+_{J^c,I})$ by $h$, while multiplying the lift of a 1-cell by $h$ multiplies the element of $\Lift(\mathfrak{Y}_{J^c,I},\mathfrak{R}^+_{J^c,I})$ by $h^{-1}$). It follows that the coefficient of $[\BSDA(Y,\Gamma;\Xi_{\norm})]^{\Z[H]}_{\comb;\,\mathrm{new}}(\Phi_0(\gamma^{\inrm}_I \otimes 1_{\Z[H]}))$ on $\Phi_1(\gamma^{\out}_J \otimes 1_{\Z[H]})$ agrees with the coefficient of $[\BSDA(Y,\Gamma;\Xi_{\norm})]^{\Z[H]}_{\comb}(\Phi_0(\gamma^{\inrm}_I \otimes 1_{\Z[H]}))$ on $\Phi_1(\gamma^{\out}_J \otimes 1_{\Z[H]})$.

   Now consider $J$ such that the outgoing $\beta$ arc is in $J$. In this case we have $\Phi'_1(\gamma^{\out}_J \otimes 1_{\Z[H]}) = h \cdot \Phi_1(\gamma^{\out}_J \otimes 1_{\Z[H]})$, and we still have $\Phi'_0 = \Phi_0$. For any $I$, our reference element of $\Lift(\mathfrak{Y}_{J^c,I},\mathfrak{R}^+_{J^c,I})$ gets multiplied by $h^{-1}$ since we multiplied the lift of a 1-cell by $h$ with no compensating change to the lifts of 2-cells. We get that the coefficient $C'_{\mathrm{concrete}}$ of $[\BSDA(Y,\Gamma;\Xi_{\mathfrak{s}'})]^{\Z[H]}_{\comb;\,\mathrm{concrete}} (\gamma^{\inrm}_I \otimes 1_{\Z[H]})$ on $\gamma^{\out}_J \otimes 1_{\Z[H]}$ is equal to $h^{-1}$ times the coefficient $C_{\mathrm{concrete}}$ of $[\BSDA(Y,\Gamma;\Xi_{\mathfrak{s}})]^{\Z[H]}_{\comb;\,\mathrm{concrete}} (\gamma^{\inrm}_I \otimes 1_{\Z[H]})$ on $\gamma^{\out}_J \otimes 1_{\Z[H]}$.

   The coefficient of $[\BSDA(Y,\Gamma;\Xi_{\norm})]^{\Z[H]}_{\comb;\,\mathrm{new}}(\Phi_0(\gamma^{\inrm}_I \otimes 1_{\Z[H]}))$ on $\Phi'_1(\gamma^{\out}_J \otimes 1_{\Z[H]})$ is, by definition, the coefficient of $[\BSDA(Y,\Gamma;\Xi_{\mathfrak{s}'})]^{\Z[H]}_{\comb;\,\mathrm{concrete}} (\gamma^{\inrm}_I \otimes 1_{\Z[H]})$ on $\gamma^{\out}_J \otimes 1_{\Z[H]}$, i.e. it is $C'_{\mathrm{concrete}} = h^{-1} \cdot C_{\mathrm{concrete}}$. Since $\Phi'_1(\gamma^{\out}_J \otimes 1_{\Z[H]}) = h \cdot \Phi_1(\gamma^{\out}_J \otimes 1_{\Z[H]})$, the coefficient of $[\BSDA(Y,\Gamma;\Xi_{\norm})]^{\Z[H]}_{\comb;\,\mathrm{new}}(\Phi_0(\gamma^{\inrm}_I \otimes 1_{\Z[H]}))$ on $\Phi_1(\gamma^{\out}_J \otimes 1_{\Z[H]})$ is $h \cdot h^{-1} \cdot C_{\mathrm{concrete}} = C_{\mathrm{concrete}}$, the same as the coefficient of $[\BSDA(Y,\Gamma;\Xi_{\norm})]^{\Z[H]}_{\comb}(\Phi_0(\gamma^{\inrm}_I \otimes 1_{\Z[H]}))$ on $\Phi_1(\gamma^{\out}_J \otimes 1_{\Z[H]})$.

   Putting these two cases for $J$ together, we get
   \[ [\BSDA(Y,\Gamma;\Xi_{\norm})]^{\Z[H]}_{\comb;\,\mathrm{new}}=[\BSDA(Y,\Gamma;\Xi_{\norm})]^{\Z[H]}_{\comb}.
   \] 

   \emph{Case (4): 1-cells corresponding to incoming $\beta$-arcs of $\Hc'$.} Here we consider two cases for $I$: either $I$ contains the incoming $\beta$ arc, or it does not. First suppose the incoming $\beta$ arc is in $I$. In this case we have $\Phi'_0(\gamma^{\inrm}_I \otimes 1_{\Z[H]}) = h \cdot \Phi_0(\gamma^{\inrm}_I \otimes 1_{\Z[H]})$, and we also have $\Phi'_1 = \Phi_1$. For any $J$, there is also a 2-cell of $\Lift(\mathfrak{Y}_{J^c,I},\mathfrak{R}^+_{J^c,I})$ whose lift also gets multiplied by $h$ as in the first part of case (3), so the reference element of $\Lift(\mathfrak{Y}_{J^c,I},\mathfrak{R}^+_{J^c,I})$ is unchanged for every $J$. We get
   \[
   [\BSDA(Y,\Gamma;\Xi_{\mathfrak{s}'})]^{\Z[H]}_{\comb;\,\mathrm{concrete}} (\gamma^{\inrm}_I \otimes 1_{\Z[H]}) = [\BSDA(Y,\Gamma;\Xi_{\mathfrak{s}})]^{\Z[H]}_{\comb;\,\mathrm{concrete}} (\gamma^{\inrm}_I \otimes 1_{\Z[H]}).
   \]
   It follows that
   \begin{align*}
   &[\BSDA(Y,\Gamma;\Xi_{\norm})]^{\Z[H]}_{\comb;\,\mathrm{new}} (\Phi_0(\gamma^{\inrm}_I \otimes 1_{\Z[H]})) \\ &=\Phi'_1([\BSDA(Y,\Gamma;\Xi_{\s'})]^{\Z[H]}_{\comb;\,\mathrm{concrete}} ((\Phi'_0)^{-1}(\Phi_0(\gamma^{\inrm}_I \otimes 1_{\Z[H]})))) \\
   &= \Phi'_1([\BSDA(Y,\Gamma;\Xi_{\s'})]^{\Z[H]}_{\comb;\,\mathrm{concrete}} (h^{-1} \cdot (\gamma^{\inrm}_I \otimes 1_{\Z[H]})))) \\
   &= h^{-1} \cdot \Phi'_1([\BSDA(Y,\Gamma;\Xi_{\s'})]^{\Z[H]}_{\comb;\,\mathrm{concrete}} (\gamma^{\inrm}_I \otimes 1_{\Z[H]}))) \\
   &= h^{-1} \cdot \Phi_1([\BSDA(Y,\Gamma;\Xi_{\s})]^{\Z[H]}_{\comb;\,\mathrm{concrete}} (\gamma^{\inrm}_I \otimes 1_{\Z[H]}))) \\
   &= h^{-1} \cdot [\BSDA(Y,\Gamma;\Xi_{\norm})]^{\Z[H]}_{\comb} (\Phi_0(\gamma^{\inrm}_I \otimes 1_{\Z[H]})).
   \end{align*}

   Now suppose the incoming $\beta$ arc is not in $I$; we have $\Phi'_0(\gamma^{\inrm}_I \otimes 1_{\Z[H]}) = \Phi_0(\gamma^{\inrm}_I \otimes 1_{\Z[H]})$, and we also have $\Phi'_1 = \Phi_1$. For any $J$, our reference element of $\Lift(\mathfrak{Y}_{J^c,I},\mathfrak{R}^+_{J^c,I})$ gets multiplied by $h^{-1}$. As in case (2), we get
   \[
   [\BSDA(Y,\Gamma;\Xi_{\mathfrak{s}'})]^{\Z[H]}_{\comb;\,\mathrm{concrete}} (\gamma^{\inrm}_I \otimes 1_{\Z[H]}) = h^{-1} \cdot [\BSDA(Y,\Gamma;\Xi_{\mathfrak{s}})]^{\Z[H]}_{\comb;\,\mathrm{concrete}} (\gamma^{\inrm}_I \otimes 1_{\Z[H]}).
   \]
   It follows that
   \begin{align*}
   &[\BSDA(Y,\Gamma;\Xi_{\norm})]^{\Z[H]}_{\comb;\,\mathrm{new}} (\Phi_0(\gamma^{\inrm}_I \otimes 1_{\Z[H]})) \\ &=\Phi'_1([\BSDA(Y,\Gamma;\Xi_{\s'})]^{\Z[H]}_{\comb;\,\mathrm{concrete}} ((\Phi'_0)^{-1}(\Phi_0(\gamma^{\inrm}_I \otimes 1_{\Z[H]})))) \\
   &= \Phi'_1([\BSDA(Y,\Gamma;\Xi_{\s'})]^{\Z[H]}_{\comb;\,\mathrm{concrete}} (\gamma^{\inrm}_I \otimes 1_{\Z[H]}))) \\
   &= \Phi'_1(h^{-1} \cdot [\BSDA(Y,\Gamma;\Xi_{\s})]^{\Z[H]}_{\comb;\,\mathrm{concrete}} (\gamma^{\inrm}_I \otimes 1_{\Z[H]}))) \\
   &= h^{-1} \cdot \Phi_1([\BSDA(Y,\Gamma;\Xi_{\s})]^{\Z[H]}_{\comb;\,\mathrm{concrete}} (\gamma^{\inrm}_I \otimes 1_{\Z[H]}))) \\
   &= h^{-1} \cdot [\BSDA(Y,\Gamma;\Xi_{\norm})]^{\Z[H]}_{\comb} (\Phi_0(\gamma^{\inrm}_I \otimes 1_{\Z[H]})).
   \end{align*}

   Putting these two cases for $I$ together, we get
   \[ [\BSDA(Y,\Gamma;\Xi_{\norm})]^{\Z[H]}_{\comb;\,\mathrm{new}}= h^{-1} \cdot [\BSDA(Y,\Gamma;\Xi_{\norm})]^{\Z[H]}_{\comb}.
   \] 
\end{proof}

\section{Sutured Alexander functors with twisted coefficients}\label{sec:AlexFunctorsInOtherSettings}

Let $H=H_1(Y)$ and $G=H/\tors(H)$. In this section we study versions of the sutured Alexander functor of Section~\ref{sec:SuturedAlexanderOverZ} over the rings $\Z[G]$ and $\Q[H]$. Both cases have their own advantages and disadvantages, and both interact with $[\BSDA]$ in a slightly different way. While the arguments in both settings follow the same general pattern, they require alternate
algebraic approaches, so we treat them separately.

The $\Z[G]$ case follows as a straightforward generalization to the case over $\Z$ and provides the closest link to the classical Alexander polynomial. The $\Q[H]$ version is algebraically our most flexible one, allowing for torsion in $H$. However, when $H$ has torsion, the theory no longer falls within the scope of \cite{FMFunctorial} and should be viewed as strictly a generalization of their constructions. 

\begin{remark}
    Throughout this section we continue to use the term ``Alexander functor'' for maps that parallel $\mathsf{A}_{\Z}(Y,\Gamma)$, even though we are not making any claim to their being strictly functorial (see Remark~\ref{rem:CohomologyClasses}). Rather than defining a functor on a cobordism category, we construct invariants $\mathsf{A}_{\Z[G]}(Y,\Gamma)$, $\mathsf{A}_{\Z[H]}(Y,\Gamma)$, $\mathsf{A}_{\F_{\chi_i}[G]}(Y,\Gamma)$, and $\mathsf{A}_{\Q[H]}(Y,\Gamma)$ each defined (up to an appropriate notions of indeterminacy) for an individual cobordism $Y$, with no composition properties studied. 
\end{remark}

\subsection{The Alexander functor over $\Z[G]$}\label{sec:AlexanderFunctorZG}

Since $G$ is free abelian and finitely generated, $\Z[G]$ is an integral domain, so the construction proceeds in close analogy with the case over $\Z$. This case is also the one most directly related to the Alexander polynomial: when $Y$ is the complement of a knot in $S^3$ with meridional sutures $\Gamma$, so $G \cong \Z$ and $\Z[G] \cong \Z[t,t^{-1}]$, the map $\mathsf{A}_{\Z[G]}$ will recover the Alexander polynomial of the knot---see \cite[Proposition~3.3]{FMFunctorial}. 

Let $(Y,\Gamma)$ be a connected sutured cobordism from $(F_0,\Lambda_0)$ to $(F_1,\Lambda_1)$. Fix a general (non-normalized) set of choices $\Xi$ as in Definition~\ref{def:BSDAZComb}, including an $\alpha$-$\alpha$ bordered Heegaard diagram $\Hc$ for $(Y,\Gamma)$; let $a$ be the number of $\alpha$-circles of $\Hc$ and let $b$ the the number of $\beta$-circles of $\Hc$. We set $d = -\chi(Y,R^+)$. Recall that 
\[
d = \rank H_1(Y,R^+) - \rank H_2(Y,R^+)=b-a.
\]

Following Section~\ref{sec:CWCoveringSpaceGeneral}, we have CW decompositions $\overline{\mathfrak{Y}}$ and $\overline{\mathfrak{R}}^+$ of spaces homotopy equivalent to $Y$ and $R^+$. After choosing lifts as in Sectin~\ref{sec:CWCoveringSpaceGeneral}, we have preferred basis elements $\overline{e}_i$ of $C^{\cell}_1(\overline{\mathfrak{Y}},\overline{\mathfrak{R}}^+)$ and $C^{\cell}_2(\overline{\mathfrak{Y}},\overline{\mathfrak{R}}^+)$, and in these bases we have a deficiency-$d$ presentation matrix $\overline{M}_{\Hc}$ for $H_1(\overline{Y},\overline{R}^+)$ with kernel $H_2(\overline{Y},\overline{R}^+)$. This implies the following fact; we omit the proof. 

\begin{lemma}\label{lem:InjectivePresentationMatrixZ[G]}
    If $H_2(\overline{Y},\overline{R}^+)=0$, then the deficiency-$d$ presentation matrix $\overline{M}_{\Hc}$ for $H_1(\overline{Y},\overline{R}^+)$ is injective; in particular, $d \geq 0$. If $H_2(\overline{Y},\overline{R}^+) \neq 0$, the matrix $\overline{M}_{\Hc}$ is never injective. 
\end{lemma}

\begin{definition}\label{def:AlexFunctionZ[G]}
    Let $d = -\chi(Y,R^+)$. Define the \emph{$\Z[G]$-valued Alexander function}
    \[
    \mathcal{A}^{\Z[G]}_{Y,\Gamma} \colon \wedge^d_{\Z[G]} H_1(\overline{Y}, \overline{R}^+) \to \Z[G],
    \]
    well-defined up to (global) multiplication by $\pm G \subset \Z[G]$, as follows. If $H_2(\overline{Y},\overline{R}^+)$ is nonzero, define $\mathcal{A}^{\Z[G]}_{Y,\Gamma}$ to be the zero function. Otherwise, $H_1(\overline{Y}, \overline{R}^+)$ admits a deficiency-$d$ presentation matrix $\overline{M}$; in this case define $\mathcal{A}^{\Z[G]}_{Y,\Gamma}(u_1 \wedge \cdots \wedge u_d)$ as follows. Choose any expression $\overline{u}_i$ of $u_i$ as a $\Z[G]$-linear combination of the generators, form a square matrix from $\overline{M}$ by adding new columns $\overline{u}_1, \ldots, \overline{u}_d$ to the right, and take the determinant, an element of $\Z[G]$. 
\end{definition}

\begin{lemma}\label{lem:AlexanderFunctionZ[G]WellDefined}
    The Alexander function $\A^{\Z[G]}_{Y,\Gamma}$ is independent of the choice of representatives $\overline{u}_1, \ldots, \overline{u}_d$ used to compute the determinant. 
\end{lemma}

\begin{proof}
    The proof of Lemma~\ref{lem:AlexanderFunctionZWellDefined} applies unchanged in this setting; one simply works over $\Z[G]$ rather than $\Z$. 
\end{proof}

\begin{proposition}\label{prop:AlexanderFunctionZ[G]IndependentOfM}
    Up to overall multiplication by $\pm G \subset \Z[G]$, the Alexander function $\mathcal{A}^{\Z[G]}_{Y,\Gamma}$ of Definition~\ref{def:AlexFunctionZ[G]} only depends on the finitely presented $\Z[G]$-module $H_1(\overline{Y}, \overline{R}^+)$ and the deficiency $d$, and does not depend on the specific choice of injective presentation matrix $\overline{M}$. 
\end{proposition}

\begin{proof}
    With $\Z[G]$ coefficients, the proof of Proposition~\ref{prop:AlexanderFunctionZIndependentOfM} largely carries over, with only minor modifications. For the steps ($\A^{\Z[G]}_{Y,\Gamma,\overline{M}} \sim \A^{\Z[G]}_{Y,\Gamma,\overline{M}_1}$) and ($\A^{\Z[G]}_{Y,\Gamma,\overline{M}_2} \sim \A^{\Z[G]}_{Y,\Gamma,\widetilde{\overline{M}}}$): over any commutative ring the determinants of column-swapped matrices differ only by some sign. Since $\pm 1 \in \pm G$, this same argument still holds over $\Z[G]$. 
    
    For ($\A^{\Z[G]}_{Y,\Gamma,\overline{M}_1} \sim \A^{\Z[G]}_{Y,\Gamma,\overline{M}_2}$): when we get to the last step where
    \[
    \det[M_1|\overline{u}_1 \cdots \overline{u}_d] =u \cdot \det [M_2|\overline{u}_1 \cdots \overline{u}_d],
    \]
    with $u \in (\Z[G])^{\times}$, since the group of units in $\Z[G]$ is $\pm G$, it follows that the Alexander functions differ by multiplication by an element of $\pm G$ that is independent of $u_1 \wedge \cdots \wedge u_d$.
\end{proof}

For the remainder of this section, suppose that we equip $(Y,\Gamma)$ with a set of normalized choices $\Xi_{\norm}$ as in Corollary~\ref{cor:NormalizedHeegaardDiagrams}. In this case, we are in the setting of Section~\ref{sec:CWCoveringSpaceNormalized}, and so we have modules $H_1(\overline{F}_0,\overline{S}^+_0)$ and $H_1(\overline{F}_1,\overline{S}^+_1)$ with preferred cellular bases $\overline{e}^{\inrm}$ and $\overline{e}^{\out}$ respectively (once we choose lifts of cells as above). Denote by $\overline{i}_*$ either of the maps 
\[
H_1(\overline{F}_i,\overline{S}^+_i) \to H_1(\overline{Y},\overline{R}^+)
\]
induced by the inclusion of pairs $\overline{i}:(\overline{F}_i,\overline{S}^+_i) \to (\overline{Y},\overline{R}^+)$ for $i=0,1$. 

\begin{proposition}\label{prop:AlexanderFunctorZ[G]Defn}
    There is a $\Z[G]$-linear map
    \[
    \mathsf{A}_{\Z[G]}(Y,\Gamma) \colon \wedge^*_{\Z[G]} H_1(\overline{F}_0,\overline{S}^+_0) \to \wedge^*_{\Z[G]} H_1(\overline{F}_1,\overline{S}^+_1),
    \]
    unique up to multiplication by $\pm G$ and homogeneous of degree $c = n_1 + \chi(Y,R^+)$, such that
    \[
    \omega(\wedge(\mathsf{A}_{\Z[G]}(Y,\Gamma) \otimes \id)(x \otimes y)) = \A^{\Z[G]}_{Y,\Gamma}\left(\overline{i}_* x \wedge \overline{i}_* y \right),
    \]
    where $x \in \wedge^*_{\Z[G]} H_1(\overline{F}_0,\overline{S}^+_0)$, $y \in \wedge^*_{\Z[G]} H_1(\overline{F}_1,\overline{S}^+_1)$, the map
    \[
    \wedge \colon \wedge^p_{\Z[G]} H_1(\overline{F}_1,\overline{S}^+_1) \otimes \wedge^q_{\Z[G]} H_1(\overline{F}_1,\overline{S}^+_1) \to \wedge^{n_1}_{\Z[G]} H_1(\overline{F}_1,\overline{S}^+_1)
    \]
    sends $z \otimes w$ to $z \wedge w$ when $p + q = n_1$ and is zero otherwise, $\omega$ denotes an arbitrary volume form on the free $\Z[G]$-module $ H_1(\overline{F}_1,\overline{S}^+_1)$, the tensor product is computed according to the super-sign rule
    \[
    (\mathsf{A}_{\Z[G]}(Y,\Gamma) \otimes \id)(x \otimes y)=(-1)^{|\mathsf{A}_{\Z[G]}(Y,\Gamma)||y|} \mathsf{A}_{\Z[G]}(Y,\Gamma)(x) \otimes y
    \]
    with $|\mathsf{A}_{\Z[G]}(Y,\Gamma)| = c$, and the Alexander functor $\A^{\Z[G]}_{Y,\Gamma}$ is defined to be zero on inputs in $\wedge^{d'} H_1(\overline{Y},\overline{R}^+)$ for $d' \neq d$.
\end{proposition}

\begin{proof}
    After making the following adjustments, the proof of Proposition~\ref{prop:AlexanderFunctorZDefn} carries over: replace each $F_i$, $S_i^+$, $Y$ and $R^+$ with $\overline{F}_i$, $\overline{S}_i^+$, $\overline{Y}$ and $\overline{R}^+$ respectively, work with $\Z[G]$-coefficients, replace $i_*$ by $\overline{i}_*$, and replace the $\Z$-basis elements $\gamma^{\mathrm{in}}_I$ for $\wedge^* H_1(F_0, S^+_0)$ and $\gamma^{\mathrm{out}}_J$ for $\wedge^* H_1(F_1, S^+_1)$ by our preferred $\Z[G]$-basis elements $\overline{e}_I^{\inrm}$ for $\wedge^*_{\Z[G]} H_1(\overline{F}_0,\overline{S}^+_0)$ and $\overline{e}_J^{\out}$ for $\wedge^*_{\Z[G]} H_1(\overline{F}_1,\overline{S}^+_1)$. For convenience we can choose the volume form $\omega$ such that
    \[
    \omega(\overline{e}^{\out}_1 \wedge \cdots \wedge \overline{e}^{\out}_{n_1}) = 1 \in \Z[G]; 
    \]any other choice of volume form differs from $\omega$ by a unit in $\Z[G]$, i.e. an element of $\pm G$, and so the result follows. 
\end{proof}

\begin{remark}\label{rem:AlexanderZ[G]SpecialCaseOfFM}
    Note that our Alexander function $\mathcal{A}^{\Z[G]}_{Y,\Gamma}$ and Alexander functor $\mathsf{A}_{\Z[G]}(Y,\Gamma)$ specialize to the Alexander function $\mathcal{A}^M_{\varphi}$ of \cite[Definition~2.2]{FMFunctorial} and Alexander functor $\mathsf{A}(M,\varphi)$ of \cite[Section~2.2]{FMFunctorial} respectively. Indeed, we recover these as special cases by assuming the topological restrictions of Remark~\ref{rem:AlexanderZSpecialCaseOfFM} and setting Florens-Massuyeau's group ``$G$'' equal to our free abelian group $G=H/\tors(H)$. 
\end{remark}

Write 
\[
[\BSDA(Y,\Gamma;\Xi_{\norm})]^{\Z[G]}_{\mathrm{comb}}
\]
for the extension of scalars for $[\BSDA(Y,\Gamma;\Xi_{\norm})]^{\Z[H]}_{\mathrm{comb}}$ along $\Z[H] \to \Z[G]$. That is, 
\[
[\BSDA(Y,\Gamma;\Xi_{\norm})]^{\Z[G]}_{\mathrm{comb}} \coloneq [\BSDA(Y,\Gamma;\Xi_{\norm})]^{\Z[H]}_{\mathrm{comb}} \otimes_{\Z[H]} 1_{\Z[G]},
\]
which is a map 
\[
[\BSDA(Y,\Gamma;\Xi_{\norm})]^{\Z[G]}_{\mathrm{comb}} \colon \wedge^*_{\Z[H]} H_1(\widehat{F}_0, \widehat{S}^+_0) \otimes_{\Z[H]} \Z[G] \to \wedge^*_{\Z[H]} H_1(\widehat{F}_1, \widehat{S}^+_1) \otimes_{\Z[H]} \Z[G].
\]
Section~\ref{sec:LiftedSurfacesHomology} established that $H_1(\widehat{F}_i, \widehat{S}^+_i)$ is freely generated over $\Z[H]$, so we have
\[
\wedge^*_{\Z[H]} H_1(\widehat{F}_i, \widehat{S}^+_i) \otimes_{\Z[H]} \Z[G] \cong \wedge^*_{\Z[G]} \left( H_1(\widehat{F}_i, \widehat{S}^+_i) \otimes_{\Z[H]} \Z[G]\right). 
\]
Furthermore, as in the beginning of this section we have an identification
\begin{equation}\label{eq:OverlineFFromWidehatF}
    H_1(\widehat{F}_i, \widehat{S}^+_i) \otimes_{\Z[H]} \Z[G] \cong H_1(\overline{F}_i, \overline{S}^+_i).
\end{equation}
Putting this all together, we may equivalently regard $[\BSDA(Y,\Gamma;\Xi_{\norm})]^{\Z[G]}_{\mathrm{comb}}$ as a map
\[
\wedge^*_{\Z[G]} H_1(\overline{F}_0,\overline{S}^+_0) \to \wedge^*_{\Z[G]} H_1(\overline{F}_1,\overline{S}^+_1),
\]
similarly to the Alexander functor $\mathsf{A}_{\Z[G]}(Y,\Gamma)$. 

\begin{theorem}\label{thm:BSDAAlexanderZG}
    The map $[\BSDA(Y,\Gamma;\Xi_{\norm})]^{\Z[G]}_{\mathrm{comb}}(Y,\Gamma)$ defined above agrees with $\mathsf{A}_{\Z[G]}(Y,\Gamma)$ up to multiplication by $\pm G$. 
\end{theorem}

\begin{proof}
    The proof closely parallels the proof of Theorem~\ref{thm:BSDAAlexanderZ}. As in the definition of the map $[\BSDA(Y,\Gamma;\Xi_{\norm})]^{\Z[H]}_{\mathrm{comb}}(Y,\Gamma)$, fix lifts of the cells of $\mathfrak{Y} \setminus \mathfrak{R}^+$ to cells of $\widehat{\mathfrak{Y}} \setminus \widehat{\mathfrak{R}}^+$, determining reference $\Spinc$ structures $\s_{J^c,I}$ and identifications
    \[
    \Phi_i \colon \wedge^* H_1(F_i,S^+_i) \otimes_{\Z} \Z[H] \xrightarrow{\cong} \wedge^*_{\Z[H]} H_1(\widehat{F}_i,\widehat{S}^+_i).
    \]
    Let $\gamma^{\inrm}_I$ and $\gamma^{\out}_J$ be basis elements of $\wedge^* H_1(F_0,S^+_0)$ and $\wedge^* H_1(F_1,S^+_1)$ respectively. Let $k = |I|$ and $l = |J|$, and assume that $l = k+c$ where $c = n_1 + \chi(Y,R^+)$ as usual. We have corresponding basis elements $-\widehat{e}^{\inrm}_I := \Phi_0(\gamma^{\inrm}_I \otimes 1_{\Z[H]})$ of $\wedge^*_{\Z[H]} H_1(\widehat{F}_0,\widehat{S}^+_0)$ and $-\widehat{e}^{\out}_J := \Phi_1(\gamma^{\out}_J \otimes 1_{\Z[H]})$ of $\wedge^*_{\Z[H]} H_1(\widehat{F}_1,\widehat{S}^+_1)$. Let $\overline{e}^{\inrm}_I$ and $\overline{e}^{\out}_J$ be the basis elements of $\wedge^*_{\Z[G]} H_1(\overline{F}_0,\overline{S}^+_0)$ and $\wedge^*_{\Z[G]} H_1(\overline{F}_1,\overline{S}^+_1)$ corresponding to $\widehat{e}^{\inrm}_I \otimes 1_{\Z[G]}$ and $\widehat{e}^{\out}_J \otimes 1_{\Z[G]}$ under the identifications of equation~\eqref{eq:OverlineFFromWidehatF}.
    
    Let $\omega_{\Z}$ denote the volume form on $H_1(F_1,S^+_1)$ defined by $\omega_{\Z}(\gamma^{\out}_1 \wedge \cdots \wedge \gamma^{\out}_{n_1}) = +1$. By tensoring with $\Z[H]$, we get a volume form $\omega_{\Z[H];\,\mathrm{concrete}}$ on the free $\Z[H]$-module $H_1(F_1,S^+_1) \otimes_{\Z} \Z[H]$. Via $\Phi_1$, we get a volume form $\omega_{\Z[H]}$ on the free $\Z[H]$-module $H_1(\widehat{F}_1,\widehat{S}^+_1)$. Tensor over $\Z[H]$ with $\Z[G]$ to get a volume form on the free $\Z[G]$-module $H_1(\widehat{F}_1,\widehat{S}^+_1) \otimes_{\Z[H]} \Z[G]$, which we identify with $H_1(\overline{F}_1,\overline{S}^+_1)$ using equation~\eqref{eq:OverlineFFromWidehatF}. We will use the resulting volume form on $H_1(\overline{F}_1,\overline{S}^+_1)$ as our choice of $\omega$.

    We will compute
    \begin{equation}\label{eq:TargetQuantityZ[G]}
        \omega(\wedge([\BSDA(Y,\Gamma;\Xi_{\norm})]^{\Z[G]}_{\comb} \otimes \id ))
    (-\overline{e}_I^{\inrm} \otimes -\overline e^{\out}_{J^c})
    \end{equation}
    and show the result agrees with $\mathcal{A}^{\Z[G]}_{Y,\Gamma}\left(\overline{i}_* (-\overline{e}_I^{\inrm}) \wedge \overline{i}_* (-\overline{e}_{J^c}^{\out})\right)$ up to an overall sign that is independent of $I$ and $J$ (the more general $\pm G$ ambiguity in the statement of the theorem is still necessary because we have fixed choices of lifts to get to this point). Observe that the quantity~\eqref{eq:TargetQuantityZ[G]} agrees with
    \begin{align*}              &\omega\left(\wedge\left(([\BSDA(Y,\Gamma;\Xi_{\norm})]^{\Z[H]}_{\comb}\otimes 1_{\Z[G]})\otimes \id\right)((-\widehat{e}_I^{\inrm} \otimes 1_{\Z[G]}) \otimes (-\widehat e^{\out}_{J^c} \otimes 1_{\Z[G]}))\right) \\
        &=\omega\Big(\wedge\Big(((\Phi_1 \circ[\BSDA(Y,\Gamma;\Xi_{\mathfrak s})]^{\Z[H]}_{\comb; \mathrm{concrete}}\circ \Phi_0^{-1})\otimes 1_{\Z[G]})\otimes \id\Big) \\
        &\qquad (\Phi_0(\gamma^{\inrm}_I \otimes 1_{\Z[H]}) \otimes 1_{\Z[G]}) \otimes (\Phi_1(\gamma^{\out}_{J^c} \otimes 1_{\Z[H]}) \otimes 1_{\Z[G]})\Big) \\
        &= (-1)^{ck} \omega \Big( (\Phi_1([\BSDA(Y,\Gamma;\Xi_{\mathfrak s})]^{\Z[H]}_{\comb;\,\mathrm{concrete}} (\gamma_I^{\inrm}\otimes 1_{\Z[H]})) \otimes 1_{\Z[G]}) \wedge (\Phi_1(\gamma^{\out}_{J^c} \otimes 1_{\Z[H]}) \otimes 1_{\Z[G]}) \Big) \\
        &= (-1)^{ck} \omega \Big(\Phi_1\Big(([\BSDA(Y,\Gamma;\Xi_{\mathfrak s})]^{\Z[H]}_{\comb;\,\mathrm{concrete}} (\gamma_I^{\inrm}\otimes 1_{\Z[H]})) \wedge (\gamma^{\out}_{J^c} \otimes 1_{\Z[H]})\Big) \otimes 1_{\Z[G]} \Big) \\
        &= (-1)^{ck} \left[\omega_{\Z[H]} \Big(\Phi_1\Big(([\BSDA(Y,\Gamma;\Xi_{\mathfrak s})]^{\Z[H]}_{\comb;\,\mathrm{concrete}} (\gamma_I^{\inrm}\otimes 1_{\Z[H]})) \wedge (\gamma^{\out}_{J^c} \otimes 1_{\Z[H]})\Big) \Big) \right]_{\Z[G]} \\
        &= (-1)^{ck} \left[\omega_{\Z[H];\,\mathrm{concrete}} \Big(([\BSDA(Y,\Gamma;\Xi_{\mathfrak s})]^{\Z[H]}_{\comb;\,\mathrm{concrete}} (\gamma_I^{\inrm}\otimes 1_{\Z[H]})) \wedge (\gamma^{\out}_{J^c} \otimes 1_{\Z[H]}) \Big) \right]_{\Z[G]} \\
    \end{align*}
    up to an overall $I,J$-independent sign, where we use our super sign convention in the third step. Thus, up to $I,J$-independent sign, \eqref{eq:TargetQuantityZ[G]} equals $(-1)^{ck + \mathrm{inv}(\sigma_{J J^c \leftrightarrow \std})}$ times the matrix entry $E \in \Z[H]$ of the map $[\BSDA(Y,\Gamma;\Xi_{\mathfrak s})]^{\Z[H]}_{\comb;\,\mathrm{concrete}}$ in column $\gamma_I^{\inrm}\otimes 1_{\Z[H]}$ and row $\gamma_{J}^{\out}\otimes 1_{\Z[H]}$, passed through the quotient map from $\Z[H]$ to $\Z[G]$.
    
    By Corollary~\ref{cor:MatrixEntryAndCellBdryDet}, for $I,J$ such that $(Y,\Gamma_{J^c,I})$ is weakly balanced ($|J|=|I|+c$), the matrix entry $E$ of $[\BSDA(Y,\Gamma;\Xi_{\mathfrak{s}})]^{\Z[H]}_{\comb;\, \mathrm{concrete}}$ in column $\gamma^{\inrm}_I \otimes 1_{\Z[H]}$ and row $\gamma^{\out}_J \otimes 1_{\Z[H]}$ is 
    \[
    E = (-1)^{\mathrm{inv}(\sigma_{JJ^c \leftrightarrow \std}) + n_1 k + ak} \det \widehat{A}
    \]
    up to $I,J$-independent sign, where $\widehat{A}$ is the matrix representing the boundary map $\partial_2$ of the chain complex for $(\widehat{\mathfrak{Y}}_{J^c,I},\widehat{\mathfrak{R}}^+_{J^c,I})$ in the bases determined by any set of lifts that represents the fixed class $\mathfrak{l} \in \Lift(\mathfrak{Y}_{J^c,I},\mathfrak{R}^+_{J^c,I})$ corresponding to the reference $\Spinc$ structure $\s_{\mathrm{ref}} = -\s_{J^c,I} \in \Spinc(Y,\partial Y, -v_{J^c,I})$, where $\s_{J^c,I}$ is being used to compute $[\BSDA(Y,\Gamma;\Xi_{\mathfrak{s}})]^{\Z[H]}_{\comb;\, \mathrm{concrete}}$. As in the proof of Theorem~\ref{thm:BSDAAlexanderZ}, we can write the matrix $\widehat{A}$ as $\kbordermatrix{
      & \alpha^{\out}_{J^c} & \alpha^c & \alpha^{\inrm}_I \\
    \beta^{\out} & -*_{J^c} & * & 0\\
    \beta^c & 0 & * & 0\\
    \beta^{\inrm} & 0 & * & *_I}$, where the middle three blocks are the matrix $\widehat{M}_{\Hc_{\norm}}$ and the entries of $*_{J^c}$ and $*_I$ are all either 0 or 1 as in Theorem~\ref{thm:BSDAAlexanderZ}. It follows that \eqref{eq:TargetQuantityZ[G]} equals
    \[
    (-1)^{ck + n_1k + ak} \det \left[ \kbordermatrix{
      & \alpha^{\out}_{J^c} & \alpha^c & \alpha^{\inrm}_I \\
    \beta^{\out} & -*_{J^c} & * & 0\\
    \beta^c & 0 & * & 0\\
    \beta^{\inrm} & 0 & * & *_I} \right]_{\Z[G]} =: (-1)^{ck + n_1k + ak} \det \widehat{A}_{\Z[G]}
    \]
    up to $I,J$-independent sign.
    
    The proof now breaks up into two cases. Suppose that $H_2(\overline{Y},\overline{R}^+) \neq 0$. Then $\mathsf{A}_{\Z[G]}(Y,\Gamma)$ is the zero map by definition. On the other hand, Lemma~\ref{lem:InjectivePresentationMatrixZ[G]} implies that $\overline{M}_{\Hc_{\norm}}$ is not injective; equivalently the columns of $\overline{M}_{\Hc_{\norm}}$ are $\Z[G]$-linearly dependent. Thus the columns of $\widehat{A}_{\Z[G]}$ are $\Z[G]$-linearly dependent, so since $\Z[G]$ is a domain, we have $\det \widehat{A}_{\Z[G]} = 0$. It follows that \eqref{eq:TargetQuantityZ[G]} is zero for all $I$ and $J$, so $[\BSDA(Y,\Gamma;\Xi_{\norm})]^{\Z[H]}_{\comb}=0$. 
    
    So suppose instead that $H_2(\overline{Y},\overline{R}^+) = 0$. Then $\overline{M}_{\Hc_{\norm}}$ is injective.     By Proposition~\ref{prop:AlexanderFunctionZ[G]IndependentOfM}, we are free to choose the presentation matrix $\overline{M}_{\Hc_{\norm}}$ to compute $\mathcal{A}^{\Z[G]}_{Y,\Gamma}\left(\overline{i}_* (-\overline{e}_I^{\inrm}) \wedge \overline{i}_* (-\overline{e}_{J^c}^{\out})\right)$; by the discussion in Section~\ref{sec:LiftedSurfacesHomology}, we can understand $i_* (-\overline{e}_I^{\inrm})$ and $i_* (-\overline{e}_{J^c}^{\out})$ in terms of
    \[
    i_*: H_1^{\cell}(\overline{\mathfrak{F}}_i,\overline{\mathfrak{S}}^+_i) \to H_1^{\cell}(\overline{\mathfrak{Y}}, \overline{\mathfrak{R}}^+).
    \]
    We get 
    \begin{align*}
    \mathcal{A}^{\Z[G]}_{Y,\Gamma}\left(\overline{i}_* (-\overline{e}_I^{\inrm}) \wedge \overline{i}_* (-\overline{e}_{J^c}^{\out})\right) &= \det \left[\kbordermatrix{
       & \alpha^c & \alpha^{\inrm}_I & \alpha^{\out}_{J^c} \\
    \beta^{\out}  & * & 0 & -*_{J^c}\\
    \beta^c & * & 0 & 0\\
    \beta^{\inrm} & * & -*_I & 0} \right]_{\Z[G]} \\
    &= (-1)^{ck+n_1 k + ak} \det \left[ \kbordermatrix{
      & \alpha^{\out}_{J^c} & \alpha^c & \alpha^{\inrm}_I \\
    \beta^{\out} & -*_{J^c} & * & 0\\
    \beta^c & 0 & * & 0\\
    \beta^{\inrm} & 0 & * & *_I} \right]_{\Z[G]}
    \end{align*}
    up to $I,J$-independent sign as in the proof of Theorem~\ref{thm:BSDAAlexanderZ}. Thus, the quantity \eqref{eq:TargetQuantityZ[G]} agrees with $\mathcal{A}^{\Z[G]}_{Y,\Gamma}\left(\overline{i}_* (-\overline{e}_I^{\inrm}) \wedge \overline{i}_* (-\overline{e}_{J^c}^{\out})\right)$ up to an overall sign that is independent of $I$ and $J$, as desired.
\end{proof}

Since we have related $[\BSDA(Y,\Gamma;\Xi_{\norm})]^{\Z[G]}_{\mathrm{comb}}$ to a quantity independent (up to multiplication by $\pm G$) of the set of choices $\Xi_{\norm}$, we are able to define the following invariant. 

\begin{definition}\label{def:BSDAIndependentZ[G]}
    Let $(Y,\Gamma)$ be a sutured cobordism.
    We define
    \[
        [\BSDA(Y,\Gamma)]^{\Z[G]}_{\comb}
    \]
    to be the map
    \[
        [\BSDA(Y,\Gamma;\Xi_{\norm})]^{\Z[G]}_{\comb}
        \colon \wedge^*_{\Z[G]} H_1(\overline{F}_0,\overline{S^+}_0) \to \wedge^*_{\Z[G]} H_1(\overline{F}_1,\overline{S^+}_1)
    \]
    for any choice of $\Xi_{\norm}$ as in Corollary~\ref{cor:CanChooseXiNormalized}. By Theorem~\ref{thm:BSDAAlexanderZG}, this map is well-defined up to multiplication by $\pm G$. 
\end{definition}

\begin{remark}\label{rem:ZHAlexanderWhenH2Zero}
    Suppose instead that throughout this section we worked over $\Z[H]$, rather than $\Z[G]=\Z[H/\tors(H)]$, allowing for torsion in $H$. Recall that in this case the ring $\Z[H]$ is not an integral domain in general; this causes issues for the most naive possible generalization of Section~\ref{sec:AlexanderFunctorZG} as follows. As noted in Remark~\ref{rem:IntroWhyPassToQ[H]}, if $H_2(\widehat{Y},\widehat{R}^+) \neq 0$, we can no longer argue as in the proof of Theorem~\ref{thm:BSDAAlexanderZG} that $\Z[H]$-linear dependence forces all maximal minors to vanish over $\Z[H]$ (i.e. the determinant to vanish over $\Z[H]$), and so if an Alexander functor were possible to define in this setting, the desired main result would only when $H_2(\widehat{Y},\widehat{R}^+)=0$.
    
    Fortunately, it is also precisely when $H_2(\widehat{Y},\widehat{R}^+)=0$, that the presentation matrix $\widehat{M}_{\Hc}$ for $H_1(\widehat{Y},\widehat{R}^+)$ from Section~\ref{sec:CWCoveringSpaceGeneral} is injective, and under this restriction the Schanuel's-lemma-inspired arguments of Proposition~\ref{lem:InjectivePresentationMatrixZ[G]} go through. So one can in fact obtain a $\Z[H]$-valued Alexander function $\A_{Y,\Gamma}^{\Z[H]}$ well-defined up to multiplication by units in $\Z[H]$, together with an associated Alexander functor $\mathsf{A}_{\Z[H]}(Y,\Gamma)$ unique up to units in $\Z[H]$, which agrees with $[\BSDA(Y,\Gamma;\Xi_{\norm})]^{\Z[H]}_{\comb}$ up to units in $\Z[H]$. The familiar line of reasoning would also culminate in an invariance statement similar to Definition~\ref{def:BSDAIndependentZ[G]}, but this and all the preceding statements would carry the caveat that $H_2(\widehat{Y},\widehat{R}^+)$ must vanish for the underlying input cobordism $(Y,\Gamma)$. 
    
    This restrictive version is not the strongest theory one might hope for. An additional difficulty in working over $\Z[H]$ is that its units are generally poorly behaved; Remark~\ref{rem:Z[H]IndependenceAndHeegaardMoves} points to one possible approach to this which we do not pursue here. The passage to $\Q[H]$ below allows us to treat cobordisms $(Y,\Gamma)$ for which $H_2(\widehat{Y},\widehat{R}^+)\neq 0$, makes the vanishing behavior of the Alexander functor more precise in this setting (see Remark~\ref{rem:AlexanderQHVanishingCriteria}), and yields a more controlled unit ambiguity. These improvements ultimately rely on the structure theory of $\Q[H]$.
\end{remark}

\subsection{The Alexander functor over $\Q[H]$}

This section provides the most general comparison between the map constructed in Definition~\ref{def:BSDAIndOfRefSpinc} and a Florens--Massuyeau style Alexander functor. We continue to assume that $Y$ is connected. 

Following the theories of Reidemeister and Turaev torsion, a standard workaround for the failure of $\Z[H]$ to be a domain when $H$ has torsion is to instead work over $\Q[H] \supset \Z[H]$, which admits a canonical decomposition as a direct sum of domains. Assuming such a decomposition, we will define the Alexander function componentwise, and then ultimately assemble these components to define an Alexander functor unique up to multiplication by units in $\Q[H]$. This restores enough algebraic control for us to make a meaningful connection with $  [\BSDA(Y,\Gamma;\Xi_{\norm})]^{\Z[H]}_{\comb}$ in the case that $H_2(\widehat{Y},\widehat{R}^+) \neq 0$. 

Set $T \coloneq \tors(H)$, and fix a splitting of $H$ as $T \oplus G$ (explicitly, a homomorphism $s \colon H \to T$ such that $s|_T = \id_T$). As in the introduction, fix a set $(\chi_1, \ldots, \chi_m)$ of representatives modulo equivalence for characters $\chi_i \colon T \to \C^*$. These extend uniquely to ring homomorphisms $\chi_i:\Z[T]\to \C$ whose images are cyclotomic fields $\F_{\chi_i}$. By the universal property for group rings, each $\chi_i$ gives a unique ring homomorphism 
\begin{align*}
    \Z[T][G] &\to \F_{\chi_i}[G] \\
    \sum_{g \in G} r_g g &\mapsto \sum_{g \in G} \chi_i\left(r_g\right) g. 
\end{align*}
Composing this with the isomorphism $\Z[H] \cong \Z[T][G]$ sending $h$ to $s(h)[h]$, where $[h]$ is the equivalence class of $h$ in $G = H/T$, yields a canonical ring homomorphism
\[
\varphi_i \colon \Z[H] \to \F_{\chi_i}[G]
\]
that makes $\F_{\chi_i}[G]$ an $\Z[H]$-algebra. Note that different choices of splitting $H \cong T \oplus G$ will produce different algebra actions. As in the introduction, there is an isomorphism of $\Q$-algebras (and $\Z[H]$-modules)
\begin{equation}\label{eq:HDecompStateSpace}
     \Q[H] \xrightarrow{(\varphi_i)_i} \bigoplus_{i=1}^m \F_{\chi_i}[G].
\end{equation}

\begin{definition}
    Let 
    \[
    H_*(Y,R^+;\Q[H]) := H_*(C_*^{\sing}(\widehat{Y}, \widehat{R}^+) \otimes_{\Z[H]} \Q[H]).
    \]
    For each $i$, let
    \[
    H_*(Y,R^+;\F_{\chi_i}[G]) := H_*(C_*^{\sing}(\widehat{Y}, \widehat{R}^+) \otimes_{\Z[H]} \F_{\chi_i}[G]), 
    \]
    where the $\Z[H]$-module structure on $\F_{\chi_i}[G]$ is determined by a choice of splitting $H \cong T \oplus G$ (so $H_*(Y,R^+;\F_{\chi_i}[G])$ depends on this choice of splitting).
\end{definition}

Note that we have an identification $H_1(Y,R^+;\Q[H]) \cong \bigoplus_i H_1(Y,R^+;\F_{\chi_i}[G])$ of modules over $\Q[H] \cong \bigoplus_i \F_{\chi_i}[G]$; the $\Q[H]$ action on the left-hand side intertwines the $\bigoplus_i \F_{\chi_i}[G]$ action on the right-hand side via the isomorphism $(\varphi_i)_i$ from $\Q[H]$ to $\bigoplus_i \F_{\chi_i}[G]$. Thus we have an isomorphism 
\begin{equation}\label{eq:AlexFunctorQ[H]DomainIdent}
\wedge^d_{\Q[H]} H_1(Y,R^+;\Q[H]) \cong \bigoplus_i \wedge^d_{\F_{\chi_i}[G]} H_1(Y,R^+;\F_{\chi_i}[G])
\end{equation}
of modules over $\Q[H] \cong \bigoplus_i \F_{\chi_i}[G]$ with the same intertwining property.

Now fix a general (non-normalized) set of choices $\Xi$ for $(Y,\Gamma)$ as in Definition~\ref{def:BSDAZComb}, including an $\alpha$-$\alpha$ bordered Heegaard diagram $\Hc$ for $(Y,\Gamma)$. By Sections~\ref{sec:CWDecomp} and ~\ref{sec:CWCoveringSpaceGeneral}, we have CW pairs $(\mathfrak{Y},\mathfrak{R}^+)$ and $(\widehat{\mathfrak{Y}},\widehat{\mathfrak{R}}^+)$, and a relative cellular chain complex $C_*(\widehat{\mathfrak{Y}},\widehat{\mathfrak{R}}^+)$ viewed as a chain complex of $\Z[H]$--modules. Passing to twisted coefficients via
\[
C^{\cell}_*(\mathfrak{Y},\mathfrak{R}^+;\Q[H])
\coloneq
C^{\cell}_*(\widehat{\mathfrak{Y}},\widehat{\mathfrak{R}}^+) \otimes_{\Z[H]} \Q[H]
\]
and
\[
C^{\cell}_*(\mathfrak{Y},\mathfrak{R}^+;\F_{\chi_i}[G])
\coloneq
C^{\cell}_*(\widehat{\mathfrak{Y}},\widehat{\mathfrak{R}}^+) \otimes_{\Z[H]} \F_{\chi_i}[G],
\]
we have canonical identifications
\[
H_*(Y,R^+;\Q[H]) \cong H_*(C_*^{\cell}(\widehat{\mathfrak{Y}},\widehat{\mathfrak{R}}^+;\Q[H]))
\]
and 
\[
H_*(Y,R^+;\F_{\chi_i}[G]) \cong H_*(C_*^{\cell}(\widehat{\mathfrak{Y}},\widehat{\mathfrak{R}}^+;\F_{\chi_i}[G]))).
\]

Recall by Section~\ref{sec:CWCoveringSpaceGeneral}, a basis $\{\widehat{e}_j\}$ over $\Z[H]$ for the chain groups of $C_*^{\cell}(\widehat{\mathfrak{Y}},\widehat{\mathfrak{R}}^+)$ (free $\Z[H]$-modules) is obtained by fixing preferred lifts of the cells of $(\mathfrak{Y},\mathfrak{R}^+)$. We get bases $\{\widehat{e}_j \otimes 1_{\Q[H]}\}$ and $\{\widehat{e}_j \otimes 1_{\F_{\chi_i}[G]}\}$ for the chain groups of $C^{\cell}_*(\mathfrak{Y},\mathfrak{R}^+;\Q[H])$ and $C^{\cell}_*(\mathfrak{Y},\mathfrak{R}^+;\F_{\chi_i}[G])$ respectively.

Over $\Q[H]$ (resp. $\F_{\chi_i}[G]$ for each $1 \leq i \leq n$), with respect to such fixed bases $\{\widehat{e}_j \otimes 1_{\Q[H]}\}$ (resp. $\{\widehat{e}_j \otimes 1_{\F_{\chi_i}[G]}\}$), there is a matrix $\widehat{M}_{\Hc}^{\;\Q[H]}$ (resp. $\widehat{M}_{\Hc}^{\;\F_{\chi_i}[G]}$) for the cellular differentials
\[
\widehat{\partial}_2^{\;\Q[H]} \colon C^{\cell}_2(\mathfrak{Y},\mathfrak{R}^+;\Q[H]) \to C^{\cell}_1(\mathfrak{Y},\mathfrak{R}^+;\Q[H])
\]
and
\[
\widehat{\partial}_2^{\;\F_{\chi_i}[G]} \colon C^{\cell}_2(\mathfrak{Y},\mathfrak{R}^+;\F_{\chi_i}[G]) \to C^{\cell}_1(\mathfrak{Y},\mathfrak{R}^+;\F_{\chi_i}[G]).
\]

\begin{remark}\label{rem:ExplicitIdentForAlexFunction}
We have 
\[
H_1(Y,R^+;\Q[H]) \cong \mathrm{coker} \widehat{M}_{\Hc}^{\;\Q[H]}
\]
and 
\[
H_1(Y,R^+;\F_{\chi_i}[G]) \cong \mathrm{coker} \widehat{M}_{\Hc}^{\;\F_{\chi_i}[G]};
\]
using these identifications, we can understand the isomorphism of equation~\eqref{eq:AlexFunctorQ[H]DomainIdent} concretely as follows. A $\Q[H]$-linear combination of basis elements $\widehat{e}_j \otimes 1_{\Q[H]}$ of $C_1^{\cell}(\mathfrak{Y},\mathfrak{R}^+;\Q[H])$ gets sent to an element of the direct sum whose $i$-th component is the linear combination of basis elements $\widehat{e}_j \otimes 1_{\F_{\chi_i}[G]}$ of $C_1^{\cell}(\mathfrak{Y},\mathfrak{R}^+;\F_{\chi_i}[G])$ whose coefficients are obtained from the original $\Q[H]$ coefficients by applying the ring homomorphism $\varphi \colon \Q[H] \to \F_{\chi_i}[G]$.
\end{remark}

Now, for the definition of Alexander functions, $\widehat{M}_{\Hc}^{\;\F_{\chi_i}[G]}$ serves as a valid deficiency-$d$ presentation matrix for $H_1(Y,R^+;\F_{\chi_i}[G])$, which is injective when $H_2(Y,R^+;\F_{\chi_i}[G])=0$, since its kernel is identified with $H_2(Y,R^+;\F_{\chi_i}[G])$. We record this fact below and omit the proof in the familiar flow. 

\begin{lemma}\label{lem:InjectivePresentationMatrixFchii[G]}
    For each $1\leq i \leq m$, if $H_2(Y,R^+;\F_{\chi_i}[G])=0$, then $d \geq 0$ and the deficiency-$d$ presentation matrix $\widehat{M}_{\Hc}^{\;\F_{\chi_i}[G]}$ for $H_2(Y,R^+;\F_{\chi_i}[G])$ is injective. If $H_2(Y,R^+;\F_{\chi_i}[G]) \neq 0$, then $\widehat{M}_{\Hc}^{\;\F_{\chi_i}[G]}$ is never injective.
\end{lemma}

\begin{definition}\label{def:AlexanderFunctionFchii[G]}
    Let $d = -\chi(Y,R^+)$. Given a fixed splitting of $H$ as $T \oplus G$, for $1 \leq i \leq n$ define the \emph{$\F_{\chi_i}[G]$-valued Alexander function} 
    \[
    \A^{\F_{\chi_i}[G]}_{Y,\Gamma} \colon \wedge^d_{\F_{\chi_i}[G]} H_1(Y,R^+;\F_{\chi_i}[G]) \to \F_{\chi_i}[G],
    \]
    well-defined up to overall multiplication by elements of $(\F_{\chi_i}^* \cdot G) \subset \F_{\chi_i}[G]$, as follows. If $H_2(Y,R^+;\F_{\chi_i}[G])\neq 0$, define $ \A^{\F_{\chi_i}[G]}_{Y,\Gamma}$ to be the zero function. Otherwise, $H_1(Y,R^+;\F_{\chi_i}[G])$ admits a deficiency-$d$ presentation matrix $\widehat{M}$, and so $\mathcal{A}^{\F_{\chi_i}[G]}_{Y,\Gamma}(u_1 \wedge \cdots \wedge u_d)$ may be defined as follows. Choose any expression $\overline{u}_i$ of $u_i$ as a $\F_{\chi_i}[G]$-linear combination of the generators, form a square matrix from $\widehat{M}$ by adding new columns $\overline{u}_1, \ldots, \overline{u}_d$ to the right, and take the determinant, an element of $\F_{\chi_i}[G]$. We define $\A^{\F_{\chi_i}[G]}_{Y,\Gamma}(u_1\wedge \cdots \wedge u_d)$ to be this element of $\F_{\chi_i}[G]$.
\end{definition}

As in the previous cases, the Alexander function $\A^{\F_{\chi_i}[G]}_{Y,\Gamma}(u_1 \wedge \cdots \wedge u_d)$ does not depend on the choice of representatives $\overline{u}_i$ of $u_i$; the proof of Lemma~\ref{lem:AlexanderFunctionZWellDefined} applies unchanged with $\F_{\chi_i}[G]$ coefficients. 

\begin{proposition}\label{prop:AlexanderFunctionFchii[G]IndependentOfM}
    Given a fixed splitting of $H$ as $T \oplus G$, then up to overall multiplication by $\F_{\chi_i}^* \cdot G \subset \F_{\chi_i}[G]$, the Alexander function $\A^{\F_{\chi_i}[G]}_{Y,\Gamma}$ of Definition~\ref{def:AlexanderFunctionFchii[G]} only depends on the finitely-presented $\F_{\chi_i}[G]$-module $H_1(Y,R^+;\F_{\chi_i}[G])$ and the deficiency $d$, and does not depend on the specific choice of injective presentation matrix $\widehat{M}$. 
\end{proposition}

\begin{proof}
    The proof of Proposition~\ref{prop:AlexanderFunctionZ[G]IndependentOfM} largely carries over in this setting, with only minor modifications. For the steps ($\A^{\F_{\chi_i}[G]}_{Y,\Gamma,\widehat{M}} \sim \A^{\F_{\chi_i}[G]}_{Y,\Gamma,\widehat{M}_1}$) and ($\A^{\F_{\chi_i}[G]}_{Y,\Gamma,\widehat{M}_2} \sim \A^{\F_{\chi_i}[G]}_{Y,\Gamma,\widetilde{\widehat{M}}}$): we note that $\pm 1$ live in $\F_{\chi_i}^*$, since $\F_{\chi_i}$ is a subfield of $\Q$, so the column-swapped determinant comparisons hold up to an element $\pm 1 \cdot e_{G} \in \F_{\chi_i}^* \cdot G$. For the step ($\A^{\F_{\chi_i}[G]}_{Y,\Gamma,\widehat{M}_1} \sim \A^{\F_{\chi_i}[G]}_{Y,\Gamma,\widehat{M}_2}$), we have 
    \[
    \det[M_1|\overline{u}_1 \cdots \overline{u}_d]=u \cdot \det [M_2|\overline{u}_1 \cdots \overline{u}_d]
    \]
    with $u \in (\F_{\chi_i}[G])^{\times}$. Since the group of units in $\F_{\chi_i}[G]$ is $\F_{\chi_i}^* \cdot G$, the Alexander functions differ by multiplication by an element of $\F_{\chi_i}^* \cdot G$ that is independent of $u_1 \wedge \cdots \wedge u_d$.
\end{proof}

Recall that for a commutative product ring $R=\prod_i^n R_i$, the units form a subgroup $(R)^{\times}=(\prod_i^n R_i)^{\times}=\prod_i^n (R_i)^{\times}$. That is, units of $R$ are tuples $(u_1,\dots,u_n)$ such that each $u_i$ is a unit in $R_i$. Under the identification of equation~\eqref{eq:HDecompStateSpace}, this implies that the group of units in $\Q[H]=\bigoplus_{i=1}^m \F_{\chi_i}[G]$ is 
\[
(\Q[H])^{\times}=\left(\prod_{i=1}^m \F_{\chi_i}[G]\right)^{\times}=\prod_{i=1}^m (\F_{\chi_i}[G])^{\times}=\prod_{i=1}^m \F_{\chi_i}^* \cdot G
\]
That is, they are tuples $(u_1,\dots,u_m)$ with each $u_i \in \F_{\chi_i}^* \cdot G$. 

\begin{definition}\label{def:AlexFunctionQ[H]}
    Define the \emph{$\Q[H]$-valued Alexander function}
    \[
    \A^{\Q[H]}_{Y,\Gamma} \colon \wedge^d_{\Q[H]} H_1(Y, R^+; \Q[H]) \to \Q[H],
    \]
    well-defined up to (global) multiplication by units in $\Q[H]$, to be the composition
    \begin{align*}
    \wedge^d_{\Q[H]} H_1(Y, R^+; \Q[H]) &\xrightarrow{\cong} \bigoplus_i \wedge^d_{\F_{\chi_i}[G]} H_1(Y, R^+; \F_{\chi_i}[G]) \\
    &\xrightarrow{\bigoplus_i \A^{\F_{\chi_i}[G]}_{Y,\Gamma}} \bigoplus_i \F_{\chi_i}[G]\\
    &\xrightarrow{\cong} \Q[H]
    \end{align*}
    where all arrows are defined using the same choice of splitting $H \cong T \oplus G$.
\end{definition}

\begin{remark}\label{rem:AlexanderQHVanishingCriteria}
    As anticipated by Remark~\ref{rem:ZHAlexanderWhenH2Zero}, this definition shows that the vanishing behavior of $\A^{\Q[H]}_{Y,\Gamma}$ is determined componentwise: whenever $H_2(Y,R^+;\F_{\chi_i}[G])\neq 0$ for some $i$, the $i^{th}$ component function of $\A^{\Q[H]}_{Y,\Gamma}$ is zero; only when $H_2(Y,R^+;\F_{\chi_j}[G])= 0$ for some $j$ does an injective deficiency-$d$ presentation matrix exist for that $j$, allowing the $j^{th}$ component function of $\A^{\Q[H]}_{Y,\Gamma}$ to act non-trivially. That is, $\A^{\Q[H]}_{Y,\Gamma}$ vanishes automatically if and only if $H_2(Y,R^+;\F_{\chi_i}[G])$ is nonzero \emph{for all} $1 \leq i \leq n$. A similar vanishing criterion holds for the Alexander functor over $\Q[H]$ defined below; we use this criterion in the proof of Theorem~\ref{thm:BSDAAlexanderQ[H]}.
\end{remark}

\begin{proposition}\label{prop:AlexanderQ[H]IndOfSplitting}
    Choose a different splitting $H \cong T \oplus G$, and in turn, fix new identifications in \eqref{eq:HDecompStateSpace}. Denote by $(\A^{\Q[H]}_{Y,\Gamma})'$ the function defined above with respect to this new splitting. Then, as maps up to multiplication by units in $\Q[H]$, we have
    \[
    \A^{\Q[H]}_{Y,\Gamma}=\left(\A^{\Q[H]}_{Y,\Gamma}\right)'.
    \]    
\end{proposition}

\begin{proof}
    The splittings are determined by homomorphisms $s, s' \colon H \to T$ such that $s|_T = \id_T$ and $s'|_T = \id_T$. The ring homomorphism $\varphi_i \colon \Z[H] \to \F_{\chi_i}[G]$ determined by $s$ sends a basis element $h$ of $\Z[H]$ to $(\chi_i(s(h)))[h] \in \F_{\chi_i}[G]$, where $[h]$ denotes the equivalence class of $h$ in $G = H / T$. Similarly, the ring homomorphism $\varphi'_i \colon \Z[H] \to \F_{\chi_i}[G]$ determined by $s'$ sends $h$ to $(\chi_i(s'(h)))[h] \in \F_{\chi_i}[G]$.

    Writing $H$ multiplicatively, the map $s's^{-1} \colon H \to T$ vanishes on $T \subset H$, so it descends to a map $u \coloneq s's^{-1} \colon G \to T$. In other words, for $h \in H$ we have $u([h]) = s'(h) s(h)^{-1}$. We have
    \begin{align*}
    \varphi_i'(h) &= (\chi_i(s'(h))[h] \\
    &= (\chi_i(u([h])s(h)))[h] \\
    &= (\chi_i \circ u)([h]) \chi_i(s(h))[h] \\
    &= (\chi_i \circ u)([h]) \varphi_i(h).
    \end{align*}
    Define $\rho_i \colon \F_{\chi_i}[G] \to \F_{\chi_i}[G]$ by, for a basis element $g$ of $\F_{\chi_i}[G]$, setting $\rho_i(g) = (\chi_i \circ u)(g)g$. We have
    \begin{align*}
        \rho_i(g_1g_2) &= (\chi_i \circ u)(g_1 g_2) g_1 g_2 \\
        &= (\chi_i(u(g_1)))(\chi_i(u(g_2))) g_1 g_2 \\
        &=\rho_i(g_1) \rho_i(g_2),
    \end{align*}
    so $\rho_i$ is a ring homomorphism; the second step uses that $u \colon G \to T$ and $\chi_i \colon T \to \F_{\chi_i}^*$ are group homomorphisms. In fact, $\rho_i$ is a ring isomorphism from $\F_{\chi_i}[G]$ to itself; its inverse is $\rho_i^{-1}(g) = ((\chi_i \circ u)(g))^{-1} g$. We also have
    \begin{align*}
        (\rho_i \circ \varphi_i)(h) &= \rho_i((\chi_i(s(h)))[h]) \\
        &= (\chi_i(s(h)))\rho_i([h]) \\
        &=(\chi_i(s(h))) (\chi_i \circ u)([h])[h] \\
        &=(\chi_i \circ u)([h]) \varphi_i(h) \\
        &= \varphi'_i(h),
    \end{align*}
    so $\varphi'_i = \rho_i \circ \varphi_i$ and $\varphi_i = \rho_i^{-1} \circ \varphi'_i$.

    Now consider the CW pair $(\mathfrak{Y},\mathfrak{R}^+)$ associated to some set of choices $\Xi$. Fix lifts of the cells of $\mathfrak{Y} \setminus \mathfrak{R}^+$ to cells of $\widehat{\mathfrak{Y}} \setminus \widehat{\mathfrak{R}}^+$, so that we have a matrix $\widehat{M}_{\Hc}$ with entries in $\Z[H]$.

    We first show that $H_2(Y,R^+;\F_{\chi_i}[G]) = 0$ when $H_2(Y,R^+;\F_{\chi_i}[G])$ is defined using $s$ if and only if it is zero when defined using $s'$. Indeed, using $s$, we have $H_2(Y,R^+;\F_{\chi_i}[G]) = 0$ if and only if the matrix $\widehat{M}^{\;\F_{\chi_i}[G]}_{\Hc}$ defined via the homomorphism $\varphi_i$ has linearly independent columns (over $\F_{\chi_i}[G]$). Using $s$, we have $H_2(Y,R^+;\F_{\chi_i}[G]) = 0$ if and only if the matrix $\widehat{M}^{\;\F_{\chi_i}[G]}_{\Hc}$ defined via the homomorphism $\varphi'_i$ has linearly dependent columns. The matrix with entries determined by $\varphi_i$ has linearly independent columns if and only if the matrix with entries determined by $\varphi'_i = \rho_i \circ \varphi_i$ has linearly dependent columns, since $\rho_i$ is a ring isomorphism from $\F_{\chi_i}[G]$ to itself.
        
    Now let $u_1 \wedge \cdots \wedge u_d \in \wedge^d_{\Q[H]} H_1(Y,R^+;\Q[H])$. Choose elements $\overline{u}_i$ in $C_1^{\cell}(\mathfrak{Y},\mathfrak{R}^+;\Q[H])$ representing $u_i \in H_1(Y,R+;\Q[H])$; each $\overline{u}_i$ is a $\Q[H]$-linear combination of basis elements $\widehat{e}_j \otimes 1_{\Q[H]}$. Applying the ring homomorphisms $\varphi_i$ or $\varphi'_i$ to the coefficients, we get $\F_{\chi_i}[G]$-linear combinations of basis elements $\widehat{e}_j \otimes 1_{\F_{\chi_i}[G]}$ of $C_1^{\cell}(\mathfrak{Y},\mathfrak{R}^+;\F_{\chi_i}[G])$, representing elements of $H_1(Y,R^+;\F_{\chi_i}[G])$. By Remark~\ref{rem:ExplicitIdentForAlexFunction} these elements are the components of the first map in Definition~\ref{def:AlexFunctionQ[H]} determined by $s$ or $s'$.

    We next claim that applying the first two maps of Definition~\ref{def:AlexFunctionQ[H]} for $s'$ to $u_1 \wedge \cdots \wedge u_d$ gives the same answer as applying the first two maps of Definition~\ref{def:AlexFunctionQ[H]} to $u_1 \wedge \cdots \wedge u_d$, then applying the ring isomorphism $\rho_i \colon \F_{\chi_i}[G] \to \F_{\chi_i}[G]$ to the $i^{th}$ component of the result. By the above discussion it suffices to consider $i$ such that $H_2(Y,R^+;\F_{\chi_i}[G]) = 0$. In this case, $\A^{\F_{\chi_i}[G]}_{Y,\Gamma}$ (computed using $s$) is the determinant of the $\Q[H]$-matrix $[(\widehat{M}_{\Hc} \otimes_{\Z[H]} \Q[H]) | \overline{u}_1 \cdots \overline{u}_d]$ after applying the ring homomorphism $\varphi_i$ to the entries, and $\A^{\F_{\chi_i}[G]}_{Y,\Gamma}$ (computed using $s'$) is the same except we apply $\varphi'_i$ rather than $\varphi$ to the entries. In this second case, we can apply $\varphi_i$ to the entries first, and then apply $\rho_i$, then finally take the determinant. Since $\rho_i$ is a ring homomorphism, this is the same as applying $\varphi_i$ to the entries, taking the determinant, and applying $\rho_i$ to the result, proving the claim.
    
    Now we analyze the last map of Definition~\ref{def:AlexFunctionQ[H]}. This map is the inverse of the isomorphism $(\varphi_i)_i$ or $(\varphi'_i)_i$ from $\Q[H]$ to $\bigoplus_i \F_{\chi_i}[G]$. For the isomorphism without the inverse, the $s'$-version is obtained from the $s$-version by postcomposing with 
    \[
    (\rho_i)_i \colon \bigoplus_i \F_{\chi_i}[G] \to \bigoplus_i \F_{\chi_i}[G].
    \]
    Thus, the $s'$-version of the inverse is obtained from the $s$-version of the inverse by precomposing with 
    \[
    (\rho_i^{-1})_i \colon \bigoplus_i \F_{\chi_i}[G] \to \bigoplus_i \F_{\chi_i}[G].
    \]
    It follows from the above claim that the composition of all three maps of Definition~\ref{def:AlexFunctionQ[H]} in the $s$ case agrees with the composition of all three maps in the $s'$ case as desired.
\end{proof}

Now we define an Alexander functor over $\Q[H]$. For the remainder of this section, suppose that we equip $(Y,\Gamma)$ with a set of normalized choices $\Xi_{\norm}$ as in Corollary~\ref{cor:NormalizedHeegaardDiagrams}. In this case, we are in the setting of Section~\ref{sec:CWCoveringSpaceNormalized}, and so we have modules $H_1(\widehat{F}_0,\widehat{S}^+_0)$ and $H_1(\widehat{F}_1,\widehat{S}^+_1)$ with preferred cellular bases $\widehat{e}^{\inrm}$ and $\widehat{e}^{\out}$ respectively. 

\begin{definition}
    Define
    \[
    H_*(F_i,S^+_i;\Q[H]) \coloneq  H_*(C_*^{\sing}(\widehat{F},\widehat{S}^+_i) \otimes_{\Z[H]} \Q[H]).
    \]
\end{definition}

Given our choices $\Xi_{\norm}$, if we let
\[
C_*^{\cell}(\mathfrak{F}_i,\mathfrak{S}^+_i;\Q[H]) := C_*^{\cell}(\widehat{\mathfrak{F}}_i, \widehat{\mathfrak{S}}^+_i) \otimes_{\Z[H]} \Q[H],
\]
then we have a canonical identification
\[
H_*(F_i,S^+_i;\Q[H]) \cong H_*(C_*^{\cell}(\mathfrak{F}_i,\mathfrak{S}^+_i;\Q[H]),
\]
Note that the differential in $C_*^{\cell}(\widehat{\mathfrak{F}}_i,\widehat{\mathfrak{S}}^+_i)$ is zero, so we have canonical identifications 
\[
H_*(F_i,S^+_i;\Q[H]) \cong H_*(C_*^{\cell}(\widehat{\mathfrak{F}}_i,\widehat{\mathfrak{S}}^+_i)) \otimes_{\Z[H]} \Q[H] \cong H_*(\widehat{F}_i, \widehat{S}^+_i) \otimes_{\Z[H]} \Q[H].
\]
If we choose lifts of the cells of $\mathfrak{Y} \setminus \mathfrak{R}^+$ to cells of $\widehat{\mathfrak{Y}} \setminus \widehat{\mathfrak{R}}^+$, we get (in particular) lifts of the cells of $\mathfrak{F}_i \setminus \mathfrak{S}^+_i$ to cells of $\widehat{\mathfrak{F}}_i \setminus \widehat{\mathfrak{S}}^+_i$ and thus identifications
\[
H_*(F_i,S^+_i;\Q[H]) \cong H_*(F_i,S^+_i) \otimes_{\Z} \Q[H].
\]

We write
\[
\widehat{i}_{\#} \colon C_1^{\sing}(\widehat{F},\widehat{S}^+) \to C_1^{\sing}(\widehat{Y}, \widehat{R}^+)
\]
for the $\Z[H]$-linear chain map induced by the inclusion of pairs $\widehat{i} \colon (\widehat{F},\widehat{S}^+) \to (\widehat{Y}, \widehat{R}^+)$. Tensoring with $\id_{\Q[H]}$ gives a map
\[
\widehat{i}_{\#} \otimes_{\Z[H]} \id_{\Q[H]} \colon C_1^{\sing}(F,S^+;\Q[H]) \to C_1^{\sing}(Y,R^+;\Q[H]).
\]
This induces a map on the level of homology given by
\[
\widehat{i}_* \coloneq \left(\widehat{i}_{\#} \otimes_{\Z[H]} \id_{\Q[H]}\right)_* \colon H_1(F,S^+;\Q[H]) \to H_1(Y,R^+;\Q[H]). 
\] 

\begin{proposition}\label{prop:AlexanderFunctorFchii[G]Defn}
    There exists a $\Q[H]$-linear map 
    \[
    \mathsf{A}_{\Q[H]}(Y,\Gamma) \colon \wedge^*_{\Q[H]} H_1(F_0,S^+_0;\Q[H]) \to \wedge^*_{\Q[H]} H_1(F_1,S^+_1;\Q[H])
    \]
    unique up to multiplication by units in $\Q[H]$ and homogeneous of degree $c = n_1 + \chi(Y,R^+)$, defined by
    \[
    \omega(\wedge(\mathsf{A}_{\Q[H]}(Y,\Gamma) \otimes \id)(x \otimes y)) = \A^{\Q[H]}_{Y,\Gamma}\left(\widehat{i}_* x \wedge \widehat{i}_* y \right),
    \]
    where $x \in \wedge^*_{\Q[H]} H_1(F_0,S^+_0;\Q[H])$, $y \in \wedge^*_{\Q[H]} H_1(F_1,S^+_1;\Q[H])$, the map
    \[
    \wedge \colon \wedge^p_{\Q[H]} H_1(F_1,S^+_1;\Q[H]) \otimes \wedge^q_{\Q[H]} H_1(F_1,S^+_1;\Q[H]) \to \wedge^{n_1}_{\Q[H]} H_1(F_1,S^+_1;\Q[H])
    \]
    sends $z \otimes w$ to $z \wedge w$ when $p+q = n_1$ and is zero otherwise, $\omega$ denotes an arbitrary volume form on the free $\Q[H]$-module $H_1(\widehat{F}_1,\widehat{S}^+_1) \otimes_{\Z[H]} \Q[H]$, the tensor product of maps is computed according to the super-sign rule
    \[
    (\mathsf{A}_{\Q[H]}(Y,\Gamma) \otimes \id)(x \otimes y)=(-1)^{|\mathsf{A}_{\Q[H]}(Y,\Gamma)||y|} \mathsf{A}_{\Q[H]}(Y,\Gamma)(x) \otimes y
    \]
    with $|\mathsf{A}_{\Q[H]}(Y,\Gamma)| = c$, and the Alexander functor $\A^{\Q[H]}_{Y,\Gamma}$ is defined to be zero on inputs in $\wedge^{dl} H_1(Y,R^+;\Q[H])$ for $d' \neq d$.
\end{proposition}

\begin{proof}
    After making the following adjustments, the proof of Proposition~\ref{prop:AlexanderFunctorZ[G]Defn} carries over to this setting: replace each $\overline{F}_i$, $\overline{S}_i^+$, $\overline{Y}$ and $\overline{R}^+$ with $\widehat{F}_i$, $\widehat{S}_i^+$, $\widehat{Y}$ and $\widehat{R}^+$ respectively, work with $\Q[H]$-coefficients, replace $\overline{i}_*$ by $\widehat{i}_*$, and replace the $\Z[G]$-basis elements $\overline{e}_I^{\inrm}$ for $\wedge^*_{\Z[G]} H_1(\overline{F}_0,\overline{S}^+_0)$ and $\overline{e}_J^{\out}$ for $\wedge^*_{\Z[G]} H_1(\overline{F}_1,\overline{S}^+_1)$ with our preferred $\Q[H]$-basis elements $\widehat{e}_I^{\inrm} \otimes 1_{\Q[H]}$ for $\wedge^*_{\Q[H]} H_1(F_0,S^+_0;\Q[H])$ and $\widehat{e}_J^{\out}\otimes 1_{\Q[H]}$ for $\wedge^*_{\Q[H]} H_1(F_1,S^+_1;\Q[H])$. Without loss of generality, we choose the volume form $\omega$ such that 
    \[
    \omega((\widehat{e}_1^{\out}\otimes 1) \wedge \cdots \wedge (\widehat{e}_{n_1}^{\out}\otimes 1)) = 1;
    \]
    any other choice of volume form will differ from $\omega$ by a unit in $\Q[H]$. The result follows. 
\end{proof}

We are almost ready to state our main theorem over $\Q[H]$. Let $[\BSDA(Y,\Gamma;\Xi_{\norm})]^{\Q[H]}_{\mathrm{comb}}$ indicate the extension of scalars for the map $[\BSDA(Y,\Gamma;\Xi_{\norm})]^{\Z[H]}_{\mathrm{comb}}$ defined in Definition~\ref{def:BSDAIndOfRefSpinc} along the canonical inclusion $\Z[H] \hookrightarrow \Q[H]$. That is,
\[
[\BSDA(Y,\Gamma;\Xi_{\norm})]^{\Z[H]}_{\mathrm{comb}} \otimes_{\Z[H]} 1_{\Q[H]} \coloneq [\BSDA(Y,\Gamma;\Xi_{\norm})]^{\Q[H]}_{\mathrm{comb}}. 
\]
This is a map 
\[
\left(\wedge^*_{\Z[H]} H_1(\widehat{F}_0, \widehat{S}^+_0)\right) \otimes_{\Z[H]} \Q[H] \to \left(\wedge^*_{\Z[H]} H_1(\widehat{F}_1, \widehat{S}^+_1)\right) \otimes_{\Z[H]} \Q[H],
\]
which we can view as a map
\[
\wedge^*_{\Q[H]} H_1(F_0,S^+_0;\Q[H]) \to \wedge^*_{\Q[H]} H_1(F_1,S^+_1;\Q[H])
\]
whose domain and codomain match those of $\mathsf{A}_{\Q[H]}(Y,\Gamma)$. 

\begin{theorem}\label{thm:BSDAAlexanderQ[H]}
    The map $[\BSDA(Y,\Gamma;\Xi_{\norm})]^{\Q[H]}_{\comb}$ described above agrees with $\mathsf{A}_{\Q[H]}(Y,\Gamma)$ up to multiplication by units in $\Q[H]$. 
\end{theorem}

\begin{proof}
   The proof is similar to Theorem~\ref{thm:BSDAAlexanderZG}. Fix lifts of the cells of $\mathfrak{Y} \setminus \mathfrak{R}^+$ to cells of $\widehat{\mathfrak{Y}} \setminus \widehat{\mathfrak{R}}^+$ as in the proof of Theorem~\ref{thm:BSDAAlexanderZG}, so that we have basis elements $\widehat{e}^{\inrm}_I$ of $\wedge^*_{\Z[H]} H_1(\widehat{F}_0,\widehat{S}^+_0)$ and $\widehat{e}^{\out}_J$ of $\wedge^*_{\Z[H]} H_1(\widehat{F}_1,\widehat{S}^+_1)$. Let $\widehat{e}^{\inrm;\Q[H]}_I$ and $\widehat{e}^{\out;\Q[H]}_J$ be the basis elements of $\wedge^*_{\Q[H]} H_1(F_0,S^+_0;\Q[H])$ and $\wedge^*_{\Q[H]} H_1(F_1,S^+_1;\Q[H])$ corresponding to $\widehat{e}^{\inrm}_I \otimes 1_{\Q[H]}$ and $\widehat{e}^{\out}_J \otimes 1_{\Q[H]}$ under the canonical isomorphisms \[
    \wedge^*_{\Q[H]} H_1(F_i,S^+_i;\Q[H]) \cong \left( \wedge^*_{\Z[H]} H_1(\widehat{F}_i,\widehat{S}^+_i) \right) \otimes_{\Z[H]} \Q[H].
    \]
    Let $\omega_{\Z}$, $\omega_{\Z[H];\mathrm{concrete}}$, and $\omega_{\Z[H]}$ be defined as in the proof of Theorem~\ref{thm:BSDAAlexanderZG}. Tensor $\omega_{\Z[H]}$ over $\Z[H]$ with $\Q[H]$ to get a volume form on the free $\Q[H]$-module 
    \[
    H_1(\widehat{F}_1,\widehat{S}^+_1) \otimes_{\Z[H]} \Q[H] \cong H_1(F_1,S^+_1;\Q[H]). 
    \]
    We use the resulting volume form on $H_1(F_1,S^+_1;\Q[H])$ as our choice of $\omega$.

    We will compute \begin{equation}\label{eq:TargetQuantityQ[H]}
        \omega(\wedge([\BSDA(Y,\Gamma;\Xi_{\norm})]^{\Q[H]}_{\comb} \otimes \id ))
    (-\widehat{e}_I^{\inrm;\Q[H]} \otimes -\widehat{e}^{\out;\Q[H]}_{J^c})
    \end{equation}
    and show the result agrees with $\mathcal{A}^{\Q[H]}_{Y,\Gamma}\left(\widehat{i}_* (-\widehat{e}_I^{\inrm;\Q[H]}) \wedge \widehat{i}_* (-\widehat{e}_{J^c}^{\out;\Q[H]})\right)$ up to an overall sign that is independent of $I$ and $J$. This is an equality of elements of $\Q[H]$; we will prove it by applying the isomorphism $\Q[H] \xrightarrow{(\varphi_i)_i} \bigoplus_i \F_{\chi_i}[G]$ to both sides, for any fixed choice of splitting $H = T \oplus G$, and checking that the $\F_{\chi_i}[G]$-components of the two sides agree for each $i$.
    
    As in the proof of Theorem~\ref{thm:BSDAAlexanderZG}, the quantity~\eqref{eq:TargetQuantityQ[H]} agrees with
    \[
    (-1)^{ck + n_1k + ak} \det \left[ \kbordermatrix{
      & \alpha^{\out}_{J^c} & \alpha^c & \alpha^{\inrm}_I \\
    \beta^{\out} & -*_{J^c} & * & 0\\
    \beta^c & 0 & * & 0\\
    \beta^{\inrm} & 0 & * & *_I} \right]_{\Q[H]} =: (-1)^{ck + n_1k + ak} \det \widehat{A}_{\Q[H]}
    \]
    up to $I,J$-independent sign. Let $\widehat{A}$ denote the matrix inside $[\cdot ]_{\Q[H]}$, so that $\widehat{A}_{\Q[H]} = [\widehat{A}]_{\Q[H]}$; the middle column of  blocks of $\widehat{A}$ (denoted $*$) form the matrix $\widehat{M}_{\Hc_{\norm}}$. Thus, if we apply $\varphi_i$ to the quantity \eqref{eq:TargetQuantityQ[H]}, we get $(-1)^{ck + n_1 k + ak} \det \widehat{A}_{\F_{\chi_i}[G]}$ where $\widehat{A}_{\F_{\chi_i}[G]} := [\widehat{A}]_{\F_{\chi_i}[G]}$ (obtained from $\widehat{A}$ by applying $\varphi_i$ to the entries).

    First consider $i$ such that $H_2(Y,R^+;\F_{\chi_i}[G]) \neq 0$; in this case 
    \[
    \varphi_i\left(\mathcal{A}^{\Q[H]}_{Y,\Gamma}\left(\widehat{i}_* (-\widehat{e}_I^{\inrm;\Q[H]}) \wedge \widehat{i}_* (-\widehat{e}_{J^c}^{\out;\Q[H]})\right)\right) = 0
    \]
    by definition, so we want to see that $\det \widehat{A}_{\F_{\chi_i}[G]} = 0$. The kernel of $\widehat{M}_{\Hc_{\norm}}^{\F_{\chi_i}[G]}$ is isomorphic to $H_2(Y,R^+;\F_{\chi_i}[G])$ which is assumed to be nonzero, so the columns of $\widehat{A}_{\F_{\chi_i}[G]}$ are linearly dependent over $\F_{\chi_i}[G]$. Since $\F_{\chi_i}[G]$ is a domain, $\det \widehat{A}_{\F_{\chi_i}[G]} = 0$ as desired.

    Now consider $i$ such that $H_2(Y,R^+;\F_{\chi_i}[G]) = 0$. In this case we can use the presentation matrix $\widehat{M}_{\Hc_{\norm}}^{\F_{\chi_i}[G]}$ to compute $\varphi_i\left(\A^{\Q[H]}_{Y,\Gamma}\left(\widehat{i}_* (-\widehat{e}^{\inrm;\Q[H]}_I) \wedge \widehat{i}_* (-\widehat{e}^{\out;\Q[H]}_{J^c})\right)\right)$. We get 
    \begin{align*}
        \varphi_i\left(\A^{\Q[H]}_{Y,\Gamma}\left(\widehat{i}_* (-\widehat{e}^{\inrm;\Q[H]}_I) \wedge \widehat{i}_* (-\widehat{e}^{\out;\Q[H]}_{J^c})\right)\right) &= \det \left[\kbordermatrix{
           & \alpha^c & \alpha^{\inrm}_I & \alpha^{\out}_{J^c} \\
        \beta^{\out}  & * & 0 & -*_{J^c}\\
        \beta^c & * & 0 & 0\\
        \beta^{\inrm} & * & -*_I & 0} \right]_{\F_{\chi_i}[G]} \\
        &= (-1)^{ck+n_1 k + ak} \det \left[ \kbordermatrix{
          & \alpha^{\out}_{J^c} & \alpha^c & \alpha^{\inrm}_I \\
        \beta^{\out} & -*_{J^c} & * & 0\\
        \beta^c & 0 & * & 0\\
        \beta^{\inrm} & 0 & * & *_I} \right]_{\F_{\chi_i[G]}},
    \end{align*}
    agreeing with the quantity \eqref{eq:TargetQuantityQ[H]} as desired.
\end{proof}

Since we have related $[\BSDA(Y,\Gamma;\Xi_{\norm})]^{\Q[H]}_{\mathrm{comb}}$ to a quantity independent (up to multiplication by units in $\Q[H]$) of the set of choices $\Xi_{\norm}$, we are able to define the following invariant. 

\begin{definition}
    Let $(Y,\Gamma)$ be a sutured cobordism
    We define
    \[
        [\BSDA(Y,\Gamma)]^{\Q[H]}_{\comb}
    \]
    to be the map
    \[
        [\BSDA(Y,\Gamma;\Xi_{\norm})]^{\Q[H]}_{\comb}
        \colon \wedge^*_{\Q[H]} H_1(F_0,S^+_0;\Q[H]) \to \wedge^*_{\Q[H]} H_1(F_1,S^+_1;\Q[H]).
    \]
    for any choice of $\Xi_{\norm}$ as in Corollary~\ref{cor:CanChooseXiNormalized}. By Theorem~\ref{thm:BSDAAlexanderQ[H]}, this map is well-defined up to multiplication by units in $\Q[H]$. 
\end{definition}

\bibliographystyle{alpha}
\bibliography{ref}

\end{document}